\documentclass[11pt]{amsart}
\usepackage{amssymb,amsmath}
\usepackage[all]{xy}
\usepackage{graphicx}
\usepackage{amscd}
\usepackage{graphicx}
\usepackage{tikz}
\usetikzlibrary{decorations.markings}
\usetikzlibrary{shapes}
\usetikzlibrary{backgrounds}
\usepackage{subfig}
\usepackage{pb-diagram}
\usepackage[all]{xy}
\usepackage{enumerate}
\usepackage{hyperref}

\usepackage{color}
\usepackage{caption}

\numberwithin{equation}{section}

\newtheorem{theorem}{Theorem}[section]
\newtheorem{corollary}[theorem]{Corollary}
\newtheorem{lemma}[theorem]{Lemma}

\newtheorem{prop}[theorem]{Proposition}

\theoremstyle{definition}
\newtheorem{definition}[theorem]{Definition}
\newtheorem{example}[theorem]{Example}
\newtheorem{remark}[theorem]{Remark}

\newcommand{\bP}{\mathbb{P}}
\newcommand{\der}{\mathrm{d}}

\newcommand{\area}{\ensuremath \mathrm{Area}}
\newcommand{\fc}{\ensuremath F^{\mathrm{cyc}}}
\newcommand\cut{\operatorname{cut}}
\newcommand\glue{\operatorname{glue}}
\newcommand\form[1]{\langle #1\rangle}
\newcommand{\PP}{\mathbb{P}}
\newcommand{\PO}{\mathbb{P}^1_{a,b,c}}
\newcommand{\CC}{\mathbb{C}}
\newcommand{\RR}{\mathbb{R}}

\newcommand{\HH}{\mathbb{H}^2}

\newcommand{\Z}{\mathbb{Z}}
\newcommand{\nat}{\mathbb{N}}

\newcommand{\AI}{A_\infty}

\newcommand{\CA}{\mathcal{A}}

\newcommand{\CS}{\mathcal{S}}

\newcommand{\WT}[1]{\widetilde{#1}}

\newcommand\cay{\operatorname{Cay}}
\newcommand{\Jac}{\mathrm{Jac}}

\title{Lagrangian Floer potential of orbifold spheres}

\begin{document}

\author[Cho]{Cheol-Hyun Cho}
\address{Department of Mathematical Sciences, Research institute of Mathematics\\ Seoul National University\\ Gwanak-ro 1\\ Gwanak-gu \\Seoul 151-747\\ Korea}
\email{chocheol@snu.ac.kr}
\author[Hong]{Hansol Hong}
\address{Department of Mathematical Sciences\\ Seoul National University\\ Gwanak-ro 1\\ Gwanak-gu \\Seoul 151-747\\ Korea}
\email{hansol84@snu.ac.kr}
\author[Kim]{Sang-hyun Kim}
\address{Department of Mathematical Sciences, Research institute of Mathematics\\ Seoul National University\\ Gwanak-ro 1\\ Gwanak-gu \\Seoul 151-747\\ Korea}
\email{s.kim@snu.ac.kr}
\author[Lau]{Siu-Cheong Lau}
\address{Department of Mathematics\\ Harvard University\\ One Oxford Street\\ Cambridge \\ MA 02138\\ USA}
\email{s.lau@math.harvard.edu}
\begin{abstract}
For each sphere with three orbifold points, we construct an algorithm to compute the open Gromov-Witten potential, which serves as the quantum-corrected Landau-Ginzburg mirror and is an infinite series in general.  This gives the first class of general-type geometries whose full potentials can be computed.  As a consequence we obtain an enumerative meaning of mirror maps for elliptic curve quotients.  Furthermore, we prove that the open Gromov-Witten potential is convergent, even in the general-type cases, and has an isolated singularity at the origin, which is an important ingredient of proving homological mirror symmetry.
\end{abstract}
\maketitle
\tableofcontents
\section{Introduction}
Mirror symmetry reveals deep relations between symplectic and complex geometry.  Closed-string mirror symmetry provides a powerful tool to enumerate rational curves by using complex deformation theory of the mirror, and open-string mirror symmetry builds a bridge between Lagrangian Floer theory in symplectic geometry and sheaf theory in complex geometry.  Mirror symmetry has brought many exciting results to geometry in the last two decades.


Let $\PO$ be an orbifold sphere with $a, b, c \geq 2$ equipped with the K\"ahler structure $\omega$ with constant curvature descended from its universal cover, which is a space-form (either the sphere, Euclidean plane, or hyperbolic plane).  The mirror of $\PO$ is a Landau-Ginzburg superpotential $W$, which is a holomorphic function defined over $\CC^3$.  Mirror symmetry between $\PO$ and $W$ has been intensively investigated \cite{S, ST, ET, Ef, R}.  On the other hand, the enumerative nature of the superpotential $W$ was less understood.  Moreover the `flat coordinate' near the large complex structure limit mirror to $\PO$ in the general-type case $\frac{1}{a} + \frac{1}{b} + \frac{1}{c} < 1$ was not investigated.

The aim of this paper is to compute the (quantum corrected) mirror superpotential $W$ of an orbifold sphere $\PO$ by appropriate counts of polygons for all $a,b,c  \geq 2$.\footnote{$\PO$ is toric when $c=1$, and the Lagrangian Floer potential in this case was known by \cite{CP}.}  $W$ will also be called the Lagrangian Floer potential or open Gromov-Witten potential, when we want to emphasize the enumerative nature of $W$ by counting polygons bounded by a certain Lagrangian immersion, which gives open Gromov-Witten invariants.  It gives an expression of the superpotential $W$ in terms of the flat coordinate.

The open Gromov-Witten potential $W$ takes the form
\begin{equation} \label{eq:W_intro}
-q^6 xyz + q^{3a} x^a + q^{3b} y^b + q^{3c} z^c + \cdots
\end{equation}
where $q$ is related to the K\"ahler flat coordinate $Q = \exp (-t)$ by the equality $Q = q^8$, and $t$ is the area of $\PO$ with respect to $\omega$.

The explicit expression of $W$ in the case $(a,b,c) = (3,3,3)$ was computed in \cite{CHL}.  For this specific case the polygons contributing to $W$ are all triangles, and so the computation is relatively easy.  In general the holomorphic polygons contributing to $W$ are rather complicated, and the number of corners of those polygons contributing to $W$ goes to infinity.  In this paper we classify all holomorphic polygons combinatorially and invented an algorithm to compute $W$.  As a consequence, we obtain an explicit expression of $W$ in the other two elliptic curve quotients where $(a,b,c)=(2,4,4)$ and $(2,3,6)$, and give a computational check that it agrees with the inverse mirror maps.

When $\frac{1}{a} + \frac{1}{b} + \frac{1}{c} < 1$, the orbifold sphere $\bP^1_{a,b,c}$ comes from quotient of a Riemann surface with genus greater than one.  It is known as the general-type case, in which the superpotential is an infinite series.  To the authors' knowledge, this is the first class of compact general-type manifolds or orbifolds\footnote{$\PO$ and its covering Riemann surface $\Sigma$ have the same expression of mirror superpotential (defined over the domains $\CC^3$ and $\CC^3 / G$ for a finite group $G$ respectively).  Thus $W$ can be considered as mirror to both $\Sigma$ and $\bP^1_{a,b,c}$.} that one has an algorithm to compute the full open Gromov-Witten potential.

In the toric case, by \cite{FOOO_MS} quantum cohomology of a toric manifold $X$ has a presentation in terms of its (quantum-corrected) Landau-Ginzburg superpotential $W$ via a natural Kodaira-Spencer map
$$ \mathfrak{ks}: QH^*(X) \stackrel{\cong}{\longrightarrow} \Jac (W). $$
Such a Kodaira-Spencer map can also be constructed for an orbifold sphere $X = \PO$, and we expect the above isomorphism should hold.  It is known as closed-string mirror symmetry.\footnote{The idea of construction and proof should be rather similar to that in \cite{FOOO_MS}, together with the fact that the Fukaya category of $X$ is generated by the Seidel Lagrangian shown in Figure \ref{fig:lagintro}.  However we are not going to prove it in this paper because it involves many analytic details.}  Hence our computation of $W$ should lead to an explicit presentation of the quantum cohomology ring.    



In the following we explain our ideas and results in more detail.

\subsection{Combinatorial geometry of the Seidel Lagrangian}
In \cite{CHL}, a geometric construction of mirror geometry based on (immersed) Lagrangian Floer theory was proposed.   It automatically comes with a functor transforming Lagrangian branes to matrix factorizations, which realizes the homological mirror symmetry conjecture.  The construction was applied to orbifold spheres $\PO$ to derive homological mirror symmetry, using the immersed Lagrangian defined by Seidel \cite{S} (see Figure \ref{fig:lagintro}).  Moreover for the elliptic curve quotient $E / \Z_3$, the construction leads to an enumerative meaning of the mirror map.

\begin{figure}[htb!]
\includegraphics[scale = 0.4]{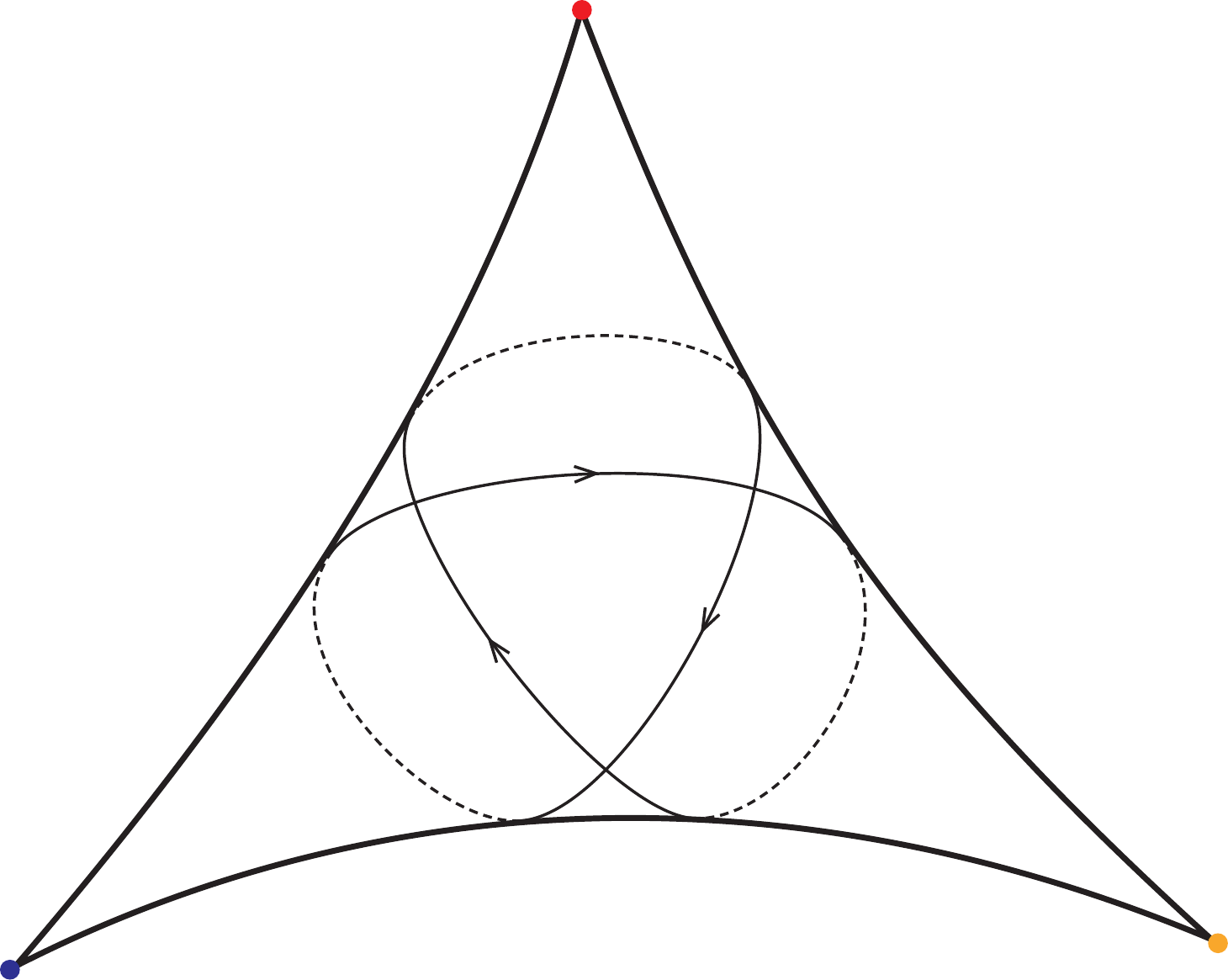}
\caption{$\mathbb{P}^1_{a,b,c}$ and the Seidel Lagrangian.}\label{fig:lagintro}
\end{figure}

The superpotential $W$ in Equation \eqref{eq:W_intro} is obtained by counting polygons bounded by the Seidel Lagrangian shown in Figure \ref{fig:lagintro}.  In order to classify all possible holomorphic polygons, we do a careful treatment to the related combinatorial geometries in the orbifold universal cover of $\PO$.

The orbifold universal cover $E$ of $\PO$ is $S^2, \RR^2, \HH$  when the orbifold Euler characteristic $\chi = \frac{1}{a} + \frac{1}{b} + \frac{1}{c} -1$ is positive, zero, or negative respectively (see Lemma \ref{lem:orbcov}). We refer them as spherical, elliptic, or hyperbolic case.

We introduce two tessellations of $E$. The first tessellation is given by the pre-image of the horizontal equator of $\PO$.  It is a triangle tessellation formed by geodesics in $E$, and we can label the triangles in either black or white such that every two adjacent triangles have different colors.  Every triangle of the tessellation is called a \emph{base triangle}.  The other one is a hexagon tessellation $Z$ of $E$; see Definition \ref{def:hexZ}.  The vetices of $Z$ are preimages of the orbifold points $A,B,C$ of $\PO$ together with centers of white base triangles which are called white vertices.  Every edge of $Z$ has exactly one white vertex, and it is labelled as $\alpha, \beta, \gamma$ when its another vertex is $A,B,C$ respectively.  The hexagon tessellation $Z$ plays a key role in the classification of polygons.  Examples of the triangle and hexagon tessellations are shown in Figure \ref{fig:444basic},  \ref{fig:ygraph}, \ref{fig:236tes}, \ref{fig:tess225}, \ref{fig:abc225}.

Now consider the pre-image of Seidel Lagrangian $L$ in $E$.  Combinatorially it is a union of boundaries of the so-called \emph{minimal $xyz$ triangles}.  There is a one-to-one correspondence between minimal $xyz$ triangles and base triangles: each minimal $xyz$ triangle is contained in one base triangle.  The vertices of a minimal $xyz$ triangle are mid-points of the three edges of the base triangle containing it.  We do the construction such that the four triangles obtained from this subdivision of a base triangle have the same area.

The pre-image of Seidel Lagrangian consists of several branches.  Each of the branches is referred as {\em a Seidel Lagrangian
in $E$}.  In the hyperbolic and elliptic cases, we show that a Seidel Lagrangian in $E$ is always embedded (instead of immersed) in $E$  (Corollary \ref{cor:bigon_no}).
In the spherical case, a Seidel Lagrangian in $E$ is immersed if and only if $(a,b,c) = (2,2,r)$ for an odd number $r$; otherwise a Seidel Lagrangian in $E$ is a circle which bisects the sphere $E$ (Lemma \ref{lem:sphdiss}).

The Seidel Lagrangian has three self-intersection points.  Each self-intersection point corresponds to a degree-one Lagrangian deformation, and it is called an immersed generator associated to a variable $x,y$ and $z$ respectively.  The immersed generators are lifted upstairs to $E$ and give $G$-equivariant deformations, where $\PO = E / G$ for a finite group $G$.

\subsection{Algorithm to count polygons}

To compute the superpotential $W$, we need to count polygons immersed in $\PO$ bounded by the Seidel Lagrangian $L$, where corners of polygons are mapped to the immersed generators $x,y$ or $z$ (see Figure \ref{fig:seidel_lag} for the minimal $xyz$ triangle as an example).

By taking the lifts of $L$ and polygons to the universal cover $E$, we show that it suffices to count embedded polygons in $E$ (Lemma \ref{lem:lift}). A polygon in $E$ which contributes to $W$ will be called a \emph{polygon for the potential}.

Given a polygon for the potential, we consider the intersection of the polygon with the hexagon tessellation $Z$, which gives what we call an $(a,b,c)$-diagram.  We find out a characterization of all $(a,b,c)$-diagrams (Lemma \ref{lem:diagram}), 
and thereby classify all polygons for the potential combinatorially.  Moreover, area formulas for polygons are deduced in terms of combinatorial data of $(a,b,c)$-diagrams (Theorem \ref{thm:area}).

Another combinatorial object we use is \emph{boundary word} of a polygon for the potential.  It is obtained by recording the intersection pattern between the tessellation $Z$ and the boundary of a polygon using three letters $\alpha, \beta, \gamma$.  We deduce its characterizing properties in Lemma \ref{lem:standard}. 

Next, we find an algorithm to generate all possible polygons for the potential in hyperbolic cases, by the so-called elementary move
(Definition \ref{defn:move}, \ref{defn:move}). Such an algorithm is needed since there are infinitely many different types of polygons for the potential in the hyperbolic cases. Roughly speaking, elementary move is a systematic way to enlarge a polygon
for the potential. We define elementary move in terms of the corresponding $(a,b,c)$-diagram.  Correspondingly we have the so-called {\em cut-glue} operation (Section~\ref{s:hyp}) on boundary words. If one of $a,b,c$ equals $2$, then the elementary move and cut-glue operation become more complicated and they are defined separately.

In Theorem \ref{thm:move}, we show that elementary moves generate all polygons for the potential (or all $(a,b,c)$-diagram) from the minimal $xyz$-triangle.  We associate a natural number $p$ to every polygon for the potential and show that each polygon is obtained by applying $p+1$ elementary moves from the minimal $xyz$-triangle in Lemma \ref{lem:pcount}.

In summary, we construct an algorithm to compute the potential in the hyperbolic case.
In Section \ref{ss:eval}, we apply the algorithm in several examples to compute $W$ inductively.  The examples are computed by running our algorithm in Mathematica. A Mathematica code for the computation of $W$ in the case of $a,b,c \geq 3$ can
be found in \cite{Ch}.

\subsection{Enumerative meaning of mirror maps of elliptic curve quotients}

The mirror map matches the flat coordinates on the K\"ahler moduli and the complex moduli of the mirror and is essential in the computation of Gromov-Witten invariants using mirror symmetry.  It arises from classical deformation theory and can be obtained from solving Picard-Fuchs equations in the B side.

Enumerative meaning of the mirror map\footnote{The mirror map we mention here is a map from the K\"ahler moduli to the complex moduli of the mirror, which is indeed the ``inverse mirror map'' to in the conventions of most existing literatures.  For simplicity we call it to be the mirror map in the rest of this introduction.} was obtained in the compact semi-Fano toric case and toric Calabi-Yau case \cite{CLT11,CLLT,CLLT12,CCLT13}.  It was shown that the mirror map equals to the so-called SYZ map, which arises from SYZ construction and is written in terms of disc invariants.

We expect that such a derivation of the mirror map from the A-side should hold in great generality.  In particular, it was shown in \cite{CHL} that for the elliptic curve quotient $E / \Z_3 = \bP^1_{(3,3,3)}$, the mirror map $\check{q}(Q)$ equals to $\frac{-\psi(q)}{\phi(q)}$ under the relation $Q = q^8$, where $Q$ is the K\"ahler parameter of $E/\Z_3$, $\psi(q)$ and $\phi(q)$ are generating functions of triangle countings where the three corners of the triangles are mapped to the immersed generators $x,y,z$ and $x,x,x$ respectively, and $q = \exp(-A)$ with $A$ being the area of the minimal $xyz$ triangle.

The mirror of the elliptic curve $E$ is again an elliptic curve, which can be written as a cubic
$$ x^3 + y^3 + z^3 + \check{q} xyz $$
in $\bP^2$.  The mirror map $\check{q}(\tilde{Q})$ of the elliptic curve $E$ is the same as that of $E/\Z_3$, under the relation $\tilde{Q} = Q^3$ between their K\"ahler parameters.  The relations $Q = q^8$ and $\tilde{Q} = Q^3$ come from the geometric facts that $E/\Z_3$ has area eight times the minimal $xyz$ triangle, and $E$ has area three times $E/\Z_3$.  Thus the mirror map of an elliptic curve also has an enumerative meaning.

In this paper we study the other two elliptic curve quotients $E/\Z_4 = \bP^1_{(2,4,4)}$ and $E/\Z_6= \bP^1_{(2,3,6)}$ and obtain an enumerative meaning of the mirror map of $E/\Z_4$ and $E/\Z_6$ (and hence enumerative meanings of the elliptic curve $E$,  whose mirror is expressed as curves in the weighted projective space $\bP^2(2,4,4)$ and $\bP^2(2,3,6)$ respectively).  $(a,b,c) = (3,3,3), (2,4,4), (2,3,6)$ are all the cases such that $\bP^1_{(a,b,c)}$ is a quotient of an elliptic curve.

We computed the explicit expressions of the superpotentials $W$ for $\bP^1_{(2,4,4)}$ and $\bP^1_{(2,3,6)}$ in Section \ref{sec:236}, \ref{sec:244} respectively.  They are weighted homogeneous polynomials whose coefficients are generating functions of polygon countings.

\begin{theorem}[Theorem \ref{thm:W236}]
The superpotential $W$ for $(a,b,c)= (2,3,6)$ is
$$W= -q xyz  + q^6 x^2 + c_y(q) y^3 + c_z(q) z^6+ c_{yz2}(q) y^2z^2 + c_{yz4} yz^4.
$$
See Theorem \ref{thm:W236} for the explicit expressions of the coefficients $c_y(q), c_z(q), c_{yz2}(q), c_{yz4}(q)$.
\end{theorem}

\begin{theorem}[Theorem \ref{thm:W244}]
The superpotential $W$ for $(a,b,c)= (2,4,4)$ is
$$W= -q xyz + q^6 x^2 + d_y(q) (y^4 +  z^4) + d_{yz}(q) y^2z^2,
$$
See Theorem \ref{thm:W244} for the explicit expressions of the coefficients  $d_y(q)$ and $d_{yz}(q)$.
\end{theorem}

By doing coordinates changes the superpotentials $W$ for $(a,b,c)= (2,3,6)$ and $(a,b,c)= (2,4,4)$ can be expressed in the forms
$$W= x^2 + y^3 + z^6 + \sigma_{236}(q) yz^4
$$
and
$$W= x^2 + y^4 + z^4 + \sigma_{244}(q) y^2z^2
$$
respectively.  See Equation \eqref{eq:sigma236} and \eqref{eq:sigma244} for the explicit expressions of $\sigma_{236}$ and $\sigma_{244}$ in terms of polygon countings.  We do computer checks that
\begin{align*}
i_{236} \left( \sigma_{236}(q) \right) =& \frac{1}{q^{48}} + 744 + 196884 q^{48} + 21493760 q^{96} + 864299970 q^{144} \\ &+ 20245856256 q^{192} + 333202640600 q^{240} + \ldots
\end{align*}
and
\begin{align*}
i_{244} \left(\sigma_{244}(q)  \right) =& \frac{1}{q^{32}} + 744 + 196884 q^{32} + 21493760 q^{64} + 864299970 q^{96} \\ &+ 20245856256 q^{128} + 333202640600 q^{160} + \ldots
\end{align*}
which coincide with the $j$-function for the elliptic curve (under the relation $\tilde{Q} = q^{48}$ and $\tilde{Q} = q^{32}$ respectively, due to the geometric facts that $E$ has area $48$ times ($32$ times) the minimal $xyz$ triangle in $E/\Z_6$ ($E/\Z_4$ resp.)).  These provide evidences that $\sigma_{236}$ and $\sigma_{244}$ are the mirror maps.

We have also computed the full superpotential for the spherical cases $\frac{1}{a} + \frac{1}{b} + \frac{1}{c} > 1$, which can be further divided into type $A$, $D$ and $E$.

The $A_n$ case is $c=1$.   The orbifold sphere $\bP^1_{a,b,1}$ is toric, and the open Gromov-Witten potential (taking $L$ to be the Lagrangian torus fibers of the moment map) is known by \cite{CP}.  Namely,
$$ W = z^a + \frac{Q}{z^b} $$
over $H^2(\bP^1_{a,b,1})$ which is one-dimensional, and $Q$ is the flat coordinate of $H^2(\bP^1_{a,b,1})$.

The $D_n$ case is $(a,b,c) = (2,2,n)$ for $n \in \nat$, and the $E_n$ case is $(a,b,c) = (2,2,n-3)$ for $n=6,7,8$ respectively.

\begin{theorem}[Theorem \ref{thm:W22r}, \ref{thm:W233}, \ref{thm:W234}, \ref{thm:W235}]
The Lagrangian-Floer potential for the spherical cases $\frac{1}{a} + \frac{1}{b} + \frac{1}{c} > 1$ are given as follows.
\begin{enumerate}
\item For $\mathbb{P}^1_{2,2,r}$,
\begin{equation} \label{eq:W_D}
W = -qxyz + q^6 x^2 + q^6 y^2 + (-1)^r q^{3r} z^r + r\sum_{k=1}^{\lfloor \frac{r}{2} \rfloor} (-1)^{r+k} q^{3r + 10k}  \frac{(r-k-1)!}{k! (r-2k)!} z^{r-2k}.
\end{equation}
\item For $\mathbb{P}^1_{2,3,3}$,
\begin{equation}
W = -qxyz + q^6 x^2 - q^9 y^3 - q^9 z^3 -4 q^{22} yz +2 q^{48}.
\end{equation}
\item For $\mathbb{P}^1_{2,3,4}$,
\begin{equation}
W = -qxyz + q^6 x^2 - q^9 y^3 + q^{12} z^4  + 5q^{25}yz^2 + 3 q^{54} z^2 +3 q^{38} y^2 - 2 q^{96}.
\end{equation}
\item For $\mathbb{P}^1_{2,3,5}$, 
\begin{equation}
\begin{array}{c}
W = -qxyz + q^6 x^2 - q^9 y^3 - q^{15} z^5  + 4 q^{60} z^4  -3q^{105} z^3 + 5 q^{150} z^2 \\
 -6 q^{28} yz^3 -9 q^{73} yz^2  -7 q^{41} y^2 z+ 5 q^{86} y^2  -2 q^{240}.
\end{array}
\end{equation}
\end{enumerate}
\end{theorem}

The flat coordinates in the spherical case was computed by Rossi \cite{R} by analyzing Frobenius structure on the moduli.  We can check that Equation \eqref{eq:W_D} matches with the calculation of \cite{R}.  Thus in the $D_n$ case, we also obtain an enumerative meaning of mirror map.

\subsection{Convergence and critical locus of the superpotential}

The Lagrangian Floer potential $W$ a priori is an infinite series, and so there is an issue whether $W$ is convergent or not.

In the spherical case $\frac{1}{a} + \frac{1}{b} + \frac{1}{c} > 1$, there are just finitely many polygons, and so $W$ is a polynomial whose coefficients are also polynomials in $q$.

In the elliptic case $\frac{1}{a} + \frac{1}{b} + \frac{1}{c} = 1$, $W$ is a polynomial whose coefficients are infinite series in $q$, see Theorem \ref{thm:W236}, Theorem \ref{thm:W244} and Theorem 6.4 of the paper \cite{CHL}.  It is easy to verify by ratio test that all of them are convergent series.

For the hyperbolic case $\frac{1}{a} + \frac{1}{b} + \frac{1}{c} < 1$, $W$ is itself an infinite series.  It is not obvious whether $W$ is  convergent or not, since there is no full explicit expression of $W$.  In this paper we prove that $W$ is still convergent.  To the authors' knowledge this is the first class of general-type manifolds or orbifolds whose superpotentials in flat coordinates are known to be convergent.

\begin{theorem}[Theorem \ref{thm:conv}] \label{thm:conv_intro}
For the hyperbolic case $\frac{1}{a} + \frac{1}{b} + \frac{1}{c} < 1$, $W(x,y,z,q)$ is complex analytic for $|x|,|y|,|z|,|q|$ small enough.
\end{theorem}

The proof of the above theorem uses the combinatorial correspondence between polygons for the potential and $(a,b,c)$ diagrams established in this paper.  Using Theorem \ref{thm:conv_intro} we prove that $W$ has an isolated singularity at the origin for the elliptic and hyperbolic cases, which is an important ingredient for proving homological mirror symmetry for $\PO$ in \cite{CHL}.

\begin{theorem}[Theorem \ref{thm:mainapply}]
Suppose $\frac{1}{a} + \frac{1}{b} + \frac{1}{c} =1$, and consider the corresponding potential $W$.  Then, $W$ has a unique singular point at the origin.
Suppose $\frac{1}{a} + \frac{1}{b} + \frac{1}{c} <1$, and consider the corresponding potential $W$. The hypersurface $W^{-1} (0)$ has an isolated singularity at the origin.
\end{theorem}

\subsection*{Acknowledgement}
The second author thanks Kyoung-Seog Lee for valuable discussions on the computation of Jacobian ideals.
The fourth author expresses his gratitude to Yefeng Shen and Jie Zhou for useful discussions on the mirror maps for elliptic curve quotients.  The work of C.H. Cho and H. Hong was supported by the National Research Foundation of Korea (NRF) grant funded by the Korea Government (MEST) (No. 2012R1A1A2003117) and (MSIP) (No. 2013042157).  The work of S.-C. Lau was supported by Harvard University.

\section{Tessellations of space-forms}
Let us consider the universal (orbifold) cover
of $\PO$ and its triangular tessellation by real equators.
The orbifold Euler characteristic (due to Thurston) of $\PO$ is given by
\begin{equation}\label{firstchern}
 \chi(\PO) =  \frac{1}{a} + \frac{1}{b} +\frac{1}{c} -1.
 \end{equation}
\begin{lemma}\label{lem:orbcov}
The universal orbifold cover $\pi:E \to \PO$ of  $\PO$ is given as
\begin{equation}
\left\{ \begin{array}{lr}
\mathbb{H}^2 & \;\;\textrm{if} \;\; \chi(\PO) <0\\
\mathbb{R}^2 &  \;\;\textrm{if} \;\; \chi(\PO) =0 \\
\mathbb{S}^2 &  \;\;\textrm{if} \;\; \chi(\PO) >0.
\end{array}\right.
\end{equation}
\end{lemma}
If $\chi(\PO) >0$, the map $\pi:E \to \PO$ is a $d$-fold orbifold cover with $d \cdot \chi(\PO) = 2$.
We denote by $L_E$ the union of all possible lifts of $L \subset \PO$ in $E$.

We will introduce two tessellations of the universal cover $E$.
The first one is a triangle tessellation which comes in a more direct way.  The second one is a hexagon tessellation $Z$ which provides
an effective way to count the polygons for the potential.

Let $M$ be a two dimensional topological space, which will be either $S^2, \RR^2, \HH$.
A \emph{cellulation of $M$} is a 2-dimensional CW-complex homeomorphic to $M$
such that its 1-skeleton is a simplicial graph and each closed 2-face is glued along a combinatorial simple closed curve.
We simply call such a CW-complex as a \emph{cellulation} when $M$ is clear from the context.
We call a 2-face simply as a \emph{face}.
If the faces are combinatorially isomorphic (isometric if there is a metric) to each other
then we call the cellulation as a \emph{tessellation}. 

Denote by $F$ the free group generated by $\CA=\{\alpha,\beta,\gamma\}$. In particular, the triangular orbifold group
\begin{equation}\label{def:G}
G:=\pi_1^{\mathrm{orb}}(\PO)= \form{\alpha,\beta,\gamma\vert\; \alpha^a=\beta^b=\gamma^c=\alpha\beta\gamma=1}
\end{equation}
is a quotient of $F$. 
We regard $F$ as the set of reduced words in $\CA\cup\CA^{-1}$.
We call an element of $\CA\cup\CA^{-1}$ as a \emph{letter}.
A consecutive sequence of letters in a word $w\in F$ is called a \emph{subword} of $w$.
If $u$ is a subword of $w$, we write $u\preccurlyeq w$.
We let $\tau(\alpha)=\beta,\tau(\beta)=\gamma,\tau(\gamma)=\alpha$.
A word is \emph{cyclically reduced} if its cyclic conjugations are all reduced.
A \emph{cyclic word} in $F$ means the set of cyclic conjugations of a cyclically reduced word in $F$.
We let $\fc$ denote the set of cyclic words in $F$. For $w\in F$, we denote by $[w]\in\fc$ the cyclic word represented by $w$. 

\subsection{Triangular tessellation by real equator}

We label the three orbifold points of $\PO$ by $A,B$ and $C$, see Figure~\ref{fig:orbifold}.  The geodesic equator which passes through the orbifold points divides the orbifold into two isometric triangles, each of which called a \emph{base triangle}.  The base triangle has angles $\pi/a, \pi/b, \pi/c$ at $A,B,C$ respectively.  The isometry type of a base triangle is denoted as $\Delta_{a,b,c}$.
We color the two base triangles by \emph{black} and \emph{white}.

The universal cover $E$ of $\PO$ is naturally tessellated by lifts of black and white base triangles so that they form a checkerboard pattern, which means two adjacent triangles must have different colors.
The labeling and transverse orientations shown in Figure~\ref{fig:orbifold} are all pulled back to $E$; see Figure~\ref{fig:ygraph}.
In particular, the positive orientation on a white base triangle in $E$ reads the vertex labels in the order $(A,B,C)$. 

Also, the pre-images of the equator  in $E$ are given by union of geodesics, which are great circles for $S^2$, straight lines for $\RR^2$,
and infinite arcs in $\HH$ whose asymptotics are perpendicular to the boundary $\partial \HH$.
The reflections across these geodesics define a reflection group $H$ given by
$$H =\langle \tau_a, \tau_b, \tau_c \mid \tau_a^2 = \tau_b^2 = \tau_c^2 = (\tau_a \tau_b)^c=( \tau_b \tau_c)^a=( \tau_c \tau_a)^b=1 \rangle.$$
Here, $\tau_a, \tau_b, \tau_c$ are reflections along edges opposite to the vertices $A,B,C$ of a base triangle in $E$ respectively.
The orbifold fundamental group $G$ is an index two subgroup of $H$.
The dual graph of this triangle tessellation can be identified with the Cayley graph of $H$.

\begin{figure}[htb!]
\includegraphics[width=.6\textheight]{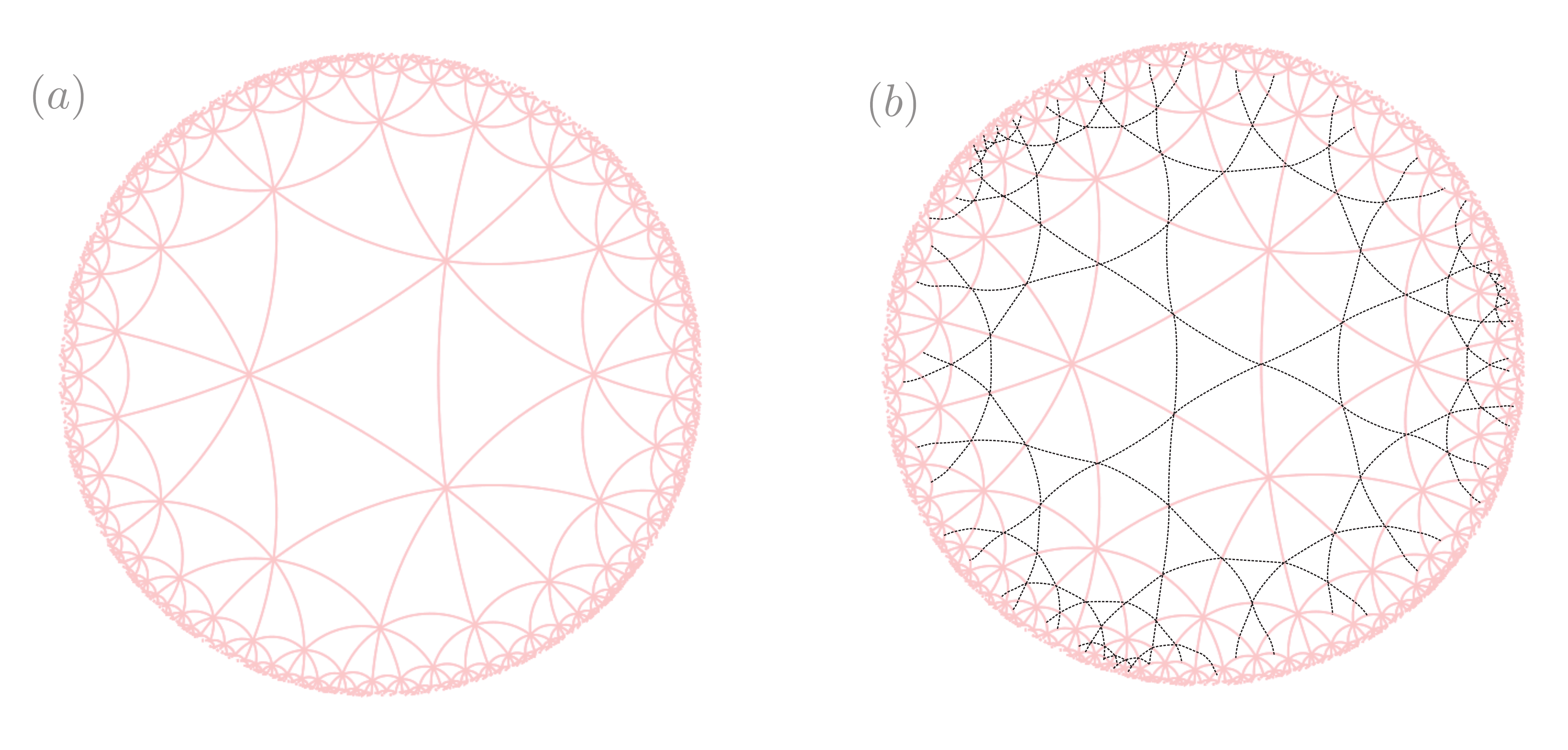}
\caption{$(a)$ Tessellation of the universal cover of $\mathbb{P}^1_{4,4,4}$ by real equator, $(b)$ Seidel Lagrangian in $\HH$}\label{fig:444basic}
\end{figure}

The Seidel Lagrangian $L \subset \PO$ can also be pulled back to the universal cover $E$ (see $(b)$ of Figure
~\ref{fig:444basic}). We remark that a lift $\WT{L} \subset E$ of the Seidel Lagrangian is not a geodesic. The pre-image $L_E \subset E$ of $L$, which is the union of all lifts of $L$, is also immersed, and the immersed points lie on the edges of the trianglular tessellation by real equators.

\subsection{Hexagon tessellation $Z$}
Now, we define a new hexagon tessellation of $E$.
\begin{figure}[htb]
  \tikzstyle {a}=[red,postaction=decorate,decoration={%
    markings,%
    mark=at position 1 with {\arrow[red]{stealth};}}]
  \tikzstyle {b}=[blue,postaction=decorate,decoration={%
    markings,%
    mark=at position .85 with {\arrow[blue]{stealth};},%
    mark=at position 1 with {\arrow[blue]{stealth};}}]
  \tikzstyle {c}=[orange,postaction=decorate,decoration={%
    markings,%
    mark=at position .7 with {\arrow[orange]{stealth};},%
    mark=at position .85 with {\arrow[orange]{stealth};},
    mark=at position 1 with {\arrow[orange]{stealth};}
}]
  \tikzstyle {av}=[red,draw,shape=circle,fill=red,inner sep=2pt]
  \tikzstyle {bv}=[blue,draw,shape=circle,fill=blue,inner sep=2pt]
  \tikzstyle {cv}=[orange,draw,shape=circle,fill=orange,inner sep=2pt]
  \tikzstyle {wv}=[black,draw,shape=circle,fill=white,inner sep=2.5pt]  
  \tikzstyle {gv}=[inner sep=0pt]
  \tikzstyle {pv}=[black,draw,shape=rectangle,fill=black,inner sep=2pt] 
  \tikzstyle {ar}=[teal,postaction=decorate,decoration={%
    markings,%
    mark=at position .35 with {\arrow[teal]{stealth};}}]
{
	\begin{tikzpicture}[thick]
   	\node [gv] at (0,0) (c) {};
   	\node [gv] at (90:2) (pp) {};
   	\node [gv] at (210:2) (qp) {};
	\node [gv] at (150:1.5) (pq) {};
   	\node [gv] at (330:2) (rp) {};
	\draw (pp) edge [bend right,ultra thick] node [gv,pos=.5] (zp) {} (qp);
	\draw (qp) edge [bend right,ultra thick] node [gv,pos=.5] (xp) {} (rp);
	\draw (rp) edge [bend right,ultra thick] node [gv,pos=.5] (yp) {} (pp);
	\node [right] at (150:2) {$z$};
	\node [above] at (270:2) {$x$};	
	\node [left] at (30:2) {$y$};		
	\node [red,right=2] at (90:2.2) {$A$};
	\node [blue,left=2] at (210:2) {$B$};	
	\node [orange,right=2] at (330:2) {$C$};	
	\draw (yp) edge [ar,teal,bend right=20,ultra thick] (zp);
	\draw (yp) edge [dashed,teal,bend left=20,ultra thick] node [pos=.35,gray,fill,star,star points=6,star point ratio=2,inner sep=2pt] {} (zp);	
	\draw (zp) edge [teal,bend right=20,ultra thick] node [pos=.5,left=1] {} (xp);
	\draw (zp) edge [ar,dashed,teal,bend left=20,ultra thick] node [pos=.5,left=1] {} (xp);	
	\draw (xp) edge [ar,teal,bend right=20,ultra thick] node [pos=.5,left=1] {} (yp);		
	\draw (xp) edge [dashed,teal,bend left=20,ultra thick] node [pos=.5,left=1] {} (yp);		
	\draw [red,ultra thick,out=-100,in=100] (pp) edge node [black,fill,pos=.49,star,star points=5,star point ratio=2,inner sep=2pt] {} (c);
	\draw [blue,ultra thick,out=40,in=200] (qp) edge (c);
	\draw [orange,ultra thick,out=140,in=-20] (rp) edge (c);
	\draw [a] (85:1.45) -- (105:1.5);
	\draw (75:1.5) node [red] {$\alpha$};
	\draw [b] (200:1.45) -- (220:1.5);
	\draw (192:1.45) node [blue] {$\beta$};
	\draw [c] (320:1.45) -- (340:1.5);
	\draw (312:1.45) node [orange] {$\gamma$};
	\draw (pp) node [av] {};
	\draw (rp) node [cv] {};
	\draw (qp) node [bv] {};
	\draw (c) node [wv] {};
	\node  [inner sep=0.9pt] at (-2.5,0) {}; 	
	\node  [inner sep=0.9pt] at (2.5,0) {}; 			
	\end{tikzpicture}}
\caption{Orbifold $\PO$. The five- and six-prong stars denote $x_0$ and $x_0'$, respectively.}\label{fig:orbifold}
\end{figure}
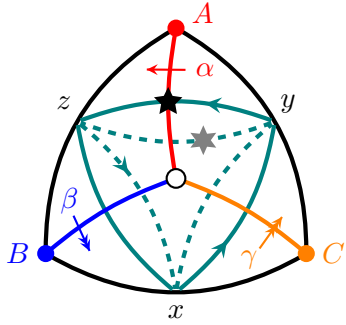
In Figure~\ref{fig:orbifold}, the Seidel Lagrangian $L$, which is the union of edges connecting $(x,z)$, $(y,z)$ and $(z,x)$, divides each base triangle into four smaller triangles. 
We assume that these four triangles have the same area $\sigma$.
Among these four triangles, the one  containing no orbifold point is called the \emph{middle triangle} or \emph{minimal $xyz$ triangle}.
We label the vertices of a middle triangle by $x,y$ and $z$ which lie on the arc $BC$, $CA$ and $AB$ of the corresponding base triangle respectively.

Consider the white base triangle of $\PO$.  We fix the center $W$, which is called a {\em white vertex}, located in the interior of the middle triangle (the white vertex is drawn as a hollow circle at the center of Figures~\ref{fig:orbifold}).  Moreover we define $x_0=AW\cap L$
and let $x_0'$ be the image of $x_0$ reflected along the equator.

We join the three cone points $A, B, C$ to $W$ by geodesic segments. The geodesic segments $AW$, $BW$ and $CW$ are labeled by $\alpha$,$\beta$ and $\gamma$ respectively.
We give transverse orientations (meaning orientations on the normal bundles) to these three geodesic segments by going counterclockwisely around $W$. We always use the convention that $A,\alpha$ is represented by the red (Apple) color, $B,\beta$ by the blue (Blueberry) color and $C,\gamma$ by the orange (Cantaloup) color.
See Figures~\ref{fig:orbifold} and~\ref{fig:ygraph}.
In particular, the positive orientation on the white base triangle in $\PO$ reads the vertex labels in the order $(A,B,C)$.

\begin{definition}[Tessellation $Z$]\label{def:hexZ}
We denote by $Y$ the union of the $\alpha,\beta$ and $\gamma$-geodesic segments in $\PO$.
 The union of the lifts of $Y$ in the universal cover $E$ gives a hexagon tessellation $Z$ of  $E$.
 \end{definition}

\begin{figure}[htb!]
\includegraphics[width=.6\textheight]{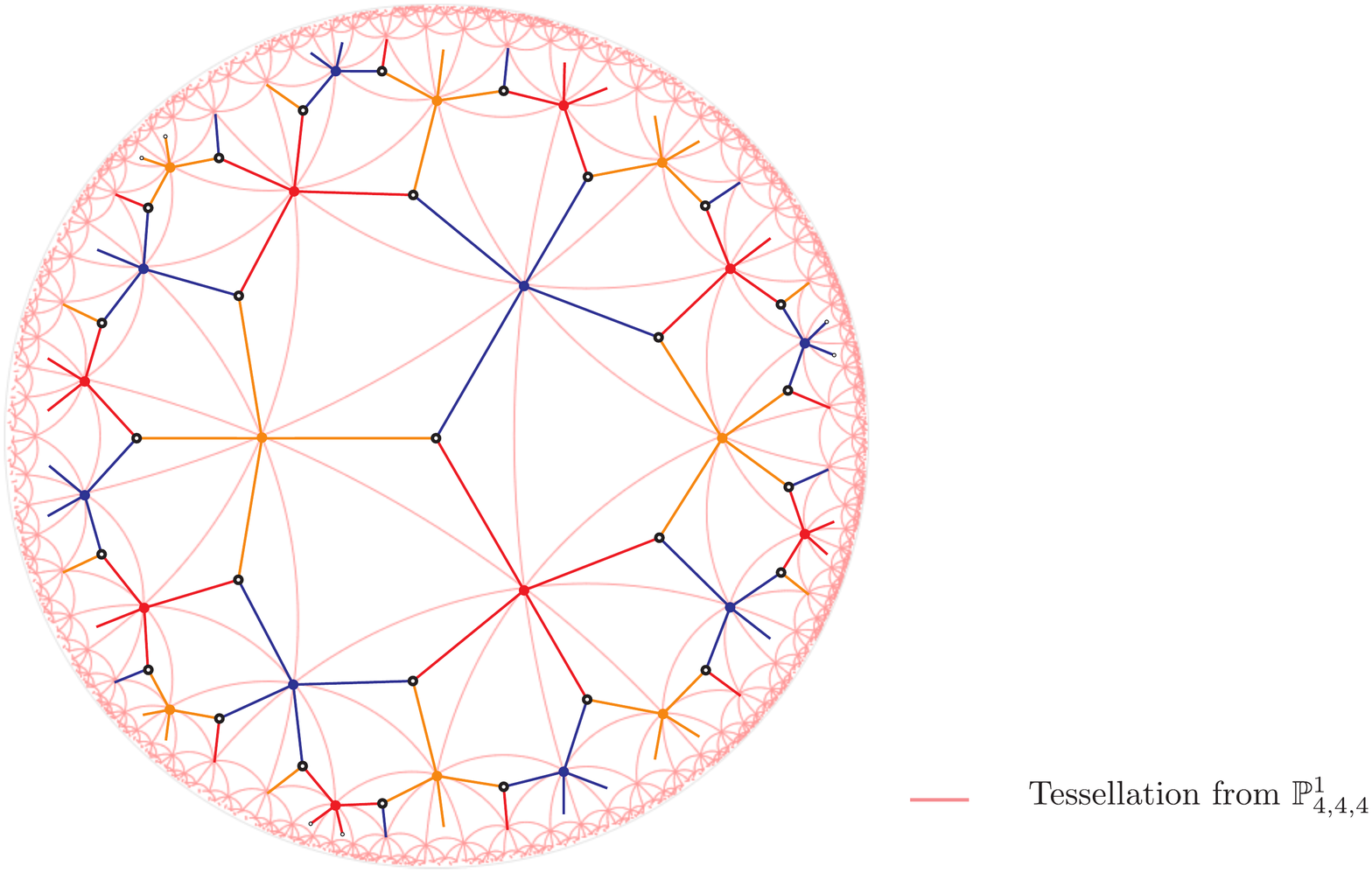}
\caption{Universal cover of $\mathbb{P}_{4,4,4}$ and its tessellation $Z$.}\label{fig:ygraph}
\end{figure}

Then each hexagon is a fundamental domain for the $ \pi_1^{\mathrm{orb}}(\PO)$-action on $E$.
Hence the labels and the orientations of the vertices and the edges of $Z$ are $\pi_1^{\mathrm{orb}}(\PO)$-invariant.
The 1-skeleton of $Z$ is simply the dual of the Cayley graph $\Gamma=\cay( \pi_1^{\mathrm{orb}}(\PO),\CA)\subseteq E$.
By a \emph{$W$-vertex} or a \emph{white vertex}, we mean a lift of $W$ from $\PO$ to $E$.
We similarly define $A$-, $B$- and $C$-vertices as lifts of $A, B$ and $C$; we call these vertices as \emph{colored}.

Now we consider the relation between the Seidel Lagrangian $L_E$ and the tessellation $Z$ of $E$.
First we define the label-reading of an oriented curve which is transverse to the edges of $Z$.
\begin{definition}\label{def:labelreading}
Let $\delta$ be an oriented curve in $E$ such that $\delta$ is transverse to the edges of $Z$.
We follow the trajectory of $\delta$.  Whenever $\delta$ crosses an edge $e$, we record the label $t$ of $e$ or its inverse $t^{-1}$, depending on whether the crossing coincides with the transverse orientation of $e$ or not.
The resulting word (after free reduction if necessary) is denoted as $w(\delta)\in F$ and called the \emph{label-reading of $\delta$}.
\end{definition}

If $\delta$ is an oriented loop, then $[w(\delta)]$ is the cyclic word represented by the label-reading of $\delta$ with respect to an arbitrary choice of the base point.
In particular, when $U$ is a holomorphic polygon with a positively oriented boundary $\partial U$,
we call $[w(\partial U)]$ to be the \emph{boundary (cyclic) word} of $U$.

The following is immediate from construction. 
The label-reading of a Seidel Lagrangian $\tilde{L}$ in $E$ (without turning at any immersed point) is the bi-infinite word
\[\cdots\gamma\beta\alpha\gamma\beta\alpha\gamma\beta\alpha\cdots
=(\gamma\beta\alpha)^\infty.\]
If we turns at an immersed generator $z$, $x$ or $y$ in the Seidel Lagrangian $L_E$, label-reading gives a word $\alpha\beta$,$\beta\gamma$ or $\gamma\alpha$ respectively.  See Figure \ref{fig:consecutiverefl} for an example.

\begin{figure}[htb!]
\includegraphics[width=.33\textheight]{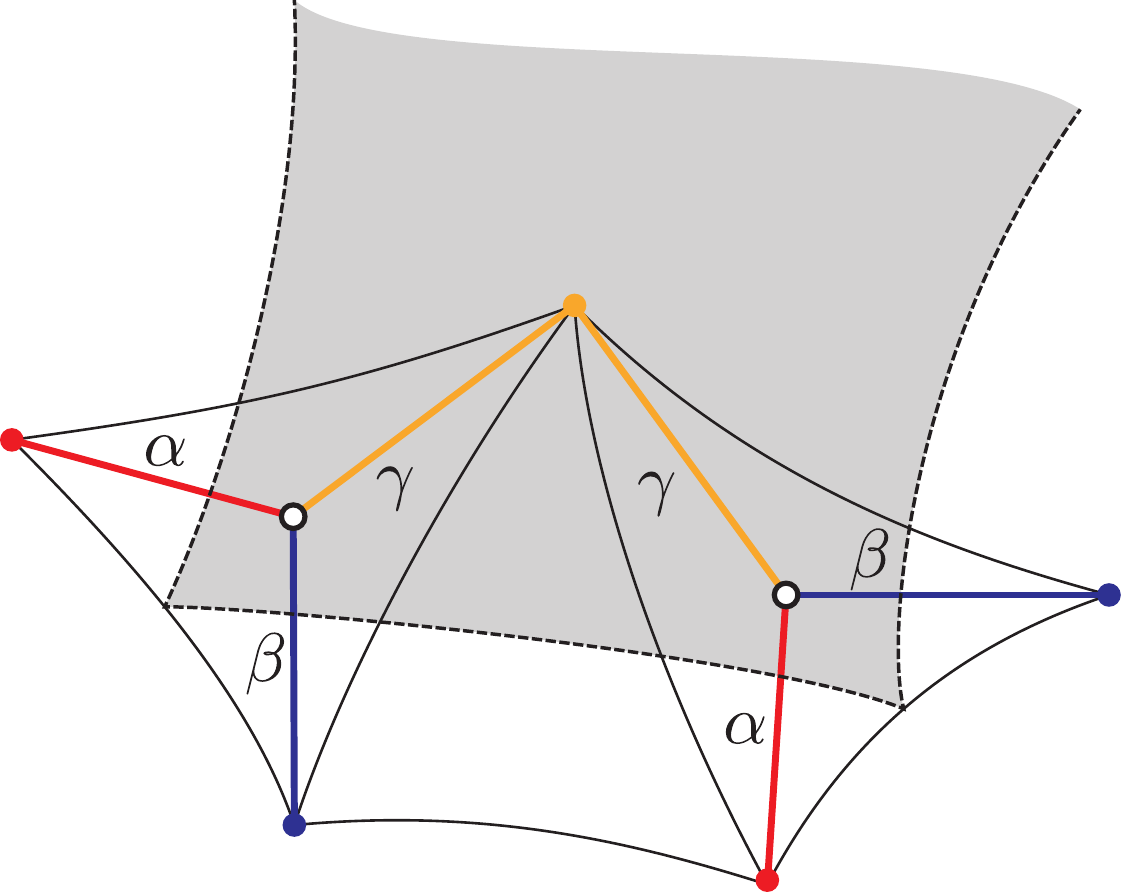}
\caption{Consecutive turns $(\alpha \beta)^2$}\label{fig:consecutiverefl}
\end{figure}

\subsection{Lifting Seidel Lagrangians}
The pre-image $L_E$ of the Seidel Lagrangian in the universal cover consists of several branches, we call a branch of $L_E$ by a
{\em Seidel Lagrangian in $E$} (or just Seidel Lagrangian if it is clear from the context).
A Seidel Lagrangian in $E$ is  embedded in most cases. 

In the spherical cases (see section \ref{sect:S2}), case by case analysis shows that  Seidel Lagrangians in $E$  are immersed (not embedded) in the case of $(2,2,r)$ for $r$ odd, and are embedded otherwise. In the elliptic cases, they are always embedded.
 
In the hyperbolic case, we prove that any Seidel Lagrangian in $E$ is always embedded (which may be intuitively clear in examples)
We give a proof using hyperbolic geometry in Corollary~\ref{cor:bigon_no}.  We present another proof of it here. Let $1/a+1/b+1/c\le 1$ be given. If a Seidel Lagrangian in $E$ is immersed (but not embedded), then there exists a minimal 1-gon with boundary on $L_E$. The corner of this 1-gon has to be mapped to one of the $x, y,z$ immersed generators from the orientation consideration. But from the area formula for a polygon for the potential (Theorem~\ref{thm:area}), such a 1-gon has to have a negative area, which is a contradiction. Note that every nonempty subword of $(\gamma\beta\alpha)^\infty$ is nontrivial by Lemma~\ref{lem:nontrivial}.

In the hyperbolic case, we want to show that there exists a compact manifold cover $\Sigma$ (instead of the hyperbolic plane $E$) of $\PO$ such that any branch of the pre-image in $\Sigma$ of the Seidel Lagrangian is embedded.

Let us regard $L$ as a based loop at $x\in\PO$ and choose a branch $\tilde L\subseteq\HH$ of the pre-image of $L$ (regarded as a path), which reads $\gamma\beta\alpha$. Suppose $p:\;(\Sigma,\tilde x)\to (\PO,x)$ is a covering map. Choose $k$ to be the smallest positive integer such that the concatenation of the $k$ copies of $L$ lifts to a loop $L':\; S^1\to \Sigma$ based at $\tilde x$.
Let us call $L'$ as the \emph{elevation} of $L$ at $\tilde x$ with respect to $p$.

\begin{prop}
If $1/a+1/b+1/c<1$, then there exists a finite cover $\Sigma$ of $\PO$ such that any elevation of $L$ in $\Sigma$ is an embedded circle. 
\end{prop}

\begin{proof}
Let $H$ be a finite index torsion free normal subgroup of $G$, and $k$ be the smallest positive integer such that $ (\gamma\beta\alpha)^k\in H$. Define W as the set of all nonempty finite sub-words of $ (\gamma\beta\alpha)^k$. Since the Fuchsian group $G$ satisfies the \emph{LERF} property \cite{Scott1978}, there exists a finite-index normal subgroup $K$ of $G$ such that $ (\gamma\beta\alpha)^k\in K$ and $W \cap K =\varnothing$. We have an elevation $L'$ satisfying the following diagram:
\[ \xymatrix{ S^1\ar[d]_-{k\mathrm{-to-}1}\ar[r]^-{L'}& \HH/ H\cap K = \Sigma\ar[d]^-{<\infty} \\ S^1\ar[r] & \HH/G } \] The hexagon tessellation of $\HH$ induces that of $\Sigma$, and $L'$ reads $(\gamma \beta \alpha)^k$ in this tessellation. If $L'$ were not embedded, then two vertices in the path $ (\gamma\beta\alpha)^k$ in $\HH$ belong to the same $(H\cap K)$--orbit. This is impossible since $W\cap K = \varnothing$.

\end{proof}



\section{Lagrangian Floer potential for the Seidel Lagrangian}
In this section, we briefly recall the definition of the Lagrangian Floer potential for the Seidel Lagrangian from \cite{CHL}.
Figure \ref{fig:lagintro} shows the Lagrangian immersion $L$ introduced by Seidel.  Its $\AI$-algebra can be defined following
the idea of Akaho-Joyce \cite{AJ} and Seidel  \cite{S}.  We consider its weak Maurer-Cartan elements
given as linear combinations of odd degree immersed generators, 
whose associated $m_0$-term is the desired Lagrangian Floer potential.

\subsection{Lagrangian Floer potential for immersed Lagrangians}
First,  the Novikov ring $\Lambda$ is defined by 
$$\Lambda:= \left\{ \sum_{i} a_i T^{\lambda_i} \mid a_i \in \CC,\quad \lambda_i \nearrow \infty \right\},
\Lambda_0:= \left\{ \sum_{i} a_i T^{\lambda_i}  \in \Lambda \mid \lambda_i \geq 0 \right\}. $$

%
%
%
According to \cite{AJ} and \cite{S}, the Floer complex $CF(L,L)$ of $L$ with itself for a Lagrangian immersion is defined by a free $\Lambda$-module generated by two types of generators.  Generators of the first type come from the singular cohomology of $L$, which is
isomorphic to the cohomology of the Morse complex of a Morse function $f$ on $L$.  For simplicity we assume $f$ has only two critical points, which are the minimum point $e$ and the maximum point $p$. Generators of the second type come from self-intersection points.
Each self intersection point has two branches, say $br_1, br_2$, and the two ordered pairs $(br_1, br_2)$ or $(br_2, br_1)$
are generators of the second type of $CF(L,L)$, and we call them to be immersed generators.  At each self intersection point $X$, one of the immersed generator has odd degree and the other one has even degree.  We denote them as $X$ and $\bar{X}$ respectively, see Figure \ref{fig:angles}.  Geometrically the immersed generators can be thought as formal smoothings at the self-intersection points.

\begin{figure}[ht]
\begin{center}
\includegraphics[height=2in]{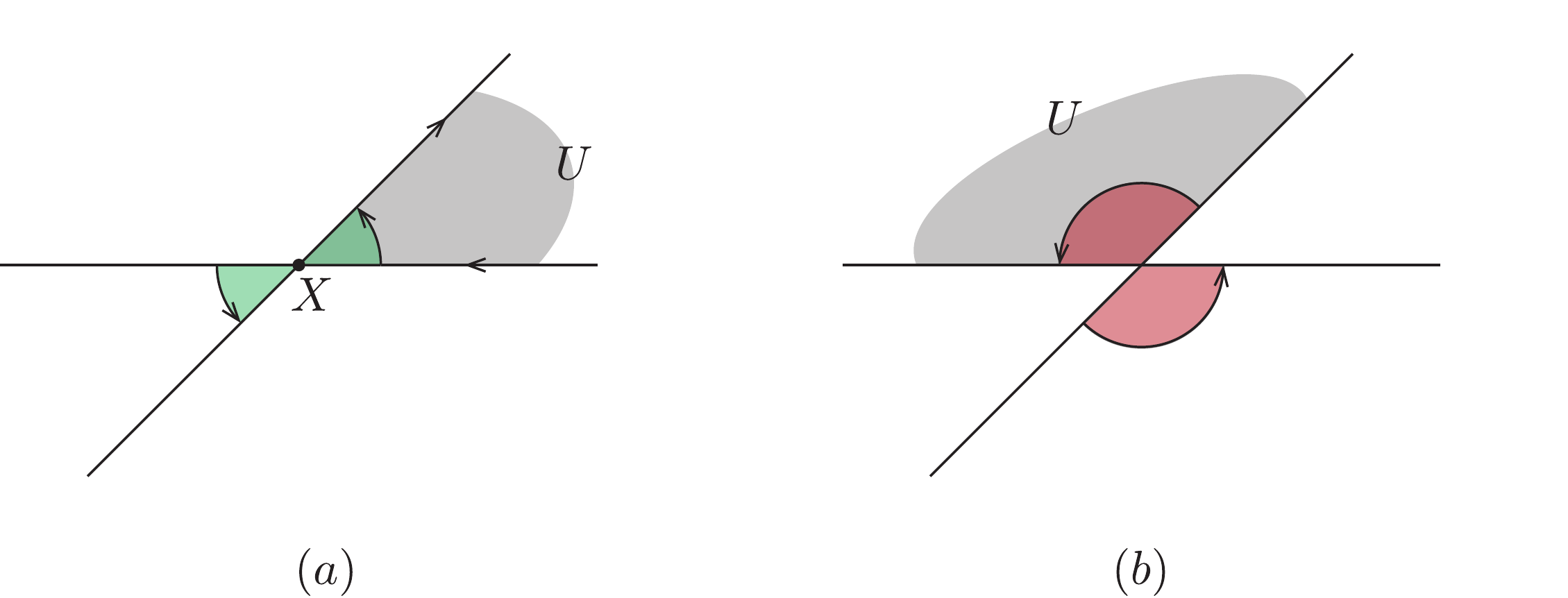}
\caption{A polygon $U$ with the input $a_i$ for (a) $a_i =X$ and (b) $a_i = \bar{X}$ }\label{fig:angles}
\end{center}
\end{figure}

The point $e$ serves as the unit for $CF(L,L)$. $CF(L,L)$ admits a natural $\Z_2$-grading for which $X$ and $p$ have odd degrees, and $\bar{X}$ and $e$ have even degrees for each self-intersection point $X$ of $L$. One can show, for example, that the Floer differential sends odd generators to even and vice versa.

$A_\infty$ operations $\{m_k\}_{k \geq 0}$ on $CF(L,L)$ can be defined using holomorphic polygons (or more generally pearly trees as in \cite{S}).  We will be interested in the cases that all inputs $a_1,\ldots,a_k$ are given by immersed generators.  $m_k$ takes the form
\begin{equation}\label{eq:mk}
m_k (a_1, \cdots, a_k) = \sum_{b,u} \pm T^{\omega (u)} \cdot b
\end{equation}
where the sum is taken over all generators $b$ of $CF(L,L)$ and (immersed) holomorphic polygons $u$ described as follows.\footnote{Rigorously speaking we need to consider moduli problems and their corresponding obstruction theories.  In our surface case all relevant moduli spaces are regular and just consist of points.  Hence we simply write $m_k$ as a genuine counting instead of some integrations over moduli.}  (Recall that $T$ is the formal variable in the Novikov ring.)  $u$ is a continuous map $(D,\partial D, \{z_0,\ldots,z_k\}) \to (X,L)$ which is holomorphic in the interior of $D$.  $z_1,\ldots,z_k$ are called input marked points (ordered couterclockwisely), and they are required to be mapped to the immersed generators $a_1, \cdots, a_k$ respectively (see Figure \ref{fig:angles} for a local picture around a corner).  $z_0$ is called the output marked point, and it is required to be mapped to $b$.  We require $u$ to be convex at $z_i$ for all $i=0,\ldots,k$, which means they have angles less than or equal to $\pi$.  In our case $a_i$'s are taken to be the immersed generators $X, Y, Z$ of the Seidel Lagrangian $L$.




The contribution from a non-constant polygon to the coefficient of $p$ in $m_k (a_1, \cdots, a_k)$ always vanishes.
The coefficient of $e$ in $m_k (a_1, \cdots, a_k)$ is given by the signed count of holomorphic polygons $u$ with convex corners $z_1,\cdots, z_k$ mapped to immersed generators $a_1,\cdots, a_k$ respectively, and the output marked point $z_0$ is mapped to $e$.

The sign in Equation \eqref{eq:mk} is defined by \cite{S}.  Pick a generic point $x_0$ on $L$ to define a non-trivial spin structure on $L$.  We can choose $x_0$ to be $e$ for computational convenience.

Fix a holomorphic polygon $u$ which is a map $(D,\partial D) \to (X,L)$ ($X$ is taken to be $\PO$ later).  If the induced orientation on each side of $\partial D$ matches with the orientation of the Lagrangian $L$, and $\partial D$ does not pass through $x_0$, then the sign of the contribution of $U$ is positive.  Disagreement of orientations along each side $\overline{a_i a_{i+1}}$ of $\partial D$ changes the sign of $u$ by $(-1)^{|a_i|}$ for $2 \leq  i \leq k$. The orientation of the first side(edge) $\overline{a_1 a_2}$ does not affect the sign and disagreement of orientations along the last side changes the sign by $(-1)^{|a_1| + |b|}$. In addition to this, we multiply by $(-1)^k$ where $k$ is the number of times $\partial u$ passes through $x_0$.  Thus the sign in Equation \ref{eq:mk} is determined.

%


The definition of $\{m_k\}_{k \geq 0}(a_1,\ldots,a_k)$ extends multilinearly to $a_1,\ldots,a_k \in CF(L,L)$.  Let $X_1,\cdots, X_k$ be all the odd degree immersed generators of $L$.  
Let $b = \sum x_i X_i \in CF(L,L)$ for $x_i \in \Lambda_0$.  

\begin{definition}
$b = \sum x_i X_i \in CF(L,L)$ is called a {\em weak Maurer-Cartan element} (bounding cochain) if
$m(e^b)$ is a constant multiple of the unit $e$. Namely,
$$\sum_{k=1}^\infty m_k(b,\cdots,b) = c(b) \cdot e$$
for $c(b) \in \Lambda_0$.
If such a weak Maurer-Cartan element exists, then $L$ is said to be {\em weakly unobstructed}.  The Lagrangian Floer potential $W$ is a function on the set of all such weak Maurer-Cartan elements $b$, which takes the value $c(b)$ defined by the above equation for each $b$.
\end{definition}

%
%
%

\subsection{$\mathbb{P}^1_{a,b,c}$ and the Seidel Lagrangian}
Now we restrict ourselves to the case of orbi-sphere $\mathbb{P}^1_{a,b,c}$ and the Seidel Lagrangian $L$.

\begin{figure}[ht]
\begin{center}
\includegraphics[height=2in]{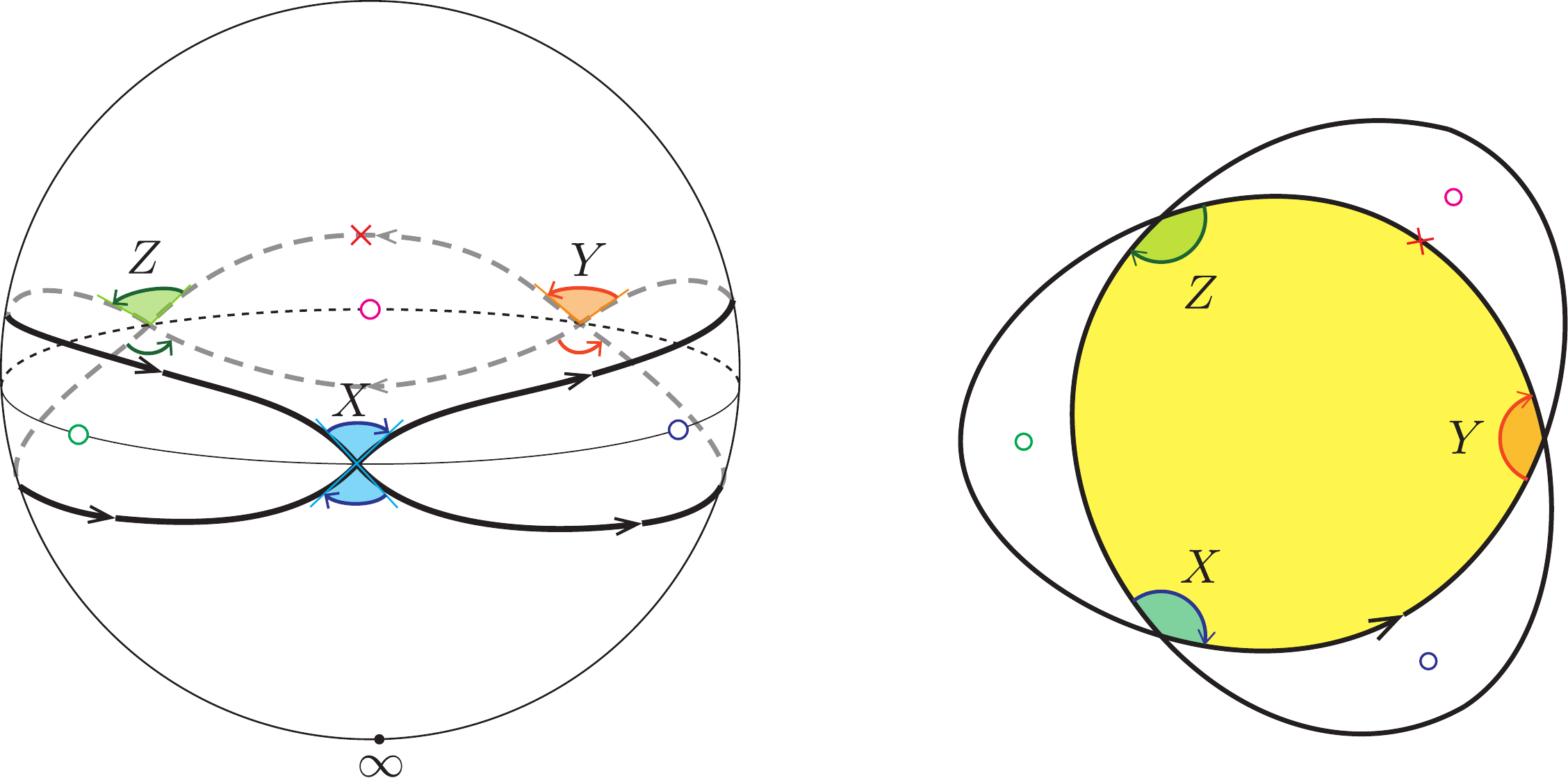}
\caption{The Seidel Lagrangian and the minimal $xyz$-triangle $U_{xyz}$: $\times$ represents both $e$ and the generic point $x_0$. }\label{fig:seidel_lag}
\end{center}
\end{figure}

As in Figure \ref{fig:seidel_lag}, $L$ has three self-intersection points which are denoted by $X$, $Y$ and $Z$ respectively. We assume that $L$ is symmetric with respect to the reflection about the equator of $\mathbb{P}^1_{a,b,c}$ passing through three orbifold points.  Each self-intersection point $X$ (or $Y$, $Z$) corresponds to two immersed generators, one has degree one which is still denoted by $X$ (or $Y$, $Z$ resp.), one has degree two which is denoted by $\bar{X}$ (or $\bar{Y}, \bar{Z}$ resp.).  Here we use the integer grading introduced by Seidel \cite{S}.

In \cite{CHL} we proved that all linear combinations of $X,Y,Z$ are weak Maurer-Catan elements.

\begin{lemma}\cite[Lemma 7.5]{CHL}
Let $V= \Lambda_0^3$.  For $(x,y,z) \in V$, define $b= xX + yY + zZ$. $b$ is a weak Maurer-Cartan element.
Hence there exists $W: V \to \Lambda_0$ such that
\begin{equation}\label{eq:defW}
m(e^b) = W(x,y,z) \cdot e
\end{equation}
\end{lemma}

$W$ is in general a formal power series in $x,y,z$.  The main purpose of this paper is to compute $W$ explicitly in the spherical and elliptic case, and to give an algorithm to compute $W$ up to any order in the hyperbolic case.

Note that holomorphic polygons contributing to $W$ must have Maslov index two by dimension counting.  Namely, the moduli space of polygons in a certain class $\beta$ with $k$ input marked points which are mapped to the degree one immersed generators $X,Y,Z$ and one output marked point which is mapped to the point $e$ has dimension $1+\mu(\beta) + k+1 - 3 - k-1 = \mu(\beta) - 2$, which is zero if and only if the Maslov index $\mu = 2$.  Hence the counting is defined for polygons with Maslov index two.  The readers are referred to \cite{CHL} for more detailed discussions on the disc moduli.  We include the definition of Maslov index for polygons for readers\rq{} convenience.

\begin{definition}
Let $u$ be a holomorphic polygon with convex corners bounded by $L$.  At each convex corner $p$, we connect the tangent space $T_pL_i$ with $T_pL_{i+1}$ by a canonical path
obtained by complex multiplications $e^{it}$ for $ 0 \leq t \leq c$ for some $c \leq \pi$. Since $p$ is an odd degree,
this provides a path between $T_pL_i$ with $T_pL_{i+1}$  as oriented spaces. Then, Maslov index $\mu$ of $u$ is defined as the winding number of the corresponding oriented Lagrangian loop in the trivialization of $u^*TE$.
\end{definition}

We need to classify all holomorphic polygons with Maslov index two whose corners are mapped to the immersed generators $X,Y,Z$.  In this paper, we will work on the universal orbifold cover $E$ of $\mathbb{P}^1_{a,b,c}$ which is $S^2, \RR^2$ or $ \mathbb{H}^2$ as explained in Lemma \ref{lem:orbcov}.
Any holomorphic polygon (which is a map from the unit disk) lifts to the universal cover, and we will count them in $E$ modulo the deck transformation group, which is $G = \pi_1^{\mathrm{orb}}(\PO)$.

\begin{lemma}
Any  holomorphic polygon in $\PO$ with boundary on $L$ 
lifts to a holomorphic polygon $U:(D^2, \partial D^2) \to (E, L_E)$.
\end{lemma}

Hence, it is enough to find all such polygons in the universal cover $E$. We will  call them to be {\em polygons for the potential in $E$} for convenience.

%
%

We finish this section by writing $W$ in terms of counting polygons for the potential in $E$.  Let $\WT{e} \in L_E$ be a lift of the point $e$ in $L$.  Counting holomorphic polygons $U$ passing through a fixed lift \textbf{$\WT{e} \in L_E$} is equivalent to counting the number of times that $\partial U$ passes through all the lifts of $e$.

Note that holomorphic polygons always appear in pairs since $L$ is preserved under the reflection about the equator of $\mathbb{P}^1_{a,b,c}$. One is a polygon $U$ whose induced boundary orientation matches with that of the Lagrangian $L$, and its partner is the reflection image of $U$ denoted by $U^{op}$, whose induced orientation on $\partial U^{op}$ is opposite to that of the Lagrangian $L$. This is because the reflection is an anti-symplectic involution while it preserves the orientation of $L$. 
 
For the polygon $U$, we define $s(U)$ to be the number of $e$'s on the boundary of $U$.  Similarly $s(U^{op})$ is defined as the number of $e$'s on the boundary of $U^{op}$, which is the same as the number $s^{\mathrm{op}}(U)$ of $e'$ on the boundary of $U$, where $e'$ is the reflection image of $e$.

The sign of the contribution of $U$ is $(-1)^{s(U)}$.  For $U^{op}$, the sign is
$$(-1)^{s^{op}(U) + (\textrm{number of sides})} = (-1)^{s(U)}$$
by Lemma 7.4 of \cite{CHL}.

Finally, when $g \in G$ preserves $U$, we need to identify $(U, \WT{e})$ with $(g(U), g(\WT{e})) = (U, g(\WT{e}))$.
Let $\eta(U)$ be the largest positive number $k$ such that the polygon $U$ in $E $ has $\mathbb{Z}/k$ symmetry from the deck transformation group $G$.  Then the counting of $U$ gives $(-1)^{s(U)} \frac{s(U)}{\eta(U)}$, and that of $U^{op}$ gives $(-1)^{s(U)} \frac{s^{op}(U)}{\eta(U)}$.  Thus we have the following proposition:

\begin{prop}[\cite{CHL}]
Let
$\omega(U)$ be the symplectic area of $U$, and $\vec{\mathbf{x}}^U := x^P y^Q z^R$ where $P,Q,R$ are the numbers of $x,y,z$-corners (vertices) in $U$ respectively.
Then, the potential of $L$ is can be written as
\begin{equation}\label{eq:formulaW}
W = \sum_U (-1)^{s(U)} \frac{s(U) + s^{\mathrm{op}}(U)}{\eta(U)} T^{\omega(U)} \vec{\mathbf{x}}^U
\end{equation}
where the sum is taken over all holomorphic polygons $U$, up to $G$ action, with Maslov index two bounded by $L_E$, whose boundary orientation matches with that of $L_E$.
\end{prop}

For instance, for the minimal $xyz$ triangle $U_{xyz}$ whose boundary orientation matches with that of $L_E$, $s( U_{xyz}) = 1$ and $s^{\rm op} (U_{xyz}) =0$.  Let $\sigma :=\omega(U_{xyz})$ and $q = T^{-\sigma}$.  As we shall see, areas of all other holomorphic polygons $U$ are integer multiple of $A$, and hence $T^{\omega(U)} = q^{k(U)}$ for $k(U) \in \Z_{>0}$.

At this moment $T$ and $q$ are considered as formal parameters, and $W$ is a formal power series.  We will put $T = e$, the exponential constant, in order to compare with the mirror map in the case of elliptic curve quotients.  Moreover, we shall prove that even in the hyperbolic case, $W$ is convergent in a neighborhood of $0$ after we put $T=e$.  Thus the Landau-Ginzburg theory is well-defined not just in a formal neighborhood of $0$ of the moduli.



\section{Holomorphic polygons and their boundary words}\label{sec:holomorphic polygons}
In this section, we discuss important topological and combinatorial properties of Seidel Lagrangians in $E$.
Then we show that the polygons for the potential in $E$ are embedded, and describe
label reading (see Definition \ref{def:labelreading}) of their boundaries.

\subsection{Combinatorial neighborhood of a Seidel Lagrangian}
Let us assume that $1/a+1/b+1/c\le 1$
and
$a\le b\le c$.
Recall we have a triangular tessellation of $E$ by $(\pi/a,\pi/b,\pi/c)$ triangles.

If $a\ge3$, a \emph{component cell} will just mean a base triangle.
If $a=2$ and $b\ge4$, then the union of the four triangles around an $A$-vertex is called a \emph{component cell}. When $a=2$ and $b=3$, the six base triangles around a $B$-vertex is called a \emph{component cell}. Note that in each of the cases, a component cell is either a geodesic triangle or quadrilateral. 

\begin{lemma}\label{lem:component}
For each Seidel Lagrangian $L$ in $E$,  there exists a geodesic $g$ of $E$ such that
for any component cell $C_0$ of $E$, $ L$ intersects  $C_0$ if and only if $g$ intersects $C_0$.
Moreover, such a geodesic is unique if $E$ is the hyperbolic plane.\end{lemma}

\begin{proof}
Let us first consider the case $a\ge 3$.
As in Figure~\ref{fig:billiard} (a), we can choose a strip $S$ containing $ L$ such that $S$ consists of six consecutive base triangles.
The condition $a\ge3$ implies that $S$ is a Euclidean or hyperbolic convex polygon.
The strip $S$ contains three copies $e_\iota,e_\tau,e_\mu$ of a side, say $AB$,
such that $e_\iota$ and $e_\tau$ contain the initial and the terminal points of the intersection of $ L$ with $S$.
We let $r$ be the length of $AB$.
For each $0\le t \le r$, let us choose $P_t\in e_\iota$ and $Q_t\in e_\tau$ such that $P_t$ and $Q_t$ are distance $t$ away from the copies of $A$ on $e_\iota$ and $e_\tau$, respectively. 
We let $f(t)$ be the distance from the copy of $A$ on $e_\mu$ to the point $e_\mu\cap P_t Q_t$.
Then $f:\;[0,r]\to[0,r]$ is a continuous function and therefore, has a fixed point $t_0$.
The natural projection $\PO\to ABC$ maps $P_{t_0}Q_{t_0}$ to a \emph{3--periodic billiard path}~\cite{Tabachnikov2005}, as shown in Figure~\ref{fig:billiard} (b).
So the bi-infinite path $g$ in $E$ which projects onto this billiard path is a geodesic.
Since every component cell intersecting $ L$ appears in some strip that is obtained by repeatedly gluing copies of $S$, we see that $g$ is the desired geodesic.

Let us assume $a=2$ and $b\ge4$.
The collection of the component cells intersecting $ L$ is obtained by repeatedly gluing component cells (which are quadrilaterals) along opposite sides.
Figure~\ref{fig:billiard} (c) shows a strip $S$, which is the union of two consecutive copies of these component cells.
Let us call the interior $A$-vertices in $S$ as $u$ and $v$.
Then by symmetry, the geodesic $g$ extending $uv$ is perpendicular to every copy of $BC$ that intersects $ L$.

Finally assume $a=2$ and $b=3$. Then each component cell (see Figure \ref{fig:lem:nbhd}) is an equilateral triangle of the interior angle $2\pi/c$. By symmetry, there exists a $3$-periodic billiard path which reflects at $A$ vertices on each side. Note that $ L$ belongs to the union of strips as in Figure~\ref{fig:billiard} (a) where each triangle is now a component cell consisting of six base triangles.
The geodesic triangle joining the three copies of $A$-vertices in each component cell $C_0$ is a 3--periodic billiard path in $C_0$.
So as in the first paragraph of this proof, we have a desired geodesic $g$.

To prove the uniqueness, let us assume $E$ is hyperbolic. In each of the above cases, the union of the component cells intersecting $g$ is a bounded neighborhood of $g$, and hence contains a unique geodesic.
\end{proof}

\begin{figure}[htb]
  \tikzstyle {a}=[red,postaction=decorate,decoration={%
    markings,%
    mark=at position 1 with {\arrow[red]{stealth};}}]
  \tikzstyle {b}=[blue,postaction=decorate,decoration={%
    markings,%
    mark=at position .85 with {\arrow[blue]{stealth};},%
    mark=at position 1 with {\arrow[blue]{stealth};}}]
  \tikzstyle {c}=[orange,postaction=decorate,decoration={%
    markings,%
    mark=at position .7 with {\arrow[orange]{stealth};},%
    mark=at position .85 with {\arrow[orange]{stealth};},
    mark=at position 1 with {\arrow[orange]{stealth};}
}]
  \tikzstyle {av}=[red,draw,shape=circle,fill=red,inner sep=1pt]
  \tikzstyle {bv}=[blue,draw,shape=circle,fill=blue,inner sep=1pt]
  \tikzstyle {cv}=[orange,draw,shape=circle,fill=orange,inner sep=1pt]
  \tikzstyle {wv}=[draw,shape=circle,fill=white,inner sep=2.5pt]  
  \tikzstyle {gv}=[draw,shape=circle,fill=black,inner sep=1pt]
  \tikzstyle {hv}=[draw,shape=circle,fill=white,inner sep=1pt]  
  \tikzstyle {ev}=[draw,teal,shape=circle,fill=teal,inner sep=1pt]
  \tikzstyle {pv}=[black,draw,shape=rectangle,fill=black,inner sep=1pt] 
  \tikzstyle {ar}=[teal,postaction=decorate,decoration={%
    markings,%
    mark=at position .35 with {\arrow[teal]{stealth};}}]
\subfloat[(a)]{
\begin{tikzpicture}[scale=.8,ultra thick]
\node [bv] at (-3.5,0) (p1) {};
\node [cv] at (-1.5,-.2) (p2) {};
\node [av] at (.5,-.2) (p3) {};
\node [bv] at (2.5,0) (p4) {};
\node [av] at (-2.5,1.73) (p5) {};
\node [bv] at (-.5,1.9) (p6) {};
\node [cv] at (1.5,1.9) (p7) {};
\node [av] at (3.5,1.73) (p8) {};
\draw (p1) -- (p2) -- (p3) -- (p4);
\draw (p5) -- (p6) -- (p7) -- (p8);
\draw (p1) -- (p5) -- (p2) -- (p6) -- (p3) -- (p7) -- (p4) -- (p8);
\draw (p1) edge node [pos=.6,ev] (l1) {} (p5);
\draw (p4) edge node [pos=.6,ev] (l2) {} (p8);
\draw (l1) edge [teal] node [ev,pos=.47] (l3) {} (l2);
\node [left] at (l1) {\small $P_t$};
\node [right,below=7] at (l1) {\small $e_\iota$};
\node [right] at (l2) {\small $Q_t$};
\node [below=2,left=1] at (p1) {\small $B$};
\node [above=2] at (p5) {\small $A$};
\node [below=2] at (p2) {\small $C$};
\node [left=3,below=8] at (l3) {\small $f(t)$};
\node [right=4,above=2] at (l3) {\small $e_\mu$};
\node [right=140,below=7] at (l1) {\small $e_\tau$};
\node [left=3,above=5] at (l1) {\small $t$};
\node [left=3,above=3] at (l2) {\small $t$};
\end{tikzpicture}
}
$\qquad$
\subfloat[(b)]{
\begin{tikzpicture}[ultra thick]
\node [bv] at (-1.5,.3) (p1) {};
\node [cv] at (1.5,.3) (p2) {};
\node [av] at (.5,2.5) (p3) {};
\draw (p1) edge [bend left=15] node [pos=.6,ev] (l3) {} (p2);
\draw (p2) edge [bend left=15] node [pos=.5,ev] (l1) {} (p3);
\draw (p3) edge [bend left=10] node [pos=.4,ev] (l2) {} (p1);
\draw (l1) edge [teal,bend left=5] (l2);
\draw (l2) edge [teal,bend left=5] (l3);
\draw (l3) edge [teal,bend left=5] (l1);
\node [gv,left=7,above=3] at (l1) {};
\node [gv,below=6] at (l1) {};
\node [above=3,right=2] at (l2) {\tiny x};
\node [left=2,below=3] at (l2) {\tiny x};
\node [hv,above=3,right=4] at (l3) {};
\node [hv, left=6,above=2] at (l3) {};
\end{tikzpicture}
} \\
\subfloat[(c)]{
\begin{tikzpicture}[ultra thick]

\node [gv] at (-2.2,-.2) (p) {};
\node [gv] at (2.2,-.2) (q) {};
\draw (p) -- (0,0) -- (q) -- (2,2) -- (-2,2) -- (p) -- (0,2) -- (q);
\draw (-2,2) -- (0,0) -- (2,2);
\draw (0,0) -- (0,2);

\foreach \x in {-2,2} {\draw (\x,2) node [bv] {}; }	
\foreach \x in {0} {\draw (\x,0) node [bv] {}; }	
\draw (p) node [cv] {};
\draw (q) node [cv] {};

\foreach \x in {0} {\draw (\x,2) node [cv] {}; }	
\foreach \x in {-1,1} {\draw (\x,1) node [av] {}; }	

\draw [dashed, teal,ultra thick,->] (-2.5,1.5)--(-.5,1.5) -- (.5,.5) -- (2.5,.5) node [teal,right] {\small $ L$};

\end{tikzpicture}
}
\caption{Finding geodesics in neighborhoods of Seidel Lagrangians.}\label{fig:billiard}
\end{figure}
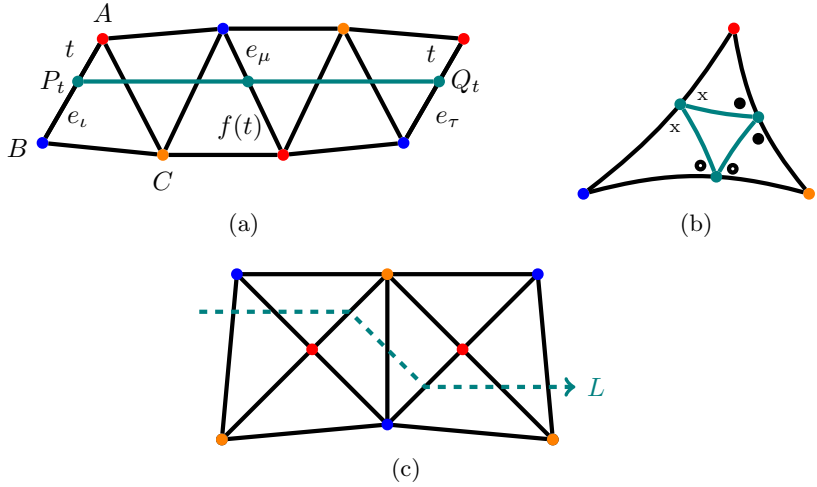

\begin{figure}[htb]
  \tikzstyle {av}=[red,draw,shape=circle,fill=red,inner sep=1pt]
  \tikzstyle {bv}=[blue,draw,shape=circle,fill=blue,inner sep=1pt]
  \tikzstyle {cv}=[orange,draw,shape=circle,fill=orange,inner sep=1pt]
  \tikzstyle {wv}=[black,draw,shape=circle,fill=white,inner sep=1pt]  
  \tikzstyle {gv}=[inner sep=0pt]
  \tikzstyle {v}=[gray,fill=gray,draw,shape=rectangle,inner sep=1pt]    
  \tikzstyle{every edge}=[-,draw]
	\begin{tikzpicture}[scale=0.75,thick]
	
	\draw (-3,0) -- (3,0)--(0,1.73)--(-3,0)--(0,-1.73)--(3,0);
	\draw (0,1.73)--(0,-1.73);
	\draw (-1.5,-1.73/2)--(0,1.73)--(1.5,-1.73/2);
	\draw (-1.5,1.73/2)--(0,-1.73)--(1.5,1.73/2);
	
	\draw [dashed, teal,ultra thick,->] (-2.5,1.73*2.5-1.23*3+.5)--(-1.23*3/1.73,.5)-- (0,.5) -- (2.23*3/4/1.73,-2.23*3/4+.5) -- (3,-2.23*3/4+.5) node [teal,right] {\tiny $ L$};

	\foreach \x in {1.73,-1.73} {\draw (0,\x) node [cv] {}; }	
	\foreach \x in {3,-3} {\draw (\x,0) node [cv] {}; }	
	
	\foreach \x in {-1.73/2,1.73/2} {\draw (1.5,\x) node [av] {}; }		
	\foreach \x in {-1.73/2,1.73/2} {\draw (-1.5,\x) node [av] {}; }		
	
	\foreach \x in {1,-1} {\draw (\x,0) node [bv] {}; }			
	
	\draw (0,0) node [left=5,below] {\tiny $v$};

	\draw (-1,1.5) node [] {\tiny $C_0$};
	\draw (1,1.5) node [] {\tiny $C_1$};	

	\node  [inner sep=0.9pt] at (-1,0) {}; 	
	\node  [inner sep=0.9pt] at (4,0) {}; 			
	\end{tikzpicture}
\caption{The combinatorial neighborhood of a Seidel Lagrangian for $a=2, b=3$.}\label{fig:lem:nbhd}
\end{figure}
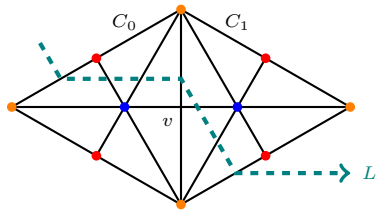

\begin{definition}\label{defn:nbhd}
For a subset $X$ of $E$,
the smallest convex collection of base triangles containing $X$ is called the \emph{combinatorial neighborhood} of $X$ and denoted as $N(X)$.
\end{definition}

\begin{lemma}\label{lem:nbhd}
The combinatorial neighborhood of a Seidel Lagrangian $ L$ in $E$ is the collection of the component cells intersecting $ L$. In particular, it is an infinite sided convex polygon.
\end{lemma}
\begin{proof}
In the Euclidean case, a combinatorial neighborhood is a flat strip and the proof is immediate. So we assume $1/a+1/b+1/c<1$. We first claim that if $ L$ intersects a component cell $C_0$, then $C_0\subseteq N( L)$. This is obvious from definition for $a\ge3$. If $a=2$ and $b\ge4$, then the claim follows from that $ L$ intersects three of the four base triangles around an $A$-vertex in $C_0$ and that the combinatorial neighborhood of these three base triangles is $C_0$. For $a=2$ and $b=3$, then let us consider the component cell $C_1$ uniquely chosen as follows: $C_1$ shares an $A$ vertex $v$ with $C_0$ such that $ L$ intersects one base triangle in $C_0$ and two base triangles in $C_1$ around $v$. See Figure~\ref{fig:lem:nbhd}. Then $ L$ intersects three base triangles from $C_0$ and also from $C_1$ and the combinatorial neighborhood of these six base triangles contains $C_0$. So $C_0\subseteq N( L)$.

It remains to show that the collection $X$ of component cells intersecting $ L$ is convex.
Recall from the proof of Lemma~\ref{lem:component} that
$X$ is obtained by repeatedly gluging component cells; see Figure~\ref{fig:billiard} (a) and (c).
Let us consider the (abstract) piecewise hyperbolic cell complex $X'$ obtained by gluing copies of a component cell in the same manner. Since each interior angle of $X'$ is non-obtuse, we have a locally isometric embedding from $X'$ onto $X$. This implies that $X'$ and $X$ are isometric by Lemma~\ref{lem:BH}.
Since $X'$ is an infinite sided convex polygon, so is $X$.
\end{proof}

The following is now immediate from Lemma~\ref{lem:nbhd}.
\begin{lemma}\label{lem:sector}
Let $L$ be a Seidel Lagrangian in $E$ and $v$ be a vertex on the boundary of $N(L)$.
If $g_0$ and $g_1$ are the geodesic half-rays extending the two sides of $N(L)$ that share the vertex $v$,
then $g_0\cup g_1$ is disjoint from $ L$.
\end{lemma}

\begin{lemma}\label{lem:real line}
Let $ L$ be a Seidel Lagrangian in $E$.
\begin{enumerate}
\item 
$ L$ is topologically a real line.
\item
If $a\ge3$, then $ L$ can be realized by a closed geodesic in $\PO$.
\end{enumerate}
\end{lemma}

\begin{proof}
The proof of (1) is immediate since $ L$ ``fellow-travels'' a geodesic line. For the part (2), we project a geodesic path in $N( L)$ onto $\PO$.
\end{proof}

\begin{corollary}\label{cor:bigon_no}
Two Seidel Lagrangians in $E$ do not bound any bigon with convex corners for $(a,b,c) \geq 3$. If one of $a, b,c $ equals 2, then there exist a unique bigon (up to the $G$-action) with area $4\sigma$ contributing to the potential, and there are no other bigons with convex corners in $E$.
\end{corollary}
\begin{proof}
Suppose two Lagrangians $ L_0,  L_1$ in $E$ intersect along an edge $e$ connecting orbifold point $B$ and $C$. Denote by $h$ the geodesic extending $e$.
Let us consider a geodesic $g_i\subseteq N( L_i)$ given by Lemma~\ref{lem:component}.
Note that $g_i$ cannot coincide with $h$ for each $i=0,1$; see Figure \ref{fig:bigon_no} (a).
In particular, $g_i\cap h$ is a single point and $N( L_i) \cap h$ is given by the edge $e$ itself. Therefore, $ L_0$ and $ L_1$ cannot cross $h$ again.

In the case that one of $a,b,c$ equals 2,  there are two possible configurations as illustrated in Figure \ref{fig:bigon_no} (b) and Figure \ref{fig:bigon_no} (c).
The first case arises if  $B$ or $C$ is $\Z/2$-point, and the second case arises if $A$ is $\Z/2$-point. In the case of Figure \ref{fig:bigon_no} (b), the same proof as
above shows that there is no bigon. In the case of Figure \ref{fig:bigon_no} (c), the figure gives the unique bigon, contributing to the potential.
\end{proof}

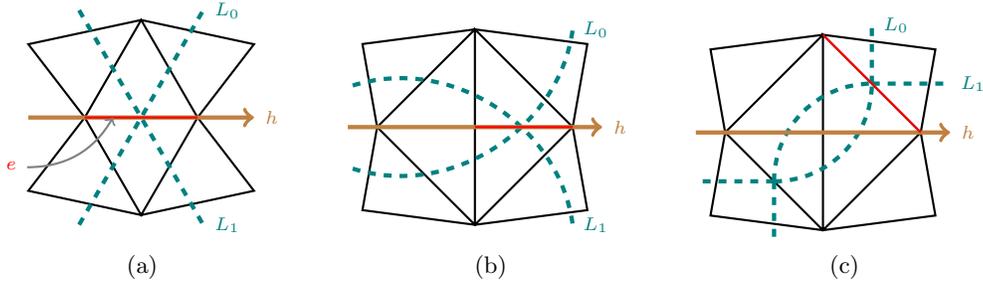
\begin{figure}[htb]
  \tikzstyle {a}=[red,postaction=decorate,decoration={%
    markings,%
    mark=at position 1 with {\arrow[red]{stealth};}}]
  \tikzstyle {b}=[blue,postaction=decorate,decoration={%
    markings,%
    mark=at position .85 with {\arrow[blue]{stealth};},%
    mark=at position 1 with {\arrow[blue]{stealth};}}]
  \tikzstyle {c}=[orange,postaction=decorate,decoration={%
    markings,%
    mark=at position .7 with {\arrow[orange]{stealth};},%
    mark=at position .85 with {\arrow[orange]{stealth};},
    mark=at position 1 with {\arrow[orange]{stealth};}
}]
  \tikzstyle {av}=[red,draw,shape=circle,fill=red,inner sep=2pt]
  \tikzstyle {bv}=[blue,draw,shape=circle,fill=blue,inner sep=2pt]
  \tikzstyle {cv}=[orange,draw,shape=circle,fill=orange,inner sep=2pt]
  \tikzstyle {wv}=[draw,shape=circle,fill=white,inner sep=2.5pt]  
  \tikzstyle {gv}=[draw,shape=circle,fill=black,inner sep=1pt]
  \tikzstyle {hv}=[draw,shape=circle,fill=white,inner sep=1pt]  
  \tikzstyle {ev}=[draw,teal,shape=circle,fill=teal,inner sep=1pt]
  \tikzstyle {pv}=[black,draw,shape=rectangle,fill=black,inner sep=2pt] 
  \tikzstyle {ar}=[teal,postaction=decorate,decoration={%
    markings,%
    mark=at position .35 with {\arrow[teal]{stealth};}}]
\subfloat[(a)]{
	\begin{tikzpicture}[scale=.75,thick]
	
	\draw (0,1.73) -- (-2,1.3) -- (-1,0) -- (0,1.73) -- (1,0) -- (2,1.3) -- (0,1.73);
	\draw (-1,0) -- (1,0) -- (2,-1.3) -- (0,-1.73) -- (-2,-1.3) -- (-1,0) -- (0,-1.73) -- (1,0);

	\draw [teal,dashed,ultra thick] (-1*1.1,-1.73*1.1) -- (1*1.1,1.73*1.1) node [right] {\tiny $ L_0$};
	\draw [teal,dashed,ultra thick] (-1*1.1,1.73*1.1) -- (1*1.1,-1.73*1.1) node [right] {\tiny $ L_1$};
	\draw [brown,ultra thick,->] (-2,0) -- (2,0) node [right] {\tiny $h$};

	\draw [red] (-1,0) -- (1,0);

	\draw (-2.3,-.6) node [red,left,below] {\tiny $e$} edge [gray,->,bend right=30] (-.5,0);

	\node  [inner sep=0.9pt] at (-2.5,0) {}; 	
	\node  [inner sep=0.9pt] at (2.5,0) {}; 			
	\end{tikzpicture}
}
$\quad$
\subfloat[(b)]{
	\begin{tikzpicture}[scale=.65,thick]
	
	\draw (0,2) -- (-2.3,1.7) -- (-2,0) -- (-2.3,-1.7) -- (0,-2) -- (2.3,-1.7) -- (2,0) -- (2.3,1.7) -- (0,2) -- (-2,0) -- (0,-2) -- (2,0) -- (0,2) -- (0,-2);
	
		\draw [teal,dashed,ultra thick,bend left=50] (-2.5,.85) edge (2,-2);
		\draw [teal,dashed,ultra thick,bend right=50] (-2.5,-.85) edge (2,2);
		\draw (2,2) node [teal,right] {\tiny $ L_0$};
		\draw (2,-2) node [teal,right] {\tiny $ L_1$};

	\draw [brown,ultra thick,->] (-2.6,0) -- (2.6,0) node [right] {\tiny $h$};

	\draw [red] (0,0) -- (2,0);


	\node  [inner sep=0.9pt] at (-2.5,0) {}; 	
	\node  [inner sep=0.9pt] at (2.5,0) {}; 		
	\end{tikzpicture}
}
$\quad$
\subfloat[(c)]{
	\begin{tikzpicture}[scale=.65,thick]
	
	\draw (0,2) -- (-2.3,1.7) -- (-2,0) -- (-2.3,-1.7) -- (0,-2) -- (2.3,-1.7) -- (2,0) -- (2.3,1.7) -- (0,2) -- (-2,0) -- (0,-2) -- (2,0) -- (0,2) -- (0,-2);
	
		\draw [teal,dashed,ultra thick,out=90,in=180] (-1,-1) edge (1,1);
		\draw [teal,dashed,ultra thick] (1,1) -- (2.6,1);
		\draw [teal,dashed,ultra thick,out=0,in=-90] (-1,-1) edge (1,1);
		\draw [teal,dashed,ultra thick]  (1,1)--(1,2.2) node [teal,right] {\tiny $ L_0$};
		\draw [teal,dashed,ultra thick]  (-1,-1)--(-2.6,-1);
		\draw [teal,dashed,ultra thick]  (-1,-1)--(-1,-2.2);
		\draw (2.6,1) node [teal,right] {\tiny $ L_1$};

			\draw [brown,ultra thick,->] (-2.6,0) -- (2.6,0) node [right] {\tiny $h$};

		\draw [red] (2,0) -- (0,2);


	\node  [inner sep=0.9pt] at (-2.5,0) {}; 	
	\node  [inner sep=0.9pt] at (2.5,0) {}; 		
	\end{tikzpicture}
	}
\caption{Corollary~\ref{cor:bigon_no}.}\label{fig:bigon_no}
\end{figure}

\subsection{Embeddedness of the polygons for the potential}
To compute the potential, we count the holomorphic polygons of Maslov index two with convex corners mapped to odd degree immersed generators. We will call them as {\em polygons for the potential} for short.  Such polygons are rigid (see \cite{S}, Section 13), and hence their countings can be defined.



\begin{lemma}
In the spherical case ($E=S^2$), the image of a lift $U$ of a polygon for the potential do not cover the whole $S^2$.
\end{lemma}
\begin{proof}
It is enough to prove that a holomorphic polygon $U$ in $S^2$ with boundary on $L_E$ whose image covers the whole $S^2$ has Maslov index greater than two.
For this, we first establish the following index formula (similar to the one in \cite{S}).
Consider two antipodal points $p_N, p_S$ of the universal cover $S^2$, 
and its complement $Y :=S^2 \setminus \{p_N, p_S\}$ admits a holomorphic volume form $\Omega= \frac{dz}{z}$,
where $z$ is the coordinate of $\CC \cong S^2 \setminus \{p_S\}$.
Consider (possibly immersed) Lagrangians in $Y$, which admit gradings with respect to $\Omega$
(Lagrangian do not bound a disc in $Y$).  Since $\Omega= \frac{dz}{z}$ has order 1 poles in $p_N$ and $p_S$, we have the following index formula. For the immersed generators $x_1,\cdots, x_n$, the expected dimension of moduli space of holomorphic polygons with
insertions $x_1, \cdots, x_n$ is given by
$$d= k - 3 - \sum_{i=1}^n \deg(x_i) + 2 (ord(u,p_N) + ord(u,p_S))$$
where $ord(u,p_N)$ (resp. $ord(u,p_S)$) is the order of intersection of the map $U$ at $p_N$ (resp. $p_S$).
Since $U$ is holomorphic, $ord(u,p_N)$ and $ord(u,p_S)$ are always non-negative.
If the image of $U$ covers $S^2$,  we have  $ord(u,p_N) + ord(u,p_S) \geq 2$.

Let us consider now the Seidel Lagrangians in $E=S^2$. We may choose two antipodal points in the pre-image of an orbifold point.
It is easy to see that every Seidel Lagrangian in $E$ when restricted to $Y$ winds around the cylinder $Y$ once if embedded, and twice if immersed. Also we can assign grading  to each Seidel Lagrangian to have value in $(0,2\pi)$.
This makes  each intersection point between two Lagrangians to have an integer grading in $[-1,2]$.

Note that we only consider odd degree intersections, and hence $\deg(x_i) = \pm 1$.
Hence $d \geq -3 + 2 (ord(u,p_N) + ord(u,p_S))$, where $d \geq 1$ if $U$ covers $S^2$.  The expected dimension is non-zero, and hence such a 
polygon $U$ whose image covers the whole $S^2$ cannot be a polygon for the potential.
\end{proof}
Hence, in the spherical case, we can take a point in $S^2$ which is not in the image of $U$,  and consider a stereographic projection to $\RR^2$. This means we may assume the polygon to be contained in a plane.

Now, we prove that lifts of polygons for the potential in the universal cover are embedded. 
We will use this to define $(a,b,c)$-diagram corresponding to the embedded image of a polygon for the potential.
We remark that we may instead define $(a,b,c)$-digram in the domain $D^2$ of an immersed polygon for the potential and obtain the same result in this paper.

First, in the spherical and elliptic cases, we will compute all polygons for the potential directly since there are finitely many of them in the spherical cases and there are finitely many types in the
elliptic cases.  The following lemma can be checked directly in these cases.  In the hyperbolic case, there are infinitely many types of polygons for the potential, and we need a more theoretical proof.

\begin{lemma}\label{lem:lift}
In the hyperbolic case, any immersed polygon $u$ for the potential lifts to a  polygon $U$ for the potential in $E$ which is an embedding.
\end{lemma}
\begin{proof}
Since $u$ can always be lifted to an immersion $U$ in $E$, it is enough to show that $U$ is an embedding.
Suppose $U: D^2 \to E$ is not embedded.
Consider the boundary $U(\partial D^2)$, which is an immersed curve in $E$.  This curve cannot be embedded, or otherwise by argument principle $U$ is an embedding.  Let $I \subset \partial D^2$ be an arc such that $U$ restricted to the interior of $I$ is injective and the two boundary points of $I$ are mapped by $U$ to the same point.  Depending on whether the image of $U(D^2)$ lies inside or outside of the curve $U(I)$,
we may assume that either $U(I)$ bounds a polygon for the potential with corners given by $x,y,z$ (the immersed corner $U(\partial I)$ should also be one of $x,y,z$-corners from the orientation consideration)
, or $U(I)$ bounds a polygon whose corners are complement of $x,y,z$-corners, possibly except at the immersed corner $U(\partial I)$ (see (a) and (b) of Figure \ref{fig:lemma42}).

\begin{figure}[htb!]
\includegraphics[height=.2\textheight]{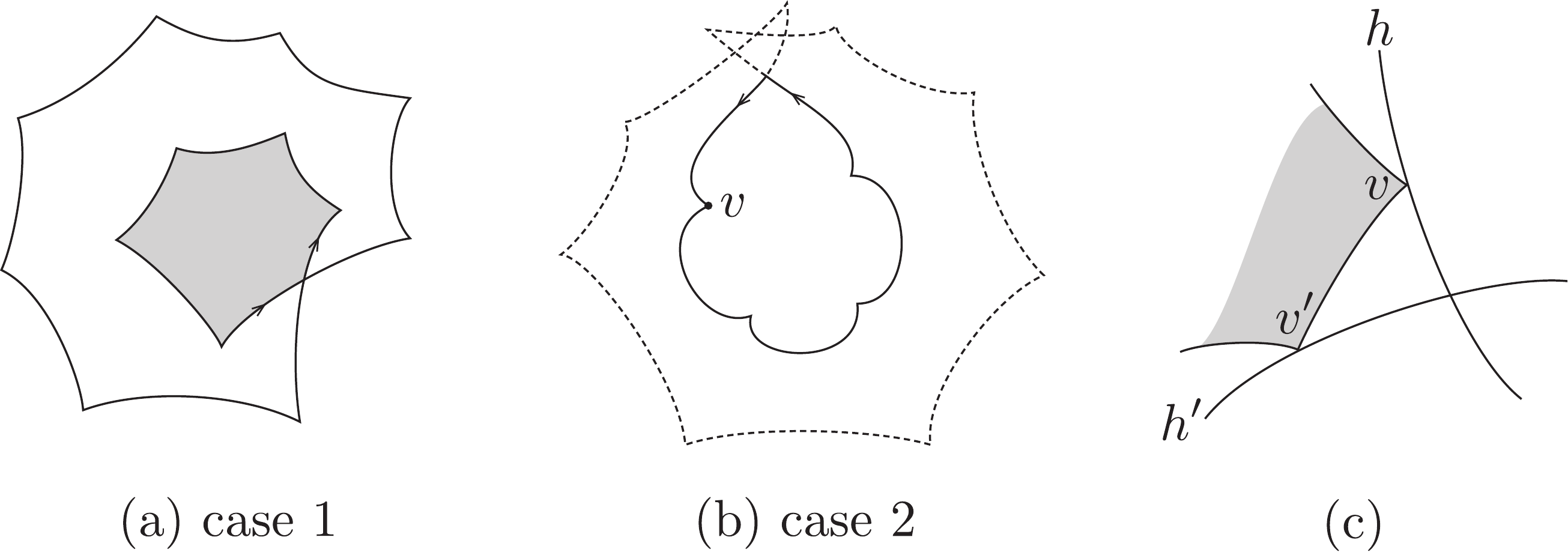}
\caption{$U(I)$ for (a) case 1 and (b) case 2}\label{fig:lemma42}
\end{figure}

In the first case as in (a) of Figure \ref{fig:lemma42}, one can notice that the Maslov index of $U$ is greater than 2, and hence such a map $U$ does not contribute to the potential.

We argue that the second case cannot occur based on the properties of Seidel Lagrangians in $E$.
Suppose $U(I)$ does not have any corners.  Then $U(I)$ is contained in one Seidel Lagrangian in $E$. This is a contradiction that Seidel
Lagrangian in $E$ is topologically a line (Corollary \ref{cor:bigon_no}).  Now suppose that $U(I)$ has at least one of $x,y,z$ corners $v$ as in (b) of Figure \ref{fig:lemma42}.  Take one of $x,y,z$-corners $v$ of $U(I)$, and we may assume that it does not have order two (otherwise as in Figure \ref{fig:bigon_no} (c), we will get a bigon).  Consider a geodesic $h$ extending the edge $e$ of the real equator intersecting the corner $v$.
The Seidel Lagrangians forming the corner $v$ do not intersect the geodesic $h$ again by Corollary \ref{cor:bigon_no}.  Even if we consider the possible next corner $v'$, it cannot cross a new geodesic $h'$ similarly.  Since we are on a plane, this shows that it is impossible for $U(I)$ to form such a polygon. This proves the lemma.
\end{proof}

\subsection{Boundary words of polygons for the potential}
We now characterize the words obtained by label-reading of boundaries of polygons.

An \emph{$x$-corner} of a polygon $U$ for the potential in $E$ is a corner of $U$ that is mapped to the immersed generator $X$.
We similarly define \emph{$y$-corner} and \emph{$z$-corner}. See Figure~\ref{fig:area} for examples.

\begin{figure}[htb!]
\includegraphics[height=.3\textheight]{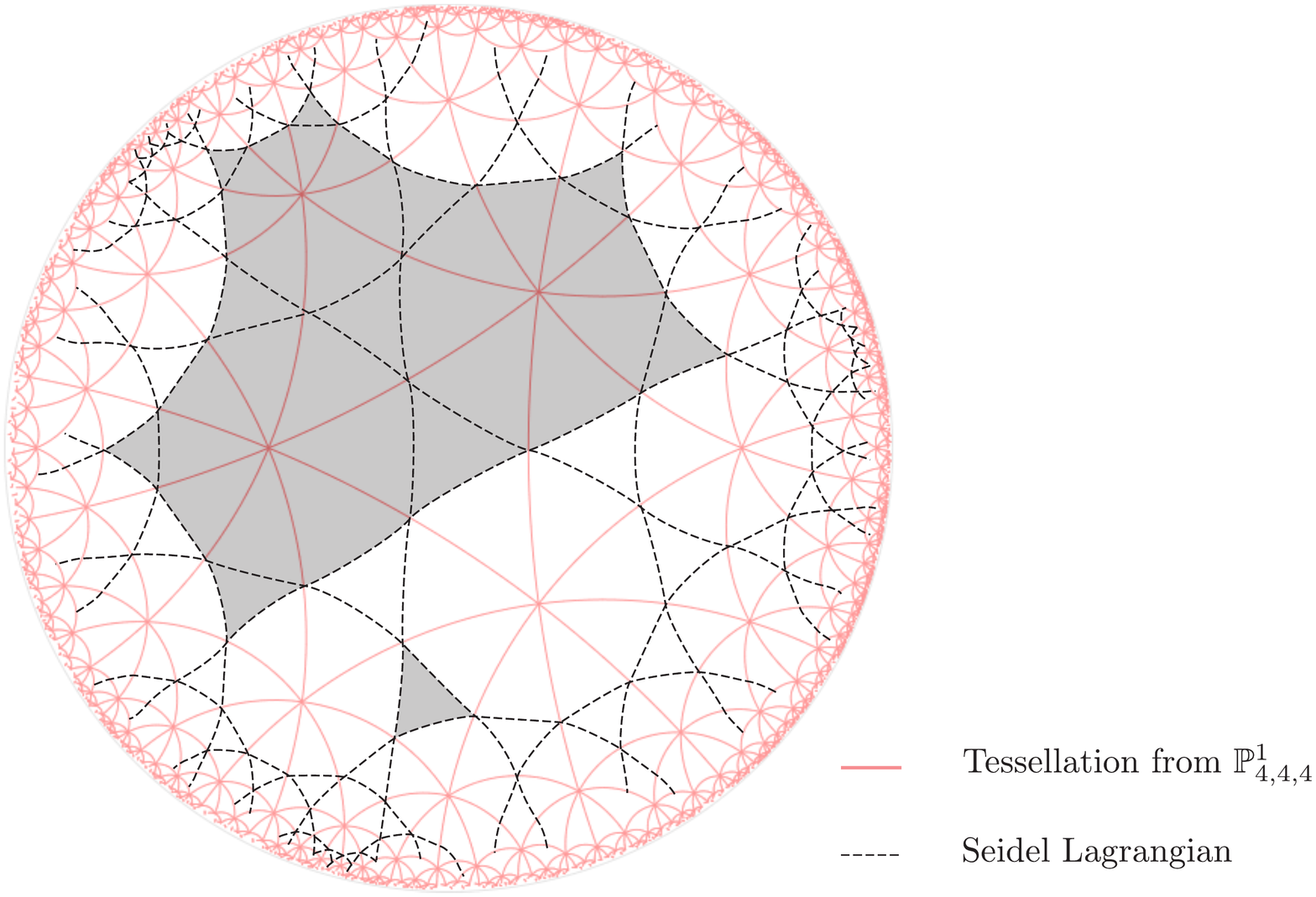}
\caption{Holomorphic polygons bounded by Seidel Lagrangians in $E$ }\label{fig:area}
\end{figure}

We explain an important property for polygons for the potential.
\begin{lemma}
Let $U$ be a polygon for the potential in $E$.
All corners of $U$ are locally contained in base triangles of the same color, which is either black or white.
\end{lemma}
\begin{figure}[htb!]
\includegraphics[width=.38\textheight]{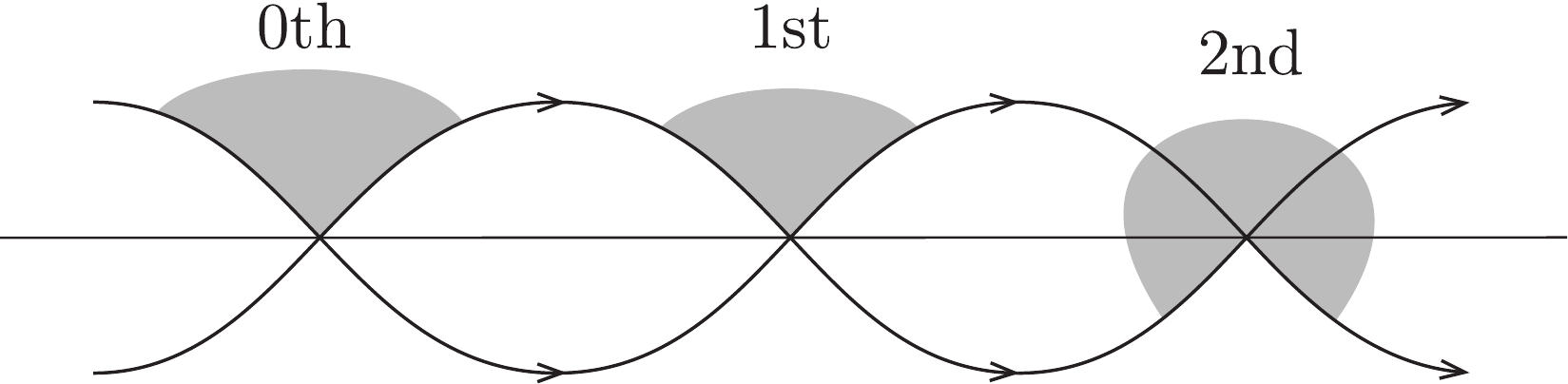}
\caption{Corners of a holomorphic polygon}\label{fig:uppercorner}
\end{figure}
\begin{proof}
The reason is that we are only allowed to have odd degree insertions for the polygons.  Starting from a corner (0-th) of $L$, if one turns again at the next (1st) corner, then these two corners occur
in the same base triangle. Suppose that we do not turn at the 1st corner and turn at the $2$nd corner instead.  But this corner has to be non-convex which is not allowed. (See Figure \ref{fig:uppercorner}.)
Following the same argument, only turning at $(2k+1)$-th corner is allowed for $k \geq 0$, which lies in the base triangle of the same color.
\end{proof}

A Seidel Lagrangian is an bi-infinite path reading $(\gamma\beta\alpha)^\infty$ in the Cayley graph of $G$.
So Lemma~\ref{lem:real line} (1) implies the following.

\begin{lemma}\label{lem:nontrivial}
Each nonempty subword of $(\gamma\beta\alpha)^\infty$ is nontrivial in $G$.
\end{lemma}

A polygon for the potential can be recorded by its boundary word defined in Definition \ref{def:labelreading}.
Each boundary word is obtained by reading labels on a simple loop in the Cayley graph of $G=\pi_1^{\mathrm{orb}}(\PO)$. In particular, we have:

\begin{lemma}\label{lem:trivial}
The boundary word of a polygon is trivial in $G$.
\end{lemma}

Let us now introduce a canonical way of representing a boundary word:

\begin{lemma}\label{lem:standard}
The boundary cyclic word of a polygon for the potential takes the form
$$w_1 w_2 w_3 w_4\cdots w_{2k}$$
such that the following hold.
\begin{enumerate}[(i)]
\item For each $i=1,2,\ldots, k$, we have $1\ne w_{2i-1}\preccurlyeq (\gamma\beta\alpha)^\infty$.
\item For each $i=1,2,\ldots, k$, the word $w_{2i}$ is a positive power of a word in $\{\alpha\beta,\beta\gamma,\gamma\alpha\}$.
\item For each $i=1,2,\ldots, 2k$, if $x$ is the last letter of $w_i$ and $y$ is the first letter of $w_{i+1}$ then $x = \tau(y)$. 
\item
Denote by $w'=w_1w_3 \cdots w_{2k-1}$ the collection of odd terms of the word $w$.
Then, $w'$ is a subword of $(\gamma \beta \alpha)^\infty$, and its word length is a multiple of 3.
\end{enumerate}
\end{lemma}
\begin{proof}
Label reading of the boundary of a  polygon can be decomposed into the following two phases. In one phase,
the boundary of a polygon turns at immersed points of $L_E$, and the word  from the turning at an immersed point
is given by one of  $\{\alpha\beta, \beta \gamma, \gamma \alpha \}$, or more generally the word
from  consecutive turns as in Figure \ref{fig:consecutiverefl}  is given by positive powers of 
$\{\alpha\beta, \beta \gamma, \gamma \alpha \}$. These give even numbered words $w_{2i}$'s. In the other phase,
 the boundary of a polygon travels along a branch $\WT{L}$ without turning and  the word is given by a finite sub-word of  $(\gamma\beta\alpha)^\infty$. These give odd numbered words $w_{2i-1}$'s. The rest of the assertion can be checked easily.
\end{proof}

The form $w=w_1 w_2 w_3 w_4\cdots w_{2k}$ in Lemma~\ref{lem:standard} is called the \emph{standard representative} (of the boundary word of a polygon for the potential in $E$).

\begin{definition}
Let $w=w_1 w_2 w_3 w_4\cdots w_{2k}$ be the standard representative of a boundary word of a polygon for the potential in $E$.  Let $w\rq{}$ be defined as in Lemma \ref{lem:standard}.

Define $[w]_0$ to be the number of $\alpha$ in $w'$, which equals the number of $\beta$ in $w'$, or the number of $\gamma$ in $w'$.  Then the word-length of $w\rq{}$ is $3[w]_0$.

Define $[w]_1$ to be the number of appearances of $\beta\gamma$ in $w$.
$[w]_1$ the number of $x$-corners of $U$.
We similarly define $[w]_2$ and $[w]_3$ as the counts of $\gamma\alpha$ and $\alpha\beta$ respectively, which are the numbers of $y$- and $z$-corners of $U$.

The number $k$ is called \emph{the number of runs} in $w$. We define $[w]_\rho=k$.
\end{definition}

For example, the following word is a standard representative of a polygon $U$ for the potential  coming from $\mathbb{P}^1_{3,4,5}$.
\[
w=\alpha \gamma (\beta\gamma)\beta (\alpha\beta)^{3}\alpha \gamma (\beta\gamma)\beta\alpha (\gamma\alpha)^{2}\gamma (\beta\gamma)\beta (\alpha\beta).\]
We see that $w\rq{} = (\alpha\gamma\beta)^3$, $[w]_0=3, [w]_1=3, [w]_2=2, [w]_3=4$
and $[w]_\rho=6$.


\section{Holomorphic discs and  $(a,b,c)$-diagram}\label{sec:abcdia}
In this section, we give a combinatorial description of a polygon for the potential, in terms of its $(a,b,c)$-diagram.
Recall that we defined  the hexagon tessellation $Z$ in $E$ in Definition \ref{def:hexZ}.
\begin{definition}\label{defn:diagram}
Let $U$ be a polygon for the potential in $E$.
The \emph{$(a,b,c)$-diagram for $U$} (or simply, a \emph{diagram}) is the cellulation of $U$ obtained by taking the intersection of $Z$ with $U$.
\end{definition}

In a diagram $X$, the interior vertices of $X$ inherit labels by $\{A,B,C,W\}$.
Moreover, the edges not on $\partial X$ inherit labels by $\CA=\{\alpha,\beta,\gamma\}$ and transverse orientations from $Z$.  The \emph{basic $(a,b,c)$-diagram} is an $(a,b,c)$-diagram with exactly one interior vertex; see Figure~\ref{fig:basic} (a).
The basic diagram corresponds to the minimal $xyz$ triangle.  A \emph{$0$-th generation $(a,b,c)$-diagram} is an $(a,b,c)$-diagram with exactly one colored vertex (and all other vertices are white). There are exactly three diagrams of the $0$-th generation depending on whether the unique colored vertex is an $A, B$ or $C$-vertex. Figure~\ref{fig:basic} (b) shows a $0$-th generation diagram with an $A$-vertex when $a=4$.

\begin{figure}[htb]
  \tikzstyle {av}=[red,draw,shape=circle,fill=red,inner sep=2pt]
  \tikzstyle {bv}=[blue,draw,shape=circle,fill=blue,inner sep=2pt]
  \tikzstyle {cv}=[orange,draw,shape=circle,fill=orange,inner sep=2pt]
  \tikzstyle {wv}=[black,draw,shape=circle,fill=white,inner sep=2.5pt]  
  \tikzstyle {gv}=[inner sep=0pt]
  \tikzstyle {pv}=[black,draw,shape=rectangle,fill=black,inner sep=2pt]  
\subfloat[(a)]{
	\begin{tikzpicture}[scale=0.6,thick]
   	\node [gv] at (0,0) (c) {};
   	\node [gv] at (120:2) (pp) {};
   	\node [gv] at (240:2) (qp) {};
   	\node [gv] at (0:2) (rp) {};

	\draw [black,ultra thick] (c) circle [radius=2];
	\draw [red,ultra thick] (pp) -- (c);
	\draw [blue,ultra thick] (qp) -- (c);
	\draw [orange,ultra thick] (rp) -- (c);

	\draw (pp) node [av] {};
	\draw (rp) node [cv] {};
	\draw (qp) node [bv] {};
	\draw (c) node [wv] {};
	\node  [inner sep=0.9pt] at (-2.5,0) {}; 	
	\node  [inner sep=0.9pt] at (2.5,0) {}; 			
	\end{tikzpicture}}
$\qquad\qquad\qquad$
\subfloat[(b)]{	
	\begin{tikzpicture}[scale=0.6,thick]
   	\node [gv] at (0,0) (c) {};
   	\node [gv] at (0:1.1) (w1) {};
   	\node [gv] at (90:1.1) (w2) {};
   	\node [gv] at (180:1.1) (w3) {};
   	\node [gv] at (270:1.1) (w4) {};
	\node [gv] at (20:2.05) (p1) {};
	\node [gv] at (-20:2.05) (q1) {};
	\node [gv] at (110:2.05) (p2) {};
	\node [gv] at (70:2.05) (q2) {};
	\node [gv] at (200:2.05) (p3) {};
	\node [gv] at (160:2.05) (q3) {};
	\node [gv] at (290:2.05) (p4) {};
	\node [gv] at (250:2.05) (q4) {};

	\draw [red,ultra thick] (w1) -- (c);
	\draw [red,ultra thick] (w2) -- (c);
	\draw [red,ultra thick] (w3) -- (c);
	\draw [red,ultra thick] (w4) -- (c);			
	\draw [orange,ultra thick] (p1) -- (w1);			
	\draw [blue,ultra thick] (q1) -- (w1);			
	\draw [orange,ultra thick] (p2) -- (w2);			
	\draw [blue,ultra thick] (q2) -- (w2);			
	\draw [orange,ultra thick] (p3) -- (w3);			
	\draw [blue,ultra thick] (q3) -- (w3);			
	\draw [orange,ultra thick] (p4) -- (w4);			
	\draw [blue,ultra thick] (q4) -- (w4);			
	\draw [black,ultra thick] (c) circle [radius=2];


	\draw (c) node [av] {};
	\draw (w1) node [wv] {};	
	\draw (w2) node [wv] {};	
	\draw (w3) node [wv] {};	
	\draw (w4) node [wv] {};	
	\node  [inner sep=0.9pt] at (-2.5,0) {}; 	
	\node  [inner sep=0.9pt] at (2.5,0) {}; 			
	\end{tikzpicture}}
\caption{The basic and the $0$-th generation $(a,b,c)$-diagrams.}\label{fig:basic}
\end{figure}
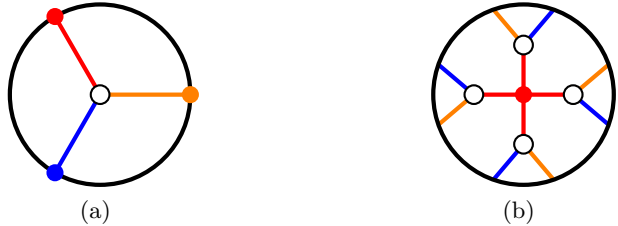

First we introduce some general terminologies. Figure~\ref{fig:tip} shows the $(4,4,4)$ diagram corresponding to the polygon in Figure~\ref{fig:area}.

\begin{definition}
Let $X'$ be an arbitrary cellulation of a closed disk.
We call a vertex on $\partial X'$ as a \emph{tip} and an edge containing a tip as a \emph{leaf}.
A non-tip vertex of a leaf is called a \emph{joint}. See Figure~\ref{fig:tip}.
A face $F$ is \emph{outer} if $F$ intersects $\partial X'$.
A face that is not outer is called \emph{inner}.
An edge contained in $\partial X'$ is called a \emph{boundary edge}.
An \emph{interior edge} means an edge that is neither a boundary edge nor a leaf.
A \emph{single joint} is a joint with exactly one leaf.
A \emph{double joint} is the one with two leaves.
The unique colored neighbor of a double joint is called a \emph{pivot}.
\end{definition}

\begin{figure}[htb]
\includegraphics[height=.16\textheight]{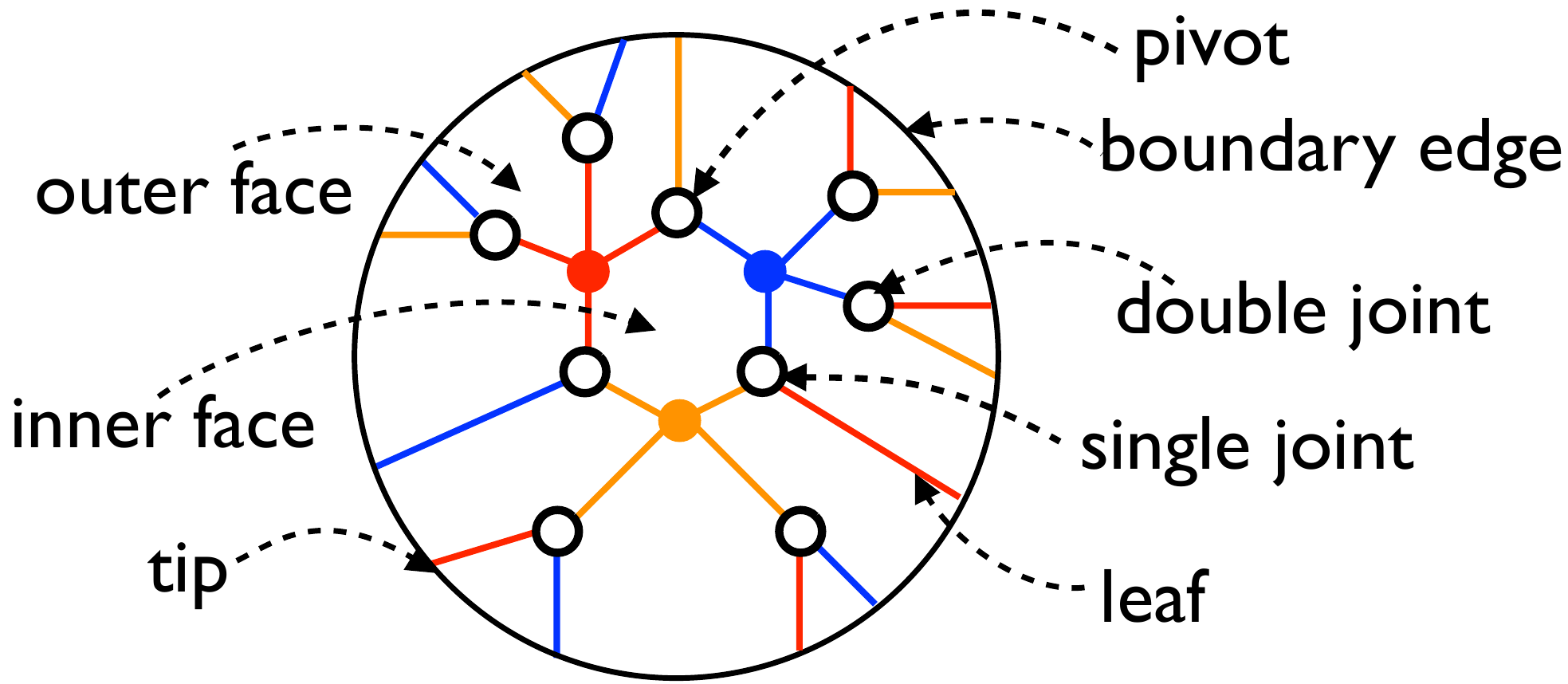}
\caption{Terminology on a cellulation.
Red (Apple), blue (Blueberry) and orange (Cantaloup) colors represent
$A,B$ and $C$ labels. }\label{fig:tip}
\end{figure}

We recall a well-known fact regarding piecewise Euclidean or hyperbolic cell complexes.
One can find a proof in~\cite[Proposition 4.14]{BH1999}, for example.

\begin{lemma}\label{lem:BH}
Suppose $E=\mathbb{R}^2$ or $\mathbb{H}^2$.
Let $S$ be a simplicial complex, which is homeomorphic to a possibly non-compact simply connected surface.
Suppose $S$ is equipped with a metric, such that each triangle is isometric to a geodesic triangle in $E$
and there are finitely many isometry types of triangles in $S$.
Assume that the cone angle is $2\pi$ at each interior vertex, and no greater than $\pi$ at each boundary vertex.
If $f:S\to E$ is a locally injective map which restricts to the isometry on each cell,
then $f$ is globally an isometric embedding.
Moreover, such $f$ always exists.
\end{lemma}

In the proof of the following lemma, we will show that every joint in an $(a,b,c)$ diagram is white.

\begin{lemma}\label{lem:diagram}
(1) 
Let $X$ be an $(a,b,c)$-diagram. Then
we have a partition $X^{(0)}\setminus\partial X = W\coprod A \coprod B \coprod C$
satisfying the following.
\begin{enumerate}[(i)]
\item
Each tip has valence three.
\item
Every interior edge joins a white vertex in $W$ to a colored vertex.
\item
For each white vertex $v$ there exist exactly three incident edges which are labeled as $(\alpha,\beta,\gamma)$ in the positive orientation around $v$.
\item
For each $A$-vertex ($B$-vertex or $C$-vertex) there exist exactly $a$ ($b$ or $c$ resp.) incident edges, and they are all labeled by $\alpha$ ($\beta$ or $\gamma$ resp.).
\item Each inner face is a hexagon.
\item\label{lem:diagram:outer} Each outer face is one of the following two types:
\begin{enumerate}[(A)]
\item a triangle formed by a boundary edge and two leaves emanating from a white joint.
\item a pentagon formed by a boundary edge, two leaves from two distinct white joints, and a length-two interior path joining those two joints.
\end{enumerate}
\end{enumerate}
(2) Conversely, suppose $X'$ is a cellulation of a disk equipped with a partition $X'^{(0)}\setminus \partial X' = W\coprod A\coprod B\coprod C$ such that the conditions from {(i)} through {(vi)} are satisfied. 
If $1/a+1/b+1/c\le1$, then $X'$ is an $(a,b,c)$-diagram and $X'^{(1)}\setminus\partial X'$ is connected.
\end{lemma}

\begin{proof}[Proof of Lemma~\ref{lem:diagram}]
(1) Suppose $X$ is an $(a,b,c)$-diagram for a polygon $U$ for the potential.
Since Seidel Lagrangians are transverse to the edges of $Z$, we have \textit{(i)}. The conditions from \textit{(ii)} through \textit{(v)} are clear from the construction of $Z$.
Each tip of $X$ is the intersection of $\partial U$ with an edge of $Z$.
We see that every joint is white.
In Figure~\ref{fig:diagram} we draw the white base triangle containing a tip and two neighboring black triangles, depending on whether the joint is single or double.
Now suppose an outer face $F$ has a boundary edge whose endpoints are $t_1$ and $t_2$ in this order. If $t_1$ and $t_2$ share a joint then $F$ is a triangle (Figure~\ref{fig:diagram} (b)). 
Otherwise, $F$ is a pentagon.

(2) 
We may assume that $a\le b\le c$. 
Intuitively, we will find a smallest convex collection $X''$ of base triangles which contains a holomorphic polygon corresponding to $X'$. For this, we let $A'$ be the set of tips which are incident with $\alpha$-leaves, and similarly define $B'$ and $C'$. 
We further cellulate $X'$ by joining each pair $\{x,y\}$ of neighbors of a $W$-vertex $v$ in such a way that the triangle $\{v,x,y\}$ contains no vertices in its interior.
By forgetting $W$-vertices, we obtain a combinatorial triangulation $X''$ from $X'$.
We metrize $X''$ by substituting a base triangle for each triangle in $X''$ such that each $A-$ or $A'$-vertex corresponds to the corner with dihedral angle $\pi/a$ and similar conditions hold for the other vertices. 

If $a=2$, we make further modification to $X''$ as follows. Suppose $x\in A'$ belongs to three triangles in $X''$. Let $y$ and $z$ be the neighbors of $x$ in $B'$ and $C'$, respectively; see Figure~\ref{fig:diagram} (c) and (d).
If $b=3$ and $y$ belongs to three triangles in $X''$, then we glue three additional base triangles about $y$ by reflection, as shown in Figure~\ref{fig:diagram} (c).
Otherwise, we glue one base triangle along $\{x,y,z\}$ as in Figure~\ref{fig:diagram} (d).
Let us still call the resulting piecewise hyperbolic or Euclidean 2-complex as $X''$.

From the valence condition, we have a well-defined map $f$ from $X''$ to $\mathbb{H}^2$ or $\mathbb{R}^2$ such that 
locally, $f$ is an isometric embedding and has a convex image.
By Lemma~\ref{lem:BH}, $f$ is globally an isometric embedding.
In particular, $f(X'')$ contains a holomorphic polygon $U$ for which $X'$ is an $(a,b,c)$-diagram.
Note that the condition \textit{(vi)} prohibits an outer face from having more than one boundary edges. This implies that $X^{(1)}\setminus \partial X$ is connected. 
\end{proof}

\begin{figure}[htb]
  \tikzstyle {av}=[red,draw,shape=circle,fill=red,inner sep=2pt]
  \tikzstyle {bv}=[blue,draw,shape=circle,fill=blue,inner sep=2pt]
  \tikzstyle {cv}=[orange,draw,shape=circle,fill=orange,inner sep=2pt]
  \tikzstyle {wv}=[black,draw,shape=circle,fill=white,inner sep=2pt]  
  \tikzstyle {gv}=[inner sep=0pt]
  \tikzstyle{every edge}=[-,draw]
\subfloat[(a) Single joint]{
	\begin{tikzpicture}[scale=0.75,thick]
	\draw (-2,0) node  [gv] (p)  {};
	\draw (0,0) node  [gv] (q)  {};
	\draw (2,0) node  [gv] (r)  {};
	\draw (-1,1.73) node  [gv] (pp)  {};
	\draw (1,1.73) node  [gv] (qq)  {};
	\draw (0,1.14) node [gv] (w) {};
	\draw (qq) -- (q) -- (p) -- (pp) -- (qq) -- (r) -- (q) -- (pp);
	\draw [red] (pp) -- (w);
	\draw [blue] (q) -- (w);
	\draw [orange] (qq) -- (w);

	\draw [dashed, teal,ultra thick,->] (-2,0.7) -- (2.2,0.7);
	\draw (1.8,1) node [above] {$U$};
	\draw [teal] (2.2,0.7) node [right] {$\partial U$};	
	\draw (pp) node [av] {};
	\draw (r) node [av] {};
	\draw (q) node [bv] {};	
	\draw (p) node [cv] {};
	\draw (qq) node [cv] {};
	\draw (w) node [wv] {};
	\node  [inner sep=0.9pt] at (-2.5,0) {}; 	
	\node  [inner sep=0.9pt] at (2.5,0) {}; 			
	\end{tikzpicture}}
$\qquad\qquad$
\subfloat[(b) Double joint]{
	\begin{tikzpicture}[scale=0.75,thick]
	\draw (-2,1.73) node  [gv] (p)  {};
	\draw (0,1.73) node  [gv] (q)  {};
	\draw (2,1.73) node  [gv] (r)  {};
	\draw (-1,0) node  [gv] (pp)  {};
	\draw (1,0) node  [gv] (qq)  {};
	\draw (0,0.8) node [gv] (w) {};
	\draw (qq) -- (q) -- (p) -- (pp) -- (qq) -- (r) -- (q) -- (pp);
	\draw [red] (pp) -- (w);
	\draw [blue] (q) -- (w);
	\draw [orange] (qq) -- (w);

	\draw [dashed, teal,ultra thick,->] (-1.2,2.078) -- (0,0) -- (1.2,2.078);
	\draw (0,2) node [above] {$U$};	
	\draw[teal] (1.2,2) node [above] {$\partial U$};	
	\draw (pp) node [av] {};
	\draw (r) node [av] {};
	\draw (q) node [cv] {};	
	\draw (p) node [bv] {};
	\draw (qq) node [bv] {};
	\draw (w) node [wv] {};
	\node  [inner sep=0.9pt] at (-2.5,0) {}; 	
	\node  [inner sep=0.9pt] at (2.5,0) {}; 				
	\end{tikzpicture}
}
\\
\vspace{.1in}
\subfloat[(c) Adding three triangles]{
	\begin{tikzpicture}[scale=0.75,thick]
	\draw (-3,0) node  [gv] (zzz)  {}  (-1,0) node [gv] (y) {}  (0,0) node [gv] (x) {} (1,0) node [gv] (yy) {};
	\draw (-1.5,-0.865) node [gv] (xp) {}  (-1.5,0.865) node [gv] (xpp) {};
	\draw (0,-1.73) node [gv] (z) {}  (0,1.73) node [gv] (zz) {} ;
	\draw (zzz) -- (yy) -- (zz) -- (zzz);
	\draw (xpp) -- (y) -- (zz) -- (z)  -- (yy);
	\draw [dashed] (zzz) -- (z) -- (y) -- (xp);

	\draw (.4,0) node [above] {$x$};
	\draw (-1,.2) node [above] {$y$};
	\draw (.4,-1.73) node [] {$z$};
	\draw (z) node [cv] {};
	\draw (zz) node [cv] {};
	\draw (zzz) node [cv] {};	
	\draw (x) node [av] {};
	\draw (xpp) node [av] {};	
	\draw (y) node [bv] {};	
	\draw (yy) node [bv] {};	
	\node  [inner sep=0.9pt] at (-4,0) {}; 	
	\node  [inner sep=0.9pt] at (2,0) {}; 		
	\end{tikzpicture}
}
$\qquad\qquad$
\subfloat[(d) Adding one triangle]{
	\begin{tikzpicture}[scale=0.75,thick]
	\draw (-1,0) node [gv] (y) {}  (0,0) node [gv] (x) {} (1,0) node [gv] (yy) {};
	\draw (0,-1.73) node [gv] (z) {}  (0,1.73) node [gv] (zz) {} ;
	\draw (zz) -- (z) -- (yy) -- (y) -- (zz) -- (yy);
	\draw [dashed] (z) -- (y);
	\draw (.4,0) node [above] {$x$};
	\draw (-1,.2) node [above] {$y$};
	\draw (.4,-1.73) node [] {$z$};
	\draw (z) node [cv] {};
	\draw (zz) node [cv] {};
	\draw (x) node [av] {};
	\draw (y) node [bv] {};	
	\draw (yy) node [bv] {};	
	\node  [inner sep=0.9pt] at (-2,0) {}; 	
	\node  [inner sep=0.9pt] at (2,0) {}; 			
	\end{tikzpicture}
}%
\caption{Lemma~\ref{lem:diagram}.}\label{fig:diagram}
\end{figure}

\begin{remark}
In part (\ref{lem:diagram:outer}) of the above lemma, 
outer pentagons can be further classified into three types.
They are pentagons with two single joints, pentagons with two double joints and pentagons with one single and one double joints.
\end{remark}

The following is immediate from Lemma~\ref{lem:diagram}.

\begin{lemma}\label{lem:diagram2}
There is a one-to-one correspondence between the set of  polygons for the potential, up to label-preserving automorphisms of $Z$
($G$-action) and the set of cellulations of a closed disk satisfying the conditions in Lemma~\ref{lem:diagram} (1), up to label-preserving combinatorial isomorphisms.
\end{lemma}


\section{Combinatorial Gauss-Bonnet formula}\label{s:area}
Let us fix positive integers $a,b,c\ge 2$ and consider the orbifold $\PO$ as in Figure~\ref{fig:orbifold}.
Recall that the Seidel Lagrangian $L_E$ divides each base triangle into four smaller triangles. 
We assume that these four triangles have the same area $\sigma$. In this section, we will find area formulas of a  polygon for the potential in terms of the following combinatorial data.
\begin{definition}\label{def:p}
Suppose $U$ is a polygon for the potential $E$. 
Denote by $v_W$ the number of white vertices contained in $U$.  Denote the number of $x$-, $y$-, $z$-corners of $U$ as $P, Q, R$ respectively.  
Moreover define $S = P+Q+R$.

If $U$ is the minimal $xyz$ triangle, we set $p = -1$.
Otherwise, we choose $p$ such that the word-length of $[w(\partial U)]$ is $2S+3p$.
i.e. $p = [\partial U]_0$.
\end{definition}
Note that for $[w] = [w(\partial(U)]$, we have $P = [w]_1, Q=[w]_2$ and $R=[w]_3$.

The above combinatorial data determine the area of a polygon for the potential, as stated in the lemma below.
The proof will be given at the end of this section.
\begin{lemma}\label{lem:area}
For  a polygon $U$ for the potential, its area is given by
$\area(U)=(8 v_W - 5S-8p)\sigma$.
\end{lemma}

\begin{theorem}\label{thm:area}
\begin{enumerate}
\item
If $\PO$ is spherical or hyperbolic, then the area of the polygon $U$ for the potential is given by
\[
\area(U) = \left(3(P+Q+R)+8 \frac{P/a+Q/b+R/c-1}{1-1/a-1/b-1/c}\right)\sigma.\]
\item
If $\PO$ is elliptic, then $P/a+Q/b+R/c=1$. In particular, $P=Q=R=1$ or $PQR=0$.
\end{enumerate}
\end{theorem}
\begin{remark}
If $\PO$ is given a hyperbolic metric,  then we would have
$\sigma = \pi(1-1/a-1/b-1/c)/4$. If $\PO$ is given a spherical metric, we have $\sigma=-\pi(1/a+1/b+1/c-1)/4$.
\end{remark}
%
%
%
%

 \begin{example}\label{ex:count}
\begin{enumerate}
\item
If $U$ is the minimal $xyz$ triangle, then we have $P=Q=R=1$.
In this case, Lemma~\ref{lem:area} recovers $\area(U)=\sigma$.
\item
Consider the holomorphic hexagon $U$ bounding the shaded hexagon in Figure~\ref{fig:area}.
We have 
\[[w(\partial U)] = [(\beta\gamma)^{2}\beta(\alpha\beta)^{2}\alpha(\gamma\alpha)^{2}\gamma].\]
The corresponding $(a,b,c)$ diagram is Figure~\ref{fig:tip}. 
We have $p=1,P=Q=R=2$ and $v_W=9$.
Lemma~\ref{lem:area} and Theorem~\ref{thm:area} both imply that $\area(U) =34\sigma$,
which can be checked by direct counting.
\end{enumerate}
\end{example}

%

For the rest of this section, we let $X$ be the $(a,b,c)$-diagram corresponding to a polygon $U$ for the potential and $w$ is a standard representative for $[w(\partial U)]$.
Lemma~\ref{lem:area} holds obviously for the basic diagram $X$.   
So we assume $X$ is not basic.

We will use the following notations.
\begin{enumerate}[(i)]
\item
$v_A,v_B,v_C,v_W$: $A,B,C,W$-vertices in $X$, respectively.
\item
We divide the set of non-leaf edges into three types.
An edge is \emph{exposed} if it is shared by an outer face and an inner face, 
An edge is \emph{immersed} if it is shared by two inner faces.
A non-leaf edge is \emph{dangling} if it is shared by two outer faces. Every dangling edge joins a pivot to a double joint.

 We denote by $e_x, e_m, e_d$ the number of  \emph{exposed},  \emph{immersed}, 
\emph{dangling} edges in $X$ respectively.
 \item
$f_i$: $i$-gon faces in $X$ for $i=3,5,6$.
\end{enumerate}
We put $\rho=[w]_\rho$, and note that we have $\rho$ pivots in $X$.
Also $3p$ and $S$ are the numbers of single and double joints, respectively.
The number of tips
is $3p+2S$, and this number coincides with that of leaves, that of the boundary edges and that of outer faces.

\begin{lemma}\label{lem:relations}
We have the following equations.
\begin{enumerate}
\item
$v_W-p =av_A+S-P=bv_B+S-Q=cv_C+S-R$.
\item
$3v_W =e_x+e_m+e_d+3p+2S$.
\item
$f_3 = S,f_5 = 3p+S$.
\item
$e_x =6p$, $e_d=S$.
\item
$6f_6 = e_x+2e_m$.
\item
$v_A+v_B+v_C-2v_W
+   f_3+f_5+f_6 =1$.
\end{enumerate}
\end{lemma}

\begin{proof}
(1) Place $a$ coins on each of the $A$-vertices in $X$, and uniformly propagate the coins into every neighbor of $A$-vertices. 
Then we will have one coin 
on each of the white vertices that do not intersect $\alpha$-leaves.
The number of the tips intersecting $\alpha$-leaves is $p + Q+R=p+S-P$.
Hence $a v_A+p+S-P = v_W$.

(2) Similarly to (1), place three coins on each of the white vertices and uniformly propagate to the adjacent \emph{edges}. 
Then every non-boundary edge will carry a coin.

(3) 
Each triangle of $X$ corresponds to a subword $\alpha\beta,\beta\gamma$ or $\gamma\alpha$ in $w$, and these subwords appear $S$ times. 
Recall that the number of outer faces is $3p+2S$.

(4)
There are $S-\rho$ pentagons which neighbors two triangles.
The number of pentagons which is adjacent to exactly one triangle is $2\rho$.
By (3) the number of pentagons which are not adjacent to a triangle is $3p-\rho$.
These three types of pentagons have $0,1$ and $2$ exposed edges, respectively.
Hence $e_x = 6p$. Note that $e_d$ is the number of double joints.

(5)
The number $6 f_6$ counts each immersed edge twice and each exposed edge once.

(6)
The number of vertices in $X$ is $v_A+v_B+v_C+v_W+3p+2S$,
that of edges is $e_x+e_m+e_d+6p+4S$
and that of faces is $f_3+f_5+f_6$. We conclude by the Euler characteristic formula $\chi(X)=1$.
\end{proof}

\begin{proof}[Proof of Lemma~\ref{lem:area}]
The Seidel Lagrangians divides  universal cover $E$ into $2a$-, $2b$-, $2c$-gons and triangles
having area $2a\sigma, 2b\sigma, 2c\sigma$ and $\sigma$, respectively.
Note that $U$ has $v_A$ copies of $2a$-gons, $v_B$ copies of $2b$-gons and $v_C$-copies of $2c$-gons. 
We find $v_W$ copies of a middle triangle inside white base triangles and $f_6$ copies inside black base triangles. So we have
\[\area(U) = (2a v_A + 2b v_B + 2c v_C + v_W+f_6)\sigma.\]
Lemma~\ref{lem:relations} (1) implies that 
$\area(U) = (7 v_W - 4S-6p+f_6)\sigma$.
From (2), (4) and (5) of the same lemma, we have
$ f_6 = v_W-2p-S$.
It follows that
$\area(U) = (8 v_W - 5S-8p)\sigma$.
\end{proof}
\begin{proof}[Proof of Theorem~\ref{thm:area}]
From (3) and (6) of Lemma~\ref{lem:relations} and $ f_6 = v_W-2p-S$,  we see
\[
1 = v_A+v_B+v_C-v_W +p+S.
\]
From (1) of the same lemma, we have
$v_A = (v_W+P-S-p)/a,v_B = (v_W+Q-S-p)/b$ and $v_C = (v_W+R-S-p)/c$.
Plugging in, we have
\[\left(1-\frac1a-\frac1b-\frac1c\right)v_W = \frac{P}a+\frac{Q}b+\frac{R}c+
\left(1-\frac1a-\frac1b-\frac1c\right)(p+S)
-1.\]
Let us first assume that $\PO$ is elliptic. 
From $1/a+1/b+1/c=1$,
we see the desired relation $P/a+Q/b+R/c=1$.
Moreover, we have either $P=Q=R=1$ or at least one of $P,Q$ or $R$ is zero.

Now assume $\PO$ is non-elliptic. We eliminate $v_W$ to see
\[
\area(U) = \left(8 \frac{P/a+Q/b+R/c-1}{1-1/a-1/b-1/c}+3S\right)\sigma.\qedhere\]
\end{proof}

The below is a consequence of the proof.
\begin{corollary}\label{cor:counts}
If $\PO$ is not an elliptic curve quotient and $U$ is a polygon for the potential,
then the following numbers corresponding to $U$ can be written as linear functions of $p,P,Q$ and $R$.
\begin{enumerate}
\item
The numbers of the $W$, $A$, $B$, $C$-vertices in $X$, respectively.
\item
The numbers of the exposed edges, the immersed edges, the dangling edges and the boundary edges, respectively.
\item
The number of triangles, pentagons and hexagons, respectively.
\end{enumerate}
\end{corollary}


\section{Elementary move}\label{s:move}
In order to compute the open Gromov-Witten potential in the elliptic or hyperbolic case, we need to enumerate infinitely many polygons for the potential.
As we have see in Section~\ref{sec:abcdia}, these polygons are in one-to-one correspondence with $(a,b,c)$ diagrams.
The purpose of this section is to describe an algorithm to generate all the $(a,b,c)$ diagrams.
The spherical case has finitely many polygons and will be treated in Section \ref{sect:S2}.
So throughout this section, we assume $1/a+1/b+1/c\le 1$.

We will see that all $(a,b,c)$ diagrams can be generated from the basic one by a sequence of  \emph{elementary moves}. The description of an elementary move for the case $a=2$ will require extra consideration, compared to the case $a,b,c\ge3$.

Recall that a Seidel Lagrangian in $E$ is an oriented curve reading the bi-infinite word $(\gamma\beta\alpha)^\infty$. Among the two closed half-spaces of $E$ that are cut by a Seidel Lagrangian $ L$, we call the one on the left of $ L$ as the \emph{left side} of $ L$. Given an $(a,b,c)$-diagram $X\subseteq E$ and a Lagrangian $ L$, the intersection of $X$ and the left side of $ L$ will be called as the \emph{left side of $ L$ in $X$}. We similarly define the \emph{right side of $ L$ (in $X$)}. Note that the left side of $ L$ in $X$ is also an $(a,b,c)$-diagram.

\begin{definition}\label{defn:move}
Suppose $X$ is an $(a,b,c)$-diagram and $ L$ is a Seidel Lagrangian intersecting the interior of $X$.
Let $X_0$ be the left side of $ L$ in $X$.
If there does not exist an $(a,b,c)$-diagram $X'$ satisfying $X_0\subsetneq X'\subsetneq X$,
then we say $X$ is obtained from $X_0$ by an \emph{elementary move along $ L$}.
\end{definition}

See Figure~\ref{fig:generic_point} for an example. Note that the three $0$-th generation diagrams are obtained from the basic diagram by elementary moves. Also, Figure~\ref{fig:tip} shows a diagram that is obtained from a $0$-th generation diagram by one elementary move.

\begin{figure}[htb]
\includegraphics[width=.5\textheight]{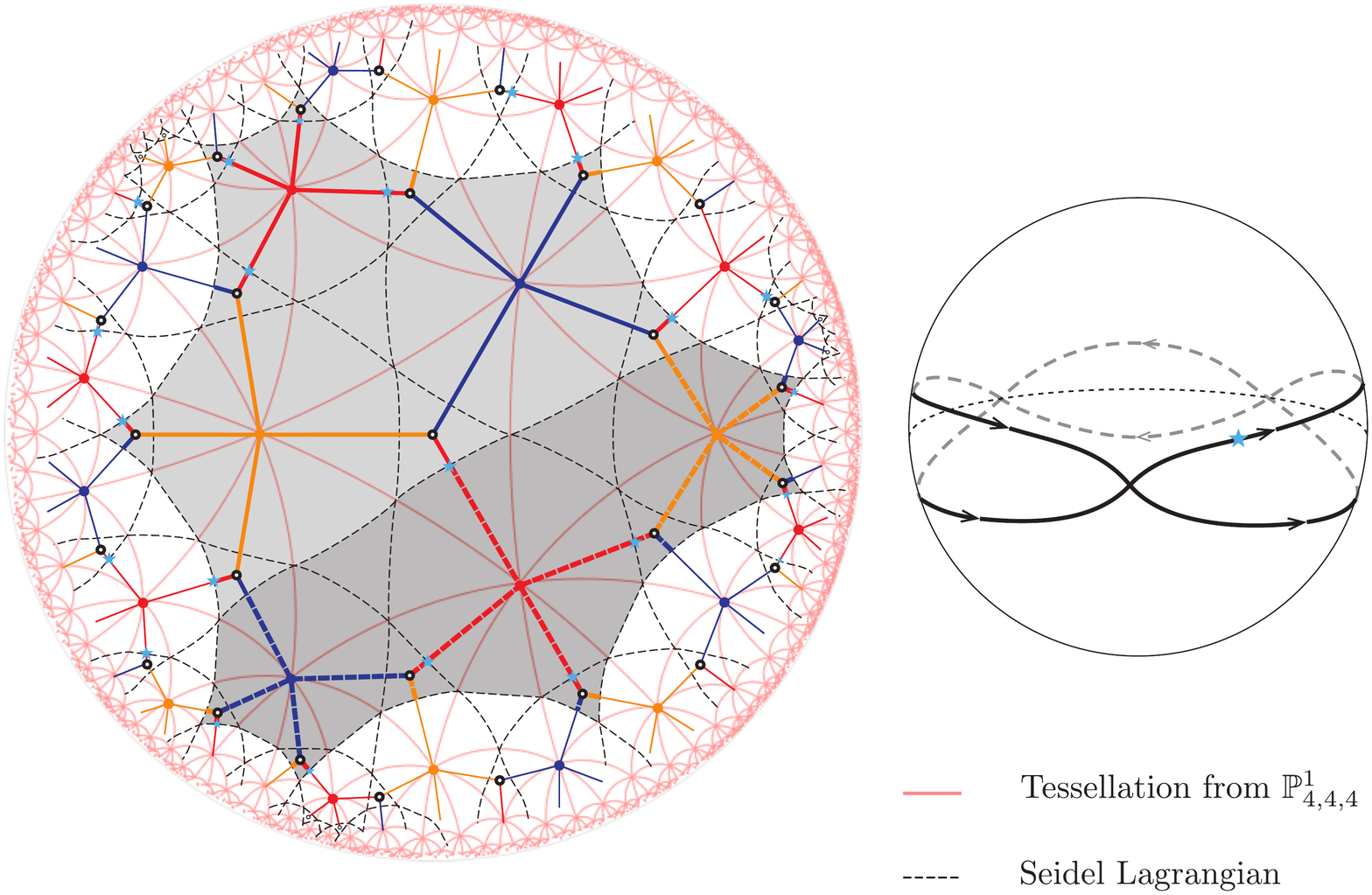}
\caption{An elementary move from a hexagon in a $(4,4,4)$-diagram to a 9-gon.}\label{fig:generic_point}
\end{figure}

To describe an elementary move in a more constructive way, let us define terminology regarding faces ``near'' from a Lagrangian. Let $ L$ be a Seidel Lagrangian in $E$. A face $F$ in $Z$ is called a \emph{Type I face along $ L$} if its interior intersects $ L$; see Figure~\ref{fig:types} (a). 

Now assume $a=2$. If a pair of Type I faces $F_1$ and $F_1'$ along $ L$ share an $A$-vertex, there uniquely exists a face $F_2$ in $Z$ on the right side of $ L$ such that $F_2$ intersects both $F_1$ and $F_1'$. Such a face $F_2$ is called a \emph{Type II face along $ L$}. See Figure~\ref{fig:types} (b) and (c). Let us further assume $b=3$. For each Type II face $F_2$, there uniquely exists a face $F_3$ on the right side of $ L$ sharing an $A$-vertex with $F_2$ and such a face $F_3$ is called a \emph{Type III face along $ L$}. Lastly, a \emph{Type IV face along $ L$} is a face on the right side of $ L$ that shares a white vertex with a Type II face and also with a Type III face along $ L$. A Type I, II, III or IV face is collectively called a \emph{typed face along $ L$}.

\begin{figure}[htb]
  \tikzstyle {av}=[red,draw,shape=circle,fill=red,inner sep=1pt]
  \tikzstyle {bv}=[blue,draw,shape=circle,fill=blue,inner sep=1pt]
  \tikzstyle {cv}=[orange,draw,shape=circle,fill=orange,inner sep=1pt]
  \tikzstyle {wv}=[black,draw,shape=circle,fill=white,inner sep=1pt]  
  \tikzstyle {gv}=[inner sep=0pt]
  \tikzstyle {v}=[gray,fill=gray,draw,shape=rectangle,inner sep=1pt]    
  \tikzstyle{every edge}=[-,draw]
\subfloat[(a) $a,b,c\ge3$]{
	\begin{tikzpicture}[scale=0.5,thick]

	\foreach \x in {-8,-2} {\draw [blue] (\x,3) -- (\x+2,3); }
	\foreach \x in {-6} {\draw [red] (\x,3) -- (\x+2,3); }	
	\foreach \x in {0} {\draw [red] (\x,3) -- (\x+1,3); }		
	\foreach \x in {-4} {\draw [orange] (\x,3) -- (\x+2,3); }		
	\foreach \x in {-9} {\draw [orange] (\x,3) -- (\x+1,3); }		
	
	\foreach \x in {-6,0} {\draw [orange] (\x,1) -- (\x,3); }		
	\foreach \x in {-4} {\draw [blue] (\x,1) -- (\x,3); }		
	\foreach \x in {-8,-2} {\draw [red] (\x,1) -- (\x,3); }		

	\foreach \x in {-9,-3} {\draw [red] (\x,1) -- (\x+2,1); }
	\foreach \x in {-5} {\draw [blue] (\x,1) -- (\x+2,1); }	
	\foreach \x in {-7,-1} {\draw [orange] (\x,1) -- (\x+2,1); }	


	\foreach \x in {-8,-6,-4,-2,0} {\draw (\x,3) node [wv] {}; }
	\foreach \x in {-7,-5,-3,-1} {\draw (\x,1) node [wv] {}; }	

	\foreach \x in {-5} {\draw (\x,3) node [av] {}; }	
	\foreach \x in {-7,-1} {\draw (\x,3) node [bv] {}; }	
	\foreach \x in {-3} {\draw (\x,3) node [cv] {}; }	
	\foreach \x in {-6,0} {\draw (\x,1) node [cv] {}; }	
	\foreach \x in {-4} {\draw (\x,1) node [bv] {}; }	
	\foreach \x in {-8,-2} {\draw (\x,1) node [av] {}; }		
	\draw [dashed, teal,ultra thick,->] (-9,2) -- (1,2) node [teal,right] {\small $ L$};

	\foreach \x in {-7,-5,-3,-1} {\draw (\x,2.5) node [] {\tiny I}; }

	\node  [inner sep=0.9pt] at (-2.5,0) {}; 	
	\node  [inner sep=0.9pt] at (2.5,0) {}; 			
	\end{tikzpicture}
}
$\quad$ 
\subfloat[(b) $a=2,b,c\ge4$]{
	\begin{tikzpicture}[scale=0.5,thick]
	\draw [blue]  (-2,4)--(0,4) (4,4)--(6,4);
	\draw [blue] (-4,1) -- (-4,4) (2,1)--(2,4);
	\draw [blue] (0,1)--(0,0) (6,0)--(6,1);

	\foreach \x in {-5,1} {	\draw [blue] (\x,2) -- (\x+3,2);}

	\foreach \x in {0} {	\draw [red] (\x,4) -- (\x+2,4);}
	\foreach \x in {6} {	\draw [red] (\x,4) -- (\x+1,4);}	
	\foreach \x in {-2,4} {	\draw [red] (\x,2) -- (\x,4);}
	
	\draw [red] (-5,4)--(-4,4) (-4,1)--(0,1) (1,0)--(1,2) (2,1)--(6,1);	

	\foreach \x in {0,6} {\draw [orange] (\x,1) -- (\x,4); }
	\foreach \x in {-4,2} {\draw [orange] (\x,0) -- (\x,1); }
	\foreach \x in {-4,2} {\draw [orange] (\x,4) -- (\x+2,4); }

	\foreach \x in {-2,4} {\draw [orange] (\x,2) -- (\x+3,2); }
	
	\foreach \x in {-4,-2,0,2,4,6} {\draw (\x,4) node [wv] {}; }
	\foreach \x in {-2,1,4} {\draw (\x,2) node [wv] {}; }	
	\foreach \x in {-4,0,2,6} {\draw (\x,1) node [wv] {}; }	
	\foreach \x in {1} {\draw (\x,4) node [av] {}; }	
	\foreach \x in {-2,4} {\draw (\x,2.5) node [av] {}; }	
	\foreach \x in {-2,4} {\draw (\x,1) node [av] {}; }	

	\foreach \x in {-1,5} {\draw (\x,4) node [bv] {}; }	
	\foreach \x in {-4,2} {\draw (\x,2) node [bv] {}; }	

	\foreach \x in {-3,3} {\draw (\x,4) node [cv] {}; }	
	\foreach \x in {0,6} {\draw (\x,2) node [cv] {}; }	

	\draw [dashed, teal,ultra thick,->] (-5,3) -- (7,3) node [teal,right] {\small $ L$};

	\foreach \x in {-3,-1,1,3,5} {\draw [below] (\x,4) node [] {\tiny I}; }
	\foreach \x in {-2,4} {\draw (\x,1.5) node [] {\tiny II}; }
	\foreach \x in {-4.3,.3,1.7,6.3} {\draw (\x,1.7) node [] {\tiny ...}; }

	\node  [inner sep=0.9pt] at (-2.5,0) {}; 	
	\node  [inner sep=0.9pt] at (2.5,0) {}; 			
	\end{tikzpicture}
}
\\
\subfloat[(c) $a=2,b=3,c\ge6$]{
	\begin{tikzpicture}[scale=0.5,thick]
	\draw (0,0) node [gv] (p00) {};
	\draw (0,1) node [gv] (p01) {};
	\draw (0,2) node [gv] (p02) {};	
	\draw (0,3) node [gv] (p03) {};
	\draw (0,4) node [gv] (p04) {};
	\draw (0,6) node [gv] (p06) {};	
	\draw (1,1) node [gv] (p11) {};
	\draw (1,2) node [gv] (p12) {};	
	\draw (1,3) node [gv] (p13) {};
	\draw (1,4) node [gv] (p14) {};
	\draw (1,6) node [gv] (p16) {};	
	\draw (2,0) node [gv] (p20) {};
	\draw (2,1) node [gv] (p21) {};
	\draw (2,4) node [gv] (p24) {};
	\draw (2,6) node [gv] (p26) {};	
	\draw (3,6) node [gv] (p36) {};		
	\draw (4,4) node [gv] (p44) {};
	\draw (4,4.5) node [gv] (p45) {};
	\draw (4,6) node [gv] (p46) {};	
	\draw (5,0) node [gv] (p50) {};
	\draw (5,6) node [gv] (p56) {};	
	\draw (6,0) node [gv] (p60) {};	
	\draw (6,3) node [gv] (p63) {};
	\draw (6,4) node [gv] (p64) {};
	\draw (6,6) node [gv] (p66) {};	
	\draw (7,4) node [gv] (p74) {};

	\draw (-6,0) node [gv] (m60) {};
	\draw (-6,1) node [gv] (m61) {};
	\draw (-6,2) node [gv] (m62) {};	
	\draw (-6,3) node [gv] (m63) {};
	\draw (-6,4) node [gv] (m64) {};
	\draw (-6,6) node [gv] (m66) {};	
	\draw (-5,1) node [gv] (m51) {};
	\draw (-5,2) node [gv] (m52) {};	
	\draw (-5,3) node [gv] (m53) {};
	\draw (-5,4) node [gv] (m54) {};
	\draw (-5,6) node [gv] (m56) {};	
	\draw (-4,0) node [gv] (m40) {};
	\draw (-4,1) node [gv] (m41) {};
	\draw (-4,4) node [gv] (m44) {};
	\draw (-4,6) node [gv] (m46) {};	
	\draw (-3,6) node [gv] (m36) {};		
	\draw (-2,4) node [gv] (m24) {};
	\draw (-2,4.5) node [gv] (m25) {};
	\draw (-2,6) node [gv] (m26) {};	
	\draw (-1,0) node [gv] (m10) {};
	\draw (-1,6) node [gv] (m16) {};	
	\draw (-7,0) node [gv] (m70) {};
	\draw (-7,2) node [gv] (m72) {};
	\draw (-7,3) node [gv] (m73) {};
	\draw (-7,4) node [gv] (m74) {};

	\foreach \x in {-2,4} {\draw [red] (\x,4) -- (\x,6); }

	\draw [red] (-7,0) edge (m63);	
	\draw [red] (m66) -- (m46) (p06) -- (p26) (p66) -- (7,6);
	\draw [red] (m52) -- (m54) (m61) -- (m41) (-1,0) -- (p03) (p01) -- (p21) (p12) -- (p14) (5,0)--(p63);	

	\draw [red] (6,1)--(7,1);
	\draw [blue] (6,2)--(7,2);

	\draw [blue] (-7,6) -- (m66);
	\draw [blue] (m63) -- (m61);

	\draw [blue] (m62) -- (m52) (m54) -- (m24) (m44) -- (m46) (m41) -- (-4,0) (p01) -- (p03) (p02)--(p12) (p14)--(p44) (p24)--(p26) (p21)--(2,0) (p46) -- (p66) (6,1)--(p63);
	\draw [blue] (m26) -- (p06);
	\draw [blue] (-4,0) -- (m41);
	
	\foreach \x in {-2,4} {\draw [orange] (\x,4) -- (\x+3,4); }
	\foreach \x in {-7} {\draw [orange] (\x,4) -- (\x+2,4); }
	\foreach \x in {-4,2} {\draw [orange] (\x,6) -- (\x+2,6); }
	\draw [orange] (-6,3) edge (-6,6) ;
	\draw [orange] (-6,1) edge (-6,0);
	\draw [orange] (6,1) edge (6,0);
	
	\draw [orange] (-6,4) edge (-7,3.4) edge (-7,3) ;
	\draw [orange] (p66)--(p63);
	\draw [orange] (p64) edge (p12) edge (p21);
	\draw [orange] (p04) edge (m52) edge (m41);
	\draw [orange] (p06) -- (p03);
	\draw [orange] (0,0) -- (p01);
	\draw (-.5,3) node [] {\tiny ...};
	\draw (5.5,3) node [] {\tiny ...};
	\draw (-6.5,2.5) node [] {\tiny ...};

	\foreach \x in {-6,-4,-2,0,2,4,6} {\draw (\x,6) node [wv] {}; }	
	\foreach \x in {-5,-2,1,4} {\draw (\x,4) node [wv] {}; }	
	\foreach \x in {-6,0,6} {\draw (\x,3) node [wv] {}; }	
	\foreach \x in {-5,1} {\draw (\x,2) node [wv] {}; }	
	\foreach \x in {-6,-4,0,2,6} {\draw (\x,1) node [wv] {}; }	

	\foreach \x in {-2,4} {\draw (\x,4.5) node [av] {}; }	

	\foreach \x in {-5,1} {\draw (\x,6) node [av] {}; }
	\foreach \x in {-5,1} {\draw (\x,3) node [av] {}; }
	\foreach \x in {-5,1} {\draw (\x,1) node [av] {}; }

	\foreach \x in {-1,5} {\draw (\x,6) node [bv] {}; }
	\foreach \x in {-4,2} {\draw (\x,4) node [bv] {}; }	
	\foreach \x in {-6,0,6} {\draw (\x,2) node [bv] {}; }

	\draw (m64) node [cv] {};
	\foreach \x in {-3,3} {\draw (\x,6) node [cv] {}; }
	\foreach \x in {0,6} {\draw (\x,4) node [cv] {}; }
	
	\draw [dashed, teal,ultra thick,->] (-7,5) -- (7,5) node [teal,right] {\small $ L$};

	\foreach \x in {-5,-3,-1,1,3,5} {\draw [below] (\x,6) node [] {\tiny I}; }
	\foreach \x in {-4,2} {\draw (\x,3) node [] {\tiny II}; }

	\foreach \x in {-5.5,.5} {\draw (\x,3) node [] {\tiny III}; }
	\foreach \x in {-4,2} {\draw (\x,1.8) node [] {\tiny IV}; }
						
	\node  [inner sep=0.9pt] at (-2.5,0) {}; 	
	\node  [inner sep=0.9pt] at (2.5,0) {}; 			
	\end{tikzpicture}
}
\caption{Three types of faces in $Z$ with respect to a Seidel Lagrangian $L$.}\label{fig:types}
\end{figure}
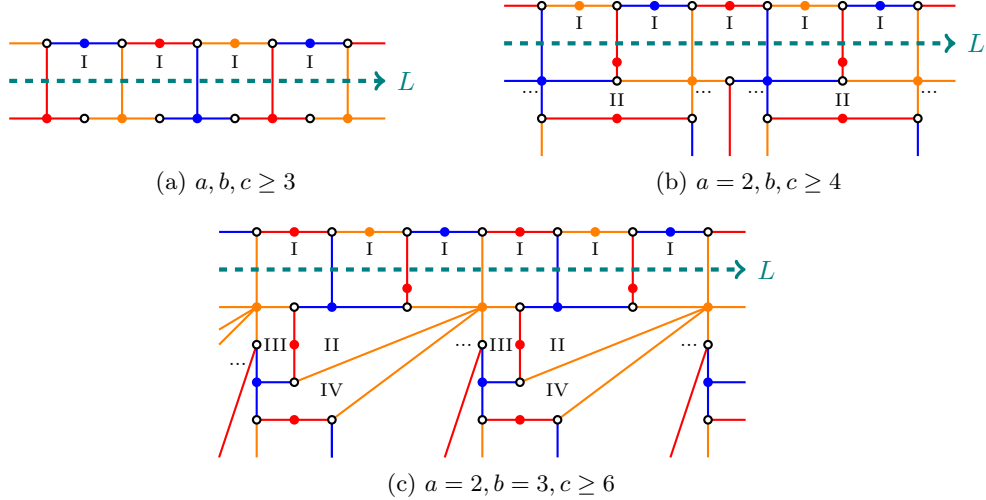

The following is a crucial inductive step for Thoerem~\ref{thm:move}.

\begin{lemma}\label{lem:hexs}
Let $ L$ be a Seidel Lagrangian and $F'$ be a face in the right side of $ L$
such that $F'$ is not typed along $ L$.
If $F'$ intersects a face which is typed along $ L$, then there exists a Seidel Lagrangian $ L'$ intersecting $F'$ such that the right side of $ L$ strictly contains the right side of $ L'$.
\end{lemma}

\begin{proof}
From the assumption and Figure~\ref{fig:move}, we observe that if $F'$ intersects a typed face $F$ then $F\cap F'$ contains a unique colored vertex $v$. First note that we have five possible cases as described below. In each of the cases,  we let $ L'$ be the unique Seidel Lagrangian intersecting $F'$ such that $v$ is in the left side of of $ L'$. 

\textbf{Case 1.} $a\ge3$ and $F'$ intersects with a Type I face $F$.

By symmetry, we may assume $F\cap F'$ contains a $C$-vertex $v$.
We let $k_1$ (or $k_2$) be the number of the faces sharing $v$ on a region between $F$ and $F'$ as shown in the Figure 18(a).
Since at least one of $k_1$ and $k_2$ is non-zero, we may assume $k_1>0$.
We let $F_2$ be the Type I face along $ L$ such that $F\cap F_2$ is an edge containing $v$.
A (part of) dual tessellation by real equators is shown in Figure~\ref{fig:hexs} (d) where $F_1=F$ and $u=v$.

\textbf{Case 2.} $a=2,b\ge4$ and $F'$ intersects with a Type I face $F$.

We may again assume $F\cap F'$ contains a $C$-vertex $v$.
Figures~\ref{fig:hexs} (a) and (e) illustrate this situation as well as the dual picture, with $u=v, F_0=F,F_0'=F'$ and $k_1,k_2\ge1$. 

\textbf{Case 3.} $a=2,b\ge4$ and $F'$ intersects with a Type II face $F$.

We can assume Case 2 did not occur, so $F\cap F'$ contains an $A$-vertex $v$ as in Figure~\ref{fig:hexs} (b). We let $u$ be the unique $C$-vertex of $F$. The dual picture is Figure~\ref{fig:hexs} (e) where $F_1=F$ and $F_1'=F'$. 

\textbf{Case 4.} $a=2,b=3$ and $F'$ intersects with a Type I face $F$.

Note that $F\cap F'$ contains a $C$-vertex $v$ as in Figure~\ref{fig:hexs} (a) with $k_1,k_2\ge 2$.
The dual figure is (f) with $u=v,F_0'=F',F_0=F$ and $k'\ge0$.

\textbf{Case 5.} $a=2,b=3$ and $F'$ intersects with a Type II, III or IV face $F$.

We assume Case 4 did not occur. Since every non-typed face intersecting a Type II face also intersects a Type I face, we see that $F$ is Type III or IV.
In Figure~\ref{fig:hexs} (d) and its dual (f), $F$ can be either $F_1$ or $F_2$, and $F'$ can be either $F_1'$ or $F_2'$. Let us mark the $C$-vertex of $F_1$ as $u$.
Note that we have $k'=0$ in (f).

In all cases, we let $g_0$ and $g_1$ be the two geodesic half-rays at $u$ which extend sides of the combinatorial neighborhood $N( L)$, as shown in the figures. Let us similarly define $g_0'$ and $g_1'$. Lemma~\ref{lem:sector} implies that the interiors of the combinatorial neighborhoods of $ L$ and $ L'$ are separated by $g_0\cup g_1$ and $g_0'\cup g_1'$.
\end{proof}

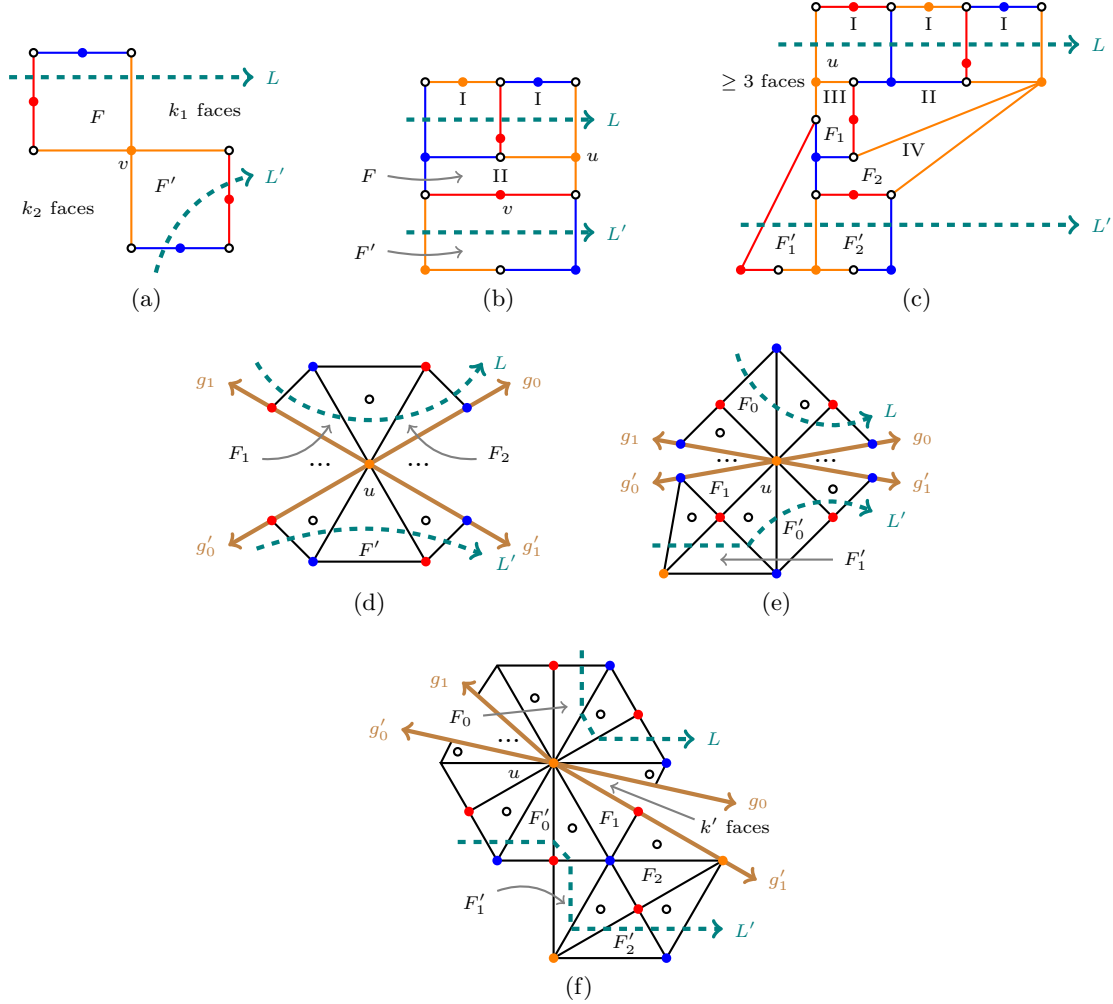
\begin{figure}[htb]
  \tikzstyle {av}=[red,draw,shape=circle,fill=red,inner sep=1pt]
  \tikzstyle {bv}=[blue,draw,shape=circle,fill=blue,inner sep=1pt]
  \tikzstyle {cv}=[orange,draw,shape=circle,fill=orange,inner sep=1pt]
  \tikzstyle {wv}=[black,draw,shape=circle,fill=white,inner sep=1pt]  
  \tikzstyle {gv}=[inner sep=0pt]
  \tikzstyle{every edge}=[-,draw]
\subfloat[(a)]{
	\begin{tikzpicture}[scale=0.65,thick]
	\draw [red] (-2,0)--(-2,2) (2,0)--(2,-2);
	\draw [blue] (-2,2)--(0,2) (0,-2)--(2,-2);
	\draw [orange] (0,0) edge (0,2) edge (-2,0) edge (2,0) edge (0,-2);
	
	\draw [dashed, teal,ultra thick,->] (-2.5,1.5) -- (2.5,1.5) node [right] {\tiny $ L$};	
 	\draw (.5,-2.5) edge [bend left,dashed, teal,ultra thick,->] (2.5,-.5);
	\draw (2.5,-.5) node [right,teal] {\tiny $ L'$};

	\draw (-.7,.7) node [] {\tiny $F$};
	\draw (1.5,.8) node [] {\tiny $k_1$ faces};
	\draw (-1.5,-1.2) node [] {\tiny $k_2$ faces};

	\draw (.7,-.7) node [] {\tiny $F'$};

	\draw (0,0) node [left=3,below] {\tiny $v$};
	\foreach \x in {-2,0} {\draw (\x,2) node [wv] {}; }	
	\foreach \x in {-2,2} {\draw (\x,0) node [wv] {}; }	
	\foreach \x in {0,2} {\draw (\x,-2) node [wv] {}; }	

	\draw (0,0) node [cv] {};	
 	\draw (2,-1) node [av] {};		
 	\draw (-2,1) node [av] {};		
 	\draw (1,-2) node [bv] {};					
 	\draw (-1,2) node [bv] {};					

	\node  [inner sep=0.9pt] at (-2.5,0) {}; 	
	\node  [inner sep=0.9pt] at (2.5,0) {}; 		
	\end{tikzpicture}}
$\quad$
\subfloat[(b)]{
	\begin{tikzpicture}[scale=0.5,thick]
	\draw [blue] (4,4)--(6,4);
	\draw [blue] (2,1)--(2,4);
	\draw [blue] (2,2) -- (4,2) (4,-1)--(6,-1)--(6,1);

	\draw [red] (4,2)--(4,4);
	\draw [red] (2,1)--(6,1);	

	\foreach \x in {6} {\draw [orange] (\x,1) -- (\x,4); }
	\foreach \x in {2} {\draw [orange] (\x,-1) -- (\x,1); }
	\draw [orange] (4,2)--(6,2)  (2,4)--(4,4) (2,-1) -- (4,-1);
	
	\foreach \x in {2,4,6} {\draw (\x,4) node [wv] {}; }
	\foreach \x in {4} {\draw (\x,2) node [wv] {}; }	
	\foreach \x in {2,6} {\draw (\x,1) node [wv] {}; }	
	\foreach \x in {4} {\draw (\x,2.5) node [av] {}; }	
	\foreach \x in {4} {\draw (\x,1) node [av] {}; }	
	\draw (4,-1) node [wv] {};
	\draw (2,-1) node [cv] {};
	\draw (6,-1) node [bv] {};
	
	\foreach \x in {5} {\draw (\x,4) node [bv] {}; }	
	\foreach \x in {2} {\draw (\x,2) node [bv] {}; }	

	\foreach \x in {3} {\draw (\x,4) node [cv] {}; }	
	\foreach \x in {6} {\draw (\x,2) node [cv] {}; }	

	\draw [dashed, teal,ultra thick,->] (1.5,3) -- (6.5,3) node [teal,right] {\tiny $ L$};
	\draw [dashed, teal,ultra thick,->] (1.5,0) -- (6.5,0) node [teal,right] {\tiny $ L'$};
	
	\draw (1,-.5) node [left] {\tiny $F'$};
	\draw (1,1.5) node [left] {\tiny $F$};
	\draw (6,2) node [right] {\tiny $u$};		
	\draw (1,-.5) edge [gray,->,bend left=-10] (3,-.5);		
	\draw (1,1.5) edge [gray,->,bend left=-10] (3,1.5);	

	\foreach \x in {3,5} {\draw [below] (\x,4) node [] {\tiny I}; }
	\foreach \x in {4} {\draw (\x,1.5) node [] {\tiny II}; }
	\foreach \x in {4.2} {\draw (\x,.6) node [] {\tiny $v$}; }
	\node  [inner sep=0.9pt] at (0,-1) {}; 	
	\node  [inner sep=0.9pt] at (8,-1) {}; 			
	\end{tikzpicture}
}	
$\quad$
\subfloat[(c)]{
	\begin{tikzpicture}[scale=0.5,thick]
	\draw (0,-1) node [gv] (p00) {};
	\draw (0,1) node [gv] (p01) {};
	\draw (0,2) node [gv] (p02) {};	
	\draw (0,3) node [gv] (p03) {};
	\draw (0,4) node [gv] (p04) {};
	\draw (0,6) node [gv] (p06) {};	
	\draw (1,1) node [gv] (p11) {};
	\draw (1,2) node [gv] (p12) {};	
	\draw (1,3) node [gv] (p13) {};
	\draw (1,4) node [gv] (p14) {};
	\draw (1,6) node [gv] (p16) {};	
	\draw (2,-1) node [gv] (p20) {};
	\draw (2,1) node [gv] (p21) {};
	\draw (2,4) node [gv] (p24) {};
	\draw (2,6) node [gv] (p26) {};	
	\draw (3,6) node [gv] (p36) {};		
	\draw (4,4) node [gv] (p44) {};
	\draw (4,4.5) node [gv] (p45) {};
	\draw (4,6) node [gv] (p46) {};	
	\draw (5,-1) node [gv] (p50) {};
	\draw (5,6) node [gv] (p56) {};	
	\draw (6,-1) node [gv] (p60) {};	
	\draw (6,3) node [gv] (p63) {};
	\draw (6,4) node [gv] (p64) {};
	\draw (6,6) node [gv] (p66) {};	
	\draw (7,4) node [gv] (p74) {};
	\draw (8,-1) node [gv] (p80) {};
	\draw (8,4) node [gv] (p84) {};
	\draw (8,4.5) node [gv] (p85) {};
	\draw (8,6) node [gv] (p86) {};
	\draw (8,7) node [gv] (p87) {};

	\draw [red] (p06) -- (p26);
	\draw [red] (-1,-1) -- (-2,-1) -- (p03) (p01) -- (p21) (p12) -- (p14);	
	\draw [red] (p44)--(p46);
	
	\draw [blue]  (p01) -- (p03) (p02)--(p12) (p14)--(p44) (p24)--(p26) (p21)--(2,-1)--(1,-1) (p46) -- (p66);

	\draw [orange] (p26)--(p46) (0,4)--(p14) (p44)--(6,4) -- (6,6);
	\draw [orange] (p64) edge (p12) edge (p21);
	\draw [orange] (p06) -- (p03);
	\draw [orange] (0,-1) -- (p01);
	\draw [orange] (-1,-1) -- 	(1,-1);		
	\foreach \x in {1,3,6} {\draw (0,\x) node [wv] {}; }
	\draw (-1,-1) node [wv] {};
	\draw (1,-1) node [wv] {};
	
	\foreach \x in {2,4} {\draw (1,\x) node [wv] {}; }	
	\foreach \x in {1,6} {\draw (2,\x) node [wv] {}; }	
	\foreach \x in {4,6} {\draw (4,\x) node [wv] {}; }		
	\foreach \x in {6} {\draw (6,\x) node [wv] {}; }

	\draw (p45) node [av] {};
	\foreach \x in {1} {\draw (\x,6) node [av] {}; }
	\foreach \x in {1} {\draw (\x,3) node [av] {}; }
	\foreach \x in {1} {\draw (\x,1) node [av] {}; }
	\draw (-2,-1) node [av] {};
	\draw (0,-1) node [cv] {};
	\draw (2,-1) node [bv] {};	
	\foreach \x in {5} {\draw (\x,6) node [bv] {}; }
	\foreach \x in {2} {\draw (\x,4) node [bv] {}; }	
	\foreach \x in {0} {\draw (\x,2) node [bv] {}; }

	\foreach \x in {3} {\draw (\x,6) node [cv] {}; }
	\foreach \x in {0,6} {\draw (\x,4) node [cv] {}; }
	\draw (.45,4.5) node [] {\tiny $u$}; 
	\draw (0,4) node [left] {\tiny $\ge 3$ faces}; 	
	\draw [dashed, teal,ultra thick,->] (-1,5) -- (7,5) node [teal,right] {\tiny $ L$};
	\draw [dashed, teal,ultra thick,->] (-2,.2) -- (7,.2) node [teal,right] {\tiny $ L'$};
	
	\draw (-.8,-.3) node [] {\tiny $F'_1$};
	\draw (1,-.3) node [] {\tiny $F'_2$};	
	
	\foreach \x in {1,3,5} {\draw [below] (\x,6) node [] {\tiny I}; }

	\foreach \x in {3} {\draw (\x,3.6) node [] {\tiny II}; }
	\foreach \x in {.5} {\draw (\x,3.6) node [] {\tiny III}; }
	\foreach \x in {.5} {\draw (\x,2.6) node [] {\tiny $F_1$}; }
	\draw (1.5,1.5) node [] {\tiny $F_2$};

	\draw (2.6,2.2) node [] {\tiny IV};
						
	\node  [inner sep=0.9pt] at (-1,0) {}; 	
	\node  [inner sep=0.9pt] at (7,0) {}; 			
	\end{tikzpicture}
}
$\quad$
\subfloat[(d)]{
	\begin{tikzpicture}[scale=.75,thick]
	\draw (0,0) node  [gv] (q)  {} node [below=5] {\tiny $u$};
	\draw (-1,1.73) node  [gv] (pu) {};
	\draw (-1.73,1) node  [gv] (pup)  {};
	\draw (1,1.73) node  [gv] (qu)  {};
	\draw (1.73,1) node  [gv] (qup)  {};
	\draw (0,1.15) node [gv] (m) {};
	\draw (-1,-1.73) node  [gv] (pd)  {};
	\draw (-1.73,-1) node  [gv] (pdp)  {};
	\draw (1,-1.73) node  [gv] (qd)  {};
	\draw (1.73,-1) node  [gv] (qdp)  {};
	\draw (-1,-1) node [gv] (pm) {};
	\draw (1,-1) node [gv] (qm) {};
		\draw (pu) -- (pup) -- (q) -- (pu) -- (qu) -- (q) -- (qup) -- (qu);
	\draw (pd) -- (pdp) -- (q) -- (pd) -- (qd) -- (q) -- (qdp) -- (qd);
	
	\draw [brown,ultra thick,->] (q) -- (-2.5,2.5/1.73) node [left] {\tiny $g_1$};
	\draw [brown,ultra thick,->] (q) -- (2.5,2.5/1.73) node [right] {\tiny $g_0$};

	\draw [brown,ultra thick,->] (q) -- (-2.5,-2.5/1.73) node [left] {\tiny $g_0'$};
	\draw [brown,ultra thick,->] (q) -- (2.5,-2.5/1.73) node [right] {\tiny $g_1'$};
	
	\draw (q) node [cv] {};	
	\draw (pdp) node [av] {};
	\draw (pup) node [av] {};
	\draw (qu) node [av] {};
	\draw (qd) node [av] {};
	\draw (pu) node [bv] {};	
	\draw (pd) node [bv] {};	
	\draw (qup) node [bv] {};	
	\draw (qdp) node [bv] {};	
	
	\draw (m) node [wv] {};
	\draw (pm) node [wv] {};
	\draw (qm) node [wv] {};		
	\draw (0,0) node [left=10] {$...$};
	\draw (0,0) node [right=10] {$...$};
	\draw (-2,1.8) edge [dashed, teal,ultra thick,->,bend right=60] (2,1.8);
	\draw (-2,-1.5) edge [dashed, teal,ultra thick,->,bend left=20] (2,-1.6);
	\draw (0,-1.5) node [] {\tiny $F'$};
	\draw [teal] (2,1.8) node [right] {\tiny $ L$};	
	\draw [teal] (2,-1.6) node [below=3,right=1] {\tiny $ L'$};	
	\draw (-2.3,.5) node [left,below] {\tiny $F_1$} edge [gray,->,bend right=30] (-.7,.7);	
	\draw (2.3,.5) node [right,below] {\tiny $F_2$} edge [gray,->,bend left=30] (.7,.7);	
		
	\node  [inner sep=0.9pt] at (-2.5,0) {}; 	
	\node  [inner sep=0.9pt] at (2.5,0) {}; 			
	\end{tikzpicture}
}
$\quad$
\subfloat[(e)]{
	\begin{tikzpicture}[scale=.75,thick]
	
	\draw (0,0) -- (0,2) -- (-1.7,.3) -- (0,0) -- (1.7,.3) -- (0,2);
	\draw (-1.7,-.3) -- (-2,-2) -- (1,1);

	\draw (0,0) -- (0,-2) -- (-1.7,-.3) -- (0,0) -- (1.7,-.3) -- (0,-2) -- (-2,-2);
	\draw (-1,1) -- (1,-1);
	
	\draw (0,0) node [left=4,below=5] {\tiny $u$};
	\draw (0,0) node [left=10] {$...$};
	\draw (0,0) node [right=10] {$...$};
	
	\draw [brown,ultra thick,->] (0,0) -- (-2.2,2.2*.3/1.7) node [left] {\tiny $g_1$};
	\draw [brown,ultra thick,->] (0,0) -- (2.2,2.2*.3/1.7) node [right] {\tiny $g_0$};

	\draw [brown,ultra thick,->] (0,0) -- (-2.2,-2.2*.3/1.7) node [left] {\tiny $g_0'$};
	\draw [brown,ultra thick,->] (0,0) -- (2.2,-2.2*.3/1.7) node [right] {\tiny $g_1'$};

	\draw (.5,1) node [wv] {};
	\draw (-1,.5) node [wv] {};	
	\foreach \x in {-1.5,-.5} {\draw (\x,-1) node [wv] {}; }	

	\draw (1,-.5) node [wv] {};
	\draw (0,0) node [cv] {};
	\draw (-2,-2) node [cv] {};	
	
	\foreach \x in {-1.7,1.7} {\draw (\x,.3) node [bv] {}; }	
	\foreach \x in {-1.7,1.7} {\draw (\x,-.3) node [bv] {}; }	
	\foreach \x in {-2,2} {\draw (0,\x) node [bv] {}; }	

	\foreach \x in {-1,1} {\draw (\x,1) node [av] {}; }	
	\foreach \x in {-1,1} {\draw (\x,-1) node [av] {}; }

	\draw (-.7,1.9) edge [dashed, teal,ultra thick,->,bend right=50] (1.7,.8);
	\draw [dashed, teal,ultra thick]  (-2.2,-1.5) -- (-.5,-1.5);
	\draw (-.5,-1.5) edge [dashed, teal,ultra thick,->,bend left=40]  (1.7,-.9);
	\draw [teal] (1.7,.7) node [above=4,right] {\tiny $ L$};	
	\draw [teal]   (1.7,-.9) node [below=3,right=1] {\tiny $ L'$};	
	\draw (-.5,1) node [] {\tiny $F_0$};
	\draw (.3,-1.2) node [] {\tiny $F_0'$};
	\draw (-1,-.5) node [] {\tiny $F_1$};
	\draw (1,-1.75) edge [gray,->] (-1,-1.75) node [right] {\tiny $F_1'$};

	\node  [inner sep=0.9pt] at (-2.5,0) {}; 	
	\node  [inner sep=0.9pt] at (2.5,0) {}; 			
	\end{tikzpicture}
}
$\quad$
\subfloat[(f)]{
	\begin{tikzpicture}[scale=0.75,thick]
	\draw (-1,-1.73) -- (-1.5,-.86) -- (1.5,.86) -- (1,1.73) -- (-1,-1.73) -- (0,-1.73) -- (0,1.73) -- (1,1.73);

	\draw (0,-1.73) -- (3,-1.73)  (2,0)--(1.5,.86);
	\draw (-1.5,-.86)--(-2,0)--(0,0) -- (1,-1.73) -- (1.5,-1.73/2) (2,0) -- (0,0) -- (3,-1.73);
	\draw (0,-1.73) -- (0,-3.46) -- (2,-3.46) -- (3,-1.73) -- (0,-3.46) -- (1,-1.73)--(2,-3.46);
	\draw (-2,0) -- (-1.8,.4) -- (0,0) -- (-1.35,1.2) -- (-1,1.73) -- (0,1.73) ;
	\draw (0,0) -- (-1,1.73);	
	\draw (2,0)--(1.8,-.4) -- (0,0);

	\draw (1.5,1.73/2);
	
	\draw [brown,ultra thick,->] (0,0) -- (1.8*1.8,-.4*1.8) node [right] {\tiny $g_0$};
	\draw [brown,ultra thick,->] (0,0) -- (-1.35*1.2,1.2*1.2) node [left] {\tiny $g_1$};

	\draw [brown,ultra thick,->] (0,0) -- (-1.8*1.5,.4*1.5) node [left] {\tiny $g_0'$};
	\draw [brown,ultra thick,->] (0,0) -- (3*1.2,-1.73*1.2) node [right] {\tiny $g_1'$};

	\draw [dashed, teal,ultra thick,->] (-1.7,-1.4) -- (0,-1.4) -- (.3,-1.73) -- 
	(.3,-1.7*1.73) -- (3,-1.7*1.73) node [teal,right] {\tiny $ L'$};

	\draw [dashed, teal,ultra thick,->] (.5,2) -- (.5,1.73/2) -- (1.5/2,.86/2) -- (2.5,.85/2) node [teal,right] {\tiny $ L$};

	\foreach \x in {1.73,-1.73} {\draw (0,\x) node [av] {}; }	
	\foreach \x in {0.86,-0.86,-2.59} {\draw (1.5,\x) node [av] {}; }		
	\foreach \x in {0.86} {\draw (-1.5,-\x) node [av] {}; }		
	\foreach \x in {0,-1.73*2} {\draw (0,\x) node [cv] {}; }

	\draw (3,-1.73) node [cv] {};
	\draw (1,-1.73) node [bv] {};
	\draw (-1,-1.73) node [bv] {};
	\draw (1,1.73) node [bv] {};
	\draw (2,-1.73*2) node [bv] {};	
	\draw (2,0) node [bv] {};

	\draw (-2.5/3,-.86) node [wv] {};
	\draw (1/3,-1.73*2/3) node [wv] {};
	\draw (-1/3,1.73*2/3) node [wv] {};

	\draw (5.5/3,-1.73*2.5/3) node [wv] {};	

	\draw (1.7,-.2) node [wv] {};	
	\draw (-1.7,.2) node [wv] {};	

	\draw (2.5/3,1.73/2) node [wv] {};
	
	\draw (2.5/3,-3.46/3-1.73/3-2.59/3) node [wv] {};

	\draw (2,-3.46/3-1.73/3-2.59/3) node [wv] {};

	\draw (1,-1) node [] {\tiny $F_1$};
	\draw (1.75,-2) node [] {\tiny $F_2$};
	\draw (1.25,-3.2) node [] {\tiny $F'_2$};	
	\draw (-1,-2.5)  node [left] {\tiny $F'_1$} edge [gray,->,bend left] (.2,-2.5);

	\draw (-.4,-.2) node [left] {\tiny $u$};
	\draw (-.8,.4) node [] {...};
	\draw (2.4,-1.1) node [right] {\tiny $k'$ faces};
	\draw (2.5,-1) edge [gray,->] (1,-.4);

	\draw (-.25,-1) node [] {\tiny $F'_0$};
	\draw (-1.25,.8)  node [left] {\tiny $F_0$} edge [gray,->] (.25,1);

	\node  [inner sep=0.9pt] at (-1,0) {}; 	
	\node  [inner sep=0.9pt] at (4,0) {}; 			
	\end{tikzpicture}
	}
\caption{Lemma~\ref{lem:hexs} and their dual pictures.}\label{fig:hexs}
\end{figure}

Let us now describe an algorithmic construction of an elementary move.
Informally speaking, an elementary move is a natural way to extend a given polygon for the potential to a strictly bigger polygon, along a given side.

Let $X_0\subseteq E$ be a given $(a,b,c)$-diagram.
Regard $X_0$ as a polygon for the potential, and choose an arbitrary side $\ell$ of the polygon.
The side $\ell$ intersects with two tips $t, t'$ of the $(a,b,c)$ diagram $X$ near the ends of $\ell$.
In other words, we let $t$ and $t'$ be the tips joined to double joints in $X_0$ such that the interval between them on $\ell$ is contained in a Seidel Lagrangian $L$. 
Let us assume that $ L$ hits $t$ before $t'$ with respect to the orientation of $ L$. Enumerate the tips intersecting $[t,t']$ as $t_0=t,t_1,t_2,\ldots,t_k=t'$ in this order.
For $i=1,2,\ldots,k-1$, each $t_i$ is joined to a single joint and contained in exactly two outer pentagons. 
We complete these pentagons to hexagons $P_1, P_2,\ldots,P_{k-1}$; these will be Type I faces along $ L$ as shown in Figure~\ref{fig:move}.
We add suitably many colored edges (ending at white vertices) to the newly added colored vertices such that all of the colored vertices have the correct valencies according to their colors. We add leaves to the white vertices so that every white vertex has valency three.  For the case $a,b,c \geq 3$, this already gives an $(a,b,c)$ diagram $X$; see Figure~\ref{fig:move} (a).
Figure~\ref{fig:move_ex} illustrates an example when $a=b=3$ and $c=4$.

Now assume $a=2$.
After completing the above process, some outer polygons may have more than five edges. More precisely, we may see outer heptagons having two leaves.
We complete each of these heptagons to hexagons by extending the two leaves to a single vertex. These hexagons are of Type II along $ L$; see Figure~\ref{fig:move} (b) and (c).
If $b,c\ge4$, then the process stops here and the result is an $(a,b,c)$-diagram $X$. 
Assume $b=3$ so that $c\ge6$. Then we again extend two leaves of each outer heptagon to a single vertex and add appropriate number of edges and vertices, so that every vertex has the correct valency. By repeating this process once more, we add suitably many Type III and IV faces along $ L$ and then obtain an $(a,b,c)$-diagram $X$. See Figure~\ref{fig:move} (c) to see the last two steps of this process.

\begin{remark}\label{rem:betagamma}
Suppose $a=2$. In the above construction, the first tip $t$ is at a $\gamma$-leaf if and only if a boundary word for $X_0$ is $(\beta\gamma)^2$, and this occurs if and only if the last tip $t'$ is at a $\beta$-leaf. 
\end{remark}

In all cases, we see that for every $(a,b,c)$-diagram $X'$ satisfying $X_0\subseteq X'$ and $\mathrm{int}(X')\cap ( L\cap X_0)\ne\varnothing$, we have $X\subseteq X'$.
In accordance with Definition~\ref{defn:move}, we say $X$ is obtained from $X_0$ by an elementary move \emph{along $ L$}. We will call the inner faces of $X$ not contained in $X_0$ as \emph{Type 1, 2, 3 or 4 faces} of this elementary move, according to the marking shown in Figure~\ref{fig:move}.

We claim that every elementary move is constructed in the above mentioned way. For this, suppose a diagram $X_1$ is obtained from a diagram $X_0$ by an elementary move such that $X_0$ is the left side of a Lagrangian $ L$ in $X_1$. By Lemma~\ref{lem:double} below, the tips of $X_0$ on $ L\cap X_1$ can be enumerated as $t_0=t,t_1,\ldots,t_{k-1},t_k=t'$ in this order such that $t$ and $t'$ are adjacent to double joints and the rests are adjacent to single joints. So the above construction is well-defined with respect to $X_0$ and $ L$, resulting in a diagram $X$ such that $X_0 \subsetneq X\subseteq X_1$. By the definition of an elementary diagram, we have $X=X_1$ and the claim is proved.

\begin{lemma}\label{lem:double}
Suppose $X$ is an $(a,b,c)$ diagram and $ L$ is a Lagrangian intersecting the interior of $X$.
If $X_0$ is the left side of $ L$ in $X$, then the first and the last leaves of $X_0$ that $ L$ hits are connected to double joints of $X_0$.
\end{lemma}

\begin{proof}
The Seidel Lagrangian segment $ L$ starts and ends at outer pentagons of $X$.  It cuts the outer pentagons in two polygons with fewer edges.  Since $X_0$ is also a diagram, its outer faces must be either pentagons or triangles.  Thus the first and last outer faces of $X_0$ adjacent to $ L$ can only be triangles.  This implies the first and last leaves of $X_0$ that $ L$ hits are connected to double joints of $X_0$.
\end{proof}

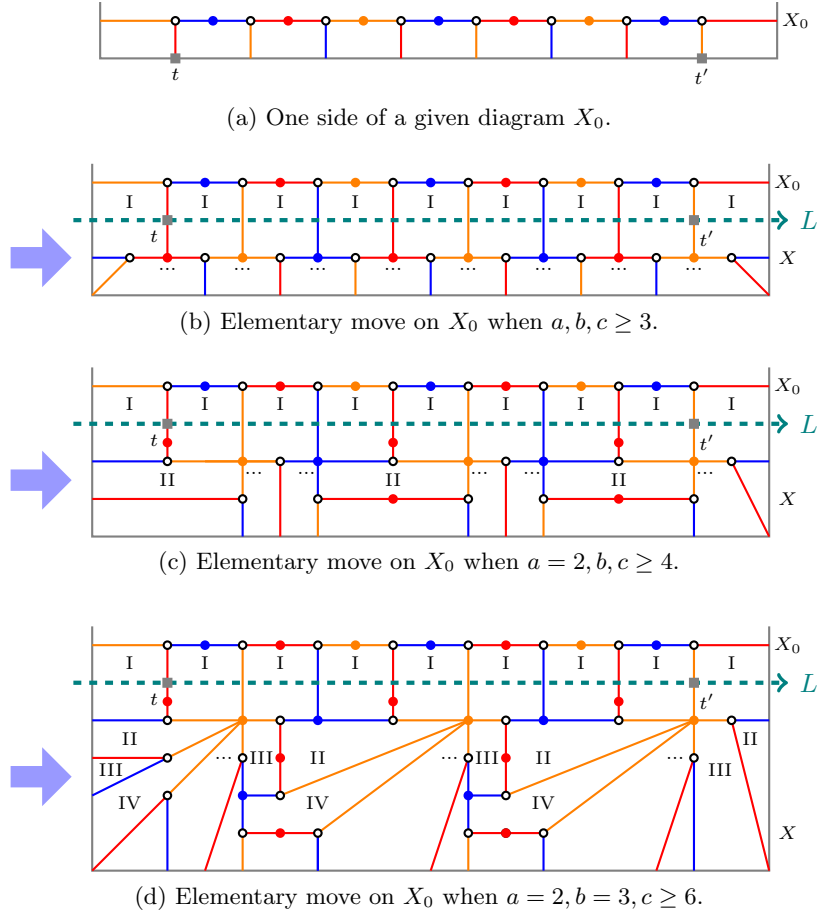
\begin{figure}[htb]
  \tikzstyle {av}=[red,draw,shape=circle,fill=red,inner sep=1pt]
  \tikzstyle {bv}=[blue,draw,shape=circle,fill=blue,inner sep=1pt]
  \tikzstyle {cv}=[orange,draw,shape=circle,fill=orange,inner sep=1pt]
  \tikzstyle {wv}=[black,draw,shape=circle,fill=white,inner sep=1pt]  
  \tikzstyle {gv}=[inner sep=0pt]
  \tikzstyle {v}=[gray,fill=gray,draw,shape=rectangle,inner sep=1.5pt]    
  \tikzstyle{every edge}=[-,draw]
\subfloat[(a) One side of a given diagram $X_0$.]{
	\begin{tikzpicture}[scale=0.5,thick]
    \node [fill=white,single arrow, draw=none, rotate=0] at (-11.5,2.5) {$\ \ \ $};

	\draw [gray] (-10,3.5) -- (-10,2) -- (8,2) -- (8,3.5);	
	\foreach \x in {-8,-2,4} {\draw [blue] (\x,3) -- (\x+2,3); }
	\foreach \x in {-6,0,6} {\draw [red] (\x,3) -- (\x+2,3); }	
	\foreach \x in {-10,-4,2} {\draw [orange] (\x,3) -- (\x+2,3); }		

	\foreach \x in {-6,0,6} {\draw [orange] (\x,2) -- (\x,3); }		
	\foreach \x in {-4,2} {\draw [blue] (\x,2) -- (\x,3); }		
	\foreach \x in {-8,-2,4} {\draw [red] (\x,2) -- (\x,3); }

	\foreach \x in {-8,-6,-4,-2,0,2,4,6} {\draw (\x,3) node [wv] {}; }

	\foreach \x in {-5,1} {\draw (\x,3) node [av] {}; }	
	\foreach \x in {-7,-1,5} {\draw (\x,3) node [bv] {}; }	
	\foreach \x in {-3,3} {\draw (\x,3) node [cv] {}; }	

	\draw (-8,2) node [v] {} node [below] {\tiny $t$};
	\draw (6,2) node [v] {} node [below] {\tiny $t'$};
	\draw (8.5,3) node [] {\tiny $X_0$};

	\end{tikzpicture}
}
\\  
\subfloat[(b) Elementary move on $X_0$ when $a,b,c\ge3$. ]{
	\begin{tikzpicture}[scale=0.5,thick]
    \node [fill=blue!40,single arrow, draw=none, rotate=0] at (-11.5,1) {$\ \ \ $};
	
	\draw [gray] (-10,3.5) -- (-10,0) -- (8,0) -- (8,3.5);	
	\foreach \x in {-8,-2,4} {\draw [blue] (\x,3) -- (\x+2,3); }
	\foreach \x in {-6,0,6} {\draw [red] (\x,3) -- (\x+2,3); }	
	\foreach \x in {-10,-4,2} {\draw [orange] (\x,3) -- (\x+2,3); }		

	\foreach \x in {-6,0,6} {\draw [orange] (\x,1) -- (\x,3); }		
	\foreach \x in {-4,2} {\draw [blue] (\x,1) -- (\x,3); }		
	\foreach \x in {-8,-2,4} {\draw [red] (\x,1) -- (\x,3); }		

	\foreach \x in {-9,-3,3} {\draw [red] (\x,1) -- (\x+2,1); }
	\draw [blue] (7,1) -- (8,1);	
	\foreach \x in {-10} {\draw [blue] (\x,1) -- (\x+1,1); }		
	\foreach \x in {-5,1} {\draw [blue] (\x,1) -- (\x+2,1); }	
	\foreach \x in {-7,-1,5} {\draw [orange] (\x,1) -- (\x+2,1); }	

	\draw [orange] (-9,1) -- (-10,0);	
	\draw [red] (7,1) -- (8,0);	
	\foreach \x in {-5,1} {\draw [red] (\x,1) -- (\x,0); }
	\foreach \x in {-7,-1,5} {\draw [blue] (\x,1) -- (\x,0); }	
	\foreach \x in {-3,3} {\draw [orange] (\x,1) -- (\x,0); }		

	\foreach \x in {-8,-6,-4,-2,0,2,4,6} {\draw (\x,3) node [wv] {}; }
	\foreach \x in {-9,-7,-5,-3,-1,1,3,5,7} {\draw (\x,1) node [wv] {}; }	

	\foreach \x in {-5,1} {\draw (\x,3) node [av] {}; }	
	\foreach \x in {-7,-1,5} {\draw (\x,3) node [bv] {}; }	
	\foreach \x in {-3,3} {\draw (\x,3) node [cv] {}; }	
	\foreach \x in {-6,0,6} {\draw (\x,1) node [cv] {}; }	
	\foreach \x in {-4,2} {\draw (\x,1) node [bv] {}; }	
	\foreach \x in {-8,-2,4} {\draw (\x,1) node [av] {}; }		
	\draw [dashed, teal,ultra thick,->] (-10.5,2) -- (8.5,2) node [teal,right] {\small $ L$};
	\draw (-8,2) node [v] {} node [left=5,below] {\tiny $t$};
	\draw (6,2) node [v] {} node [right=5,below] {\tiny $t'$};
	\draw (8.5,3) node [] {\tiny $X_0$};
	\draw (8.5,1) node [] {\tiny $X$};

	\foreach \x in {-9,-7,-5,-3,-1,1,3,5,7} {\draw (\x,2.5) node [] {\tiny I}; }
	\foreach \x in {-8,-6,-4,-2,0,2,4,6} {\draw (\x,.7) node [] {\tiny ...}; }

	\node  [inner sep=0.9pt] at (-2.5,0) {}; 	
	\node  [inner sep=0.9pt] at (2.5,0) {}; 			
	\end{tikzpicture}
}
\\ 
\subfloat[(c) Elementary move on $X_0$ when  $a=2,b,c\ge4$.]{
	\begin{tikzpicture}[scale=0.5,thick]
    \node [fill=blue!40,single arrow, draw=none, rotate=0] at (-11.5,1.5) {$\ \ \ $};	
	\draw [gray] (-10,4.5) -- (-10,0) -- (8,0) -- (8,4.5);	
	\draw [blue] (-8,4) -- (-6,4) (-2,4)--(0,4) (4,4)--(6,4);
	\draw [blue] (-4,1) -- (-4,4) (2,1)--(2,4) (7,2)--(8,2);
	\draw [blue] (-6,1)--(-6,0) (0,1)--(0,0) (6,0)--(6,1);
	\draw [blue] (-10,2)--(-8,2);

	\foreach \x in {-5,1} {	\draw [blue] (\x,2) -- (\x+3,2);}

	\foreach \x in {-6,0,6} {	\draw [red] (\x,4) -- (\x+2,4);}
	\foreach \x in {-8,-2,4} {	\draw [red] (\x,2) -- (\x,4);}
	
	\draw [red]  (7,2)--(8,0);
	\draw [red] (-10,1)--(-6,1) (-5,0)--(-5,2) (-4,1)--(0,1) (1,0)--(1,2) (2,1)--(6,1);	

	\foreach \x in {-6,0,6} {\draw [orange] (\x,1) -- (\x,4); }
	\foreach \x in {-4,2} {\draw [orange] (\x,0) -- (\x,1); }
	\foreach \x in {-10,-4,2} {\draw [orange] (\x,4) -- (\x+2,4); }
	
	\foreach \x in {-8,-2,4} {\draw [orange] (\x,2) -- (\x+3,2); }
	
	\draw [orange] (-7,2)--(-5,2);	
	\foreach \x in {-8,-6,-4,-2,0,2,4,6} {\draw (\x,4) node [wv] {}; }
	\foreach \x in {-8,-5,-2,1,4,7} {\draw (\x,2) node [wv] {}; }	
	\foreach \x in {-6,-4,0,2,6} {\draw (\x,1) node [wv] {}; }	
	\foreach \x in {-5,1} {\draw (\x,4) node [av] {}; }	
	\foreach \x in {-8,-2,4} {\draw (\x,2.5) node [av] {}; }	
	\foreach \x in {-2,4} {\draw (\x,1) node [av] {}; }	

	\foreach \x in {-7,-1,5} {\draw (\x,4) node [bv] {}; }	
	\foreach \x in {-4,2} {\draw (\x,2) node [bv] {}; }	

	\foreach \x in {-3,3} {\draw (\x,4) node [cv] {}; }	
	\foreach \x in {-6,0,6} {\draw (\x,2) node [cv] {}; }	

	\draw [dashed, teal,ultra thick,->] (-10.5,3) -- (8.5,3) node [teal,right] {\small $ L$};
	\draw (-8,3) node [v] {} node [left=5,below] {\tiny $t$};
	\draw (6,3) node [v] {} node [right=5,below] {\tiny $t'$};
	\draw (8.5,4) node [] {\tiny $X_0$};
	\draw (8.5,1) node [] {\tiny $X$};

	\foreach \x in {-9,-7,-5,-3,-1,1,3,5,7} {\draw [below] (\x,4) node [] {\tiny I}; }
	\foreach \x in {-8,-2,4} {\draw (\x,1.5) node [] {\tiny II}; }
	\foreach \x in {-5.7,-4.3,.3,1.7,6.3} {\draw (\x,1.7) node [] {\tiny ...}; }
	\end{tikzpicture}
}
\\
\subfloat[(d) Elementary move on $X_0$ when  $a=2,b=3,c\ge6$.]{
	\begin{tikzpicture}[scale=0.5,thick]
    \node [fill=blue!40,single arrow, draw=none, rotate=0] at (-11.5,2.5) {$\ \ \ $};
	
	\draw (0,0) node [gv] (p00) {};
	\draw (0,1) node [gv] (p01) {};
	\draw (0,2) node [gv] (p02) {};	
	\draw (0,3) node [gv] (p03) {};
	\draw (0,4) node [gv] (p04) {};
	\draw (0,6) node [gv] (p06) {};	
	\draw (1,1) node [gv] (p11) {};
	\draw (1,2) node [gv] (p12) {};	
	\draw (1,3) node [gv] (p13) {};
	\draw (1,4) node [gv] (p14) {};
	\draw (1,6) node [gv] (p16) {};	
	\draw (2,0) node [gv] (p20) {};
	\draw (2,1) node [gv] (p21) {};
	\draw (2,4) node [gv] (p24) {};
	\draw (2,6) node [gv] (p26) {};	
	\draw (3,6) node [gv] (p36) {};		
	\draw (4,4) node [gv] (p44) {};
	\draw (4,4.5) node [gv] (p45) {};
	\draw (4,6) node [gv] (p46) {};	
	\draw (5,0) node [gv] (p50) {};
	\draw (5,6) node [gv] (p56) {};	
	\draw (6,0) node [gv] (p60) {};	
	\draw (6,3) node [gv] (p63) {};
	\draw (6,4) node [gv] (p64) {};
	\draw (6,6) node [gv] (p66) {};	
	\draw (7,4) node [gv] (p74) {};
	\draw (8,0) node [gv] (p80) {};
	\draw (8,4) node [gv] (p84) {};
	\draw (8,4.5) node [gv] (p85) {};
	\draw (8,6) node [gv] (p86) {};
	\draw (8,7) node [gv] (p87) {};
		
	\draw (-6,0) node [gv] (m60) {};
	\draw (-6,1) node [gv] (m61) {};
	\draw (-6,2) node [gv] (m62) {};	
	\draw (-6,3) node [gv] (m63) {};
	\draw (-6,4) node [gv] (m64) {};
	\draw (-6,6) node [gv] (m66) {};	
	\draw (-5,1) node [gv] (m51) {};
	\draw (-5,2) node [gv] (m52) {};	
	\draw (-5,3) node [gv] (m53) {};
	\draw (-5,4) node [gv] (m54) {};
	\draw (-5,6) node [gv] (m56) {};	
	\draw (-4,0) node [gv] (m40) {};
	\draw (-4,1) node [gv] (m41) {};
	\draw (-4,4) node [gv] (m44) {};
	\draw (-4,6) node [gv] (m46) {};	
	\draw (-3,6) node [gv] (m36) {};		
	\draw (-2,4) node [gv] (m24) {};
	\draw (-2,4.5) node [gv] (m25) {};
	\draw (-2,6) node [gv] (m26) {};	
	\draw (-1,0) node [gv] (m10) {};
	\draw (-1,6) node [gv] (m16) {};	
	\draw (-7,0) node [gv] (m70) {};
	\draw (-7,2) node [gv] (m72) {};
	\draw (-7,3) node [gv] (m73) {};
	\draw (-7,4) node [gv] (m74) {};
			
	\draw (-8,0) node [gv] (m80) {};
	\draw (-8,1) node [gv] (m81) {};
	\draw (-8,2) node [gv] (m82) {};			
	\draw (-8,3) node [gv] (m83) {};
	\draw (-8,4) node [gv] (m84) {};			
	\draw (-8,4.5) node [gv] (m85) {};
	\draw (-8,6) node [gv] (m86) {};			
	\draw (-8,7) node [gv] (m87) {};			
	
	\draw [gray] (-10,6.5) -- (-10,0) -- (8,0) -- (8,6.5);
	\foreach \x in {-8,-2,4} {\draw [red] (\x,4) -- (\x,6); }

	\draw [red] (-7,0) edge (m63);	
	\draw [red] (m66) -- (m46) (p06) -- (p26) (p66) -- (8,6) (p74) -- (8,0);
	\draw [red] (m52) -- (m54) (m61) -- (m41) (-1,0) -- (p03) (p01) -- (p21) (p12) -- (p14) (5,0)--(p63);

	\draw [blue] (-8,6) -- (m66);
	\draw [blue] (-8,4) -- (-10,4);
	\draw [blue] (m63) -- (m61);

	\draw [blue] (m62) -- (m52) (m54) -- (m24) (m44) -- (m46) (m41) -- (-4,0) (p01) -- (p03) (p02)--(p12) (p14)--(p44) (p24)--(p26) (p21)--(2,0) (p46) -- (p66) (6,0)--(p63) (p74)--(8,4);
	\draw [blue] (m26) -- (p06);
	\draw [blue] (-4,0) -- (m41);
	\draw [blue] (-8,3) -- (-10,2);
	\draw [blue] (-8,2) -- (-8,0);
	\draw [red] (-8,3) -- (-10,3);
	\draw [red] (-8,2) -- (-10,0);
	
	\foreach \x in {-8,-2,4} {\draw [orange] (\x,4) -- (\x+3,4); }

	\foreach \x in {-10,-4,2} {\draw [orange] (\x,6) -- (\x+2,6); }
	\draw [orange] (m63) edge (m66) ;
	\draw [orange] (m61) edge (-6,0);
	
	\draw [orange] (-6,4) edge (-8,3) edge (-8,2) ;
	\draw [orange] (p66)--(p63);
	\draw [orange] (p64) edge (p12) edge (p21);
	\draw [orange] (p04) edge (m52) edge (m41);
	\draw [orange] (p06) -- (p03);
	\draw [orange] (0,0) -- (p01);
	\foreach \x in {-6.5,5.5,-.5} {\draw (\x,3) node [] {\tiny ...};}

	\foreach \x in {-8,-6,-4,-2,0,2,4,6} {\draw (\x,6) node [wv] {}; }	
	\foreach \x in {-8,-5,-2,1,4,7} {\draw (\x,4) node [wv] {}; }	
	\foreach \x in {-8,-6,0,6} {\draw (\x,3) node [wv] {}; }	
	\foreach \x in {-8,-5,1} {\draw (\x,2) node [wv] {}; }	
	\foreach \x in {-6,-4,0,2,} {\draw (\x,1) node [wv] {}; }	

	\foreach \x in {-8,-2,4} {\draw (\x,4.5) node [av] {}; }	

	\foreach \x in {-5,1} {\draw (\x,6) node [av] {}; }
	\foreach \x in {-5,1} {\draw (\x,3) node [av] {}; }
	\foreach \x in {-5,1} {\draw (\x,1) node [av] {}; }

	\foreach \x in {-7,-1,5} {\draw (\x,6) node [bv] {}; }
	\foreach \x in {-4,2} {\draw (\x,4) node [bv] {}; }	
	\foreach \x in {-6,0} {\draw (\x,2) node [bv] {}; }

	\draw (m64) node [cv] {};
	\foreach \x in {-3,3} {\draw (\x,6) node [cv] {}; }
	\foreach \x in {0,6} {\draw (\x,4) node [cv] {}; }
	
	\draw [dashed, teal,ultra thick,->] (-10.5,5) -- (8.5,5) node [teal,right] {\small $ L$};
	
	\draw (-8,5) node [v] {} node [left=5,below] {\tiny $t$};
	\draw (6,5) node [v] {} node [right=5,below] {\tiny $t'$};
		
	\draw (8.5,6) node [] {\tiny $X_0$};
	\draw (8.5,1) node [] {\tiny $X$};

	\foreach \x in {-9,-7,-5,-3,-1,1,3,5,7} {\draw [below] (\x,6) node [] {\tiny I}; }
	\foreach \x in {-4,2} {\draw (\x,3) node [] {\tiny II}; }
	\foreach \x in {-9,7.5} {\draw (\x,3.5) node [] {\tiny II}; }
	\foreach \x in {-9.5,6.7} {\draw (\x,2.7) node [] {\tiny III}; }

	\foreach \x in {-5.5,.5} {\draw (\x,3) node [] {\tiny III}; }
	\foreach \x in {-9,-4,2} {\draw (\x,1.8) node [] {\tiny IV}; }
						
	\node  [inner sep=0.9pt] at (-2.5,0) {}; 	
	\node  [inner sep=0.9pt] at (2.5,0) {}; 			
	\end{tikzpicture}
}
\caption{Three cases of elementary moves. The regions marked with dots may contain extra double joints, but neither single joints nor inner faces.}\label{fig:move}
\end{figure}

\begin{figure}[htb]
  \tikzstyle {av}=[red,draw,shape=circle,fill=red,inner sep=1pt]
  \tikzstyle {bv}=[blue,draw,shape=circle,fill=blue,inner sep=1pt]
  \tikzstyle {cv}=[orange,draw,shape=circle,fill=orange,inner sep=1pt]
  \tikzstyle {wv}=[black,draw,shape=circle,fill=white,inner sep=1pt]  
  \tikzstyle {gv}=[inner sep=0pt]
  \tikzstyle {v}=[gray,fill=gray,draw,shape=rectangle,inner sep=1.5pt]    
  \tikzstyle{every edge}=[-,draw]
	\begin{tikzpicture}[scale=0.5,thick]
    \node [fill=blue!40,single arrow, draw=none, rotate=-90] at (-1,4.8) {$\ \ \ $};	
	\draw [gray] (-10,7.5) -- (-10,6) -- (8,6) -- (8,7.5);	
	\foreach \x in {-8,-2,4} {\draw [blue] (\x,7) -- (\x+2,7); }
	\foreach \x in {-6,0,6} {\draw [red] (\x,7) -- (\x+2,7); }	
	\foreach \x in {-10,-4,2} {\draw [orange] (\x,7) -- (\x+2,7); }		
	\foreach \x in {-7,-1,5} {\draw [blue] (\x,7) -- (\x,7.5); }
	\foreach \x in {-5,1} {\draw [red] (\x,7) -- (\x,7.5); }
	\foreach \x in {-3,3} {\draw [orange] (\x,7) -- (\x-.4,7.5); }
	\foreach \x in {-3,3} {\draw [orange] (\x,7) -- (\x+.4,7.5); }
	\foreach \x in {-6,0,6} {\draw [orange] (\x,6) -- (\x,7); }		
	\foreach \x in {-4,2} {\draw [blue] (\x,6) -- (\x,7); }		
	\foreach \x in {-8,-2,4} {\draw [red] (\x,6) -- (\x,7); }		
	\foreach \x in {-8,-6,-4,-2,0,2,4,6} {\draw (\x,7) node [wv] {}; }

	\foreach \x in {-5,1} {\draw (\x,7) node [av] {}; }	
	\foreach \x in {-7,-1,5} {\draw (\x,7) node [bv] {}; }	
	\foreach \x in {-3,3} {\draw (\x,7) node [cv] {}; }	
	\draw (-8,6) node [v] {} node [left=5,below] {\tiny $t$};
	\draw (6,6) node [v] {} node [right=5,below] {\tiny $t'$};
	\draw (8.5,7) node [] {\tiny $X_0$};

	\draw [gray] (-10,3.5) -- (-10,-1) -- (8,-1) -- (8,3.5);	
	\foreach \x in {-8,-2,4} {\draw [blue] (\x,3) -- (\x+2,3); }
	\foreach \x in {-6,0,6} {\draw [red] (\x,3) -- (\x+2,3); }	
	\foreach \x in {-10,-4,2} {\draw [orange] (\x,3) -- (\x+2,3); }		

	\foreach \x in {-7,-1,5} {\draw [blue] (\x,3) -- (\x,3.5); }
	\foreach \x in {-5,1} {\draw [red] (\x,3) -- (\x,3.5); }
	\foreach \x in {-3,3} {\draw [orange] (\x,3) -- (\x-.4,3.5); }
	\foreach \x in {-3,3} {\draw [orange] (\x,3) -- (\x+.4,3.5); }

	\foreach \x in {-6,0,6} {\draw [orange] (\x,0) -- (\x,3); }		
	\foreach \x in {-4,2} {\draw [blue] (\x,1) -- (\x,3); }		
	\foreach \x in {-8,-2,4} {\draw [red] (\x,1) -- (\x,3); }		

	\foreach \x in {-9,-3,3} {\draw [red] (\x,1) -- (\x+2,1); }
	\draw [blue] (7,1) -- (8,1);	
	\foreach \x in {-10} {\draw [blue] (\x,1) -- (\x+1,1); }		
	\foreach \x in {-5,1} {\draw [blue] (\x,1) -- (\x+2,1); }	
	\foreach \x in {-7,-1,5} {\draw [orange] (\x,1) -- (\x+2,1); }	

	\draw [orange] (-9,1) -- (-10,-1);	
	\draw [red] (7,1) -- (8,-1);	
	\foreach \x in {-5,1} {\draw [red] (\x,1) -- (\x,-1); }
	\foreach \x in {-7,-1,5} {\draw [blue] (\x,1) -- (\x,-1); }	
	\foreach \x in {-3,3} {\draw [orange] (\x,1) -- (\x,-1); }		
	\foreach \x in {-6,0,6} {\draw [red] (\x,0) -- (\x-.5,-1);}
	\foreach \x in {-6,0,6} {\draw [blue] (\x,0) -- (\x+.5,-1);}

	\foreach \x in {-8,-6,-4,-2,0,2,4,6} {\draw (\x,3) node [wv] {}; }
	\foreach \x in {-9,-7,-5,-3,-1,1,3,5,7} {\draw (\x,1) node [wv] {}; }	
	\foreach \x in {-6,0,6} {\draw (\x,0) node [wv] {};}

	\foreach \x in {-5,1} {\draw (\x,3) node [av] {}; }	
	\foreach \x in {-7,-1,5} {\draw (\x,3) node [bv] {}; }	
	\foreach \x in {-3,3} {\draw (\x,3) node [cv] {}; }	
	\foreach \x in {-6,0,6} {\draw (\x,1) node [cv] {}; }	
	\foreach \x in {-4,2} {\draw (\x,1) node [bv] {}; }	
	\foreach \x in {-8,-2,4} {\draw (\x,1) node [av] {}; }		
	\draw [dashed, teal,ultra thick,->] (-10.5,2) -- (8.5,2) node [teal,right] {\small $ L$};
	\draw (-8,2) node [v] {} node [left=5,below] {\tiny $t$};
	\draw (6,2) node [v] {} node [right=5,below] {\tiny $t'$};
	\draw (8.5,1) node [] {\tiny $X$};

	\node  [inner sep=0.9pt] at (-2.5,0) {}; 	
	\node  [inner sep=0.9pt] at (2.5,0) {}; 			
	\end{tikzpicture}
\caption{An example of an elementary move when $(a,b,c)=(3,3,4)$.}\label{fig:move_ex}
\end{figure}
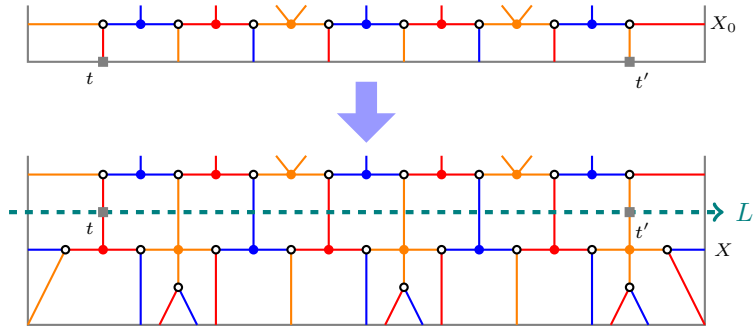

We now prove the main result of this section.

\begin{theorem}\label{thm:move}
Every non-empty $(a,b,c)$-diagram can be obtained from the basic $(a,b,c)$-diagram by a finite sequence of elementary moves.
\end{theorem}

\begin{proof}
We use induction on the number of inner faces in a given $(a,b,c)$-diagram $X$.
Suppose first that $X$ has no inner faces and is not a basic diagram.
Then $X$ has at least one colored vertex $v$.
Since every face containing $v$ is outer, it follows that $v$ is the only colored vertex
and $X$ is a $0$-th generation diagram.

Assume that $X$ has at least one inner face.
There exists a Seidel Lagrangian $ L$ passing through this inner face.
We let $X_0$ be the left side of $ L$.
Denote as $X'$ the $(a,b,c)$-diagram obtained from $X_0$ by an elementary move along $ L$.
By definition, we have $X_0\subsetneq X'\subseteq X$.
Since $X'$ has an inner face, it is neither basic nor of the $0$-th generation.
Let us choose $ L$ so that the number of inner faces in $X\setminus X'$ is minimal.
We claim that this number is zero.

Enumerate the Type I faces along $ L$ that intersects $X_0$ as $(F_0,F_1,\ldots,F_k)$ 
so that $F_0$ and $F_k$ are outer in $X_0$.
If all the outer faces of $X'$ in the right side of $ L$ are still outer in $X$,
then $X=X'$ and the claim is proved.
So let us assume to have a hexagon face $F'\subseteq X\setminus X_0$ such that $F'$ is in the right side of $ L$ and $F'\cap X'$ is an outer face in $X'$.
If $F'\cap X'$ is a triangle, then the neighboring pentagons of $F'\cap X'$ in $X'$ become inner in $X$ as well. Hence we may assume $F'\cap X'$ to be an outer pentagon of $X'$.
Since $X'$ is not a 0-th generation diagram, $F'$ intersects with at least one of Type I, II, III or IV faces of the elementary move from $X_0$ to $X'$. See Figure~\ref{fig:move}.
By Lemma~\ref{lem:hexs}, we have a Lagrangian $ L'$ which intersects with $F'$ such that the right side of $ L'$ is strictly smaller than that of $ L$. This is a contradiction to the minimality assumption.
\end{proof}


\section{Potential in the hyperbolic case}\label{s:hyp}
Throughout this section, we assume $1/a+1/b+1/c<1$ and $a\le b\le c$.
Since $G=\pi_1^{\mathrm{orb}}(\PO)$ acts transitively on the hexagons of $Z$, we have the following.
\begin{lemma}\label{lem:unique}
Given a cyclic word $[w]\in\fc$ there exists at most one polygon $U$ for the potential up to the action of $G$
such that $[w(\partial(U)]=[w]$.
\end{lemma}

So we can compute the potential for $\PO$ by counting boundary words of polygons for the potential.
We have seen that we can generate polygons for the potential (equivalently, $(a,b,c)$ diagrams) by applying elementary moves. 
This elementary move on a diagram will correspond to a transformation called \emph{cut-glue operation} on its boundary word.

\subsection{Cut--glue operations}\label{ss:cutglue}
For $t\in\CA$ we denote by $o(t)$ the order of $t$ in $G$.
Recall that we denote $\tau(\alpha) = \beta, \tau(\beta) =\gamma, \tau(\gamma) =\alpha$. 

\begin{definition}\label{defn:cut}
Given $\theta \in \CA$ and $k\ge0$,
we let $w$ is the length-$(k+2)$ subword of $(\gamma\beta\alpha)^\infty$ starting with the letter $\tau(\theta)$.
If $\theta'$ is the last letter of $w$, then we define 
\[\cut(\theta,k)=\theta w \tau(\theta') \in F\]
and call $\cut(\theta,k)$ as a \emph{cut word}.
If $a=2$, then we further assume that $\theta\ne\beta$ and that the last letter of $\cut(\theta,k)$ is not $\gamma$.
\end{definition}

Note that given a cut word $\theta_1 w_1w_2 \cdots w_{k+2} \tau(w_{k+2})$, only the first  pair
$\theta_1 w_1$ and the last pair $w_{k+2}\tau(w_{k+2})$ are not in the order of $(\gamma\beta\alpha)^\infty$.
One can think of a cut word as the label-reading of a path starting from a corner to the next corner of a polygon for the potential.
Hence, cut word arises as the label-reading  of the interval on a boundary of an $(a,b,c)$-diagram
to which an elementary move applies; see Figure~\ref{fig:move}. 

The reason why we pose the restriction is for $a=2$ case is that, the only boundary words of holomorphic polygons which contain $\beta\gamma$ are $\alpha\beta\gamma$ and $(\beta\gamma)^2$. As we will see later, the whole set of polygons for the potential  can be generated without considering elementary moves for these two specific polygons; see Theorem~\ref{thm:enum}.
When $a=2$, we have four types of cut words. Namely for $k\ge0$, we have:
\[
\cut(\alpha,3k) = \alpha\cdots\beta, \cut(\alpha,3k+1) = \alpha\cdots \alpha, \cut(\gamma,3k-1)=\gamma\cdots\beta, \cut(\gamma,3k)=\gamma\cdots\alpha.\]

Now let us define a glue word case-by-case.

\textbf{Case 1.} $a\ge3$.

Given $\theta\in \CA$ and $k\ge0$, we let $\theta_1\theta_2\cdots$ be the infinite subword of $(\gamma\beta\alpha)^\infty$ starting with $\theta_1=\theta$ and define
\[\glue(\theta,k)=
\theta_1 (\theta_2\theta_1) \left(\prod_{i=1}^{k+2} (\theta_{i+1}\theta_i)^{o(\theta_{i+2})-3} \right)(\theta_{k+3}\theta_{k+2})\theta_{k+3}. \]

\textbf{Case 2.} $a\ge2,b\ge4$.

We have four subcases. For $k\ge0$, we define:

\begin{eqnarray*}
\glue(\alpha,3k) &=& \alpha (\gamma\alpha)^{b-3} \left( \gamma\beta (\alpha\beta)^{c-4} \alpha (\gamma\alpha)^{b-4}\right)^{k} (\gamma\alpha) \gamma\beta \\
 \glue(\alpha,3k+1) &=& \alpha (\gamma\alpha)^{b-3}  \gamma\beta (\alpha\beta)^{c-4} \left(\alpha (\gamma\alpha)^{b-4}\gamma\beta (\alpha\beta)^{c-4}   \right)^{k} (\alpha\beta) \alpha \\ 
 \glue(\gamma,3k-1)&=& \gamma \beta (\alpha \beta)^{c-3} \left(\alpha (\gamma\alpha)^{b-4} \gamma\beta (\alpha\beta)^{c-4} \right)^{k-1} \alpha (\gamma\alpha)^{b-3} \gamma\beta \\ 
\glue(\gamma,3k)&=&\gamma \beta (\alpha \beta)^{c-3} \left(\alpha (\gamma\alpha)^{b-4} \gamma\beta (\alpha\beta)^{c-4} \right)^{k} (\alpha\beta) \alpha
\end{eqnarray*}

\textbf{Case 3.} $a\ge2,b=3$.

We again have four subcases. For $k\ge0$, we define:

\begin{eqnarray*}
\glue(\alpha,3k) &=& \alpha\gamma\beta (\alpha\beta)^{c-5} (\alpha\gamma\beta(\alpha\beta)^{c-6})^{k-1} (\alpha\beta) \alpha\gamma\beta  \\
 \glue(\alpha,3k+1) &=& \alpha\gamma\beta (\alpha\beta) ((\alpha\beta)^{c-6} \alpha\gamma\beta)^k (\alpha\beta)^{c-4} \alpha \\ 
 \glue(\gamma,3k-1)&=& \gamma\beta (\alpha\beta)^{c-4} (\alpha\gamma\beta(\alpha\beta)^{c-6} )^{k-1} (\alpha\beta) \alpha\gamma\beta \\ 
\glue(\gamma,3k)&=&\gamma\beta (\alpha\beta)^{c-4} (\alpha\gamma\beta(\alpha\beta)^{c-6} )^{k} (\alpha\beta)^2 \alpha
\end{eqnarray*}

In each of the cases, we call the word $\glue(\theta,k)\in F$ as a \emph{glue word}. 
A glue word is the label-reading of the newly generated boundary interval of an $(a,b,c)$-diagram from an elementary move. So an elementary move does a cut--glue operation on the boundary word:

\begin{definition}
Suppose $[w],[w']\in\fc$ and
$w'$ is obtained by replacing a subword $\cut(\theta, k )$ of $w$ by
$\glue(\theta, k )$ for some $\theta \in\CA$ and $ k \ge0$.
Then we say $[w']\in\fc$ is obtained by a \emph{cut--glue operation} on $[w]\in \fc$.
\end{definition}

\begin{lemma}\label{lem:pcount}
Suppose that we obtain polygon $U_{k+1}$ from a single cut--glue operation of the polygon $U_{k}$. 
Then, we have $$p(U_{k+1}) = p(U_k) +1.$$ 
\end{lemma}
\begin{proof}
Recall that $p(U_k) = [\partial U_k]_0$, which was obtained from the number of odd terms in $[\partial U_k]$ in Lemma \ref{lem:standard}
divided by 3. Note that in the cut-word, the odd terms are $w_2 \cdots w_{k+1}$ since $\theta_1 w_1$, and $w_{k+2} \tau(w_{k+2})$
give words in $\{\alpha \beta, \beta \gamma, \gamma \alpha\}$.
And in the glue-word, the odd terms are $\theta_1\theta_2\cdots \theta_{k+3}$. Hence
the number of odd terms increase by $3$, and hence $p$ increase by $1$.
\end{proof}

Therefore, we may use $p$ to denote the generation number of cut--glue operation.
Let us define $\mathcal{S}^{-1} =\{[\alpha\beta\gamma]\}$
and $\CS^{p}\subseteq\fc$ be the set of cyclic words obtained by applying cut--glue operation once on cyclic words on $\CS^{p-1}$. We call the cyclic words in $\CS^p$ as \emph{$p$-th generation words}.
It is immediate to see that $\CS^{0}$ is the boundary words of the three $0$-the generation diagrams.
We also see that $\CS^{1}$ is the boundary word of the unique first generation diagram, which contains precisely one interior face. So we have:
\[
\CS^0 = \{ 
[(\alpha\beta)^c], [(\beta\gamma)^a], [(\gamma\alpha)^b]
\},\quad
\mathcal{S}^1 = \{ [\alpha (\gamma\alpha)^{b-2}\gamma(\beta\gamma)^{a-2}\beta(\alpha\beta)^{c-2}]\}.\]

We let $\mathcal{S}=\cup_{p\ge-1}\CS^p$. 

\begin{theorem}\label{thm:enum}
There is a one-to-one correspondence $\Phi$ between
the set of polygons for the potential up to $G$-action
and the set $\CS$, defined by $\Phi(U) = [w(\partial U)]$.
\end{theorem}

\begin{proof}
Note first that $\CS^1$ is obtained by applying a cut-glue operation on some word in $\CS^0$, even in the case when $a=2$.
Theorem~\ref{thm:move} and Lemma~\ref{lem:unique} completes the proof.
\end{proof}

\subsection{Computation of the potential}\label{ss:eval}

Since each $\CS^p$ is finite and every cyclic word in $\CS^p$ has length greater than $p$, we clearly have an algorithmic enumeration of $\CS$ and hence the enumeration of the terms in the potential. 

Suppose $[w]\in \mathcal{S}^p$ where $p\ge-1$.
If $p\ge0$, we see that $p=[w]_0$.
Motivated by Lemma~\ref{lem:area}, we define for $P,Q,R\in\mathbb{Z}$:
\begin{equation}\label{eq:areapqr1}
\area(P,Q,R) =3(P+Q+R)+8 \frac{P/a+Q/b+R/c-1}{1-1/a-1/b-1/c}.
\end{equation}
Given $[w]\in\fc$, there exists $u\in F$ such that $u$ is root-free and $w = u^n$ for some $n>0$.
We denote this unique number $n$ as $\eta([w])$.

\begin{theorem}\label{thm:term}
Let $U\subseteq\mathbb{H}^2$( or $ S^2$) be a polygon for the potential such that the label-reading $[w]$ of $\partial U$ 
belongs to $\CS^p$ for some $p\ge-1$.
Then for $S=[w]_1+[w]_2+[w]_3$
and
$A=\area([w]_1,[w]_2,[w]_3)$,
the term of the disk potential corresponding to $U$ is given by
\[
(-1)^{p+[w]_2+[w]_3}
\frac{2p+[w]_1+[w]_2+[w]_3}{\eta([w])}
q^{\area([w]_1,[w]_2,[w]_3)} x^{[w]_1}y^{[w]_2}z^{[w]_3}.
\]
\end{theorem}
\begin{proof}
Since
$s(U)$
is the number of $yz$ arcs on $\partial U$, 
we see that $s(U)$ is the sum of the number of $\alpha$-leaves from single joints
and the numbers of double joints with dangling $\beta$- or $\gamma$-edges
in the corresponding $(a,b,c)$-diagram.
So we have
$s(U) = p+[w]_2+[w]_3$. We analogously get $s^{\mathrm{op}}(U) = p+[w]_1$.
\end{proof}

Now, the potential $W$ can be found in the following way.
We can find words (holomorphic polygons) in each generation $ \mathcal{S}^p$ inductively, by considering
all possible cut--glue operations on the words in  $ \mathcal{S}^{p-1}$.
The above theorem gives a  combinatorial way to write down a potential term from a given word.
Note that given any polygon $U$ for the potential $W$, the index $p$ (which is a finite number) of the boundary word of a polygon $U$ tells us that $U$ is counted at the generation $ \mathcal{S}^p$.
Hence, this provides an algorithm to compute the potential $W$. 

\begin{example}\label{ex:enum}
Let $a=b=c=4$. Then 
$\CS^{-1}=\{[w^{-1}_1=\alpha\beta\gamma]$,
$\CS^0=\{[w^0_1=(\beta\gamma)^{4}],[w^0_2=(\gamma\alpha)^{4}],[w^0_3=(\alpha\beta)^{4}]\}$.
By applying a cut--glue operation on $[(\beta\gamma)^{4}]$ we get
\[
w^1_1=\lefteqn{\overbrace{\phantom{(\beta\gamma)\beta(\alpha\beta)}}^{c_1}}
		(\beta\gamma)\beta
		\lefteqn{\underbrace{\phantom{(\alpha\beta) (\alpha\beta)}}_{c_2}}
		(\alpha\beta)
		\lefteqn{\overbrace{\phantom{(\alpha\beta)\alpha\beta^{-1}}}^{c_3}}
		(\alpha\beta) \alpha
		\lefteqn{\underbrace{\phantom{(\gamma\alpha)(\gamma\alpha)}}_{c_4}}
		(\gamma\alpha)
		\lefteqn{\overbrace{\phantom{(\gamma\alpha) \gamma(\beta\gamma)}}^{c_5}}
		(\gamma\alpha) \gamma		(\beta\gamma)
\]
Since $[w^1_1]$ is invariant under the cycle permutation
$\alpha\to \beta\to\gamma\to\alpha$, we see that $\CS^1=\{[w_1^1]\}$.
In the above equality, we see that $c_1,\ldots,c_5$ are cut subwords of $w_1^1$. We let $c_6$ be the cut subword $(\beta\gamma)$
formed by the last and the first two letters $w_1^1$.
There are six different cut--glue operations on $[w_1]$ corresponding $c_1,\ldots,c_6$. 
Among these, $c_1,c_3$ and $c_5$ result in the same cyclic word
\[
[w^2_1]=[\beta(\alpha\beta)^{2}\alpha(\gamma\alpha)\gamma(\beta\gamma)^2\beta\cdot(\alpha\beta)\alpha(\gamma\alpha)^2\gamma(\beta\gamma)].
\]
The three cut--glue operations $c_2,c_4$ and $c_6$ yield distinct cyclic words
\begin{eqnarray*}
[w^2_2]&=&[(\beta\gamma)\beta\cdot\alpha(\gamma\alpha)^2\gamma(\beta\gamma)^2\beta\cdot\alpha\beta^{-2}\gamma(\beta\gamma)], \\
{[w^2_4]}&=&[(\beta\gamma)\beta(\alpha\beta)^2\alpha\cdot\gamma(\beta\gamma)^2\beta(\alpha\beta)^2\alpha\cdot\gamma(\beta\gamma)], \\
{[w^2_6]}&=&[(\beta\gamma)\beta(\alpha\beta)^2\alpha(\gamma\alpha)^2\gamma\cdot\beta(\alpha\beta)^2\alpha(\gamma\alpha)^2\gamma].
\end{eqnarray*}
Using Theorem~\ref{thm:term} we have the potential of $\mathbb{P}_{4,4,4}$ (shown up to the fifth generation):
\tiny
\begin{eqnarray*}
W&=&
-xyzq+x^{4}q^{12}+y^{4}q^{12}+z^{4}q^{12}-8x^{2}y^{2}z^{2}q^{34}
+6x^{4}z^{4}q^{56}+13x^{3}y^{3}z^{3}q^{67}+6y^{4}z^{4}q^{56}+6x^{4}y^{4}q^{56}
\\ &&
-32x^{2}y^{2}z^{6}q^{78}-34x^{5}yz^{5}q^{89}-32x^{6}y^{2}z^{2}q^{78}-28x^{4}y^{4}z^{4}q^{100}-34xy^{5}z^{5}q^{89}-34x^{5}y^{5}zq^{89}
\\ &&
-32x^{2}y^{6}z^{2}q^{78}
+63x^{3}y^{3}z^{7}q^{111}+10x^{4}z^{8}q^{100}-237x^{6}y^{2}z^{6}q^{122}-237x^{2}y^{6}z^{6}q^{122}+10y^{4}z^{8}q^{100}
\\ &&
-795x^{5}y^{5}z^{5}q^{133}+63x^{7}y^{3}z^{3}q^{111}
+10x^{8}z^{4}q^{100}+10x^{8}y^{4}q^{100}-261x^{6}y^{6}z^{2}q^{122}-1086x^{4}y^{8}z^{4}q^{144}
\\ &&
+63x^{3}y^{7}z^{3}q^{111}-1164x^{4}y^{4}z^{8}q^{144}-1112x^{8}y^{4}z^{4}q^{144}
+10y^{8}z^{4}q^{100}+10x^{4}y^{8}q^{100}-2604x^{6}y^{6}z^{6}q^{166}
\\ &&
-125x^{5}yz^{9}q^{133}-1458x^{7}y^{3}z^{7}q^{155}-1485x^{3}y^{7}z^{7}q^{155}-150xy^{5}z^{9}q^{133}-72x^{2}y^{2}z^{10}q^{122}
-203x^{9}y^{5}z^{5}q^{177}
\\ &&
-1350x^{7}y^{7}z^{3}q^{155}-261x^{5}y^{9}z^{5}q^{177}
-125x^{9}yz^{5}q^{133}-150x^{9}y^{5}zq^{133}-78x^{8}z^{8}q^{144}-232x^{5}y^{5}z^{9}q^{177}
\\ && -72x^{10}y^{2}z^{2}q^{122} -125xy^{9}z^{5}q^{133}-125x^{5}y^{9}zq^{133}-52y^{8}z^{8}q^{144}-78x^{8}y^{8}q^{144}-72x^{2}y^{10}z^{2}q^{122}\cdots
\end{eqnarray*}
\normalsize
\end{example}

\begin{example}\label{ex:enum2}
We use Theorem~\ref{thm:term} to write down the disk potential of
$\mathbb{P}_{3,4,5}$ up to the fifth generation as
\tiny
\begin{eqnarray*}
 W&=&-xyzq +x^{3}q^{9} +y^{4}q^{12} -z^{5}q^{15} 
+8xy^{2}z^{3}q^{34} 
 +6x^{2}z^{6}q^{56} +13xy^{3}z^{5}q^{67} -11x^{2}y^{4}zq^{53} 
\\ &&+34x^{2}yz^{8}q^{89} -60x^{3}y^{2}z^{4}q^{75} 
 +72xy^{4}z^{7}q^{100} -96x^{2}y^{5}z^{3}q^{86} -154x^{3}y^{3}z^{6}q^{108} +77x^{2}y^{2}z^{10}q^{122} \\&&+92xy^{5}z^{9}q^{133} 
 -2331x^{3}y^{4}z^{8}q^{141} -19x^{4}z^{7}q^{97} -210x^{2}y^{6}z^{5}q^{119} +36x^{4}y^{4}z^{2}q^{94} +6y^{8}z^{8}q^{144} 
\\&& +1175x^{4}y^{5}z^{4}q^{127} +38x^{3}y^{7}zq^{105} -312x^{4}yz^{9}q^{130} -2106x^{2}y^{7}z^{7}q^{152} +1120x^{3}y^{5}z^{10}q^{174} 
\\&&+702x^{2}y^{3}z^{12}q^{155} -858x^{4}y^{6}z^{6}q^{160} +78x^{3}z^{13}q^{144} +135x^{4}y^{2}z^{11}q^{163} +644xy^{6}z^{11}q^{166} +87x^{2}y^{8}z^{9}q^{185} 
\\&&+56x^{5}y^{4}z^{9}q^{182} +242x^{5}y^{2}z^{5}q^{116} -75x^{5}y^{3}z^{7}q^{149} +888x^{3}y^{8}z^{3}q^{138} -81xy^{10}z^{6}q^{163} +87y^{9}z^{10}q^{177}
\\&& -6y^{5}z^{15}q^{180} +48x^{5}y^{7}z^{2}q^{146}+\cdots
\end{eqnarray*}
\normalsize
\end{example}


\section{Potential in the (2,3,6)-case}\label{sec:236}
In this section we compute the potential for the elliptic case $(a,b,c) = (2,3,6)$.  We will see in Section \ref{sec:mir236} that the countings of polygons in this section give an enumerative meaning of the mirror map.  The triangle and hexagon tessellations are shown in Figure \ref{fig:236tes}.

\begin{figure}[htb!]
\includegraphics[scale = 0.5]{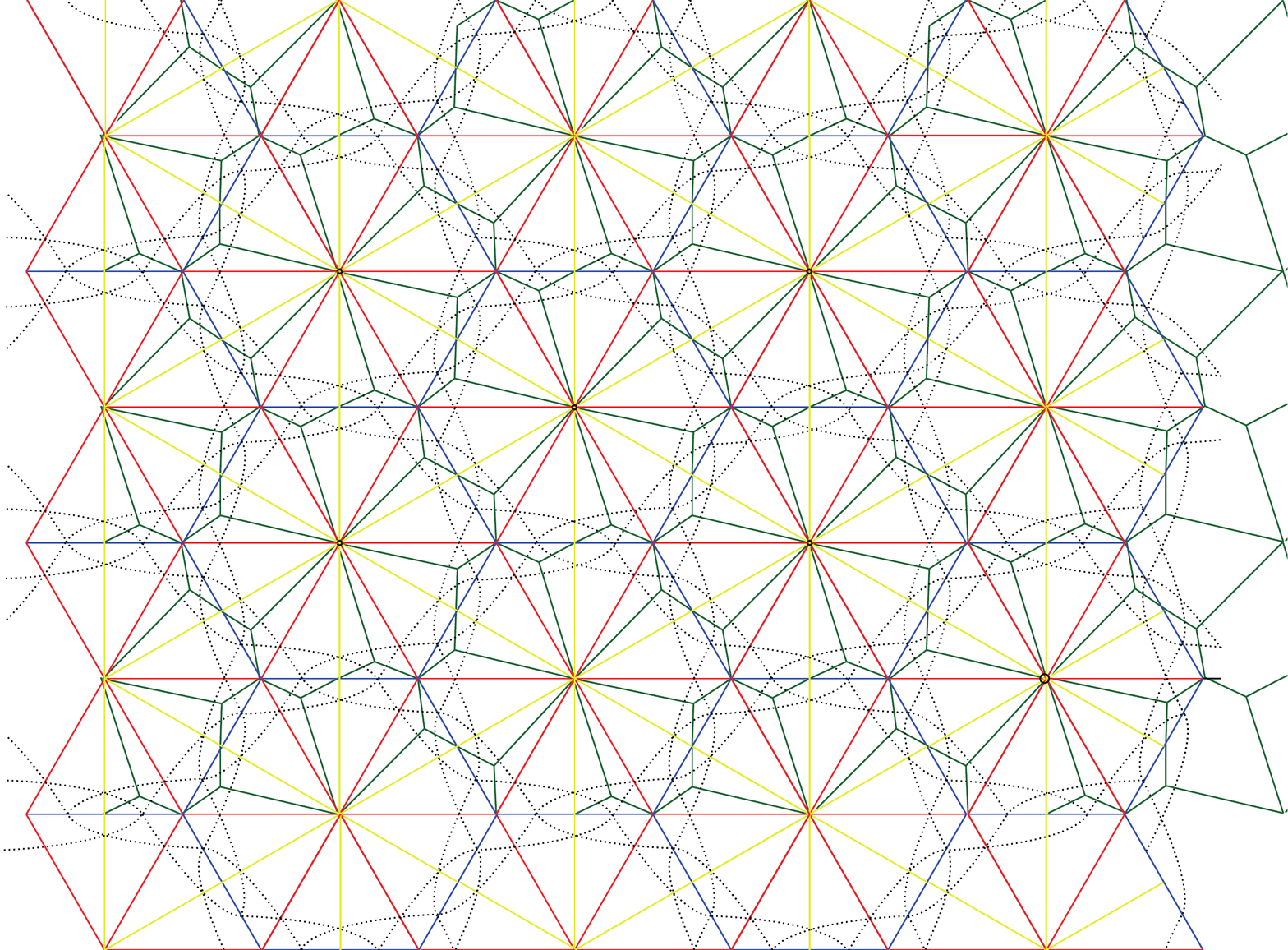}
\caption{The triangle and hexagon tessellations for the elliptic curve quotient $\bP^1_{2,3,6} = E / \Z_6$.  The Seidel Lagrangians are shown by dotted curves.}\label{fig:236tes}
\end{figure}

\begin{theorem}\label{thm:W236}
The potential $W$ for $(a,b,c)= (2,3,6)$ is
$$W= q^6 x^2 - q xyz + c_y(q) y^3 + c_z(q) z^6  + c_{yz2}(q) y^2z^2 + c_{yz4}(q) yz^4,
$$
where 
\begin{align*}
c_y(q) &= \sum_{a\ge0}(-1)^{a+1} (2a+1) q^{48A(a-1,0,0,0) + 9};\\
c_{yz2}(q) &= \sum_{n\ge a\ge0} \big(
(-1)^{n-a}(6n-2a+8) q^{48A(n,a,0,0)-4} + (2n+4) q^{48 A(n,a,n-a,0)- 4} \big);\\
c_{yz4}(q) &= 
\sum_{a,b\ge0, n\ge a+b} (-1)^{n-a-b} (6n-2a-2b+7) q^{48 A(n,a,b,0)-17};\\
c_z(q) &= \sum (-1)^{n-a-b-c}\left({6n - 2a - 2b - 2c + 6}\over{\eta(n,a,b,c)}\right)
\cdot q^{48 A(n,a,b,c)-30}
\end{align*}
where 
$$A(n,a,b,c) :=  {n+2\choose 2} - {a+1 \choose 2} - {b+1 \choose 2} - {c+1 \choose 2},$$
the summation in the expression of $c_z(q)$ is taken over
$(n,a,b,c)\in T_1\coprod T_2\coprod T_3\coprod T_6$,
\begin{align*}
T_6 =& \{(3a, a, a, a) \; : a\ge0\}, \\ 
T_3 =& \{(n,a,a,a)\;: n>3a\ge0\}, \\
T_2 =& \{(a+b+c,a,b,c) \;: a,b,c\ge0\text{ such that }a<\min(b,c)\text{ or }a=c< b \}, \\
T_1 =& \{ (a+b+c+k,a,b,c) \;: k \in \Z_{>0}, a,b,c\text{ are distinct non-negative integers such that } \\
& a<\min(b,c) \text{ or }a=c< b\},
\end{align*}
and $\eta(n,a,b,c)=i$ for $(n,a,b,c)\in T_i$.
\end{theorem}

The $C$-vertices form a tessellation of $\mathbb{R}^2$ by equilateral triangles.
Given a polygon $U$ for the potential, the convex hull $U_0$ of the $C$-vertices contained in $U$  is a (possibly degenerate) hexagon; such a hexagon is called an \emph{Eisenstein lattice hexagon}.
As was noticed in \cite{Th}, the hexagon $U_0$ is obtained from an equilateral lattice triangle of side-length $n$ by cutting off three equilateral lattice triangles of side-lengths $a,b$ and $c$; see Figure \ref{fig:eisenstein_hexa}. 
We require $n \ge \max(a+b, b+c, c+a)$ and $a,b,c\ge0$.
The expression
\[A(n,a,b,c) =  {n+2\choose 2} - {a+1 \choose 2} - {b+1 \choose 2} - {c+1 \choose 2}\]
is the number of $C$-vertices in $U_0$.
From Theorem \ref{thm:area} (2), we have the relation $P/2+Q/3+R/6 = 1$. Since each $P, Q,R$ of a polygon gives th (non-negative) exponents of $x,y,z$ in $W$,  we see that the potential has only
$xyz, x^2, z^6, y^2 z^2, y^3$ and $yz^4$ terms.

Knowing the number $v_C$ of $C$-vertices in a polygon $U$ for the potential, we  obtain its area
as $\area(U) = 48v_C +3S-8R$ from the area formula in Lemma \ref{lem:area}.

\begin{figure}[htb!]
\includegraphics[width=.3\textheight]{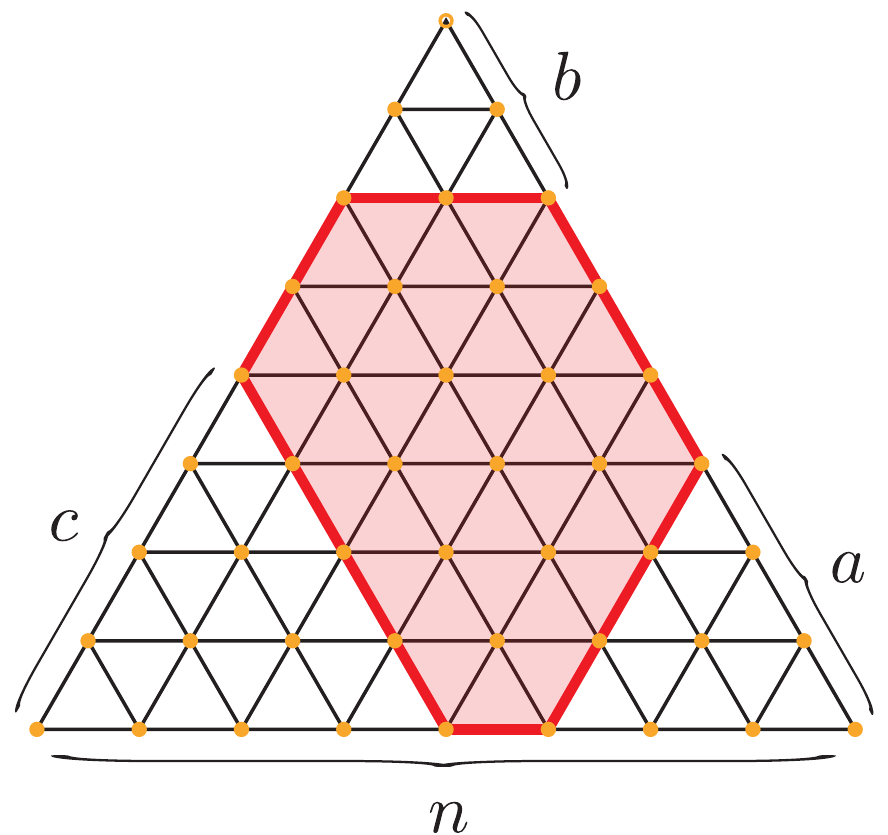}
\caption{The Eisenstein lattice hexagon}\label{fig:eisenstein_hexa}
\end{figure}

The following two terms $c_{x} x^2 + c_{xyz} xyz$ can be calculated easily:
\begin{equation}\label{eq:236x}
q^6 x^2- qxyz
\end{equation}
Here, the coefficient of $x^2$ is given by the unique  bigon with two $x$-corners
passing through $e$ at two points, which has area $\sigma^6$.
Due to $\Z/2$-symmetry, we count it as $2 \times \frac{1}{2} = 1$.
Now, the coefficient of $xyz$ is also given by a single term. Starting from $x$-corner, one note
that any $xyz$-polygon (contributing to the potential) should lie inside the bigon with two $x$-corners. Hence only the minimal triangle is possible for $xyz$.

\subsection{The coefficient of $z^6$ term}
We have $P=Q=0$ and $R=S=6$
and $U$ has six $z$-corners,
Let us denote the side-lengths of $U_0$ as $(r_0, r_1, \ldots, r_5) = ( a , n - a - b, b, n - b - c, c, n - c - a)$ for $r_j\ge0$ and $a,b,c\ge0$. (See Figure \ref{fig:z6hexa}.)
We have $v_C =A(n,a,b,c)$.
The boundary word of $U$ is
\[w = \prod_{j=0}^5 \left( \beta( \alpha \gamma\beta)^{r_j} \alpha \right).\]
The exponents $r_j$ in $w$ can be verified by observing that the side of $U_0$ with length $r_j$ yields $(r_j+1)$ copies of subwords $\beta\alpha$ in the boundary cyclic word $[w] = [\partial U]$. 
\begin{figure}[htb!]
\includegraphics[width=.4\textheight]{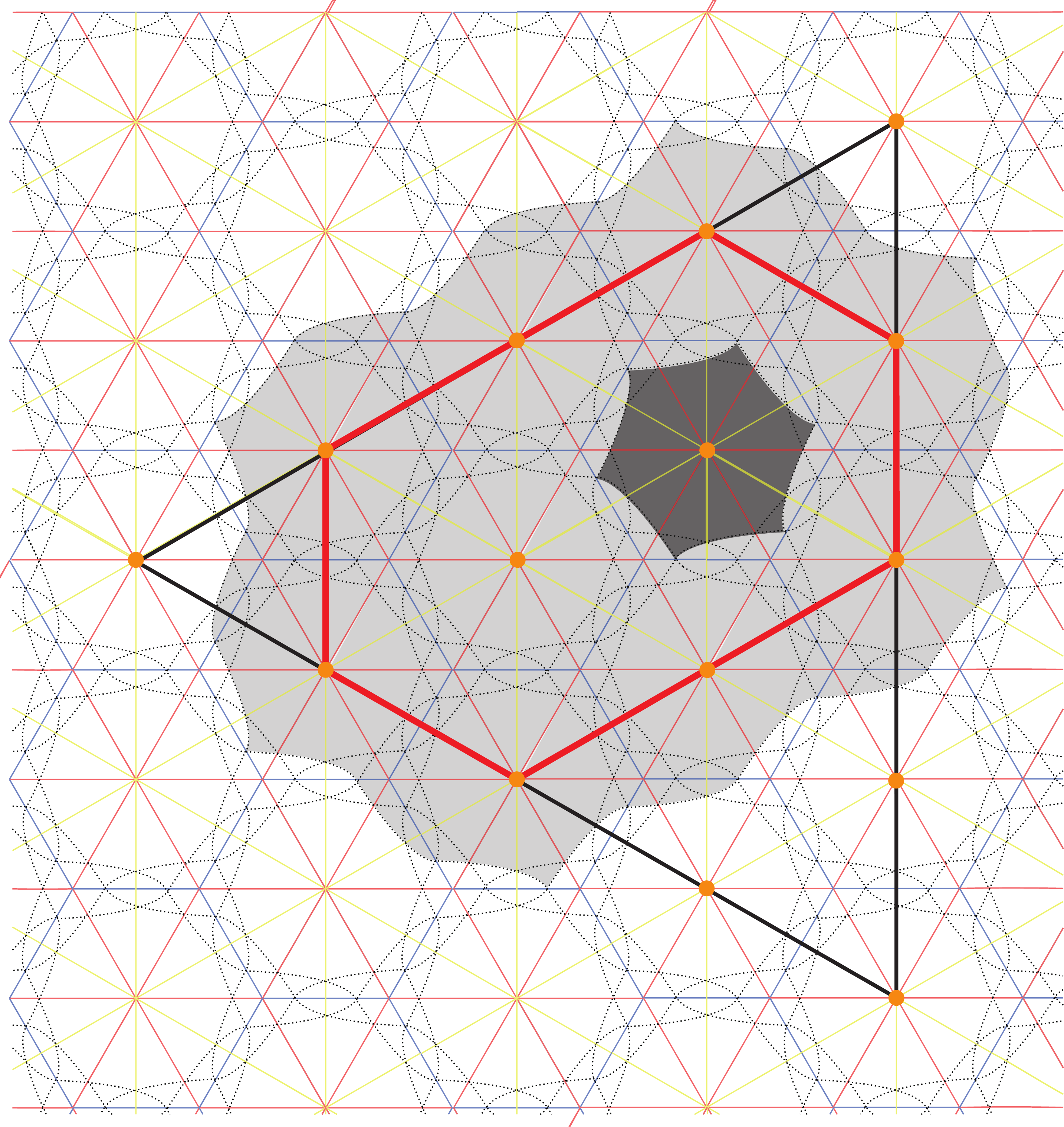}
\caption{$z^6$-hexagon with $(n,a,b,c) = (5,2,1,1)$}\label{fig:z6hexa}
\end{figure}

Note that the above multiplication is grouped so that each $j$-th term corresponds to the $j$-th side of $U_0$.
To enumerate  polygons for the potential,
we need to consider the parameters $(n,a,b,c)\in \Z_{\ge0}^4$ so that each cyclic word is counted exactly once.
For example, the quadruples $(5,2,2,2)$ and $(4,1,1,1)$ both yield the same cyclic word
$[((\beta \alpha \gamma)^2\beta\alpha(\beta \alpha \gamma)\beta\alpha)^3]$.
Let us denote by $\eta(n,a,b,c)$ the maximal period of the cyclic word that comes from the parameter $(n,a,b,c)$.
In other words, $\eta(n,a,b,c)$ is the period of cyclic symmetry of $(r_0, r_1, \ldots, r_5)$
and this is equal to $\eta$-value of the polygon $U$.
Note that
\begin{equation*}
\eta(n,a,b,c) :=  \left\{
\begin{array}{lll}
6 & \mbox{if} \,\,\, a=b=c=n/3 &\\
3 & \mbox{if} \,\,\, a=b=c  & (\mbox{but, not} \,\,\, a=n/3) \\
2 & \mbox{if} \,\,\, a+b+c=n & (\mbox{but, not} \,\,\, a=b=c)   \\ 
\end{array}\right.
\end{equation*}
and $\eta(n,a,b,c)=1$ in all the remaining cases.
In order to avoid duplicate counting, we define 
\begin{eqnarray*}
T_6&=& \{(3a, a, a, a) \; : a\ge0\}, \\ 
T_3 &=& \{(n,a,a,a)\;: n>3a\ge0\}, \\
T_2 &=& \{(a+b+c,a,b,c) \;: a,b,c\ge0\text{ such that }a<\min(b,c)\text{ or }a=c< b \} \\
T_1 &=& \{ (a+b+c+k,a,b,c) \;: k \in \Z_{>0}, a,b,c\text{ are distinct non-negative integers such that } \\
& & a<\min(b,c) \text{ or }a=c< b\}.
\end{eqnarray*}

From the relation 
\[(\text{the word length of }w) = 2p + 3S\]
we have $ p = \sum_j r_j = 3n - a - b - c$.
Using $\area(U) = 48 v_C - 30$, we obtain the $c_z(q) z^6$-term in the potential  where
$c_z(q)$ is given as follows:
\begin{equation}\label{eq:236z}
\sum(-1)^{n-a-b-c}\left({6n - 2a - 2b - 2c + 6}\over{\eta(n,a,b,c)}\right)
\cdot q^{48 A(n,a,b,c)-30}
\end{equation}
where the sum is taken over
$(n,a,b,c)\in T_1\coprod T_2\coprod T_3\coprod T_6$.
Note that $\eta(n,a,b,c)=i$ for $(n,a,b,c)\in T_i$.

\subsection{$y^2 z^2$-term}
The computation is similar to the $z^6$-term. 
We have $P=0, Q=R=2$ and $S=4$
and the convex hull $U_0$ of the $C$-vertices contained in such a polygon $U$ for the potential  has two $y$-corners and two $z$-corners.
The $y$-corners correspond to the angles $\pi/3$ and the $z$-corners correspond to the angles $2\pi/3$ in $U_0$.
Hence $U_0$ is either a trapezoid or a parallelogram, depending on whether or not the two $y$-corners are adjacent.

\emph{Trapezoid case.}
Let the side-lengths be given as $(n,n-a, a, n-a)$ for $n\ge a$ and $a\ge0$. (See Figure \ref{fig:y2z2trape}.)
\begin{figure}[htb!]
\includegraphics[width=.4\textheight]{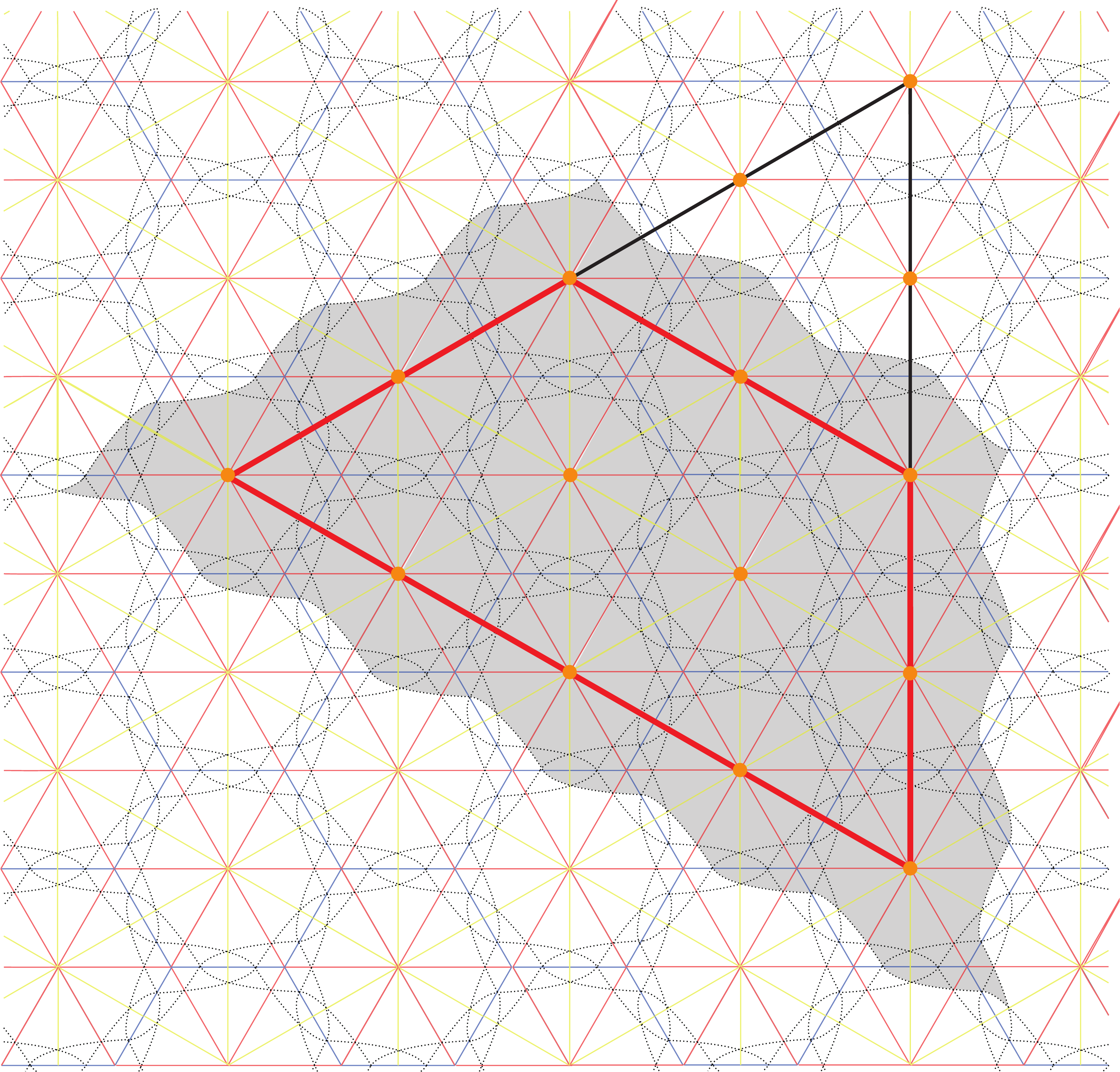}
\caption{$y^2 z^2$-trapezoid with $n=5$ and $a=2$}\label{fig:y2z2trape}
\end{figure}
This corresponds to the words of the form:
\[w =(\alpha (\gamma\beta\alpha)^{n-a} \gamma\beta\alpha)
\cdot(
\beta
(\alpha\gamma\beta)^a \alpha)
\cdot 
(\beta (\alpha\gamma\beta)^{n-a} \alpha\gamma)
\cdot
(\alpha
(\gamma\beta\alpha)^{n+1}\gamma).\]
The above expression is grouped so that each term corresponds to a side of $U_0$.
We have $p=3n-a+2, \eta=1$
and $\area(U) = 48 v_C +3S - 8R = 48A(n,a,0,0)-4$.
So the potential has the terms as follows:
\[
\sum_{n\ge a\ge0}
(-1)^{n-a}(6n-2a+8) q^{48A(n,a,0,0)-4}y^2z^2.\]
\emph{Parallelogram case.}
Let the side-lengths be given as $(a,n-a, a, n-a)$ for $n\ge a\ge0$.
(See Figure \ref{fig:y2z2para}.)
\begin{figure}[htb!]
\includegraphics[width=.5\textheight]{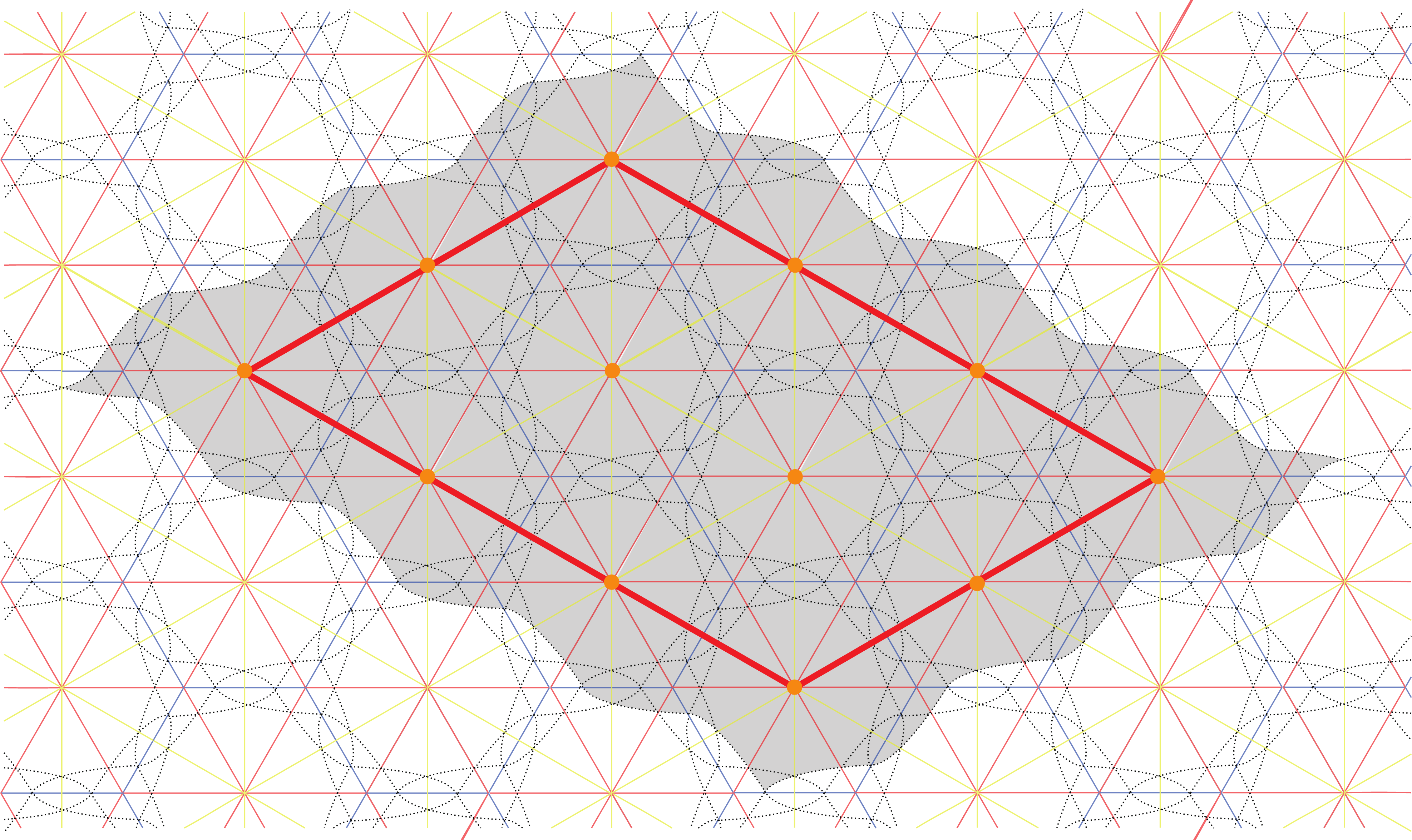}
\caption{$y^2 z^2$-parallelogram with $n=7$ and $a=3$}\label{fig:y2z2para}
\end{figure}

This corresponds to the words:
\[w =(\alpha(\gamma\beta \alpha)^{a+1}) ( \beta\alpha \gamma)^{n-a+1})^2.\]
We have $p=2n+2, \eta=2$
and $\area(U) =48 A(n,a,n-a,0)- 4$.
The corresponding term in the potential is
$$
\sum_{n\ge a\ge0}
(2n+4) q^{48 A(n,a,n-a,0)- 4}y^2z^2.$$

Hence the term $c_{yz2} y^2z^2$ in the potential $W$ is given as
\begin{equation}\label{eq:236yz2}
\sum_{n\ge a\ge0} \big(
(-1)^{n-a}(6n-2a+8) q^{48A(n,a,0,0)-4} + (2n+4) q^{48 A(n,a,n-a,0)- 4} \big) 
y^2z^2.
\end{equation}

\subsection{$y^3$-term}
Similarly, the $C$-vertices in the corresponding polygon for the potential
form an equilateral lattice triangle such that the number of $C$-vertices on one side is $a\ge0$. We have $v_C = A(a-1,0,0,0)$ and $\area(U) = 48 v_C +3S - 8R = 48 A(a-1,0,0,0) + 9$.
The boundary word will then be:
\[
w = (\alpha( \gamma\beta\alpha)^{a} \gamma)^3.\]
From $p=3a$ and $\eta = 3$, we obtain the potential terms
\begin{equation}\label{eq:236y}
\sum_{a\ge0}(-1)^{a+1} (2a+1) q^{48A(a-1,0,0,0) + 9} y^3.
\end{equation}

An example of a triangle for $y^3$-term is drawn in Figure \ref{fig:y3tri}.
\begin{figure}[htb!]
\includegraphics[width=.35\textheight]{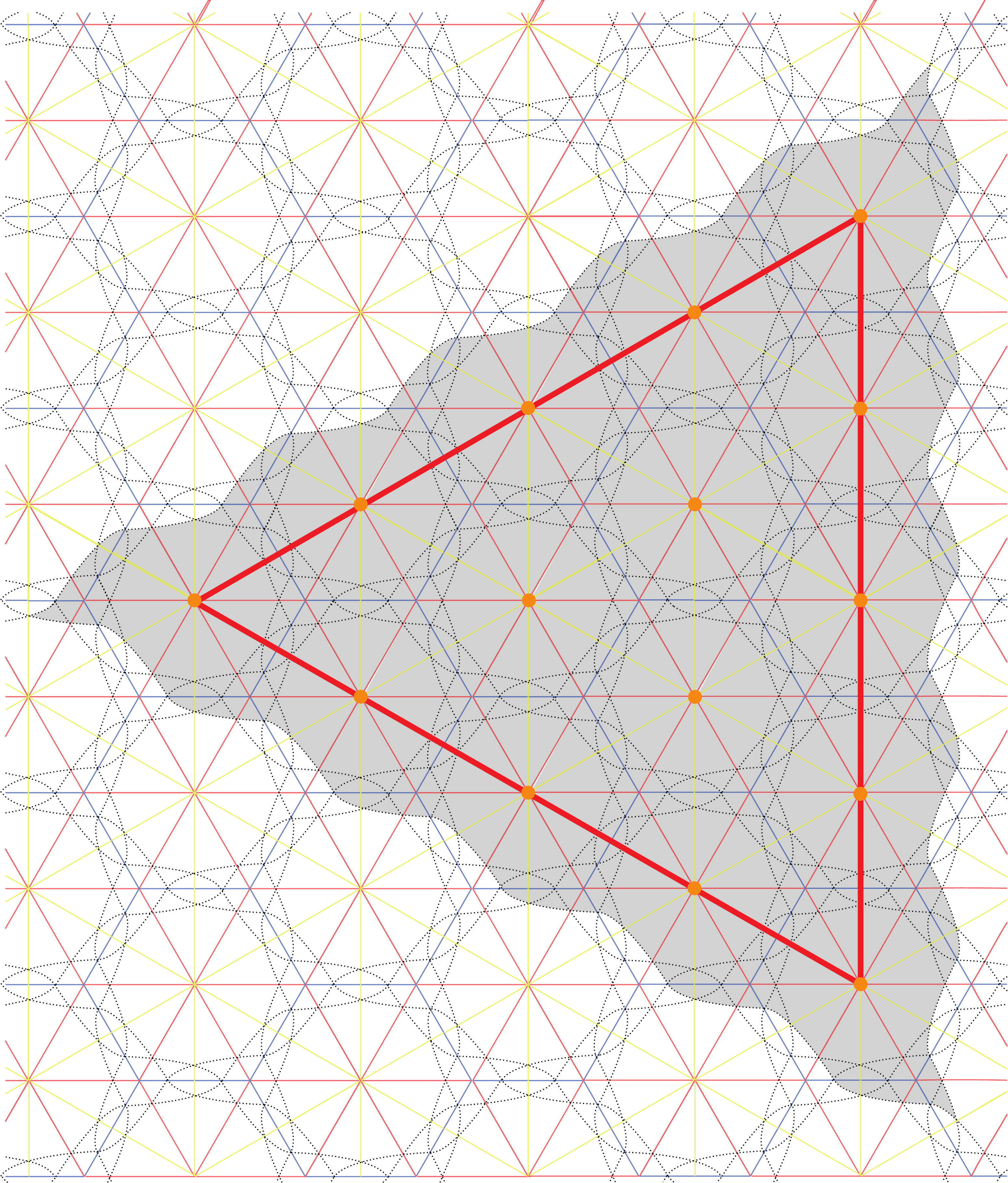}
\caption{$y^3$-triangle with $a=5$}\label{fig:y3tri}
\end{figure}

\subsection{$yz^4$-term}
The $C$-vertices in the polygon for the potential form a lattice pentagon
of side-lengths, say $(n-a,a,n-a-b,b,n-b)$ 
for $a,b\ge0$ and $n\ge a+b$.
This polygon is obtained from an equilateral lattice triangle of side-length $n$
by cutting off lattice triangles of side-lengths $a$ and $b$ at corners. (See Figure \ref{fig:yz4pent}.)
\begin{figure}[htb!]
\includegraphics[width=.35\textheight]{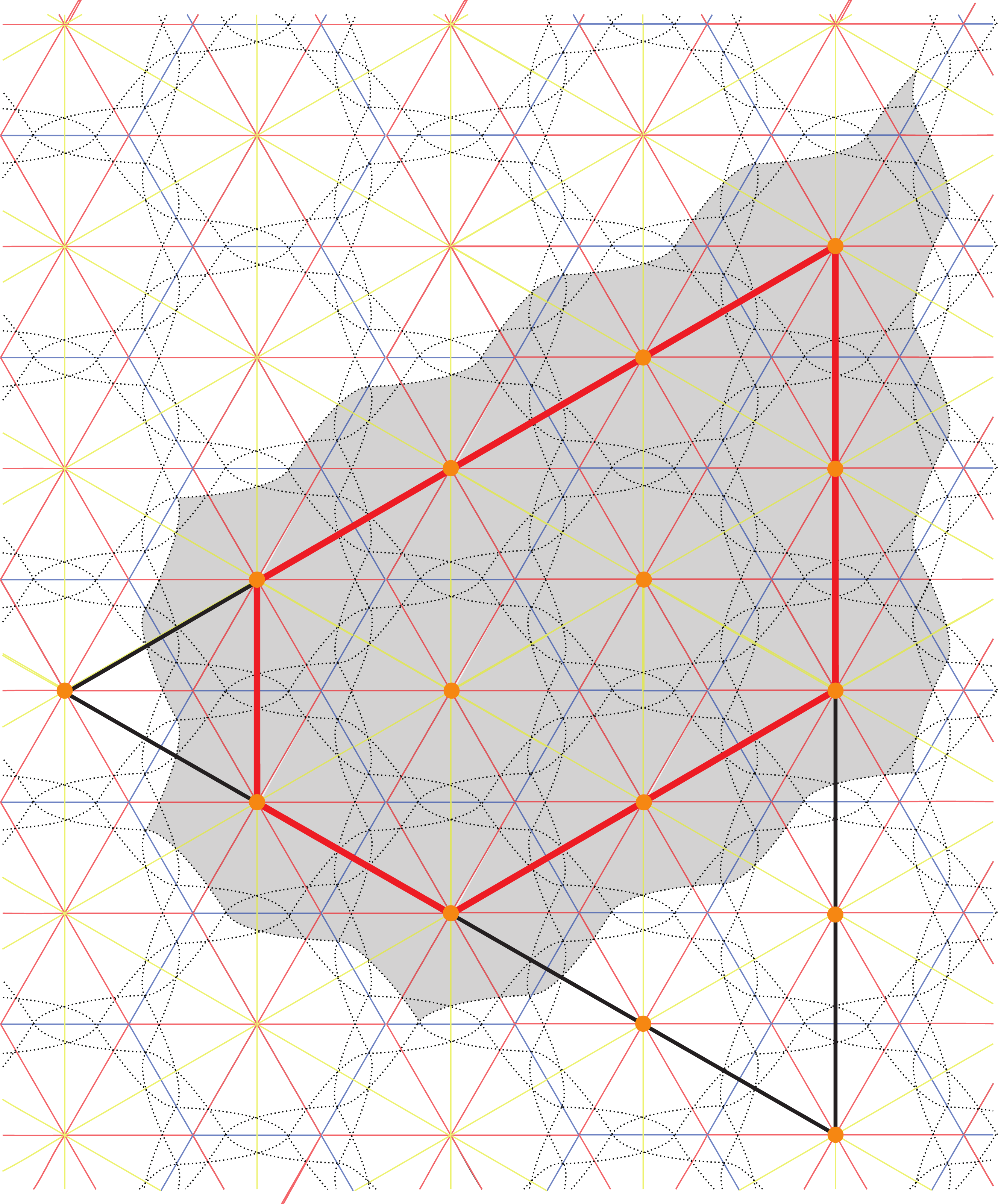}
\caption{$y z^4$-pentagon with $(n,a,b)=(5,2,1)$}\label{fig:yz4pent}
\end{figure}
The boundary word is of the form:
\[w =(\alpha(\gamma\beta \alpha)^{n-b+1})
(\beta ( \alpha \gamma\beta)^{b}\alpha)
(\beta ( \alpha \gamma\beta)^{n-a-b}\alpha)
(\beta ( \alpha \gamma\beta)^{a}\alpha)
(\beta ( \alpha \gamma\beta)^{n-a}\alpha\gamma)
.\]
We have $p=3n-a-b+1, \eta = 1$
and
$\area(U) = 48 v_C + 3S - 8R = 48 A(n,a,b,0)-17$.
So we get the potential terms
\begin{equation}\label{eq:236yz4}
\sum_{a,b\ge0, n\ge a+b}
(-1)^{n-a-b}
(6n-2a-2b+7) q^{48 A(n,a,b,0)-17}yz^4.
\end{equation}


\section{Potential in the (2,4,4)-case}\label{sec:244}

In this section, we compute the potential for the elliptic case $(a,b,c) = (2,4,4)$.  We will see in Section \ref{sec:mir244} that the countings of polygons in this section give an enumerative meaning of the mirror map.

\begin{theorem}\label{thm:W244}
The potential $W$ for $(a,b,c)= (2,4,4)$ is
$$W= q^6 x^2 - q xyz + d_y(q) y^4 + d_z(q) z^4 + d_{yz}(q) y^2z^2,
$$
where 
\begin{align*}
d_y(q) = d_z(q) &= \sum_{0 \le r} (2r+1) q^{16 (2r+1)^2-4}
+\sum_{0\le r < s} (2r+2s+2) q^{16 (2r+1)(2s+1)-4}, \\
d_{yz}(q) &= \sum_{r\ge1, s\ge1} \big( - (4r+4s-2) q^{16(2r-1)2s-4} +  (2r+2s) q^{64rs-4} \big).
\end{align*}
\end{theorem}

From Theorem \ref{thm:area} (2), we have $P/2 + Q/4 + R/4=1$.
The following two terms when $P \geq 1$, can be computed as in $(2,3,6)$ case:
\begin{equation}\label{eq:244x}
q^6 x^2- qxyz
\end{equation}

So let us now assume $P=0$ and $Q+R = 4 = S$.
Let $U$ be a polygon for the potential.
It will be useful to recall from Lemma \ref{lem:relations} that
$4 v_B + 4 - Q = 4 v_C + 4 - R$ and hence, $v_C = (v_B+v_C - (Q-R)/4)/2$.
It follows from the area formula in Lemma \ref{lem:area}
\[\area(U)= 8v_W  - 5S - 8p= 32 v_C + 12 - 8R
= 16(v_B+v_C)-4.\]
Schematically, the Seidel Lagrangians give a tessellation of $\mathbb{R}^2$ by unit squares
and each square contains either a $B$-vertex or a $C$-vertex, in the checkerboard pattern. (See Figure \ref{fig:sch_4gon}.) A polygon for the potential is a rectangle formed by these unit squares.
The types ($y$- or $z$-) of the corners on a side of this rectangle coincide
if and only if the length of that side is odd.
From this, we see that the potential has $y^4, z^4$ and $y^2z^2$ terms other than $xyz$ and $x^2$ terms.
\begin{figure}[htb!]
\includegraphics[width=.25\textheight]{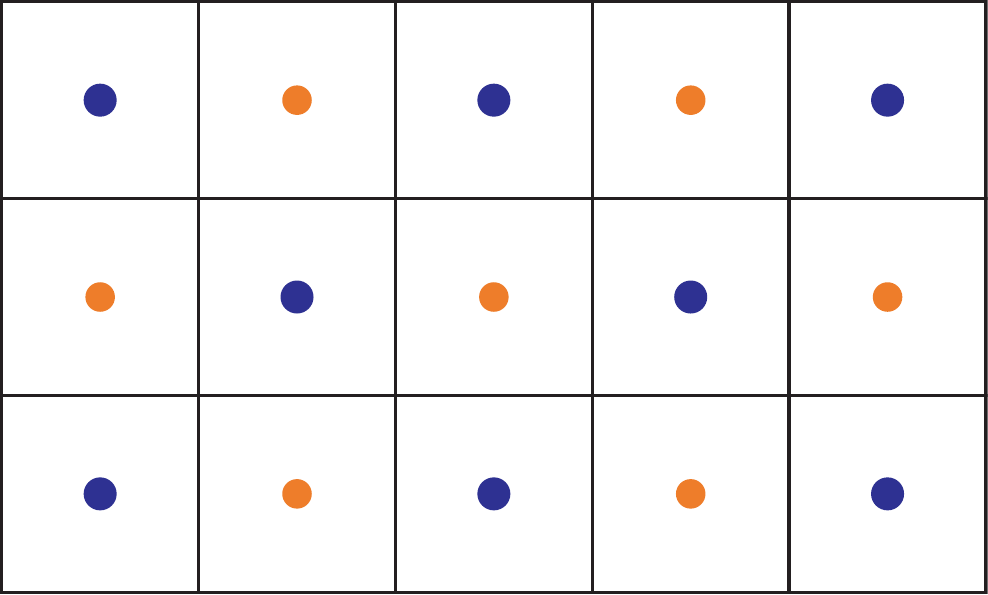}
\caption{A schematic rectangle : a unit rectangle with $B$(resp. $C$)-vertex has four $y$-(resp. $z$-)corners.}\label{fig:sch_4gon}
\end{figure}

\subsection{$y^4$-terms}
By the discussion above, the side-lengths of the corresponding holomorphic rectangle are all odd, say $2r+1$ and $2s+1$ for some $r,s\ge0$.
The boundary word will be:
\[w =((\alpha\gamma\beta)^r \alpha\gamma (\alpha\gamma\beta)^s \alpha\gamma)^2.\]
Note $\eta(U) = 2$ if $r\ne s$, and $\eta(U) = 4$ if $r=s$.
We have $v_B + v_C = (2r+1)(2s+1)$, and so,
$\area(U) = 16 (2r+1)(2s+1)-4$. 
From $p = 2(r+s)$, we obtain the potential terms for $d_{y}(q) y^4$:
\begin{eqnarray}\label{eq:244y}
&\sum_{0\le r} (2r+1) q^{16 (2r+1)^2-4} y^4  \nonumber \\
&+\sum_{0\le r < s} (2r+2s+2) q^{16 (2r+1)(2s+1)-4} y^4.
\end{eqnarray}
\begin{figure}[htb!]
\includegraphics[width=.4\textheight]{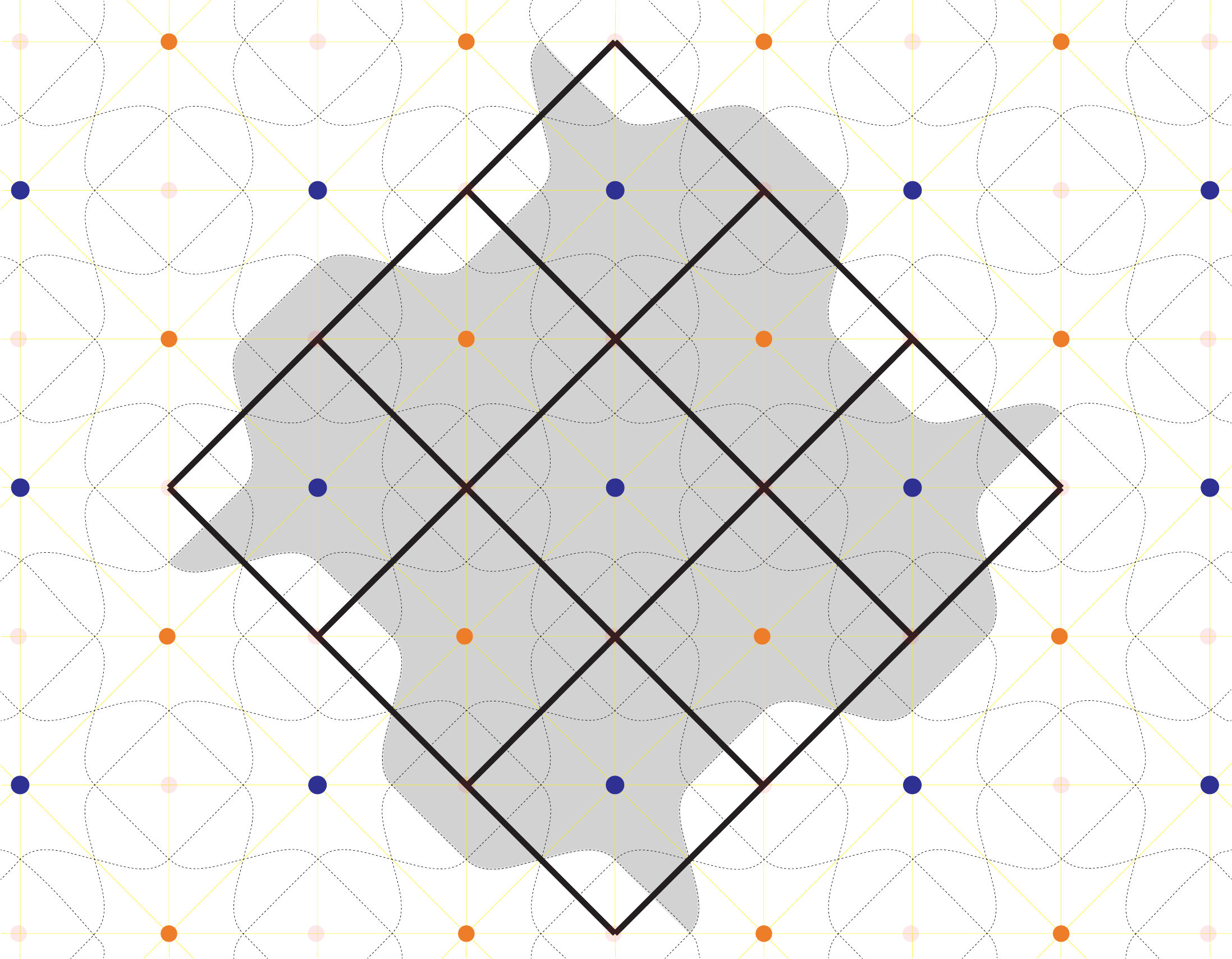}
\caption{$y^4$-rectangle with $r=s=1$}\label{fig:real_y4}
\end{figure}

\subsection{$z^4$-terms}
The computation is very similar to $y^4$-terms.
The boundary word is:
\[w =((\beta\alpha\gamma)^r \beta\alpha (\beta\alpha\gamma)^s \beta\alpha)^2\]
for $r,s\ge0$.
The potential terms  for $d_{z}(q) z^4$ are:
\begin{eqnarray}\label{eq:244z}
&\sum_{0\le r} (2r+1) q^{16 (2r+1)^2-4} z^4  \nonumber \\
&+\sum_{0\le r < s} (2r+2s+2) q^{16 (2r+1)(2s+1)-4} z^4.
\end{eqnarray}

 \subsection{$y^2 z^2$-terms}
 \emph{Case 1.} The two $y$-corners are adjacent. (Figure \ref{fig:realyz}.)
 \begin{figure}[htb!]
\includegraphics[width=.3\textheight]{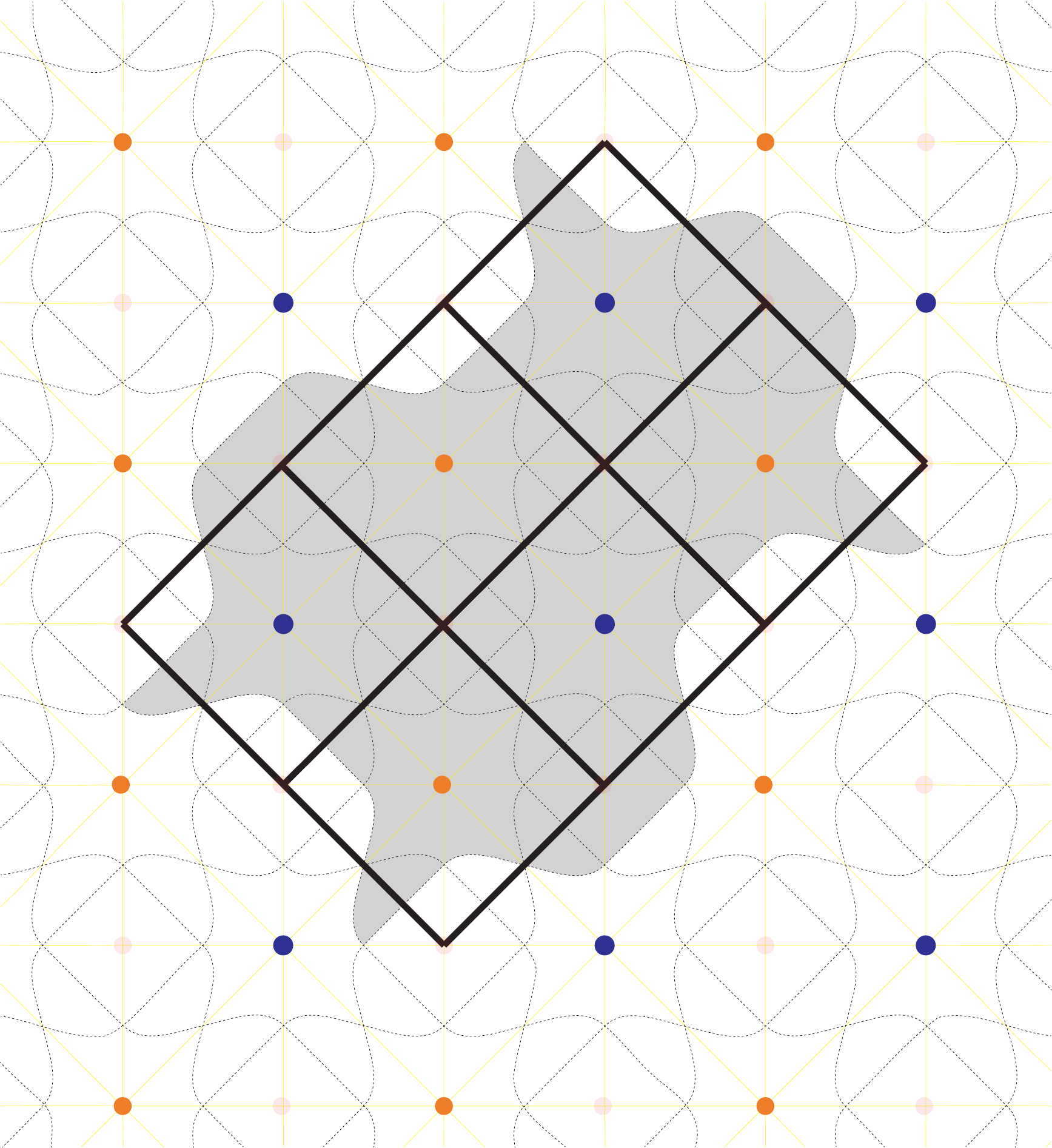}
\caption{$y^2z^2$-rectangle with $r=s=1$}\label{fig:realyz}
\end{figure}

 In this case, one side-length of the corresponding lattice rectangle is $(2r+1)$
 and a neighboring side-length is $2s$ for some $r\ge0$ and $s\ge1$.
 We have $v_B + v_C = (2r+1)2s$.
The boundary word is 
\[w =(\alpha\gamma\beta)^r\alpha\gamma (\alpha\gamma\beta)^s \alpha
 (\beta\alpha\gamma)^r \beta\alpha(\beta\alpha\gamma)^s\]
and $\eta = 1$. From $p = 2r + 2s - 1$, we have the potential terms:
\[
\sum_{r\ge0, s\ge1} - (4r+4s+2) q^{16(2r+1)2s-4} y^2z^2.
\]

 \emph{Case 2.} The two $y$-corners are not adjacent.
The side-lengths of the corresponding lattice rectangle can be written as
$(2r, 2s, 2r, 2s)$ for some $r,s\ge1$.
We have $v_B + v_C = 4rs$, $\eta = 2$
and 
\[
w = 
((\alpha\gamma\beta)^r\alpha (\beta\alpha\gamma)^s)^2.\]
Since $p = 2r+2s-2$, we have the potential terms:
\[
\sum_{r\ge1, s\ge1} (2r+2s) q^{64rs-4} y^2z^2.
\]

Hence, the term $d_{yz}(q) y^2z^2$ of $W$ is given as
\begin{equation}\label{eq:244yz}
\sum_{r\ge1, s\ge1} \big( - (4r+4s-2) q^{16(2r-1)2s-4} +  (2r+2s) q^{64rs-4} \big) y^2z^2.
\end{equation}


\section{Enumerative meaning of mirror maps of elliptic curve quotients}
In Section \ref{sec:236} and \ref{sec:244}, we computed the potentials for $\bP^1_{a,b,c}$ when $(a,b,c) = (2,3,6)$ and $(2,4,4)$.  Together with the case $(a,b,c) = (3,3,3)$, these are all the cases where $\bP^1_{a,b,c}$ is a global quotient of an elliptic curve.
To consider the K\"{a}hler parameter of the orbi-sphere  $\bP^1_{a,b,c}$,  we define $Q$ with $Q = q^8$.
Since $q=e^{-\sigma}$ , we have  $Q = e^{-8\sigma}$, where $\sigma$ is the minimal $XYZ$ triangle,
and $8\sigma$ is the area of  $\bP^1_{a,b,c}$.

In \cite[Section 6.1 and 6.2]{CHL} the potential was computed for the case $(a,b,c) = (3,3,3)$.  Moreover it was proved that the coefficients agrees with the inverse mirror map $\check{q}(Q)$.  Namely, the potential of $\bP^1_{3,3,3}$ takes the form
$$ W = \phi(q) (x^3 + y^3 + z^3) - \psi(q) xyz $$
where $\phi(q)$ and $\psi(q)$ are series in $q$, and $- \log q$ is the area of the minimal $XYZ$ triangle (see \cite[Section 6.1]{CHL}, where $q_\alpha$ there is denoted by $q$ here).  By a change of coordinates on $(x,y,z)$, $W$ can be rewritten as
$$ (x^3 + y^3 + z^3) - \frac{\psi(q)}{\phi(q)} xyz.$$
It was shown that 
$$\check{q}(Q = q^8) = - \psi(q) / \phi(q) $$
and this gives an enumerative meaning of the inverse mirror map $\check{q}(Q)$. 

In this section, we do numerical verifications that our polygon countings in Section \ref{sec:236} and \ref{sec:244} give the inverse mirror maps of $\bP^1_{2,3,6}$ and $\bP^1_{2,4,4}$ respectively.

\subsection{The (2,4,4) case} \label{sec:mir244}
Recall that the potential takes the form
$$W= q^6 x^2 + d_y(q) y^4 + d_y(q) z^4 - q xyz + d_{yz}(q) y^2z^2.
$$
First we eliminate the $xyz$ term by coordinate changes.
Take the change of coordinates $x \mapsto x / q^3$ (and $y,z$ remain unchanged), and we obtain
$$ x^2 + d_y(q) y^4 + d_y(q) z^4 - \frac{xyz}{q^2} + d_{yz}(q) y^2z^2.  $$
Then take the change of coordinates $x \mapsto x + \frac{yz}{2q^2}$ and the potential becomes
$$ x^2 + d_y(q) y^4 + d_y(q) z^4 + \left(d_{yz}(q) - \frac{1}{4q^4} \right) y^2z^2.  $$
Finally, take the change $y \mapsto y / d_y^{\frac{1}{4}}$, $z \mapsto z / d_y^{\frac{1}{4}}$, we get
\begin{equation}\label{eq:244co}
 x^2 + y^4 + z^4 + \frac{d_{yz}(q) - (4q^4)^{-1}}{d_y} y^2z^2.  
\end{equation}

\begin{equation} \label{eq:sigma244}
\sigma(q) := \frac{d_{yz}(q) - (4q^4)^{-1}}{d_y}
\end{equation}
takes the form
$$-\frac{1}{4 q^{16}} - 5 q^{16} + \frac{31 q^{48}}{2} - 54 q^{80} + \ldots $$
and so $\sigma = \infty$ when $q = 0$.

We conjecture that $\sigma(q)$ equals to the inverse mirror map $\check{q}(Q)$ under the relation $Q = q^8$.

The mirror map $Q(\check{q})$ takes the form
$$ Q(x) = \frac{\pi_B(x)}{\pi_A(x)} $$
where $x = \frac{\check{q}^2}{4}$,
$\pi_A(x), \pi_B(x)$ satisfies the Picard-Fuchs equation
$$ x(1-x) \frac{\der^2 Q}{\der x^2} + \left( \frac{1-3x}{2} \right) \frac{\der Q}{\der x} - \frac{Q}{16} = 0. $$


Let
$$ i_{244}(\check{q}) := \frac{16 (\check{q}^2+12)^3}{(\check{q}^2-4)^2}.  $$
By Saito\rq{}s theory on simple elliptic singularities \cite{Saito}, the inverse mirror map $\check{q}(Q)$ satisfies
$$ i_{244}(\check{q}(Q)) = j(Q^4) $$
where $j(x)$ is the $j$-function for elliptic curves, which takes the form
$$ \frac{1}{x} + 744 + 196884 x + 21493760 x^2 + 864299970 x^3 + 20245856256x^4 + 333202640600 x^5 + \ldots $$
The power of $Q^4$ comes from the geometric fact that $\bP^1_{2,4,4} = E / \Z_4$, and hence the elliptic curve $E$ has area four times $\bP^1_{2,4,4}$.

By using Mathematica, we have verified that
\begin{align*}
i_{244} \left( \frac{d_{yz}(q) - (4q^4)^{-1}}{d_y} \right) =& \frac{1}{q^{32}} + 744 + 196884 q^{32} + 21493760 q^{64} + 864299970 q^{96} \\ &+ 20245856256 q^{128} + 333202640600 q^{160} + \ldots
\end{align*}
This gives a strong evidence that the inverse mirror map can be obtained from polygon countings:
$$\check{q}(Q = q^8) = \frac{d_{yz}(q) - (4q^4)^{-1}}{d_y}.$$

\subsection{The (2,3,6) case} \label{sec:mir236}
Recall that the potential takes the form
$$W= q^6 x^2 + c_y(q) y^3 + c_z(q) z^6 - q xyz + c_{yz2}(q) y^2z^2 + c_{yz4} yz^4.$$

We do changes of coordinates to simplify the expression.  As in the $(2,4,4)$ case, we can do change of coordinates on $x$ to eliminate the term $xyz$ and obtain
$$ x^2 + c_y(q) y^3 + \left( c_{yz2}(q) - (4q^4)^{-1}\right)y^2z^2 +  c_{yz4} yz^4 + c_z(q) z^6.$$

To eliminate the term $y^2 z^2$, take the coordinate change
$$ y \mapsto y - \frac{c_{yz2}(q) - (4q^4)^{-1}}{3 c_y(q)} z^2. $$
The potential becomes
\begin{multline*}
x^2 + c_y(q) y^3 + \left(c_{yz4}(q) - \frac{c_{yz2}^2(q)}{3 c_y(q)} - (48 q^8 c_y(q))^{-1} + \frac{c_{yz2}(q)}{6 q^4 c_y(q)} \right)yz^4 \\ + \left(c_z(q) + \frac{2 c_{yz2}^3(q)}{27 c_y^2(q)} - \frac{c_{yz2}(q) c_{yz4}(q)}{3 c_y(q)} - (864 q^{12} c_y^2(q))^{-1} + \frac{c_{yz2}(q)}{72 q^8 c_y^2(q)} - \frac{c_{yz2}^2(q)}{18 q^4 c_y^2(q)} + \frac{c_{yz4}(q)}{12 q^4 c_y(q)} \right)z^6.
\end{multline*}
Finally, take $y \mapsto y / c_y^{\frac{1}{3}}(q)$ and
$$ z \mapsto z \left(c_z(q) + \frac{2 c_{yz2}^3(q)}{27 c_y^2(q)} - \frac{c_{yz2}(q) c_{yz4}(q)}{3 c_y(q)} - (864 q^{12} c_y^2(q))^{-1} + \frac{c_{yz2}(q)}{72 q^8 c_y^2(q)} - \frac{c_{yz2}^2(q)}{18 q^4 c_y^2(q)} + \frac{c_{yz4}(q)}{12 q^4 c_y(q)} \right)^{-\frac{1}{6}}$$
one obtains
\begin{equation}\label{eq:236co}
 x^2 + y^3 + z^6 + \sigma(q) yz^4
\end{equation}
where
\begin{multline} \label{eq:sigma236}
\sigma(q) = \left(c_{yz4}(q) - \frac{c_{yz2}^2(q)}{3 c_y(q)} - (48 q^8 c_y(q))^{-1} + \frac{c_{yz2}(q)}{6 q^4 c_y(q)} \right) c_y^{-\frac{1}{3}}(q)\\
\cdot \left(c_z(q) + \frac{2 c_{yz2}^3(q)}{27 c_y^2(q)} - \frac{c_{yz2}(q) c_{yz4}(q)}{3 c_y(q)} - (864 q^{12} c_y^2(q))^{-1} + \frac{c_{yz2}(q)}{72 q^8 c_y^2(q)} - \frac{c_{yz2}^2(q)}{18 q^4 c_y^2(q)} + \frac{c_{yz4}(q)}{12 q^4 c_y(q)} \right)^{-\frac{2}{3}}.
\end{multline}

$\sigma(q)$ takes the form
$$ - \frac{3}{2^{2/3}} - 864 \cdot 2^{1/3} q^{48} - 352512 \cdot 2^{1/3} q^{96} + \ldots $$
and so $\sigma(q) = - \frac{3}{2^{2/3}}$ at $q=0$.

We conjecture that $\sigma(q)$ equals to the inverse mirror map $\check{q}(Q)$ under the equality $Q = q^8$.
 The mirror map $Q(\check{q})$ takes the form
$$ Q(x) = \frac{\pi_B(x)}{\pi_A(x)} $$
where $x = -\frac{4\check{q}^3}{27}$,
$\pi_A(x), \pi_B(x)$ satisfies the Picard-Fuchs equation
$$ x(1-x) \frac{\der^2 Q}{\der x^2} + \frac{2-5x}{3} \cdot \frac{\der Q}{\der x} - \frac{7 Q}{144} = 0. $$
Let
$$ i_{236}(\check{q}) := \frac{1728 (4 \check{q}^3)}{27 + 4 \check{q}^3}.  $$
By Saito\rq{}s theory on simple elliptic singularities, the mirror map $\check{q}(Q)$ satisfies
$$ i_{236}(\check{q}(Q)) = j(Q^6) $$
where $j(x)$ is the $j$-function for elliptic curves, which takes the form
$$ \frac{1}{x} + 744 + 196884 x + 21493760 x^2 + 864299970 x^3 + 20245856256x^4 + 333202640600 x^5 + \ldots $$
The power of $Q^6$ comes from the geometric fact that $\bP^1_{2,3,6} = E / \Z_6$, and hence the elliptic curve $E$ has area six times $\bP^1_{2,3,6}$.

By using Mathematica, we have verified that
$$ i_{236} \left( \sigma(q) \right) = \frac{1}{q^{48}} + 744 + 196884 q^{48} + 21493760 q^{96} + 864299970 q^{144} + 20245856256 q^{192} + 333202640600 q^{240} + \ldots $$
This gives a strong evidence that the inverse mirror map can be obtained from polygon countings:
$$\check{q}(Q = q^8) = \sigma (q).$$


\section{Spherical cases} \label{sect:S2}
We compute open Gromov-Witten potentials for the spherical case, where 
$\frac{1}{a} + \frac{1}{b} + \frac{1}{c} >1$.  In this case the orbifold $\mathbb{P}^1_{a,b.c}$ has a universal cover $\mathbb{S}^2$.
Spherical cases can be divided into the following types.
\begin{enumerate}
 \item[(i)] $(1,a,b)$ for any $a,b \geq 1$,
 \item[(ii)]   $(2,2,r)$ with $r \geq 2$,
 \item[(iii)] $(2,3,4)$, $(2,3,5)$, $(2,3,6)$.
\end{enumerate}
They are known as the $A,D,E$-types respectively.  
Since the open Gromov-Witten potential has only finitely many terms, we will find polygons for the potential directly without using the algorithm introduced before.

We will use elementary move given in Definition \ref{defn:move}.  One can check that if $V$ is obtained by applying elementary move to $U$ (i.e. $\nexists V'$ with $U \subsetneq V' \subsetneq V$), then the value $p$ of $V$ is greater by one than that of $U$. Thus, the generation of a polygon makes sense even in the spherical case.
Note also that the area formula and the formula for the number of generic points on the boundary of polygons are still valid for spherical cases. These spherical cases are known to have close relationships with $ADE$-singularities,

The case of $(2,2,r)$ for odd $r$ is rather special in the sense that a lift of Seidel Lagrangian in the universal cover $E=S^2$ is still immersed. This is the only case that Seidel Lagrangian in $E$ is not embedded. In particular, we have a 1-gon, which
gives a linear term in the corresponding potential $W$. Hence the origin is not a critical point of $W$ in the $(2,2, odd)$ case.

For the spherical cases except $(2,2,\textrm{odd})$, Seidel Lagrangian in the universal cover $E$ becomes a (topological) circle.
\begin{lemma}\label{lem:sphdiss}
In the spherical cases except $(2,2,\textrm{odd})$-case, any Seidel Lagrangian in $E$ bisects the area of the sphere.
\end{lemma}
\begin{proof}
This can be proved by counting the number of middle triangles. There is another proof using the area formula
in Theorem \ref{thm:area}. We take the disc $D$ bounded by a lift of Seidel Lagrangian and apply the area formula.
We have $P=Q=R=0$ as there are no corners, and hence the area of $D$ is given by
$\frac{8\sigma}{\chi}$ for  $\chi = \frac{1}{a} + \frac{1}{b} + \frac{1}{c}-1$. Since $S^2 \to \PO$ is
a $\frac{2}{\chi}$-fold cover, and hence the area of $S^2$ is $\frac{2 \cdot 8 \sigma}{\chi}$, this proves the claim.
\end{proof}
For the spherical cases except $(2,2,\textrm{odd})$, there are holomorphic discs (0-gons) with boundary on Seidel Lagrangians in $E$ which contributes to the constant term in the potential $W$. One can find that the words corresponding to these discs are given as:
\begin{equation}
\begin{array}{rl}
(2,2,r=even) :& (\gamma \beta \alpha)^{r/2} \\
(2,3,3) :& (\gamma \beta \alpha)^2 \\
(2,3,4) :& (\gamma \beta \alpha)^3 \\
(2,3,5) :& (\gamma \beta \alpha)^5, \\
\end{array}
\end{equation}
which shows the rotational symmetries (i.e. $\eta \neq 1$).

\subsection{$(2,2,r)$-case}
The universal cover $S^2$ is tessellated by preimages of upper and lower hemispheres in $\mathbb{P}^1_{2,2,r}$ as drawn in (a) of Figure \ref{fig:tess225}.  The preimage in $S^2$ of the Seidel Lagrangian is drawn as dashed curves in (b) of Figure \ref{fig:tess225}. To read off information of polygons for the potential effectively, we again use $(a,b,c)$-diagram constructed in Section \ref{sec:abcdia}. See Figure \ref{fig:abc225}. As before, subwords $(\beta \gamma)$, $(\gamma \alpha)$ and $(\alpha \beta)$ represent $x$, $y$ and $z$ respectively.

\begin{figure}[htb]
\begin{center}
\subfloat[(a)]{
\includegraphics[height=.24\textheight]{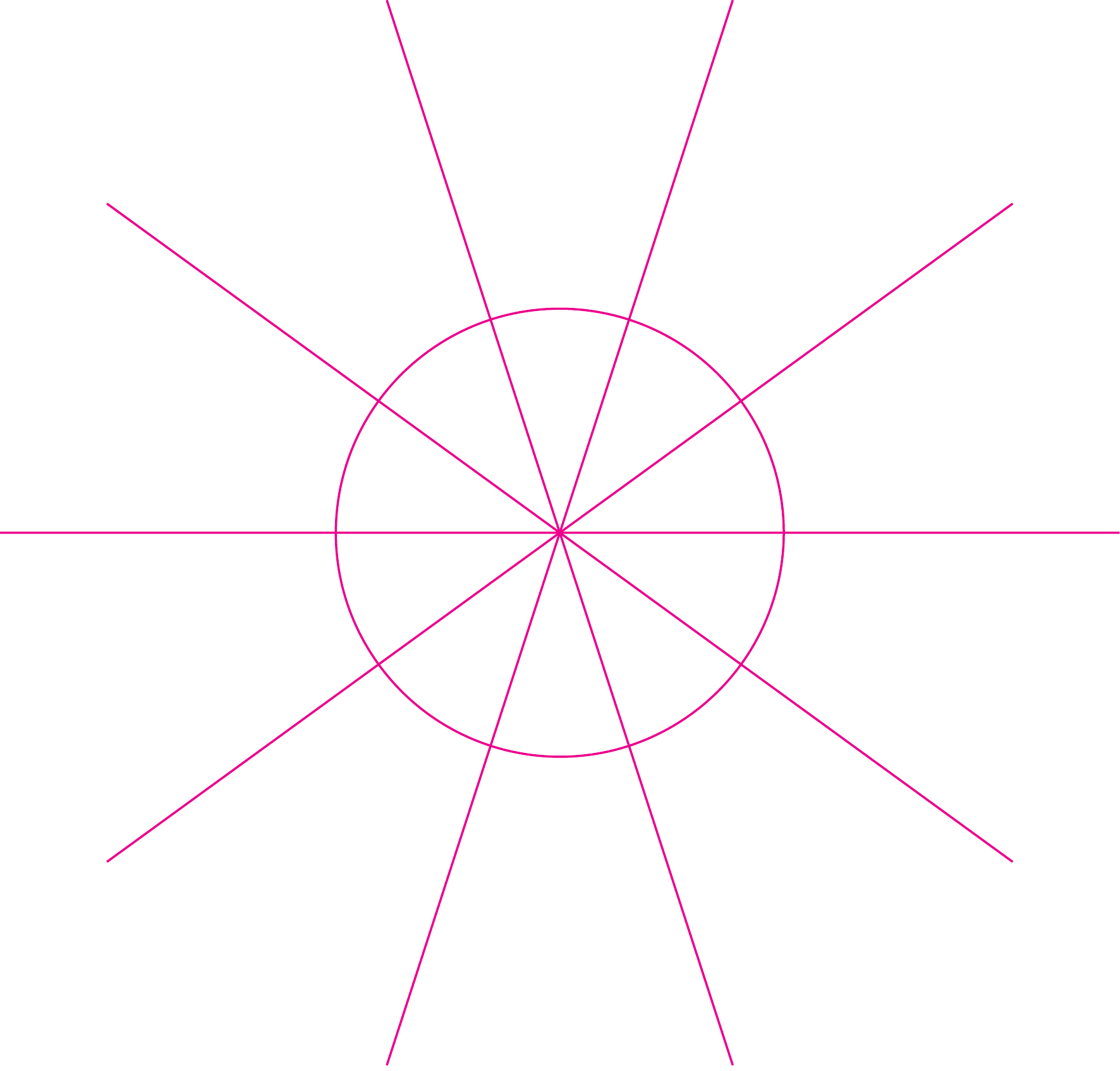}
}
$\qquad\qquad$
\subfloat[(b)]{
\includegraphics[height=.24\textheight]{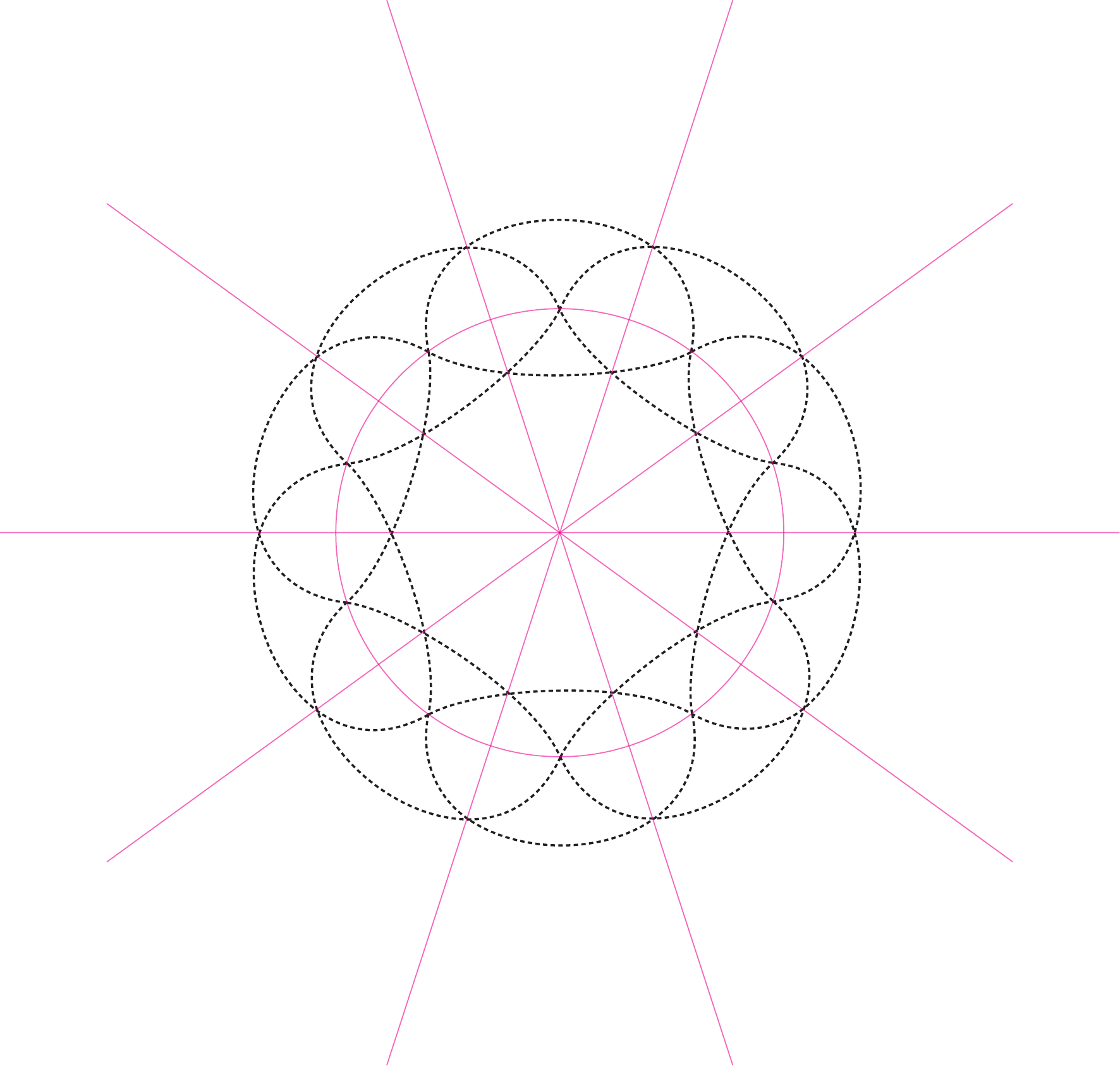}
}
\end{center}
\caption{Tessellation of $S^2$ induced by $\mathbb{P}^1_{2,2,5}$ and the preimage of the Seidel Lagrangian (one of vertices with weight $r(=5)$ lies at infinity.)}\label{fig:tess225} 
\end{figure}

\begin{figure}[ht]
\begin{center}
\includegraphics[height=3in]{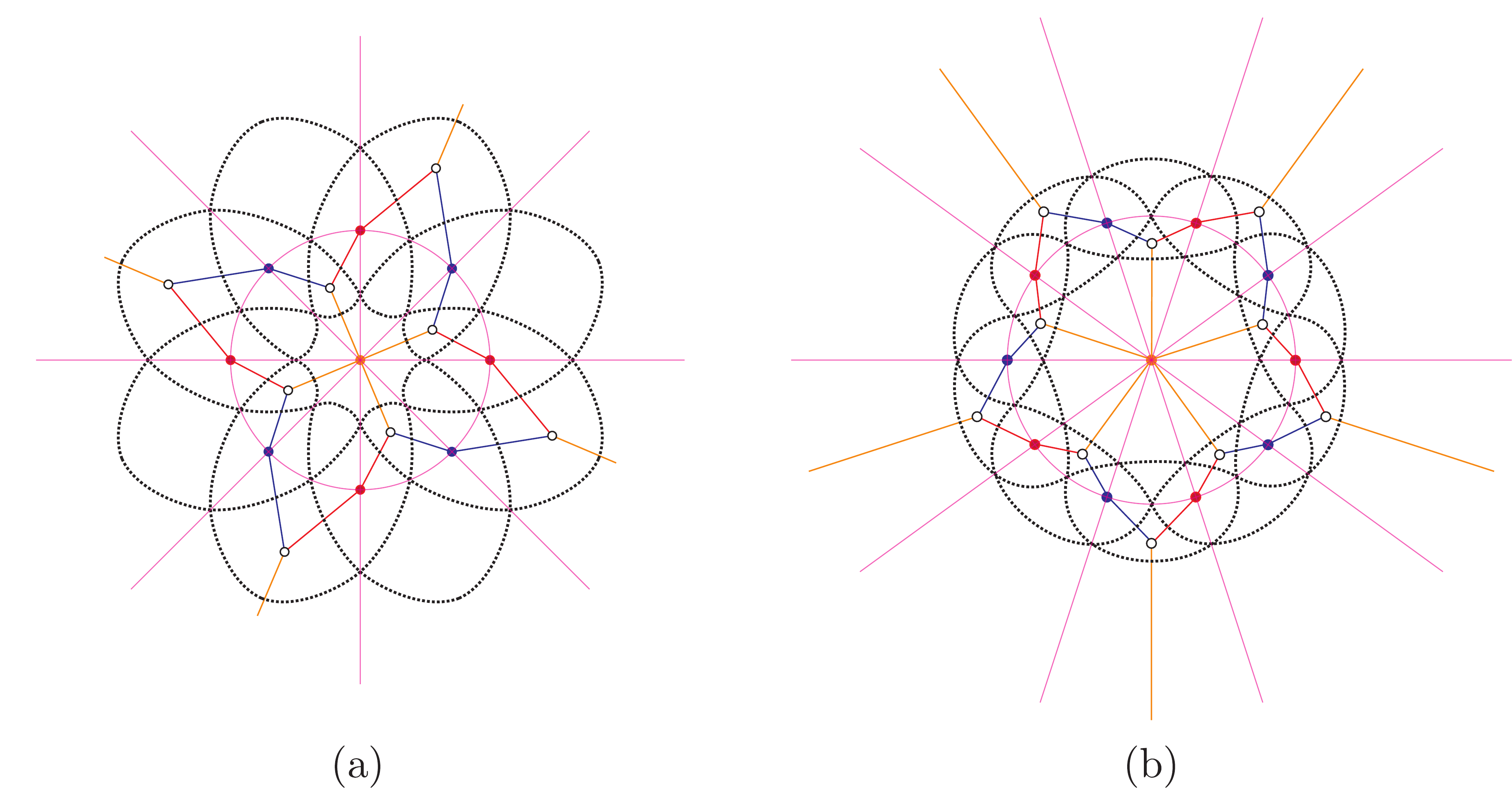}
\caption{$(a,b,c)$-diagram for (a) $(2,2,4)$ and (b) $(2,2,5)$ }\label{fig:abc225}
\end{center}
\end{figure}

As usual, we have the term $-qxyz$ from the minimal $xyz$-triangle. 
Also note that  the total area of $E$ is $16r\sigma$ since $S^2 \to \mathbb{P}^1_{2,2,r}$ is $2r$-sheeted.\\

\noindent{\bf (i)}
From the picture, there are only two $2$-gons which give $x^2$-terms, $U_{x^2}$ corresponding to the word $(\gamma \alpha)^2$ and its reflection image. The total number of contribution from this pair is $1$ because of $\Z_2$-symmetry (i.e. $\eta (U_{x^2}) = 2$). One can also check that the sign factor is given by $s(U_{x^2}) =2$, and the area of $U_{x^2}$ is $q^6$. In conclusion, we get the term $q^6 x^2$. Computation for $y^2$-term is essentially the same and we get $q^6 y^2$. Observe that if a polygon contains a $x$-corner, then it should be either the basic triangle (for $xyz$) or the 2-gon $U_{x^2}$.\\

\noindent{\bf (ii)}
Observe that we have two $r$-gons corresponding to $z^r$ one of which has the word expression $(\alpha \beta)^r$ and lies at the center of the diagram (Figure \ref{fig:abc225}) containing $0$. Denote this polygon by $U_{z^r}$. Its reflection image, then, includes the $c$-vertex at infinity. With help of the area formula and the formula for $s(U_{z^r})$, we obtain the term $(-1)^r q^{3r} z^r$. Note that the number of $\alpha$-leaves for $U_{z^r}$ is $r$ but $U_{z^r}$ admits $\Z_r$-symmetry and hence, $\eta(U_{z^r})=r$.

Now, we can choose a side $e$ in $U_{z^r}$ and perform the elementary move with respect to $e$ to get $z^{r-2}$ term as Figure \ref{fig:move225} illustrates. Namely, if we try to extend the polygon through the side $e$, the only possible move adds
a bigon to the image, removing two corners of the polygon. 
 Let us call this process {\em resolving $e$}. After this process, one can resolve another side $e'$ if it is not adjacent to $e$.
Repeating this process, we get
$$z^r, z^{r-2}, z^{r-4}, \cdots,z^{r-2k},\cdots, z^{r-\lfloor r/2\rfloor}$$
where $\lfloor r/2\rfloor$ is the greatest integer which is not bigger than $r/2$. For example, the polygon obtained by resolving $k$-sides (with no two sides in succession) of the original $r$-gon simultaneously will give $z^{r-2k}$-term.
(See Figure \ref{fig:move225} for $z^{r-4}$-polygons obtained by elementary moves associated with two sides of the central $r$-gon.)
If a polygon $U_k$ gives the term $z^{r-2k}$ (hence, $U_0=U_{z^r}$), then the formula in Theorem \ref{thm:term} tells us that
the area of $U_k$ is $q^{3k +10}$ and
$$s(U_k) + s^{\rm op} (U_k) =  (r-k) + (k) = r.$$
For $r$ is odd, $k = \lfloor r/2\rfloor$, then $U_k$ gives $1$-gon, contributing to the linear term of $W$.
\begin{figure}[ht]
\begin{center}
\includegraphics[height=4in]{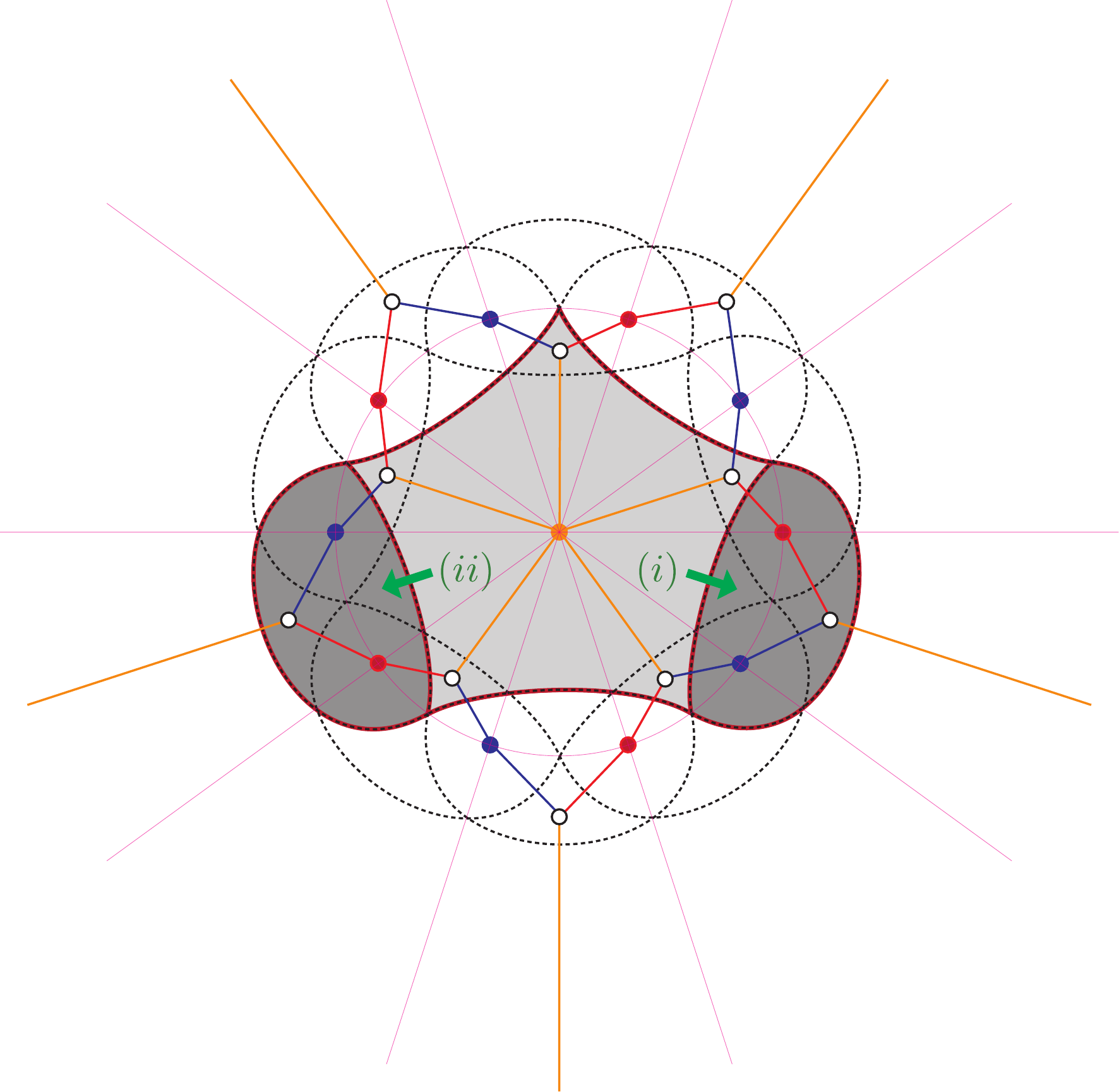}
\caption{Polygons contributing to $z^{r-2k}$}\label{fig:move225}
\end{center}
\end{figure}

To count them properly, let us discuss these polygons in more systematic way. For an (abstract) $r$-gon with sides $\{a_1, \cdots, a_r\}$ arranged in counter-clockwise order, we choose a sub-collection of $k$-sides $\{a_{i_1}, a_{i_2}, \cdots, a_{i_k}\}$ with $i_1 < i_2 < \cdots < i_k$ and $ i_{j+1} - i_{j} >1$ for $j \neq r$ and $(i_1+r) - i_{k} >0$. (This sub-collection represents the set of sides which are to be resolved.) The last condition simply means that the sub-collection does not contain a pair of consecutive sides. 
\begin{figure}[ht]
\begin{center}
\includegraphics[height=2in]{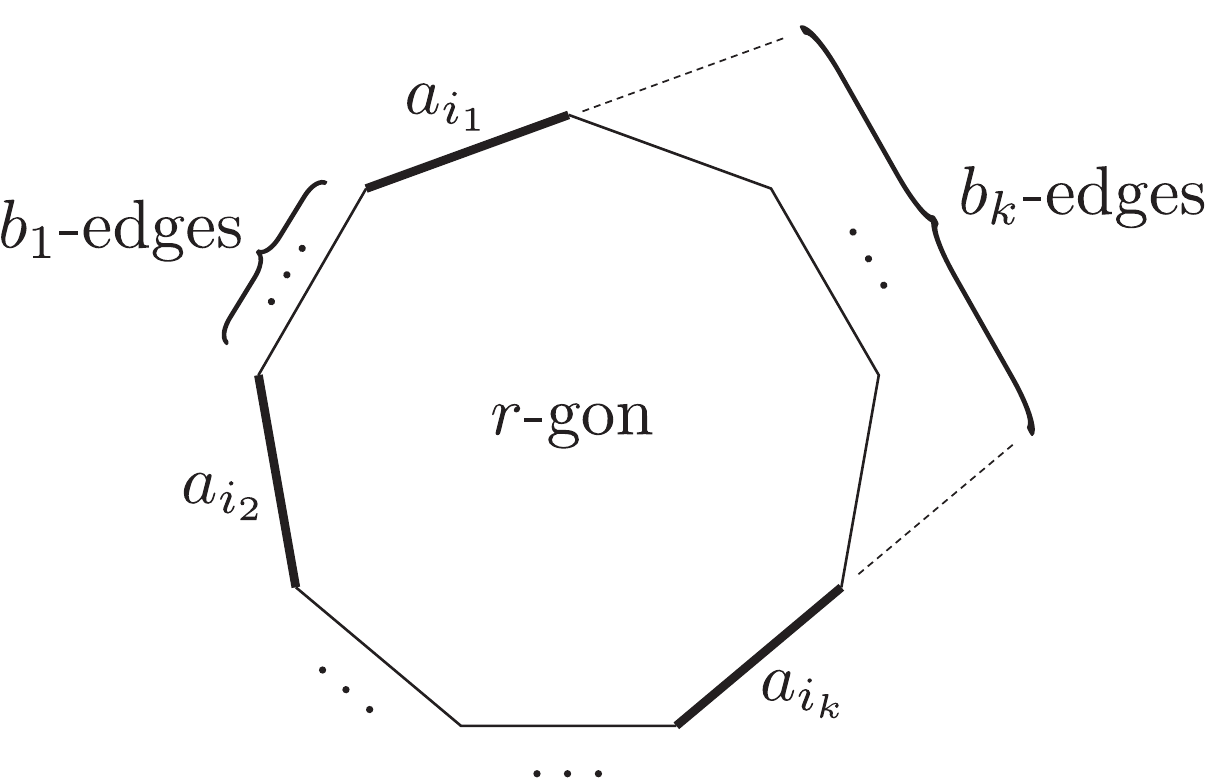}
\caption{Sub-collection of $k$-sides in the $r$-gon}\label{fig:rgon}
\end{center}
\end{figure}
Define $b_j$ by
\begin{equation*}
\begin{array}{c}
b_1 := {i_2} - {i_1} -1, \quad b_2 := {i_3} - {i_2} -1 ,\quad \cdots ,\quad b_{k-2} = i_{k-1} - i_{k-2} -1\\
b_{k-1} := {i_k} - {i_{k-1}} -1 , \quad b_k := (i_1 +r) - i_k -1.
\end{array}
\end{equation*}
$b_j$ are the number of sides between two sides $a_{i_{j+1}}$ and $a_{i_j}$, when traveling around the boundary of the $r$-gon in the counter-clockwise direction. (See Figure \ref{fig:rgon}.)

We see that the number of sub-collections satisfying the above conditions is equivalent to the number of solutions $(b_1, \cdots, b_k)$ satisfying
$$b_1 + \cdots b_k = r -k$$
with $b_j \geq 1$. Letting $b_j' = b_j -1$, we find that this number equals the number of solutions $(b_1', \cdots, b_k')$ satisfying
\begin{equation}\label{eq:H}
b_1' + \cdots b_k' = r -2k
\end{equation}
with $b_j' \geq 0$, and hence, is precisely given by
\begin{equation}\label{eq:HH}
_{k} H_{r-2k} =  \left(\begin{array}{c}r-k-1 \\ r-2k \end{array}\right) = \frac{(r-k-1)!}{(k-1)! (r-2k)!},
\end{equation}
where $H$ represents the {\em combination with repetition}.
That is, there are $ \frac{(r-k-1)!}{(k-1)! (r-2k)!}$-ways of choosing $k$-sides of the $r$-gon which can be resolved simultaneously to give a polygon corresponding to $z^{r-2k}$.

These polygons are lifts of those in $\mathbb{P}^1_{2,2,r}$, and different polygons in $E$ may be identified in the quotient.
In these cases, such identification is given by the $\Z_k$-action. Namely, rotating $(b_1', \cdots, b_k')$ to 
$(b_2', \cdots, b_k', b_1')$. Hence we may divide the count in \eqref{eq:HH} by  $k$.
(If $b_1', \cdots, b_k'$ is symmetric, then $\Z_k$-action on this set may be trivial, but the symmetry factor $\eta$ of these
polygons has to be considered also. The overall contribution to $W$ is given as the division by $k$.
We also need to multiply by $r$ which is the number of  $e$ in the boundary of the polygon.

As a result, the part of potential obtained by these $r-2k$ gons are exactly
$$(-1)^r q^{3r} z^r + r\sum_{k=1}^{\lfloor \frac{r}{2} \rfloor} (-1)^{r+k} q^{3r + 10k}  \frac{(r-k-1)!}{k! (r-2k)!} z^{r-2k}$$
where the first $r$ in the second summand comes from the number of generic points on the boundary of polygons.\\

\begin{theorem}\label{thm:W22r}
The Lagrangian Floer potential of the Seidel Lagrangian in $\mathbb{P}^1_{2,2,r}$ is given as
\begin{equation}\label{eq:pot22r}
-qxyz + q^6 x^2 + q^6 y^2 + (-1)^r q^{3r} z^r + r\sum_{k=1}^{\lfloor \frac{r}{2} \rfloor} (-1)^{r+k} q^{3r + 10k}  \frac{(r-k-1)!}{k! (r-2k)!} z^{r-2k}.
\end{equation}
\end{theorem}

\begin{remark}
\eqref{eq:pot22r} coincides with the mirror potential of $\mathbb{P}^1_{2,2,r}$ given in \cite[p. 283]{R} (but the powers of  $q$ are somewhat different). This could be because that we have chosen the Seidel Lagrangian to divide $\mathbb{P}^1_{a,b,c}$
into 8 pieces of equal area, and another configuration (which is still invariant under reflection along the equator)  might be related to that of \cite{R}.
\end{remark}

\subsection{$(2,3,3)$-case ($E_6$)}
We next deal with three exceptional cases making use of the explicit $(a,b,c)$-diagram for each case. 

We begin with the case of $(2,3,3)$. The universal cover $S^2 \to \mathbb{P}^1_{a,b,c}$ is 12-fold, whose area gives $q^{96}$. The potential contains $-q xyz$ due to this minimal triangle.
We also have terms $q^6 x^2$, $-q^9 y^3$ and $-q^9 z^3$ from polygons corresponding to words $(\beta \gamma)^2$, $(\gamma \alpha)^3$ and $(\alpha \beta)^3$, respectively.

In addition to these, one can apply the elementary move to any side of $y^3$-polygon ($=(\gamma \alpha)^3$)
to get a $2$-gon $U_{yz}$ with vertices $y$ and $z$ (or one can start with  $z^3$-polygon for the same result due to the symmetry between $y$ and $z$); see Figure \ref{fig:e6yz}.  By applying our formula (Theorem \ref{thm:term}) to $U_{yz}$, we see that $U_{yz}$ together with its reflection image contributes to the potential by $-4 q^{22} yz$.

\begin{figure}[ht]
\begin{center}
\includegraphics[height=3in]{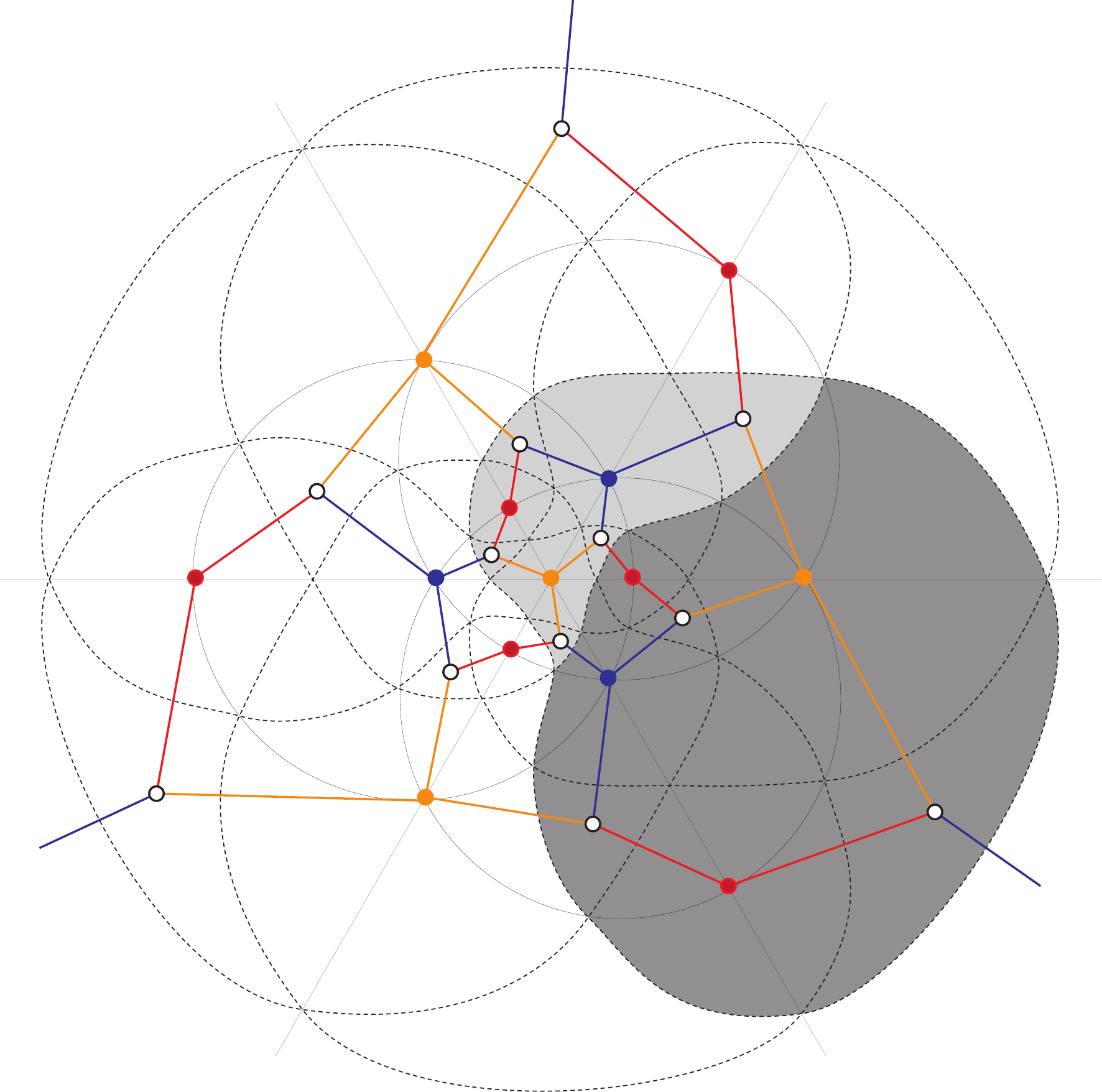}
\caption{$U_{yz}$ and the $0$-gon}\label{fig:e6yz}
\end{center}
\end{figure}

We perform elementary move once again with respect to any side of $U_{yz}$ to get a constant term. Indeed, the outcome is a smooth disk without any turns whose boundary is a single branch of the preimage of the Seidel Lagrangian. (See the darker region of Figure \ref{fig:e6yz} which is produced after the second elementary move.) Recall from Lemma \ref{lem:sphdiss} that the boundary circle divides $S^2$ into two regions with the same area. Hence, it gives the constant term $2 q^{48}$ where the number of generic points on its boundary and its reflection image is $4$, but there is $\Z_2$-symmetry for this disc($=(\gamma \beta \alpha)^2$).

\begin{theorem}\label{thm:W233}
The Lagrangian Floer potential of the Seidel Lagrangian in $\mathbb{P}^1_{2,3,3}$ is given as
\begin{equation}\label{eq:pot233}
-qxyz + q^6 x^2 - q^9 y^3 - q^9 z^3 -4 q^{22} yz +2 q^{48}.
\end{equation}
\end{theorem}

\subsection{$(2,3,4)$-case ($E_7$)}
In this case, $S^2 \to \mathbb{P}^1_{2,3,4}$ is a 24-fold cover and hence, the area of the whole sphere is given by $q^{192}$. As before, the minimal triangle and the 0-th generation polygons from words $(\beta \gamma)^2$, $(\gamma \alpha)^3$ and $(\alpha \beta)^4$ give rise to the following four terms:
$$ -qxyz + q^6 x^2 - q^9 y^3 + q^{12} z^4.$$
Denote the polygons for $y^3$ and $z^4$ by $U_{y^3}$ and $U_{z^4}$. We apply elementary move to $U_{y^3}$ (or $U_{z^4}$) to obtain  a triangle $U_{yz^2}$ corresponding to the term, $5 q^{25} yz^2$,
which is the first generation polygon.

We can further apply elementary moves to three sides of $U_{yz^2}$ but two of them coincides. So, we have essentially two second generation polygons which give $y^2$ and $z^2$ terms, respectively. It is easy to see that both of $2$-gons admit $\Z_2$-symmetry. In fact, these two polygons have word expressions $(\gamma \beta \alpha (\gamma \alpha))^2$ and $(\alpha \gamma \beta (\alpha \beta))^2$; see Figure \ref{fig:z2e7} for the $z^2$-polygon.  After counting the number of generic points on their boundary, we see that their contributions are
$$3 q^{38} y^2 + 3 q^{54} z^2. $$

\begin{figure}[ht]
\begin{center}
\includegraphics[height=3in]{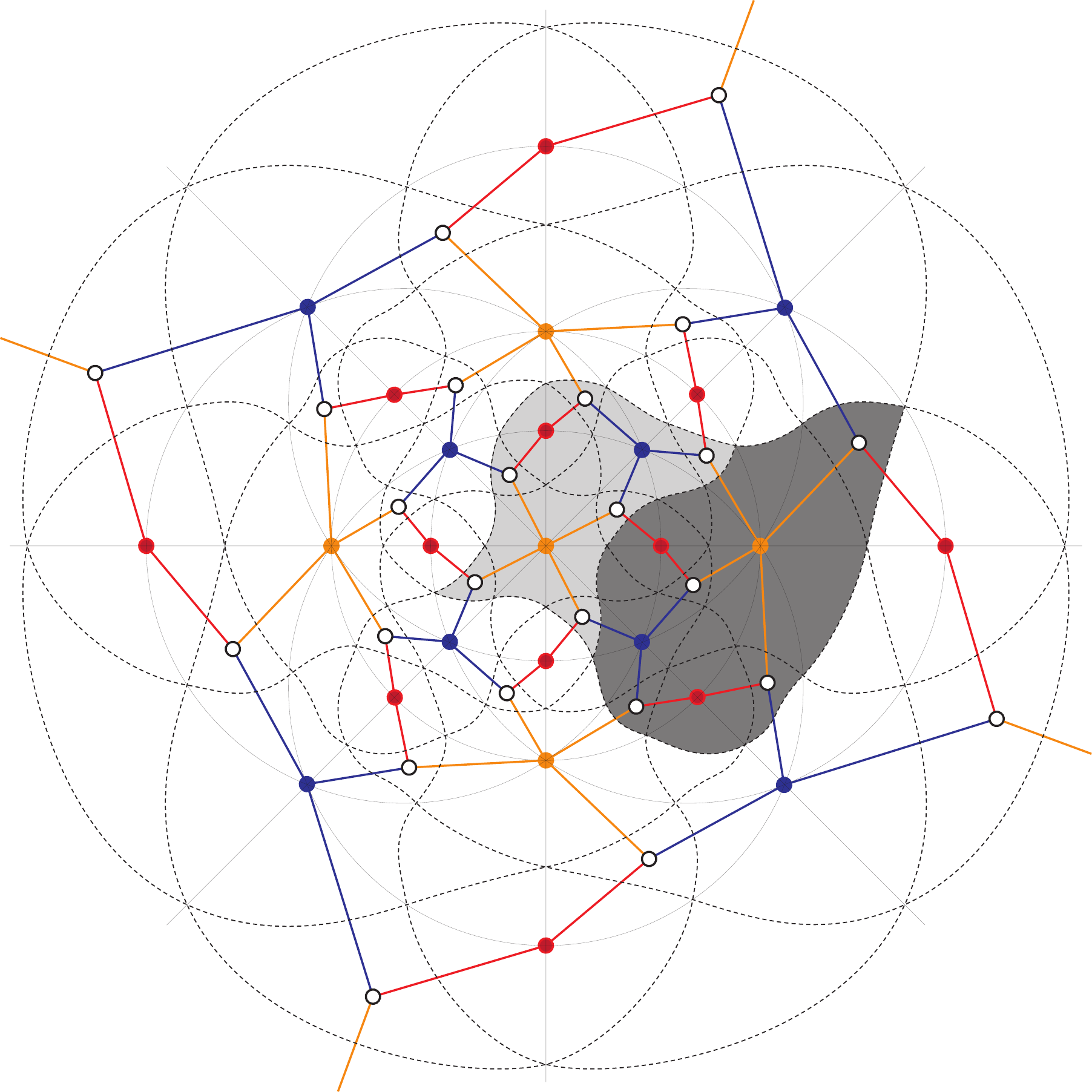}
\caption{$yz^2$- and $z^2$- polygons}\label{fig:z2e7}
\end{center}
\end{figure}

Finally, a third generation polygon become a disc (0-gons) whose corresponding word is $(\gamma \beta \alpha)^3$.
It has $\Z_3$-symmetry, and the corresponding constant  term of $W$ is $-2 q^{96}$.
\begin{theorem}\label{thm:W234}
The Lagrangian Floer potential of the Seidel Lagrangian in $\mathbb{P}^1_{2,3,4}$ is given as
\begin{equation}\label{eq:pot234}
-qxyz + q^6 x^2 - q^9 y^3 + q^{12} z^4  +5 q^{25} yz^2+ 3 q^{54} z^2 +3 q^{38} y^2 - 2 q^{96}.
\end{equation}
\end{theorem}

\subsection{$(2,3,5)$-case ($E_8$)}
Finally, we find the potential for $(2,3,5)$-case, which has the most complicated diagram among spherical $(a,b,c)$-diagrams. In this case $S^2 \to \mathbb{P}^1_{2,3,5}$ is a $60$-sheeted cover and the total area of $S^2$ is $q^{480}$.

 Computations for the first few terms are the same as in previous cases. We have
$$ -qxyz + q^6 x^2 - q^9 y^3 - q^{15} z^4$$
from the minimal triangle, $(\beta \gamma)^2$, $(\gamma \alpha)^3$ and $(\alpha \beta)^5$ with considering their symmetries carefully.
We  run the elementary-move from the $y^3$-triangle as in  $(2,3,4)$-case. The first move produces a $4$-gon corresponding to  the term $-6 q^{28} yz^3 $ in the potential.

The second elementary moves with respect to sides of this $4$-gon essentially yield two non-isomorphic polygons which give rise to $y^2 z$- and $z^4$-terms;
$$-7 q^{41} y^2 z + 4 q^{60} z^4. $$
Note that the $z^4$-polygon has the word expression $(\alpha \gamma \beta (\alpha \beta)^2)^2$ and hence, admits $\Z_2$-symmetry; see Figure \ref{fig:yz2e8} below.

\begin{figure}[ht]
\begin{center}
\includegraphics[height=3in]{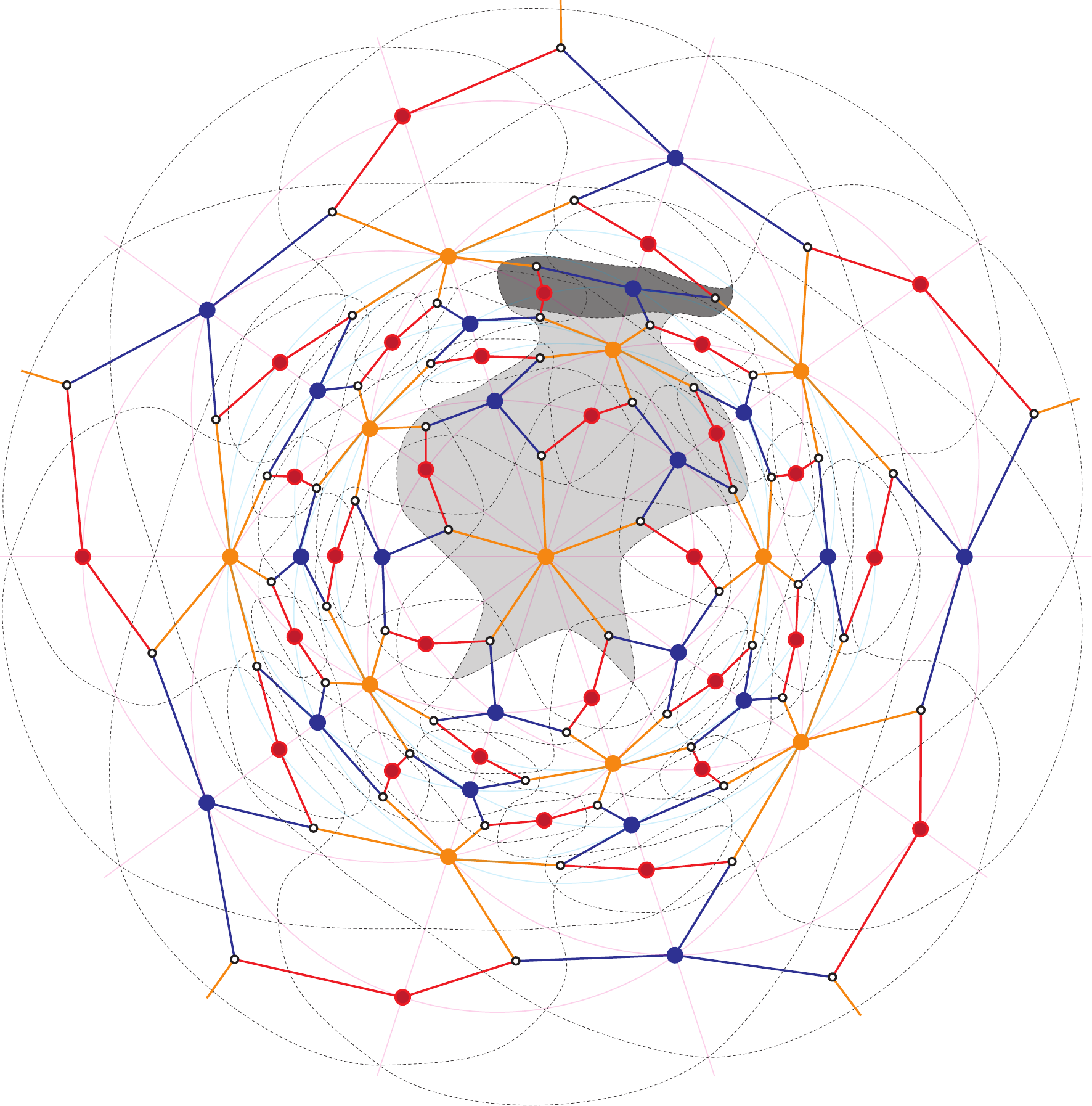}
\caption{The $z^4$-polygon (lighter one) and the $yz^2$-polygons (darker one)}\label{fig:yz2e8}
\end{center}
\end{figure}

The third elementary moves from the second generation polygons produce two polygons corresponding to $y z^2$- and $z^3$-terms, which induces
$$ -9 q^{73} yz^2 -3q^{105} z^3$$
in the potential; see Figure \ref{fig:yz2e8} for the $yz^2$-polygon and \ref{fig:z3e8} for the $z^3$-polygon. Here, the word for the $z^3$-polygon is given by $(\alpha \gamma \beta (\alpha \beta))^3$ and has $\Z_3$-symmetry.

\begin{figure}[ht]
\begin{center}
\includegraphics[height=3in]{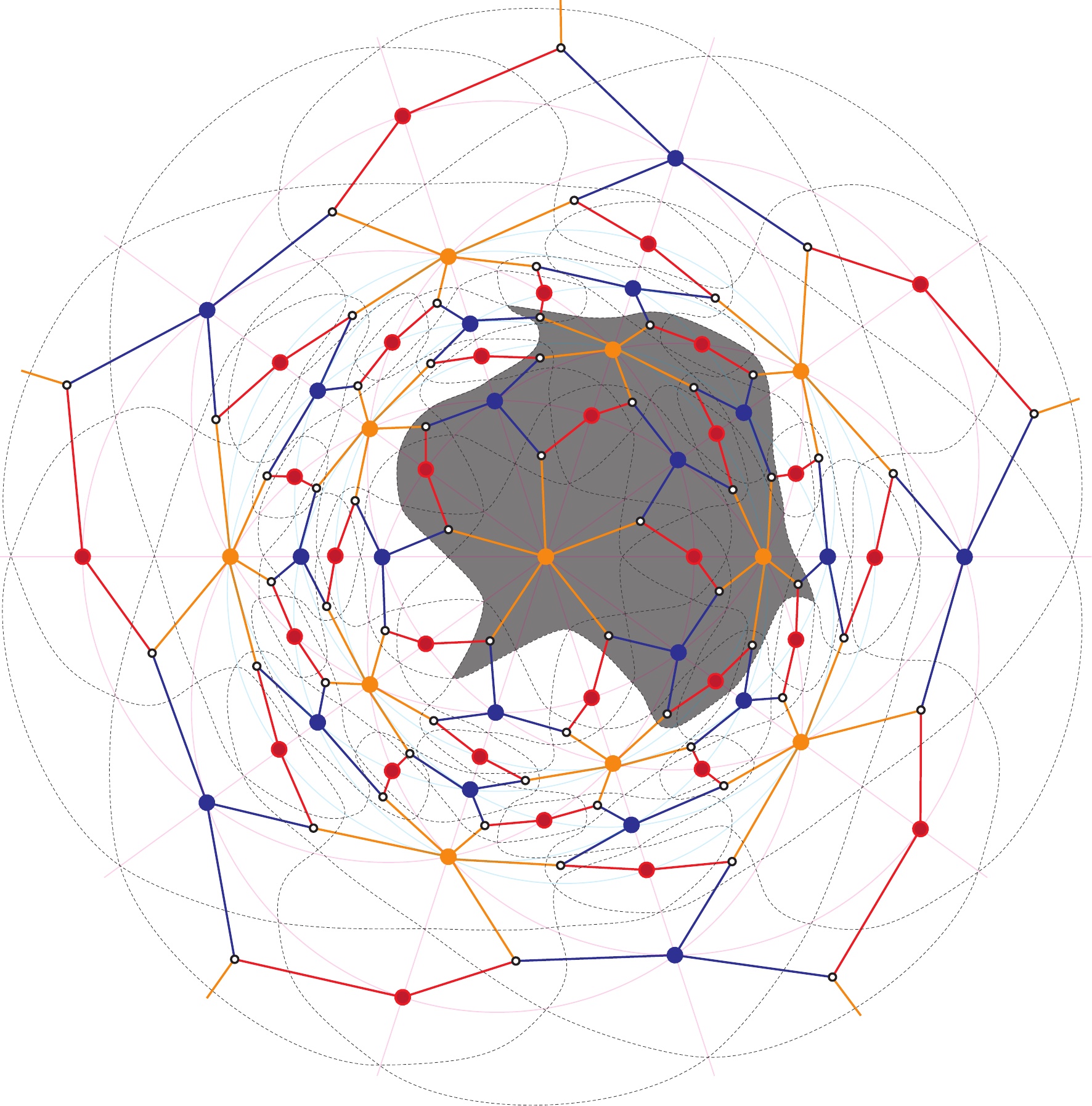}
\caption{The $z^3$-polygon}\label{fig:z3e8}
\end{center}
\end{figure}

There are two kinds of the fourth generation polygons, $y^2$- and $z^2$-polygons both of which have $\Z_2$-symmetries. When expressed as words, they are $(\alpha \gamma \beta (\alpha \beta)^2)^2$ and $(\gamma \beta \alpha \gamma \beta \alpha (\gamma \alpha))^2$, respectively. Thus, they induces the following part of the potential:
$$ 5 q^{86} y^2 + 5 q^{150} z^2.$$

Finally, the last elementary move yields a disc (0-gon) with  area $q^{240}$. This disc corresponds to the word $(\gamma \beta \alpha)^5$, and hence, admits $\Z_5$-symmetry. So, it contributes to the potential by $-2 q^{240}$.

In summary, we have the following.
\begin{theorem}\label{thm:W235}
The Lagrangian Floer potential of the Seidel Lagrangian in $\mathbb{P}^1_{2,3,5}$ is given as
\begin{equation}\label{eq:pot235}
\begin{array}{c}
-qxyz + q^6 x^2 - q^9 y^3 - q^{15} z^5  + 4 q^{60} z^4  -3q^{105} z^3 + 5 q^{150} z^2 \\
 -6 q^{28} yz^3 -9 q^{73} yz^2  -7 q^{41} y^2 z+ 5 q^{86} y^2  -2 q^{240}.
\end{array}
\end{equation}
\end{theorem}


\section{Applications}
We let $W$ to be the potential for $\PO$, where $\frac{1}{a} + \frac{1}{b} + \frac{1}{c} \leq1 $.
The main goal of this section is to prove the following theorem.

\begin{theorem}\label{thm:mainapply}
\begin{enumerate}
\item If $\frac{1}{a} + \frac{1}{b} + \frac{1}{c} =1$, then $W$ has a unique singular point at the origin.
\item If $\frac{1}{a} + \frac{1}{b} + \frac{1}{c} <1$, then the hypersurface $W^{-1} (0)$ has an isolated singularity at the origin.
\end{enumerate}
\end{theorem}
The proof in the elliptic cases are rather simple using the normal form of the potential.
But the proof in the hyperbolic cases are rather complicated since $W$ is an infinite series and we do not have an explicit form of other potential.
Let us first discuss the elliptic cases.
\begin{prop}
If $\frac{1}{a} + \frac{1}{b} + \frac{1}{c} =1$,  then $W$ has a unique singular point at the origin.
\end{prop}
\begin{proof}
Let us first consider the $(3,3,3)$-case. The potential after coordinate change is given by
 $W = x^3+y^3 +z^3 - \sigma(q) xyz$ where $\sigma(q)$ is the (inverse) mirror map in \cite{CHL}.
 The critical point equations $\partial_xW = \partial_yW=\partial_zW=0$ give
$$3x^2-\sigma xy = 3y^2 -\sigma xz = 3z^2 -\sigma xy =0.$$
Hence we get $3^3 x^2y^2z^2 = \sigma^3 x^2y^2z^2$. These equations are satisfied if
$x=y=z=0$ or $\sigma =3$. But $\sigma(q) \neq 3$, and hence we have shown that the only singularity of $W$ is the origin.

The proof for $(2,4,4)$ and $(2,3,6)$ are similar using the normal forms \eqref{eq:244co}, \eqref{eq:236co}.
In the $(2,4,4)$ case, we have $W = x^2+y^4 +z^4 + \sigma(q) y^2z^2$, and the critical point equations
give $$2x = 4y^3 +2yz^2\sigma = 4z^3 + 2zy^2\sigma =0.$$ 
Hence $x=0$ and we also have $4^2y^3z^3 = 4y^3z^3\sigma^2$. Since $\sigma \neq 4$, we have $y=0, z=0$, which proves the claim.

In the $(2,3,6)$ case, we have $W = x^2 + y^3 +z^6 + \sigma(q) yz^4$, and the critical point equations give
$$2x= 3y^2 + \sigma z^4 = 6z^5 + 4\sigma yz^3=0.$$
Hence, $x=0$ and we also have $3y^2 \cdot (6z^5)^2 =  (-\sigma z^4)(4 \sigma yz^3)^2$.
Since $\sigma^3 \neq - 3 \cdot 6^2 / 4^2$, we have $y=z=0$, which proves the claim.
\end{proof}

Let us consider the hyperbolic cases $\frac{1}{a} + \frac{1}{b} + \frac{1}{c} < 1$ for the rest of the paper.

We will divide the proof of Theorem \ref{thm:mainapply} into three cases (Proposition \ref{prop:hyperisol3}, \ref{prop:hyperisolc6}, \ref{prop:hyperisolab4}). In the first case when $a,b,c \geq 3$, the potential has a nice property (Lemma \ref{lem:Wxylyzm}), which makes the derivatives of $W$ comparably simple. 

The Lemma \ref{lem:Wxylyzm} may not hold  when one of $a,b,c$ is $2$. For these cases, we will alternatively use the Weierstrass preparation theorem to simplify the derivatives of $W$. In order to use the Weierstrass theorem, we first prove that $W$ is convergent over complex numbers (where we consider $q$ as a fixed complex number).

\subsection{Convergence of open Gromov-Witten potential}
Using the area formula (together with the set of formulas given in Lemma \ref{lem:relations}), we prove that the potential $W$ converges for small $x$, $y$, $z$ and $q$ when $\frac{1}{a} + \frac{1}{b} + \frac{1}{c} <1$. We will use this fact later to show that the hypersurface $W^{-1} (0)$ has an isolated singularity at zero for $\mathbb{P}^1_{2,3,c}$ with $c>6$ (see Subsection \ref{subsec:isolc6}).

\begin{theorem} \label{thm:conv}
There exists a positive number $\epsilon$ such that $W(x,y,z,q)$ is absolutely convergent for $|x|, |y|, |z| < \epsilon$ for  small enough $q$.  Hence the potential $W$ is an analytic function on a neighborhood of $(0,0,0)$ for small enough $q$.
\end{theorem}

\begin{proof}
Recall from \eqref{eq:areapqr1} that the area of the holomorphic polygon $U$ with $P$ $x$-corners, $Q$ $y$-corners and $R$ $z$-corners is given (in terms of the area of $\sigma$) by
$${\rm Area} (U)= 3 (P+Q+R) + 8 \frac{P/a + Q/b + R/c -1}{1-1/a - 1/b - 1/c}$$
%
Therefore, we see that $k_1 S <{\rm Area} (U) < k_2 S$ for some $k_1, k_2 >0$, for $S=P+Q+R$ large enough.

Also, the coefficient of the monomial corresponding to this polygon is 
$$(2p + S) / \eta < 2p +S$$
up to sign where $\eta$ is the number of symmetries that $U$ admits. We now estimate this number purely in terms of $S$. By $(3)$ of Lemma \ref{lem:relations}, we have the relation $f_5 = 3p + S$ where $f_5$ is the number of $5$-gons in the $(a,b,c)$-diagrams for $U$. Therefore, $2p + S < f_5$. Since there are just finitely many different types of $5$-gons, we can take the minimal area $l$ of $5$-gons so that $l f_5 < {\rm Area} (U) < k_2 S$. Hence $f_5 < k_3 S$ for some $k_3 >0$. Combining these two, we obtain the following inequality: $ 2p + S < k_3 S$. In particular, $p < A_0 S$ for some large integer $A_0$.

We next estimate the number of polygons which have the same $S$. Let us denote this number by $N_S$.
We first fix a generation $p$, and count the maximal number of $p$-generation polygons with the same $S=P+Q+R$. To get such polygons, we should plug in $S$-corners between two consecutive letters in $(\gamma \beta \alpha)^p$ (Lemma \ref{lem:standard}). (We remark that not all of these choices give genuine holomorphic polygons.) There are at most $_{3p} H_S$-ways of doing this where $_S H_{3p}$ means the combination with repetitions. Therefore,
\begin{eqnarray*}
N_S &\leq& \sum_{p=1}^{A_0 S} \left. \right. _{3p} \! H_S 
\leq \sum_{p=1}^{AS} \left. \right._{p} \! H_S  \qquad (\mbox{for} \,\, A > 3A_0) \\
&=& \sum_{p=1}^{AS} \left. \right. _{S+1} \! H_{p-1} 
\leq \sum_{p=0}^{AS} \left. \right. _{S+1} \! H_p  = _{S+2} \! H_{AS} 
\end{eqnarray*}
where we used identities $_n  H_r = _{n+r-1} \! C_r = _{n+r-1} \! C_{n-1} = _{r+1} \! H_{n-1} $ and $_n H_0 + _n\! H_1 +\cdots + _n \! H_r = _{n+1} \! H_r$. Thus,
$$N_S \leq \left. \right._{S+2} \! H_{AS} = \left(\begin{array}{c}(A+1) S +1 \\AS\end{array}\right) = \frac{((A+1)S +1)!}{(AS)! (S+1)!}.$$
We define this upper bound by $b_S:=\left. \right._{S+2} \! H_{AS}$. Observe that
$$ b_{S+1} / b_S = \frac{ ((A+1)S +A+2) \cdots ((A+1)S + 2 ) }{(AS+A) \cdots (AS +1) \cdot (S+1) }.$$
Since the degrees of both denominator and numerator coincide, the limit $\lim_{S \to \infty} b_{S+1} / b_S$ is given by the ratio of the leading coefficients, that is, 
\begin{equation}\label{eq:ratiobs}
\lim_{S \to \infty} b_{S+1} / b_S = \frac{(A+1)^{A+1}}{A^A}.
\end{equation}
%

Assume now that $|x|, |y|, |z|, |q| < \epsilon <1$. Then, $e^{-{\rm Area}(U)} < e^{-k_1 S}$. Therefore if we arrange monomials of $W$ in $q$-th power, and apply the above estimates, we obtain
\begin{eqnarray*}
|W| &<&  \sum_{S \geq 1} N_S \, (k_3 S) \, \epsilon^S (q^{k_1})^S \\
  &<& k_3 \sum_{S \geq 1} b_S  S  \, (\epsilon^{k_1+1} )^S
\end{eqnarray*}
It is easy to see that the last sum converges for small $\epsilon$ by ratio test \eqref{eq:ratiobs}.
\end{proof}

\subsection{Preparations for isolatedness}
In order to prove Theorem \ref{thm:mainapply},
we need the following preliminary lemmas which show special properties of the hyperbolic potentials. The first lemma concerns the hyperbolic potentials for $a,b,c \geq 3$.

%
%

\begin{lemma}\label{lem:Wxylyzm}
If $a,b,c \geq 3$, then $W$ does not contain monomials of the forms
\begin{equation}\label{eq:linhigh}
x y^l, y z^m, z x^n.
\end{equation}
for $l,m,n \geq 0$.
\end{lemma}

\begin{proof} 
The proof will be based on the word expression of the cut-glue operation and the presentation of $\pi_1^{orb} (\mathbb{P}^1_{a,b,c}) / [\pi_1^{orb} (\mathbb{P}^1_{a,b,c}),\pi_1^{orb} (\mathbb{P}^1_{a,b,c})]$.
We have to show that there does not exist a  polygon for the potential whose corners are all the same except one corner. Suppose to the contrary that there exists such a polygon $U$. Without loss of generality, we may assume that $U$ consists of a single $x$-corner and $l$ number of $y$-corners. 
Note that the generation $p$ of  $U$ should be bigger than $1$.

Let $\mathcal{S}$ be the set of polygons $U'$ for the potential such that $U' \varsubsetneq U$, 
and that $U'$ contains the unique $x$-corner of $U$. $\mathcal{S}$ is a nonempty set because there is a minimal triangle incident to this $x$-corner. Therefore, $\mathcal{S}$ is a finite partially ordered set (with the ordering given by ``$\subset$"), and we can take a maximal polygon $V$ for the potential in $\mathcal{S}$. By the definition of the elementary move (Definition \ref{defn:move}), we see that $U$ is obtained from $V$ by a single step elementary move. 

There are exactly two corners $C_1$ and $C_2$ of $V$ which are not corners of $U$.
If $C_1$ or $C_2$ is a $y$-corner, then from the explicit formula in Section~\ref{ss:cutglue} we see that $a=2$ or $c=2$; this is a contradiction.
It is also easy to see that the corners should be $z$ and $x$-corners, in the counter-clockwise.
Since the glue word corresponding to $U\setminus V$ has  no $xy$ or $yz$ terms, 
  the glue word formula in Section~\ref{ss:cutglue} implies that $a=c=3$.
\begin{figure}[htb]
\includegraphics[height=.18\textheight]{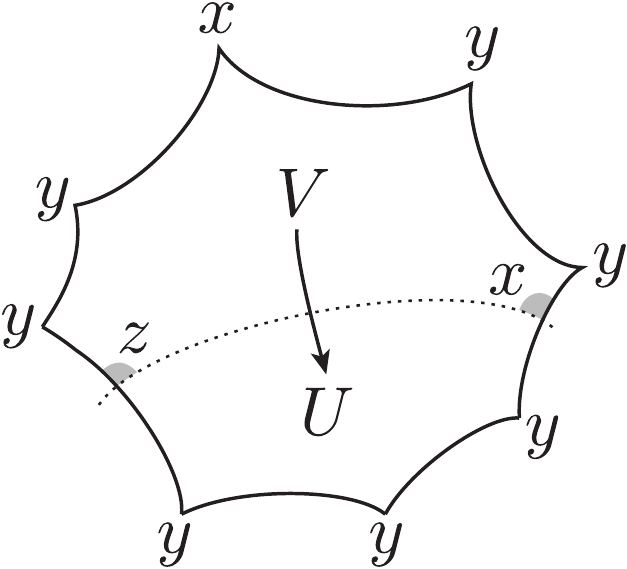}
\caption{The elementary move from $V$ to $U$}\label{fig:Vac3}
\end{figure}

For this remaining case, we argue as follows.
Recall that the fundamental group of $\mathbb{P}^1_{3,b,3}$ has the following presentation
$$\langle \alpha, \beta, \gamma : \alpha^3 = \gamma^3 = \beta^{b} = \alpha \beta \gamma =1 \rangle.$$
Therefore, its abelianization is given by
\begin{eqnarray*}
H_1^{orb} (\mathbb{P}^1_{3,b,3} ) &=& \langle [\alpha], [\beta], [\gamma] : 3[\alpha] = 3 [\gamma] = b \cdot [\beta]  = [\alpha] + [\beta] + [\gamma] =0 \rangle \\
 &=& \left\{
 \begin{array}{lc}
 \Z_3 \langle [\alpha] \rangle \oplus \Z_3 \langle [\beta] \rangle \oplus \Z_3 \langle [\gamma] \rangle / \langle ([\alpha], [\beta], [\gamma]) \rangle & 3\,|\,b   \\
 \Z_3 \langle [\alpha] \rangle \quad ( \it{with}\,\,\, [\gamma] = -[\alpha] \,\,\,\it{and}\,\,\, [\beta]=0) & otherwise
 \end{array}
 \right.
\end{eqnarray*}
where $[\alpha]$, $[\beta]$ and $[\gamma]$ denote the image of $\alpha$, $\beta$ and $\gamma$ in the abelianizaiton.

Now, let $w= w_1 w_2 \cdots w_{2k}$ be the word corresponding to $U$ where $w_{even}$ represents corners of $U$ and the product of $w_{odd}$'s is a subword of $(\gamma \beta \alpha)^{\infty}$ (see Lemma \ref{lem:standard}). Since $[\gamma] + [\beta] + [\alpha] =0$ in the abelianization, it follows that
$$[w] = [w_2] [w_4] \cdots [w_{2k}] = [(\beta \gamma)(\gamma \alpha)^l]=-[\alpha] - l [\beta]$$
where the second equality comes from the fact that $U$ only consist of $a$- and $b$-corners. As $[w] =0$ and $U$ being a holomorphic polygon, we have $[w] = -[\alpha] - l [\beta] =0$. However, this relation is not  possible from the presentation of $H_1^{orb} (\mathbb{P}^1_{3,b,3})$ above. This proves the lemma.

\end{proof}

The following lemma holds for all $a,b,c \geq 2$, whose corollaries (Corollary \ref{cor:specialpropW1}, \ref{cor:specialpropW2}) will crucial to prove Theorem \ref{thm:mainapply}.
\begin{lemma}\label{lem:specialpropW0}
Let $U$ and $V$ be polygons for the potential with $V \subsetneq U$, and let $l$  be a side of $V$ such that $l \nsubseteq \partial U$. Consider the polygon $V'$ for the potential which is obtained by applying the elementary move to $V$ along the side $l$. Then, $V' \subset U$.  
\end{lemma}

\begin{proof}
\begin{figure}[htb]
\includegraphics[height=.18\textheight]{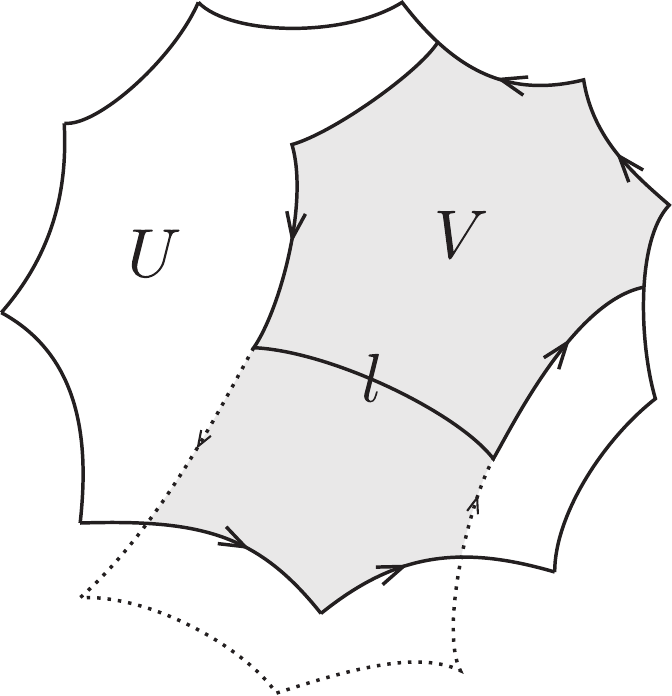}
\caption{The elementary move from $V(\subset U)$}\label{fig:elemovesubset}
\end{figure}
Suppose that $V'$ is not contained in $U$ as in Figure \ref{fig:elemovesubset}.
Considering the boundary orientation of $U$ and $V'$,  the new corners of the intersection $U \cap V'$ has odd degrees, and hence  $U \cap V'$ is also a polygon for the potential (see Figure \ref{fig:elemovesubset}). Hence $U \cap V'$ is properly contained in $V'$, which contradicts the minimality of $V'$ (see Definition \ref{defn:move}).
\end{proof}

\begin{lemma}
Consider the hyperbolic case.
For a  polygon $U$ for the potential, denote by $P$, $Q$, $R$ the numbers of $x$-, $y$- and $z$-corners of $U$, respectively. If  $U'$ is another polygon for the potential with $P' \geq P$, $Q' \geq Q$ and $R' \geq R$, then $\rm{Area} (U') \geq \rm{Area} (U)$.
\end{lemma}

\begin{proof}
It follows directly from the area formula (Theorem \ref{thm:area}) since $1- \frac{1}{a} - \frac{1}{b} - \frac{1}{c} >0$.
\end{proof}

\begin{corollary}\label{cor:specialpropW1}
If $U$ is a holomorphic $l$-gon for a hyperbolic potential all of whose corners are of the same type, say $y$, then $l \geq b$. 
\end{corollary}

\begin{proof}
Take any vertex $v$ of $U$. Clearly, the minimal triangle containing $v$ lies inside $U$. By Lemma \ref{lem:specialpropW0}, we get the $0$-th generation polygon $V$ inside $U$ which is obtained by the elementary move applying to this minimal triangle along the side opposite to $v$-corner.

Recall that the number of $b$-corners in $0$-th generation polygon is exactly $b$. Since the area of $U$ is bigger than or equal to that of $V$, the area formula (see (1) of Theorem \ref{thm:area}) for hyperbolic potentials tells us that $l \geq b$. i.e. the area is an increasing function on the number of $b$-corners.
\end{proof}

\begin{corollary}\label{cor:specialpropW2}
If one of $a$, $b$, $c$ becomes $2$, say $c=2$, then the terms of the form $xy^l, yz^m, zx^n$ in  \eqref{eq:linhigh} may exist with the following restrictions.
\begin{enumerate} 
\item
Let $(a,b,c) = (3,b,2)$ with $b>6$.
Then, there is a polygon $U_{xy^{l}}$ for the potential which consists of $l$ $y$-corners and one $x$-corner, and the minimal $l$ among all such polygons is $b-2$. Furthermore, if we have a polygon $U_{x^k y}$ for the potential whose corners are all $x$ except one $y$-corner, then the number of $x$-corners is greater than or equal to $\lfloor \frac{b-1}{2} \rfloor$. Here, $\lfloor r \rfloor$ denotes the greatest integer not bigger than $r$.
\item
Let $(a,b,c) = (a,b,2)$ with $b \geq a \geq 4$.
If there exists a polygon $U_{x^l y}$ for the potential whose corners are all $x$ except one $y$-corner, then the number of $x$-corners, $l$, satisfies $l > \frac{(a-2)(b-2)}{2}$. A similar statement holds for the number of $y$-corners, $k$, in the polygon $U_{x y^k}$ whose corners are all $x$ except one $y$-corner. 
\end{enumerate}

\end{corollary}

\begin{proof}$\left. \right.$

\noindent(1) Observe that the operation applied to a side between two $y$-corners with the minimal distance eliminates these two $y$-corners, and produces a single $x$-corner (``$\textnormal{glue}(\gamma,0)$"; See Section~\ref{s:hyp} and  Figure \ref{fig:prop2W1} (a).
\begin{figure}[htb]
\includegraphics[height=.25\textheight]{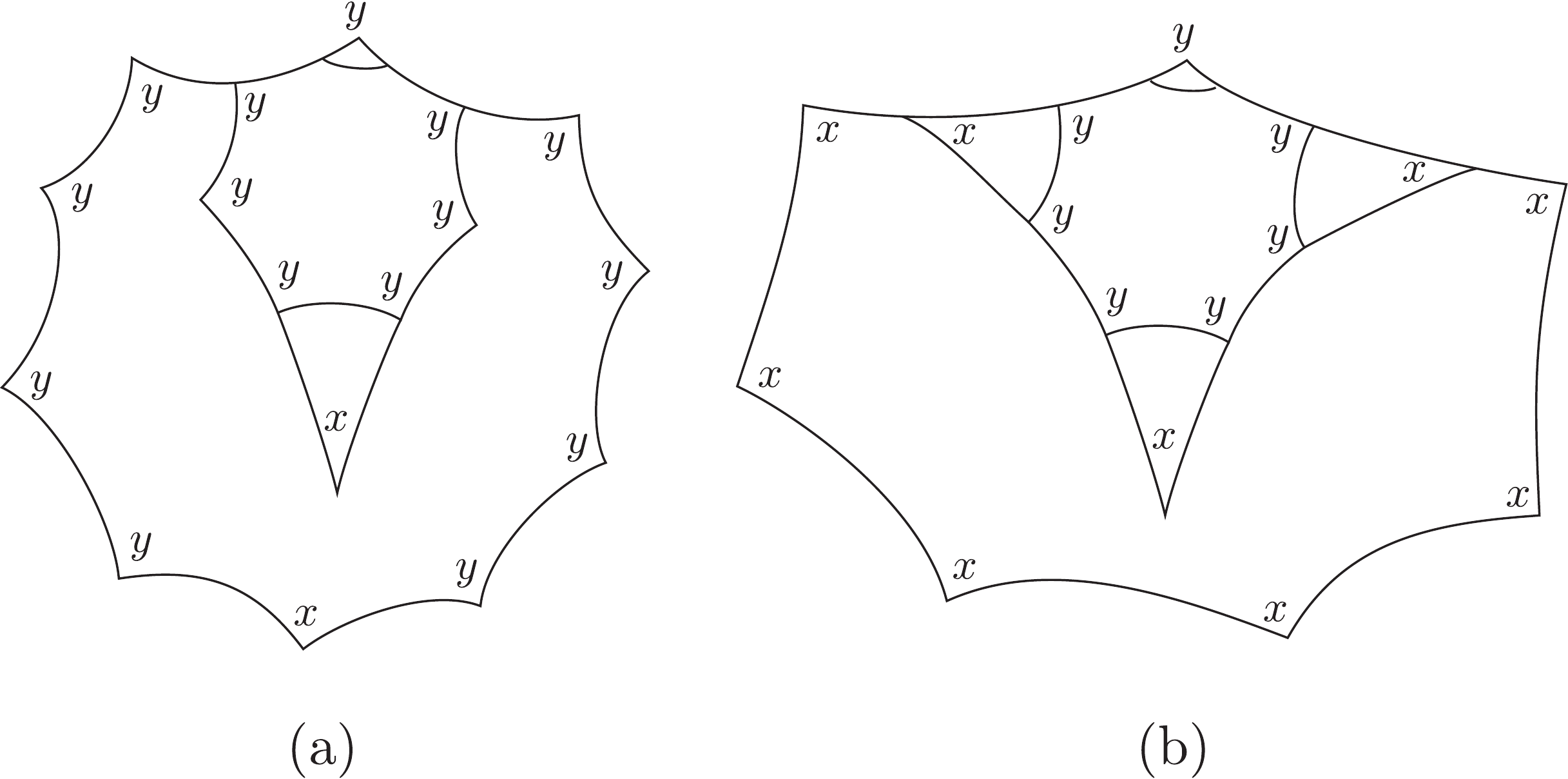}
\caption{(a) $U_{xy^l}$ and (b) $U_{x^k y}$ for $(3,b,2)$ with $b >6$}\label{fig:prop2W1}
\end{figure}

Take any $y$-corner of $U_{xy^l}$ and consider the minimal triangle containing this corner. By Lemma \ref{lem:specialpropW0}, the $0$-th generation polygon $U_{y^b}$ lies inside $U_{xy^l}$. One can apply the elementary move further to $U_{y^b}$ to get the polygon $U_{xy^{b-2}}$ which still lies inside $U_{xy^l}$ by Lemma \ref{lem:specialpropW0}. (See (a) of Figure \ref{fig:prop2W1}.) Then, by the same argument as in the proof of Corollary \ref{cor:specialpropW1}, we have $l \geq b-2$ because the area of $U_{xy^l}$ is bigger or equal to that of $U_{xy^{b-2}}$.

Next, consider the polygon $U_{x^k y}$. As before, we have the $0$-th generation polygon $U_{y^b}$ inside $U_{x^k y}$ as shown in (b) of Figure \ref{fig:prop2W1}. Observe that we can apply the elementary move to $U_{y^b}$ along at least $\lfloor\frac{b-1}{2} \rfloor$-sides of $U_{y^b}$ which are mutually disjoint. Again by comparing the area of the resulting polygon and $U_{x^k y}$,  we have
$k \geq \lfloor \frac{b-1}{2} \rfloor $. Here, $k \geq 3$ as $b \geq 7$.

\noindent(2)
By symmetry, it suffices to consider $U_{x^k y}$ only. Then, the proof of (2)  is almost the same as that of the second part of (1), except the fact that the elementary move applying to the $0$-th generation polygon $U_{y^b}$ creates $(a-2)$ $x$-corners. See Figure \ref{fig:prop2W2}. Therefore, after applying $\lfloor \frac{b-1}{2} \rfloor$ elementary moves to $U_{y_b}$ in the same manner as in (1), we have $((a-2) \lfloor \frac{b-1}{2} \rfloor)$ $x$-corners.
\begin{figure}[htb]
\includegraphics[height=.25\textheight]{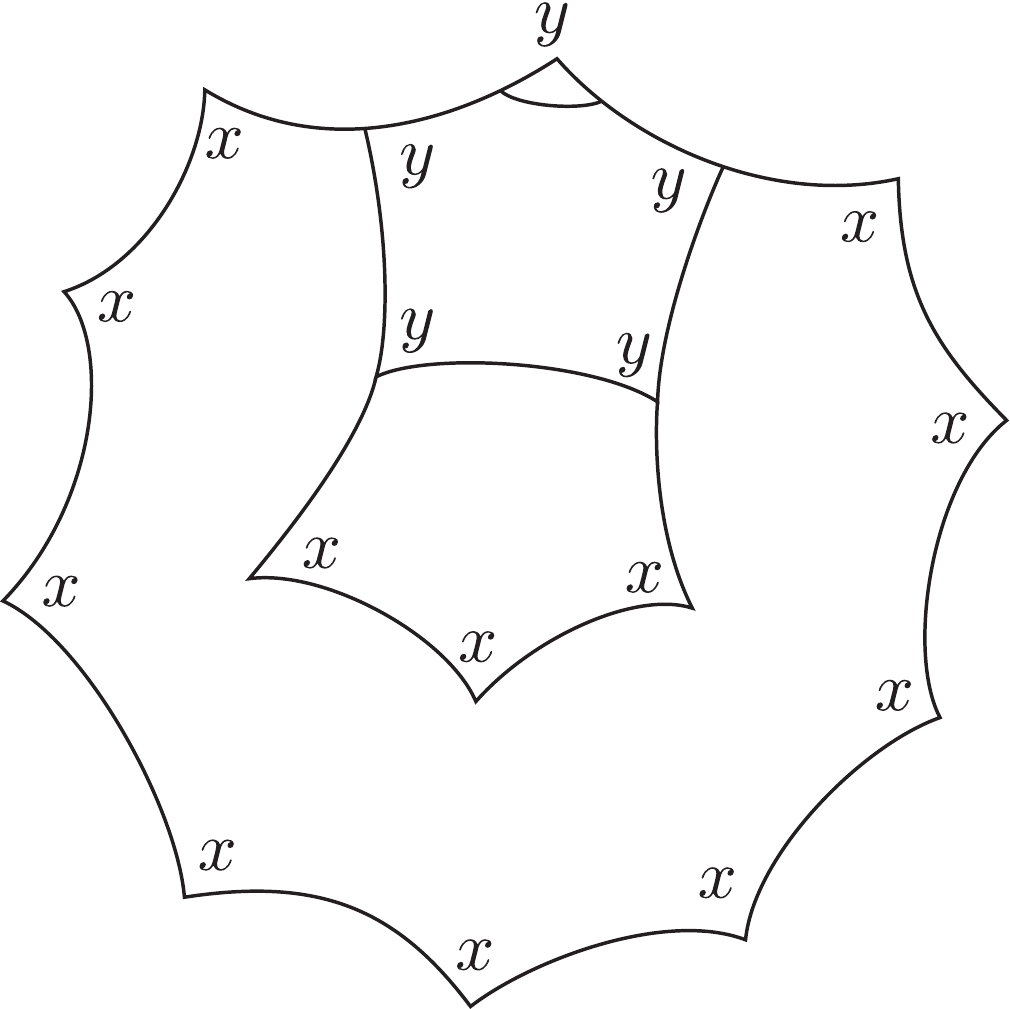}
\caption{$U_{x^k y}$ for $(a,b,2)$ with $b \geq a \geq 4$}\label{fig:prop2W2}
\end{figure}
\end{proof}

\begin{remark}\label{rem:specialpropW2}
The estimate in (2) of Corollary \ref{cor:specialpropW2} implies the following. Since we have assumed that $b \geq a$, $b$ is bigger or equal to $5$, and hence, $l \leq  2 (a-2) $. As  $a \geq 4$, we see that $l \geq a$.
%
\end{remark}

\subsection{Isolatedness for $(a,b,c)$ with $a,b,c\geq3$}
We are now ready to prove the main proposition in this section.

\begin{prop}\label{prop:hyperisol3}
For $a,b,c \geq 3$ with $\frac{1}{a} + \frac{1}{b} + \frac{1}{c} <1$, the origin is the only critical point of $W$ that lies over the critical value zero.
\end{prop}

\begin{proof}
In order to look at the critical fiber $W^{-1} (0)$, we consider the {\em Jacobian ideal ring} of the potential $W$ localized at $W=0$ which is defined by
\begin{equation}\label{eq:locJac}
{\rm Jac} (W)|_{W=0} :=\dfrac{ \Lambda [[x,y,z ]] }{(W, \partial_x W, \partial_x W, \partial_x W)}.
\end{equation}
From now on, we write $W_{x_i}$ instead of $\partial_{x_i} W$ for simplicity. We claim that $x^a$, $y^b$ and $z^c$ are contained in the denominator of \eqref{eq:locJac}. From the lemma, we can write the potential $W$ as
$$W= p_x x^a + p_y y^b+ p_z z^c + Axyz+ B x^2 y^2 + C y^2 z^2 +D z^2 x^2$$
where $p_x$, $p_y$ and $p_z$ only depend on $x$, $y$ and $z$ respectively, and $A$ is an element of $\Lambda [[x,y,z ]]$ and $B$, $C$ and $D$ are elements of $\Lambda [[x,y ]]$, $\Lambda [[y,z ]]$ and $\Lambda [[x,z ]]$, respectively. In particular, $A$ is a unit in $\Lambda[[x,y,z ]]$ because of the term $-q xyz$ in the potential, and so are $p_x$, $p_y$ and $p_z$ by Corollary \ref{cor:specialpropW1}. Then,
\begin{equation}\label{eq:Jacrel1}
\begin{array}{l}
 W_x = a p_x x^{a-1} +  (A + xA_x) yz +  (2B + xB_x) xy^2+ (2C + xC_x) z^2 x\\
 W_y = b p_y y^{b-1} +  (A + yA_y) zx +  (2B + yB_y) x^2 y+ (2D + yD_y) y z^2 \\
 W_z = c p_z z^{c-1} +  (A + zA_z) xy +  (2C + zC_z) y^2 z+ (2D + zD_z) z x^2 
\end{array}
\end{equation}
Note that $A+xA_{x}$, $A+yA_{y}$ and $A+zA_{z}$ are all invertible since they start with $-q \in \Lambda$. From the first equation (multiplied by $x$), we know that
\begin{equation}\label{eq:xa}
a p_x x^a + (A+ xA_x) xyz +  (2B + xB_x) x^2y^2+ (2C + xC_x) z^2 x^2
\end{equation}
is in $(W, W_x, W_y, W_z)$.

From the second and the third equation of \eqref{eq:Jacrel1},
\begin{equation}\label{eq:Jacrel2}
\begin{array}{l}
 zx = (A + yA_y)^{-1} y   \left( b p_y y^{b-2} +    (2B + yB_y) x^2 + (2D + yD_y)  z^2 \right) \\
  xy =   (A + zA_z)^{-1} z \left( c p_z z^{c-2} +   (2C + zC_z) y^2 + (2D + zD_z) x^2 \right)
\end{array}
\end{equation}
up to the Jacobian relations. Combining \eqref{eq:xa} and \eqref{eq:Jacrel2}, we obtain
$$ ap_x x^a + f_x xyz =0$$
in the ring \eqref{eq:locJac} where $f_x$ is given by
\begin{equation*}
\begin{array}{ll}
f_x =&  A+ xA_x + \dfrac{(2B+xB_x)(c p_z z^{c-2} +   (2C + zC_z) y^2 + (2D + zD_z) x^2 )xy}{(A+zA_z)} \\
&+ \dfrac{(2C+xC_x)(b p_y y^{b-2} +    (2B + yB_y) x^2 + (2D + yD_y)  z^2) zx}{(A+zA_z)}
\end{array}
\end{equation*}
and hence is a unit beginning with $-q$. To summarize, we have 
\begin{equation}\label{eq:xaybzc}
\begin{array}{l}
p_x x^a = \left(\frac{q}{a} + \cdots \right) xyz \\
p_y y^b = \left(\frac{q}{b} + \cdots \right) xyz \\
p_z z^c = \left(\frac{q}{c} + \cdots \right) xyz 
\end{array}
\end{equation}
in the Jacobian ideal ring localized at $W=0$. (Here, ``$\cdots$" does not contain a constant term.) Plugging these into the formula of $W$ and replacing factors $xy$, $yz$, $zx$ (in $Bxy \underline{(xy)}$, $Cyz \underline{(yz)}$, $Dzx \underline{(zx)}$, respectively) as in \eqref{eq:Jacrel2}, we get
$$ \left( q\left(\frac{1}{a} + \frac{1}{b} + \frac{1}{c} -1 \right) + \cdots \right) xyz =0$$
in \eqref{eq:locJac}. The coefficient of $xyz$ in the above equation is certainly invertible since $\frac{1}{a} + \frac{1}{b} + \frac{1}{c} \neq 0$ for hyperbolic $(a,b,c)$, and hence, $xyz \in (W, W_x, W_y, W_z)$, and from \eqref{eq:xaybzc}, $x^a, y^b, z^c\in (W, W_x, W_y, W_z)$ as claimed. Therefore, if $W= W_x = W_y = W_z =0$, then $x=y=z=0$. 
\end{proof}

\begin{corollary}\label{cor:jacfindim}
The Jacobian ideal ring ${\rm Jac} (W)|_{W=0}$ \eqref{eq:locJac} is finite dimensional over $\Lambda$.
\end{corollary}
\begin{proof}
Since $x^a, y^b, z^c, xyz$ are in $(W, W_x, W_y, W_z)$, it is easy to see that $\{x^m y^n, y^n z^l, z^l x^m\}$ with $0 \leq m \leq a-1$, $0 \leq n \leq b-1$, $0 \leq l \leq c-1$ generate this ring as a $\Lambda$-vector space.
\end{proof}


\subsection{Isolatedness for $(3,b,2)$ with $b>6$}\label{subsec:isolc6}
The remaining case is when one of $a,b,c$ becomes $2$.
In this case, there can exist terms appearing in \eqref{eq:linhigh}:. (Indeed, $xy^{b-2}$ always exists in the potential for $(3,b,2)$ as shown in the proof of Corollary \ref{cor:specialpropW2}.) Thus, the argument given in the previous subsection is no longer valid. However,  
we can take a more direct approach as follows.

\begin{definition}
A Weierstrass polynomial $P(z)$ is
$$z^k + g_{k-1} z^{k-1} + ... + g_0$$
where for every $i = 0,\ldots,k-1$,
$g_i (z_2, ..., z_n)$
is analytic and
$g_i (0, ..., 0) = 0$.
\end{definition}

\begin{theorem} [Weierstrass preparation]  For an analytic function $f$, if
$f(0, ...,0) = 0$, but $f(z, z_2, ..., z_n)$ as a power series has a term only involving $z$, we can write (locally near $(0, ..., 0)$)
$$f(z, z_2, ..., z_n) = P(z)h(z, z_2, ..., z_n)$$
with $h$ analytic and $h(0, ..., 0)$ not $0$, and $P$ a Weierstrass polynomial.  Moreover, if the highest degree of terms only involving $z$ is $k$, then the degree of $P$, considered as a polynomial in $z$, is $k$. 
\end{theorem}

We remark that we can apply this theorem to our potential $W$ since $W$ converges near the origin (with a small $q$), so that it is an analytic function defined on a small neighborhood of the origin.

\begin{prop}\label{prop:hyperisolc6}
Let $W$ be the mirror potential for $\mathbb{P}^1_{3,b,2}$ with $b > 6$. Then, the hypersurface $W^{-1} (0)$ has an isolated singularity only at the origin.
\end{prop}

\begin{proof}
By Corollary \ref{cor:specialpropW1} and (1) of Corollary \ref{cor:specialpropW2},
the potential can be written as
$$W=  p_1 x^3 + p_2 y^b +p_3 z^2 + s xyz + t xy^{b-2} + u x^{k} y + v x^2 y^2$$
for some $p_1, p_2 , p_3, s, t,u,v$ in $\Lambda [[x,y,z]]$ and$k \geq 3$ from (1) of Corollary \ref{cor:specialpropW2}. 
Moreover, $p_1, p_2, p_3,s,t$ are units in the formal power series ring and, $t$ and $u$ are series only involving $y$ and $x$, respectively. 
Then, from the relation
$$W_z= 2p_3 z + s xy =0$$
in the Jacobian ideal ring, we may eliminate the last variable $z$ and obtain a two-variable power series of the form
\begin{equation} \label{eq:W}
\tilde{W} = A x^3 + B y^c + C x^2 y^2 + D(y) xy^{b-2} + E(x) x^{k} y,
\end{equation}
where $A,B,C,D$ are units, and $D$ and $E$ are series only in respectively $y$ and $x$.
Observe that
\begin{align}\label{eq:Wx}
\tilde{W}_x =  (3A + xA_x) x^2 + (2C + x C_x) x y^2  + D y^{b-2} + (kE + xE_x) x^{k-1} y,\\
\label{eq:Wy}
\tilde{W}_y = (b B + y B_y) y^{b-1} + (2C + y C_y) x^2 y + ((b-2)D+ y D_y) x y^{b-3} + E x^k.
\end{align}

Applying the Weierstrass preparation theorem to \eqref{eq:Wx}, we see that
\begin{equation}\label{eq:WPx}
\tilde{W}_x = (x^2 - g_1(y) x - g_2(y)) \cdot U
\end{equation}
where $g_1(0) =g_2(0)=0$ and $U$ is a unit in $\Lambda[[x,y]]$. Note that $U(0) =3A(0)$ by comparing the coefficients of $x^2$ in \eqref{eq:Wx} and \eqref{eq:WPx}. (By $P(0)$ for a power series $P$, we mean the constant term of $P$.) Thus, $x^2 = g_1(y) x +g_2(y)$ in the Jacobian ideal ring. 

Now, we put $x=0$ for both \eqref{eq:Wx} and \eqref{eq:WPx} to get $g_2(y) = V_1 y^{b-2} $ where $V_1$ is a unit in $\Lambda[[x,y]]$ and $V_1(0) = D(0)$. Now, we take the $x$-derivatives of \eqref{eq:Wx} and \eqref{eq:WPx}:
\begin{eqnarray*}
\tilde{W}_{xx} &=& (2x - g_1(y) ) \cdot U +  (x^2 - g_1(y) x - y^{b-2} \cdot (unit)) \cdot U_x  \\
&=& (6A + 6 x A_x + x^2 A_{xx} )x + (2C +4x C_x + x^2 C_{xx}) y^2 \\
&& + (k(k-1)E + 2k x E_x +x^2 E_{xx}) x^{k-2} y
\end{eqnarray*}
Since $b-2 >2$, the only $y^2$-term in the first equation can appear in $g_1(y)$. Thus, we have $g_1(y) = V_2 y^2 $ for some unit $V_2$ with $V_2(0) = - \frac{2C(0)}{U(0)} = - \frac{2C(0)}{3A(0)}$ by comparing the leading term for $y$. In summary,  $x^2 =  V_2  xy^2  + V_1  y^{b-2} $, and hence,
\begin{equation}\label{eq:subytox}
y^{b-2} = \frac{1}{V_1} x^2 - \frac{V_2}{V_1} x y^2
\end{equation}
holds in the Jacobian ideal ring. Set $V_3 := \frac{1}{V_1}$ and $V_4:=-\frac{V_2}{V_1}$. Then, 
$$V_3(0) = \frac{1}{D(0)}, \qquad V_4(0) =  \frac{2C(0)}{3A(0) D(0)}.$$

By considering Equation \eqref{eq:W} and \eqref{eq:Wy} and eliminating $y^b$, we see that
$$A_1 x^3 + C_1 x^2 y^2 + D_1 xy^{b-2} + E_1 x^k y = 0,$$
where $A_1(=A)$, $C_1$ and $D_1$ are units. We compute $C_1$ and $D_1$ precisely for later purpose:
$$C_1 = C - B \dfrac{ 2C + yC_y}{bB + y B_y}$$ 
$$ D_1 = D - B \dfrac{ (b-2)D  + y D_y }{bB + y B_y}.$$
Thus, $C_1(0)=\frac{b-2}{b} C(0)$, and $D_1 (0) =  \frac{2}{b} D(0) $.

We next substitute $V_3 x^2 + V_4 xy^2$ for $y^{b-2}$ to get
\begin{equation} \label{eq:A2C2}
A_2 x^3 + C_2 x^2 y^2 + E_1 x^k y =0.
\end{equation}
Here, $A_2$ and $C_2$ are units whose constant terms are given by
\begin{equation*}
\begin{array}{l}
A_2(0) = A(0) +D_1(0) V_3(0) = A(0) + \dfrac{2}{b}.\\
C_2(0)=  C_1(0) + D_1(0) V_4 (0)=\dfrac{b-2}{b} C (0) + \dfrac{4}{3b} \dfrac{C(0)}{A(0) }.
\end{array}
\end{equation*}

By the same substitution applying to the equation $x \tilde{W}_x=0$ \eqref{eq:Wx}, we have 
\begin{equation}\label{eq:A3C3}
A_3 x^3 +C_3 x^2 y^2 + E_1 x^k y =0
\end{equation}
for some units $A_3$ and $C_3$ with
\begin{equation*}
\begin{array}{l}
A_3(0)= 3A(0) + D(0) V_3(0) =3A(0) +1,\\
C_3(0)= 2C(0) + D(0) V_4(0)=2C(0)\left( 1 + \dfrac{1}{3A(0)}\right).
\end{array}
\end{equation*}
Finally, eliminating $x^2 y^2$-terms from Equation \eqref{eq:A2C2} and \eqref{eq:A3C3},
$$A_4 x^3 + E_2 x^k y = x^3 (A_4 + E_2 x^{k-3} y) = 0$$
We claim that the coefficient $A_4$ of $x^3$ is a unit in $\Lambda [[ x,y]]$, or equivalently $A_2 C_3- A_3 C_2$ begins with a constant term, which will imply that $x^3 \cdot (unit) =0$ in the Jacobian ideal ring (i.e. $x^3 \in (\tilde{W}, \tilde{W}_x, \tilde{W}_y)$) since $E_2 x^k$ is divisible by $x^3$ (i.e. $k \geq 3$). 

Indeed,
\begin{equation}\label{eq:A2C3}
A_2 (0) C_3 (0)- A_3 (0) C_2 (0)=\frac{6-b}{b} C(0) \left( A(0) + \frac{1}{3}\right),
\end{equation}
and \eqref{eq:A2C3} cannot be zero since $b > 6$, and $A(0)$ begins with a power of $q$. 

Now, since $k \geq 3$ and $C_3$ is a unit, \eqref{eq:A3C3} tells us that $x^2 y^2 \in (\tilde{W}, \tilde{W}_x, \tilde{W}_y)$. Finally, by applying \eqref{eq:subytox} to $yz^{c-2}$-term in $y \tilde{W}_y =0$, the only term which does not involve neither $x^i $ with $i \geq 3$ nor $x^2 y^2$ is $y^b$, and hence, $y^b \in (\tilde{W}, \tilde{W}_x, \tilde{W}_y)$.

%
%
%

\end{proof}

For the case when $(a,b,c)=(a,b,2)$ with $b \geq a \geq 4$, the proof goes almost parallel except the fact that we have two unknown weights $a$,$b$ whereas $b$ is the only such a term previously.

\begin{prop}\label{prop:hyperisolab4}
Let $W$ be the mirror potential for hyperbolic $\mathbb{P}^1_{a,b,2}$ with $b \geq a \geq 4$. Then, the hypersurface $W^{-1} (0)$ has an isolated singularity only at the origin.
\end{prop}

\begin{proof}
The potential in this case is of the following form:
\begin{equation}
W= p_1 x^a + p_2 y^b + p_3 z^2 + s xyz + t xy^{k} + u x^l y +v x^2 y^2
\end{equation}
where $s$ is a unit in $\Lambda[[x,y,z]]$ because of the minimal triangle, and so are $p_1$, $p_2$, $p_3$ (Corollary \ref{cor:specialpropW1}). Also $t$ and $u$ only depends on $y$ and $x$, respectively. From Remark \ref{rem:specialpropW2}, we may assume that $l \geq a$. As before, we use the relation
$$W_z= 2p_3 z + s xy =0$$
in the Jacobian ideal ring to eliminate the variable $z$ and obtain a two-variable power series of the form
\begin{equation} \label{eq:W2}
\tilde{W} = A x^a + B y^b + C x^2 y^2 + D(y) xy^{k} + E(x) x^{l} y,
\end{equation}
where $A,B,C$ are units in $\Lambda[[x,y]]$, and $D$ and $E$ are series only in respectively $y$ and $x$. Taking the $x$- and $y$-derivatives of $\tilde{W}$, we have
\begin{align}\label{eq:Wx2}
\tilde{W}_x =  (a A + x A_x) x^{a-1} + (2C + x C_x) x y^2  + D y^{k} + (l E + xE_x) x^{l-1} y,\\
\label{eq:Wy2}
\tilde{W}_y = (b B + y B_y) y^{b-1} + (2C + y C_y) x^2 y + (k D+ y D_y) x y^{k-1} + E x^l.
\end{align}
By applying Weierstrass preparation theorem to \eqref{eq:Wx2}, we can find a series $G$ in $\Lambda [[x,y]]$ such that
\begin{equation}\label{eq:WeierWx2}
\tilde{W}_x = (x^{a-1} - h_1 (y) x^{a-1} - \cdots - h_{a-1} (y) - h(y) ) \cdot U
\end{equation}
where $h_i$'s and $h$ vanish at $y=0$ and $U$ is a unit.
Comparing \eqref{eq:Wx2} and \eqref{eq:WeierWx2}, we see that $U(0) = a A(0)$ and $h(y) = V \cdot y^k$ for some unit $V$ with $V(0) = D(0)$. 

Let us write $G xy:= h_1 (y) x^{a-1} + \cdots + h_{a-1} (y)$ for simplicity which makes sense since each $h_i$ is divisible $y=0$. Then,
\begin{equation}\label{eq:WeierWx2G}
\tilde{W}_x = (x^{a-1} -G xy - V y^k ) \cdot U,
\end{equation}
and we claim that $G$ is divisible by $y$.
Indeed, by taking the $x$-derivatives of \eqref{eq:Wx2} and \eqref{eq:WeierWx2G} indepedently,
\begin{equation}\label{eq:WeierWxx2Ga}
\tilde{W}_{xx} = ((a-1)x^{a-2} -G_x x y - G y ) \cdot U + (x^{a-1} -G xy - V y^k ) \cdot U_x
\end{equation}
and
\begin{equation}\label{eq:WeierWxx2Gb}
\begin{array}{cl}
\tilde{W}_{xx} =& ( a (a-1) A + 2a x A_x  + x^2 A_{xx} ) x^{a-2}+ (2C + 4 x C_x + x^2 C_{xx} ) y^2 \\
& + ( l(l-1)E +2l x E_x  + x^2 E_{xx}  )  x^{l-2} y
\end{array}
\end{equation}
Plugging $x=0$ in both of equations, it follows that
$ -U G \cdot y= 2C \cdot y^2 $,
and since $U$ and $C$ are units, $G = y \bar{G}$ for some unit $\bar{G}$ with $\bar{G}(0) = -\frac{2C(0)}{U(0)}$. 
Now, from \eqref{eq:WeierWx2G}, we have the following relation in the Jacobian ideal:
\begin{equation}\label{eq:1VG}
y^k = \frac{1}{V} (x^{a-1} - \bar{G} xy^2)
\end{equation}

Observe that we can eliminate ``$By^b$" term in \eqref{eq:W2} making use of the relation ``$y \times \eqref{eq:Wy2} =0$". This leads us to the following relation in the Jacobian ideal:
 $$A x^a + C_1 x^2 y^2 + D_1 xy^k + E_1 x^l y=0$$
 with units $C_1$ and $D_1$ satisfying $C_1 (0) = \frac{b-2}{b } C(0)$ and $D_1 (0) =  \frac{b- k}{b} D(0) $. We, then, substitute $\frac{1}{V} (x^{a-1} - \bar{G} xy^2)$ \eqref{eq:1VG} for $y^k$ in $D_1 xy^k$ (in the above equation) to get
 \begin{equation}\label{eq:A2C2ab4}
 A_2 x^a + C_2 x^2 y^2 + E_1 x^l y =0
 \end{equation}
where $A_2$ and $C_2$ are units with
\begin{eqnarray*}
&&A_2(0)= A(0) + \frac{ D_1 (0)}{V(0) }=  A(0) + \frac{b-k}{b} ,\\
&&C_2(0)= C_1(0) - D_1(0) \frac{ \bar{G} (0)}{V(0)} = \frac{b-2}{b} C(0) + \frac{2b-2k}{ab} \frac{ C(0)}{ A(0)}.
\end{eqnarray*}

On the other hand, we can repeat the same procedure with $y \tilde{W}_y =0$ to have
\begin{equation}\label{eq:A3C3ab4}
A_3 x^a + C_3 x^2 y^2 + E_3 x^l y =0
\end{equation}
for some units $A_3$ and $C_3$ with
\begin{eqnarray*}
&&A_3(0)= a A(0) + \frac{ D(0)}{V(0) }= a A(0) + 1 ,\\
&&C_3(0)= 2 C (0) - D(0) \frac{ \bar{G} (0)}{V(0)} = 2C(0) + \frac{2C(0)}{a A(0)}  .
\end{eqnarray*}

Lastly, combing \eqref{eq:A2C2ab4} and \eqref{eq:A3C3ab4}, we see that
$$A_4 x^a + E_2 x^l y = x_a (A_4 + E_2 x^{l-a} y) =0$$
with 
$$A_4 (0) = A_2 (0) C_3 (0) - A_3(0) C_2( 0) = -\frac{ (a-2)(b-2) -4}{b} \left( A(0) + \frac{1}{a} \right)$$
which shows that $A_4$ is a unit since $(a-2)(b-2) > 4$. As $(l-a) \geq 0$ from (2) of Corollary \ref{cor:specialpropW2}, we have $x^a \in (\tilde{W}, \tilde{W}_x, \tilde{W}_y)$. The remaining step is precisely the same as that of the proof of Proposition \ref{prop:hyperisolc6}, and we omit.
\end{proof}

Finally, we remark that Corollary \ref{cor:jacfindim} also holds in these cases by the same reason.
%
%

\appendix

\bibliographystyle{amsalpha}

\end{document}